\documentclass[10pt, a4paper, review]{article}

\usepackage{fourier} % fourier for math mode
\usepackage[oldstyle, proportional]{libertine} % libertine as main font

\DeclareMathAlphabet\mathbfcal{OMS}{cmsy}{b}{n} % Cal boldface font
\DeclareMathAlphabet{\mathcal}{OMS}{Zplm}{m}{n} % Cal from lmodern when using fourier, in order to be compatible
\SetMathAlphabet{\mathcal}{bold}{OMS}{zplm}{b}{n} % with mathbfcal 

\usepackage[a4paper, tmargin=1.0in,bmargin=1.0in, lmargin=0.7in, rmargin=0.7in]{geometry}
\usepackage{bm}
\usepackage{graphicx}
\usepackage[dvipsnames]{xcolor}
\usepackage{amsmath}
\usepackage{blindtext}
\usepackage{hyperref}
\usepackage[noabbrev,nameinlink]{cleveref}
\usepackage{amssymb}
\usepackage{amsthm}
\usepackage{amsmath}
\usepackage{placeins}
\usepackage{multirow}
\usepackage{cases}
\usepackage{float}
\usepackage{stmaryrd}
\usepackage{breakcites} % To break long citations

\newfloat{algorithm}{t}{lop} % To make algorithm floating

\newcommand{\fd}{Finite Difference~}
\newcommand{\lbm}{lattice Boltzmann~}
\newcommand{\matricial}[1]{\bm{#1}}
\newcommand{\vectorial}[1]{\bm{#1}}
\newcommand{\boldOther}[1]{\bm{#1}}
\newcommand{\discrete}[1]{\mathsf{#1}}
\newcommand{\timeShift}{\discrete{z}}
\newcommand{\identity}{\matricial{I}}
\newcommand{\schemeMatrix}{\boldOther{\discrete{E}}}
\newcommand{\schemeMatrixAsymptotic}{\boldOther{\mathcal{E}}}
\newcommand{\companionMatrixScheme}{\matricial{\discrete{Q}}}
\newcommand{\transportMoment}{\boldOther{\discrete{T}}}
\newcommand{\transportMomentAsymptotic}{\boldOther{\mathcal{T}}}
\newcommand{\asymptoticEquivalence}{\asymp}
\newcommand{\relaxationMatrix}{\matricial{S}}
\newcommand{\momentLetter}{m}
\newcommand{\discreteMoment}{\discrete{\momentLetter}}
\newcommand{\timeVariable}{t}
\newcommand{\spaceVariable}{x}
\newcommand{\spaceStep}{\Delta \spaceVariable}  
\newcommand{\timeStep}{\Delta \timeVariable}
\newcommand{\relatives}{\mathbb{Z}}
\newcommand{\naturals}{\mathbb{N}}
\newcommand{\spatialDimensionality}{d}
\newcommand{\lattice}{\spaceStep \relatives^{\spatialDimensionality}}
\newcommand{\timeLattice}{\timeStep \naturals}
\newcommand{\relaxationParameter}{s}
\newcommand{\integerInterval}[2]{\llbracket #1, #2\rrbracket}
\newcommand{\integerIntervalClosedOpen}[2]{\llbracket #1, #2\llbracket}
\newcommand{\determinant}{\textrm{det}}
\newcommand{\indiceMoments}{i}
\newcommand{\indiceDistributions}{j}
\newcommand{\velocityNumber}{q}
\newcommand{\indiceTime}{n}
\newcommand{\coefficientFDScheme}{\discrete{c}}

\newcommand{\solutionCauchy}{u}
\newcommand{\solutionCauchyInitial}{\solutionCauchy^{\circ}}
\newcommand{\pointWiseDiscretisationInitialDatum}{\discreteMoment_1^{\circ}}
\newcommand{\strong}[1]{\textbf{#1}}
\newcommand{\initialisationOperator}{\discrete{w}}
\newcommand{\initialisationOperatorAsymptotic}{\omega}
\newcommand{\initialisationOperatorCoefficients}{w}
\newcommand{\reals}{\mathbb{R}}
\newcommand{\complex}{\mathbb{C}}
\newcommand{\equilibriumCoefficient}{\epsilon}
\newcommand{\collisionMatrix}{\matricial{K}}
\newcommand{\canonicalBasisVector}{\vectorial{e}}
\newcommand{\termAtOrder}[2]{#1^{(#2)}}
\newcommand{\bigO}[1]{O(#1)}
\newcommand{\duboisOperatorEntry}{\mathcal{G}}
\newcommand{\duboisOperator}{\boldOther{\duboisOperatorEntry}}
\newcommand{\polynomialEquilibrium}{\pi}
\newcommand{\indiceFreeOne}{\ell}
\newcommand{\indiceFreeTwo}{p}
\newcommand{\indiceFreeThree}{r}
\newcommand{\indiceOrder}{h}
\newcommand{\maxOrder}{H}
\newcommand{\indiceColumn}{j}
\newcommand{\latticeVelocity}{\lambda}
\newcommand{\matrixEntries}[3]{#1_{#2 #3}}
\newcommand{\basicShiftLetter}{\discrete{x}}
\newcommand{\scheme}[2]{$\textrm{D}_{#1}\textrm{Q}_{#2}$}
\newcommand{\transportVelocity}{V}
\newcommand{\collided}{\star}
\newcommand{\distributionDensity}{f}
\newcommand{\discreteDistributionDensity}{\discrete{\distributionDensity}}
\newcommand{\momentMatrix}{\matricial{M}}
\newcommand{\diagMatrix}{\text{diag}}
\newcommand{\discreteVelocityNormalized}{c}
\newcommand{\genericFunction}{\phi}
\newcommand{\matrixSpace}[2]{\mathcal{M}_{#1}(#2)}
\newcommand{\nonTriviallyRelaxingMomentsNumber}{Q}
\newcommand{\ringSpaceOperators}{\reals[\basicShiftLetter_1, \basicShiftLetter_1^{-1}, \dots, \basicShiftLetter_{\spatialDimensionality}, \basicShiftLetter_{\spatialDimensionality}^{-1}]}
\newcommand{\ringSpaceOperatorsOneD}{\reals[\basicShiftLetter_1, \basicShiftLetter_1^{-1}]}
\newcommand{\ringTimeSpaceOperators}{\reals[\timeShift, \basicShiftLetter_1, \basicShiftLetter_1^{-1}, \dots, \basicShiftLetter_{\spatialDimensionality}, \basicShiftLetter_{\spatialDimensionality}^{-1}]}
\newcommand{\ringContinuousTimeSpaceOperators}{\reals[\partial_{\timeVariable}, \partial_{\spaceVariable_1}, \dots, \partial_{\spaceVariable_{\spatialDimensionality}}]}
\newcommand{\formalSeries}[2]{#1 \llbracket #2 \rrbracket}
\newcommand{\linearGroup}[2]{\text{GL}_{#1}(#2)}
\newcommand{\boundedDomainNumberPoints}{N_x}
\newcommand{\transpose}[1]{#1^{\mathsf{t}}}
\newcommand{\testFunction}{\phi}
\newcommand{\definitionEquality}{:=}
\newcommand{\conjugate}[1]{\overline{#1}}

\newcommand{\blockNumberGinzburg}{W}
\newcommand{\rvline}{\hspace*{-\arraycolsep}\vline\hspace*{-\arraycolsep}}
\newcommand{\indiceBlock}{\indiceFreeOne}

\newcommand{\adjugateMatrix}{\text{adj}}
\newcommand{\annhilitaingPolyGinzburgWithOrder}[1]{\Psi_{#1}}
\newcommand{\annhilitaingPolyGinzburg}{\annhilitaingPolyGinzburgWithOrder{2}}
\newcommand{\annhilitaingPolyGinzburgWithOrderFourier}[1]{\fourierTransformed{\Psi}_{#1}}
\newcommand{\annhilitaingPolyGinzburgFourier}{\annhilitaingPolyGinzburgWithOrderFourier{2}}

\newcommand{\indiceMultiIndexDiscrete}{\mathfrak{e}}
\newcommand{\indiceMultiIndexDifferential}{\mathfrak{n}}
\newcommand{\nonNegativeReals}{\reals_{+}}
\newcommand{\nonZeroNaturals}{\naturals^{*}}
\newcommand{\genericDiscreteOperator}{\discrete{d}}

\newcommand{\symmetricPart}{\discrete{S}}
\newcommand{\antisymmetricPart}{\discrete{A}}
\newcommand{\eigenvalueLetter}{g}
\newcommand{\fourierTransformed}[1]{\hat{#1}}
\newcommand{\frequency}{\xi}
\newcommand{\amplificationPolynomial}{\Phi}

\newcommand{\iniScheme}{initialisation scheme}
\newcommand{\iniSchemes}{initialisation schemes}
\newcommand{\stScheme}{starting scheme}
\newcommand{\stSchemes}{starting schemes}
\newcommand{\bulkScheme}{bulk \fd scheme}

\newcommand{\amplificationFactorStartingScheme}[1]{\fourierTransformed{\discrete{\eigenvalueLetter}}^{[#1]}}

\newtheorem{theorem}{Theorem}%  meant for continuous numbers
\newtheorem{proposition}{Proposition}% 
\newtheorem{corollary}{Corollary}% 
\theoremstyle{remark}
\newtheorem{example}{Example}%
\newtheorem{remark}{Remark}%

\theoremstyle{definition}
\newtheorem{definition}{Definition}%

\floatstyle{ruled}
\newfloat{algorithm}{thp}{lop}
\floatname{algorithm}{Algorithm}

\providecommand{\keywords}[1]{\textbf{\textit{Keywords---}} #1}
\providecommand{\amsCat}[1]{\textbf{\textit{MSC---}} #1}

\usepackage[normalem]{ulem}

\begin{document}

% \title{From \lbm to multi-step methods: a modified equation analysis for the initial time boundary layer}
\title{Initialisation from \lbm to multi-step \fd methods: modified equations and discrete observability}
% % \title[Initialisation of linear lattice Boltzmann schemes with one conserved moment]{Initialisation of linear lattice Boltzmann schemes with one conserved moment: discrete and asymptotic analyses}

\author{\strong{Thomas Bellotti} (\href{thomas.bellotti@polytechnique.edu}{thomas.bellotti@polytechnique.edu}) \\ CMAP, CNRS, \'Ecole polytechnique, Institut Polytechnique de Paris, 91120 Palaiseau\footnote{Current email: \href{bellotti@math.unistra.fr}{bellotti@math.unistra.fr}. Current affiliation: IRMA, Universit\'e de Strasbourg, 67000 Strasbourg.}}
% \address{\orgdiv{CMAP}, \orgname{CNRS, \'Ecole polytechnique, Institut Polytechnique de Paris}, \orgaddress{\street{Palaiseau}, \postcode{91120}, \country{France}}}}

% \authormark{Thomas Bellotti}

% \corresp[*]{Corresponding author: \href{thomas.bellotti@polytechnique.edu}thomas.bellotti@polytechnique.edu}

% \received{}{12}{2022}
% % \revised{Date}{0}{Year}
% % \accepted{Date}{0}{Year}

% %\editor{Associate Editor: Name}

% % \boxedtext{
% % \begin{itemize}
% % \item Key boxed text here.
% % \item Key boxed text here.
% % \item Key boxed text here.
% % \end{itemize}}

\maketitle

\begin{abstract}
Latitude on the choice of initialisation is a shared feature between one-step extended state-space and multi-step methods.
The paper focuses on \lbm schemes, which can be interpreted as examples of both previous categories of numerical schemes.
We propose a modified equation analysis of the \iniSchemes{} for \lbm methods, determined by the choice of initial data. 
These modified equations provide guidelines to devise and analyze the initialisation in terms of order of consistency with respect to the target Cauchy problem and time smoothness of the numerical solution.
In detail, the larger the number of matched terms between modified equations for initialisation and bulk methods, the smoother the obtained numerical solution.
This is particularly manifest for numerical dissipation.
Starting from the constraints to achieve time smoothness, which can quickly become prohibitive for they have to take the parasitic modes into consideration, we explain how the distinct lack of observability for certain \lbm schemes---seen as dynamical systems on a commutative ring---can yield rather simple conditions and be easily studied as far as their initialisation is concerned. 
This comes from the reduced number of \iniSchemes{} at the fully discrete level.
These theoretical results are successfully assessed on several lattice Boltzmann methods.
\end{abstract}
 \\
\amsCat{65M75, 65M06, 65M15, 93B07}

% % \showthe\textwidth

\section{Introduction}

Numerical analysis features two notable frameworks where the knowledge of the initial state for numerical schemes is incomplete: one-step extended state-space methods (\emph{e.g.} relaxation schemes, kinetic schemes, gas-kinetic schemes, \emph{etc.}) and multi-step methods.
On the one hand, \lbm schemes have historically been considered in the realm of the one-step extended state-space methods \cite{kuznik2013mesoscopic}. From this standpoint, they have previously been compared \cite{graille2014approximation, simonis2020relaxation} to approximations of systems of conservation laws taking the form of relaxation systems \emph{\`a la} Jin-Xin \cite{jin1995relaxation}, and interpreted as peculiar discretisations of these systems when collision and transport terms are split and the relaxation time tends to zero proportionally to the time step.
Both in the relaxation systems and the \lbm schemes, conserved and non-conserved quantities are present at the same time but only conserved ones appear in the original system of conservation laws at hand.
Although the initialisation of the non-conserved quantities remains free in principle, it has important repercussions on the behaviour of the solution---such as the formation of time boundary layers---both for the relaxation systems and the \lbm schemes.
On the other hand, in recent works \cite{suga2010accurate, dellacherie2014construction, fuvcik2021equivalent, bellotti2021fd, bellotti2021equivalentequations}, \lbm schemes have been thought and recast---as far as the evolution of the conserved quantities of interest is concerned---as multi-step \fd schemes.
Unsurprisingly, multi-step schemes both for Ordinary \cite{ascher1998, hairer2008, hundsdorfer2006monotonicity, hundsdorfer2003monotonicity} and Partial Differential Equations \cite{gustafsson1995time, strikwerda2004finite} need to be properly initialised by some starting procedure with desired features, for example, consistency.
When \lbm schemes are seen in their original formulation, where conserved and non-conserved moments mingle, the initialisation of the non-conserved moments can be freely devised.
Once the \lbm schemes are recast as corresponding multi-step \fd schemes \cite{bellotti2021fd} solely on the conserved moments, the choice of initialisation for the conserved and non-conserved moments determines what the \iniSchemes{} feeding the corresponding \bulkScheme{} at the beginning of the simulation are.

The previous discussion highlights that for numerical methods such as \lbm schemes, the information gap between initial conditions for the target system of conservation laws and the numerical method must be filled and thus the issue of providing decision tools to throng this hollow clearly manifests. 
Furthermore, one must be careful when comparing numerical schemes to the continuous problem they aim at approximating, because: ``Finite difference approximations have a more complicated ``physics'' than the equations they are designed to simulate. The irony is no paradox, however, for finite differences are used not because the numbers they generate have simple properties, but because those numbers are simple to compute'', see \cite[Chapter 5]{trefethen1996}.
Since the seminal paper of Warming and Hyett \cite{warming1974modified}, the method of the modified equation \cite{gustafsson1995time, strikwerda2004finite, carpentier1997derivation} has proved to be a valuable tool to describe such ``complicated physics''.
Moreover, since \lbm schemes (respectively, their corresponding \fd schemes) feature non-physical moments (respectively, parasitic modes/eigenvalues), these terms play a role in the consistency of the initialisation routines---contrarily to what happens in the bulk---creating a rich yet complex dynamics.

In the framework of \lbm schemes, previous efforts \cite{van2009smooth} (under acoustic scaling), \cite{caiazzo2005analysis,junk2015l2, huang2015initial} (under diffusive scaling) have provided the first guidelines to establish the initial conditions, relying on asymptotic expansions both on the conserved and non-conserved variables.
One aim of these studies has been to suppress initial oscillating boundary layers being part of the ``more complicated physics'' of the discrete numerical method evoked in \cite{trefethen1996}, which are however absent in the solution of the target conservation law.
Even if the techniques introduced in these works guarantee the elimination of the initial oscillating boundary phenomena, no precise quantitative analysis of their inner structure has been presented.
Moreover, since the non-conserved moments do not have an analogue in the continuous problem, these procedures are---despite the fact of providing good indications---intrinsically formal. Finally, these works have only addressed the initialisation of specific \lbm schemes, namely the \scheme{1}{2} for \cite{van2009smooth}, the \scheme{1}{2} and \scheme{1}{3} for \cite{junk2015l2} and the \scheme{2}{9} for \cite{caiazzo2005analysis, huang2015initial}.

Inspired by the open questions left by previous works in the literature, the present contribution aims at being the first general study on the initialisation of \lbm schemes.
The pivotal tool that we introduce is a modified equation analysis for the initial conditions/\stSchemes{} and provides explicit constraints for general \lbm schemes that guarantee a sufficient order of consistency of the \iniSchemes{}, thus avoid order reductions of the overall method.
The modified equations are obtained by considering that the choice of initial data shapes the \stSchemes{} on the conserved variables of interest.
Since the non-conserved moments are eliminated, the analyses we perform rely on less formal assumptions than the ones available in the literature.
Pushing this tool one order further in the discretisation parameter, we meticulously describe the internal structure of the initial oscillating boundary layers, caused by incompatible numerical features---in particular, dissipation---between initialisation and bulk schemes. Previous works \cite{van2009smooth} have just certified the existence of these oscillations in numerical simulations. 
Let us insist once again on the fact that the dissipation of the physical (or consistency) mode for the \iniSchemes{} is driven both by the physical and parasitic eigenvalues of the \bulkScheme.
Another novelty in our work is the characterisation---by seeing \lbm methods as dynamical systems on a commutative ring and exploiting the concept of observability---of a vast well-known class of \lbm schemes with a reduced number of \iniSchemes{}, irrespective of the number of non-conserved moments.
The initial motivation to introduce the concept of observability is---for this class of schemes---to successfully determine the constraints needed to eliminate initial oscillating boundary layers due to a dissipation mismatch.

We consider \lbm schemes with one conserved moment, for the sake of keeping the discussion and the notations simple and essential.
The extension to several conserved moments can be envisioned in the spirit of our previous works \cite{bellotti2021fd,bellotti2021equivalentequations}.
We particularly concentrate on the widely adopted acoustic scaling \cite{dubois2019nonlinear} between time and space steps. 
The diffusive scaling \cite{zhao2017maxwell, zhang2019lattice} is succinctly discussed with the very same techniques at the end of the work.
Moreover, we consider linear schemes \cite{van2009smooth}, hence the equilibria for the non-conserved moments are linear functions of the conserved one.
The \lbm schemes we focus on aim at approximating the solution of the following $\spatialDimensionality = 1, 2, 3$ dimensional linear Cauchy problem
\begin{numcases}{}
    \partial_{\timeVariable} \solutionCauchy (\timeVariable, \vectorial{\spaceVariable}) + \vectorial{\transportVelocity} \cdot \nabla_{\vectorial{\spaceVariable}} \solutionCauchy (\timeVariable, \vectorial{\spaceVariable}) = 0, \qquad (\timeVariable, \vectorial{\spaceVariable}) \in \nonNegativeReals \times \reals^{\spatialDimensionality},\label{eq:CauchyEquation} \\
    \solutionCauchy (0, \vectorial{\spaceVariable}) = \solutionCauchyInitial (\vectorial{\spaceVariable}), \qquad \vectorial{\spaceVariable} \in \reals^{\spatialDimensionality}, \label{eq:CauchyInitialDatum}
\end{numcases}
with velocity $\vectorial{\transportVelocity} \in \reals^{\spatialDimensionality}$ and initial datum $\solutionCauchyInitial$, which is a smooth function defined everywhere in $\reals^{\spatialDimensionality}$. 
In this work, we only consider, contrarily to \cite{van2009smooth}, explicit initialisations, to keep the presentation simple.
However, the analysis of implicit initialisations can be done with the same techniques.

The paper is structured as follows.
At the beginning, \Cref{sec:LBMSchemes} introduces the general \lbm schemes treated in our study and \Cref{sec:CorrespondingFD} briefly recalls the main points about the reformulation of \lbm schemes as \fd schemes, away from the initial time.
This reformulation has allowed \cite{bellotti2021equivalentequations} to rigorously study the consistency of \lbm schemes apart from their initialisation and now characterises the number of needed \iniSchemes{}. 
In \Cref{sec:ModifiedEquations}, we introduce the modified equation analysis of these \stSchemes{} and find the constraints under which they are consistent with the same equation as the bulk \fd scheme.
The examples and numerical simulations of \Cref{sec:Illustrations} are introduced to corroborate the theoretical findings of \Cref{sec:ModifiedEquations} and---pushing the computation of the modified equations of the \stSchemes{} one order further---we describe the internal structure of the initial oscillating boundary layers in the case of simple 1d schemes.
One particular scheme also stimulates the discussion of the following \Cref{sec:Ginzburg}, where we re-evaluate the number of initialisation schemes at the discrete level more closely, thanks to the introduction of the notion of observability for the \lbm schemes. This allows to clearly identify and study a category of schemes for which the study of the initial conditions is greatly simplified and thus the constraints to avoid initial oscillating boundary layers can be easily established.
We eventually conclude in \Cref{sec:Conclusions}.

\section{Lattice Boltzmann schemes}\label{sec:LBMSchemes}

To introduce the multiple-relaxation-times \lbm schemes under the D'Humi\`eres formalism \cite{dhumieres1992}, which can be used to handle the previous Cauchy problem \eqref{eq:CauchyEquation}, \eqref{eq:CauchyInitialDatum} and the present paper is concerned by, one considers the following building blocks.
\begin{itemize}
    \item Time and space steps $\timeStep$ and $\spaceStep$, which are linked through a scaling. In the paper, since the target problem \eqref{eq:CauchyEquation} is hyperbolic, the acoustic scaling $\timeStep = \spaceStep/\latticeVelocity$ for a given fixed lattice velocity $\latticeVelocity > 0$ is taken, unless otherwise indicated.
    \item Discrete velocities $\vectorial{\discreteVelocityNormalized}_1, \dots , \vectorial{\discreteVelocityNormalized}_{\velocityNumber} \in \relatives^{\spatialDimensionality}$, with $\velocityNumber \in \nonZeroNaturals$.
    \item An invertible moment matrix $\momentMatrix \in \linearGroup{\velocityNumber}{\reals}$.
    \item A relaxation matrix $\relaxationMatrix = \diagMatrix(\relaxationParameter_1, \relaxationParameter_2, \dots, \relaxationParameter_{\velocityNumber})$, where $\relaxationParameter_{\indiceMoments} \in ]0, 2]$ for $\indiceMoments \in \integerInterval{2}{\velocityNumber}$ and $\relaxationParameter_1 \in \reals$.
    \item The equilibrium coefficients $\boldOther{\equilibriumCoefficient} \in \reals^{\velocityNumber}$ such that $\equilibriumCoefficient_1 = 1$, meaning that the first moment is conserved.
\end{itemize}

\begin{algorithm}
    \begin{itemize}
        \item Given $\boldOther{\discreteMoment}(0, \vectorial{\spaceVariable}) \in \reals^{\velocityNumber}$ for every $\vectorial{\spaceVariable} \in \lattice$.
        \item For $\indiceTime \in \naturals$
        \begin{itemize}
            \item{\strong{Collision}}. Using the collision matrix $\collisionMatrix \definitionEquality \identity - \relaxationMatrix (\identity - \boldOther{\equilibriumCoefficient} \otimes \canonicalBasisVector_1 )$, it reads
            \begin{equation}\label{eq:LBMCollisionPhase}
                \boldOther{\discreteMoment}^{\collided} (\indiceTime\timeStep, \vectorial{\spaceVariable}) = \collisionMatrix \boldOther{\discreteMoment}(\indiceTime\timeStep, \vectorial{\spaceVariable}), \qquad \vectorial{\spaceVariable} \in \lattice.
            \end{equation}
            The post-collision distribution densities are recovered by $\boldOther{\discreteDistributionDensity}^{\collided} (\indiceTime\timeStep, \vectorial{\spaceVariable}) = \momentMatrix^{-1} \boldOther{\discreteMoment}^{\collided} (\indiceTime\timeStep, \vectorial{\spaceVariable})$ on every point $\vectorial{\spaceVariable} \in \lattice$ of the lattice.
            \item{\strong{Transport}}, which reads
            \begin{equation}\label{eq:LBMTransportPhase}
                \discreteDistributionDensity_{\indiceDistributions} ((\indiceTime + 1)\timeStep, \vectorial{\spaceVariable}) = \discreteDistributionDensity_{\indiceDistributions}^{\collided} (\indiceTime\timeStep, \vectorial{\spaceVariable} - \spaceStep\vectorial{\discreteVelocityNormalized}_{\indiceDistributions}), \qquad \vectorial{\spaceVariable} \in \lattice, \quad \indiceDistributions \in \integerInterval{1}{\velocityNumber}.
            \end{equation}
            The moments at the new time step are obtained by $\boldOther{\discreteMoment} ((\indiceTime+1)\timeStep, \vectorial{\spaceVariable}) = \momentMatrix \boldOther{\discreteDistributionDensity} ((\indiceTime+1)\timeStep, \vectorial{\spaceVariable})$ on every point $\vectorial{\spaceVariable} \in \lattice$.
        \end{itemize}
    \end{itemize}
    \caption{\label{alg:LBMScheme}Lattice Boltzmann scheme.}
\end{algorithm}

The \lbm scheme then reads as in \Cref{alg:LBMScheme}.
The way of devising this algorithm---\emph{i.e.} choosing the discrete velocities $\vectorial{\discreteVelocityNormalized}_1, \dots , \vectorial{\discreteVelocityNormalized}_{\velocityNumber}$, the moment matrix $\momentMatrix$, the collision parameters $\relaxationMatrix$, and the equilibrium coefficients $\boldOther{\equilibriumCoefficient}$---in order to obtain consistency with \eqref{eq:CauchyEquation} has been formally studied with different strategies in a myriad of papers \cite{lallemand2000theory, junk2005asymptotic, dubois2008equivalent, yong2016theory, dubois2019nonlinear, chai2020multiple} to cite a few, and rigorously in our recent contribution \cite{bellotti2021equivalentequations}, relying on an exact algebraic elimination of the non-conserved moments which are not present in \eqref{eq:CauchyEquation}.
The present work particularly focuses on the choice of $\boldOther{\discreteMoment}(0, \cdot)$.
For future use, we introduce the number $\nonTriviallyRelaxingMomentsNumber$ of non-conserved moments which do not relax to their equilibrium value during the collision phase \eqref{eq:LBMCollisionPhase}:
\begin{equation}\label{eq:DefinitionNonTrivialRelaxations}
    \nonTriviallyRelaxingMomentsNumber \definitionEquality \# \bigl \{ \relaxationParameter_{\indiceMoments} \neq 1 ~ : ~ \indiceMoments \in \integerInterval{2}{\velocityNumber} \bigr \} \in \integerIntervalClosedOpen{0}{\velocityNumber}.
\end{equation}
Roughly speaking, the larger $\nonTriviallyRelaxingMomentsNumber$, the stronger the ``entanglement'' between moments in the scheme.
Remark that, since the corresponding column in $\collisionMatrix$ is zero, there is even no need to specify the initial value $\discreteMoment_{\indiceMoments}(0, \cdot)$ when $\relaxationParameter_{\indiceMoments} = 1$, for $\indiceMoments \in \integerInterval{2}{\velocityNumber}$. This comes from the fact that the post-collisional values of these moments are entirely determined by their values at equilibrium.

\section{Corresponding Finite Difference schemes and initialisation schemes}\label{sec:CorrespondingFD}

Recasting the \lbm scheme as a multi-step \fd scheme \cite{bellotti2021fd} has allowed to rigorously define the notions of stability and consistency \cite{bellotti2021equivalentequations}, without a precise consideration on the role of the initial conditions, which is indeed the aim of the present paper.
In \Cref{sec:CorrespondingFD}, we briefly recall the essential concepts on the way of rewriting \lbm schemes as \fd ones, showing that the choice of initial data determines what the \iniSchemes{}, used to start up the multi-step \fd scheme, are.

\subsection{Finite Difference scheme in the bulk}

The transport phase \eqref{eq:LBMTransportPhase} can be rewritten on the moments $\boldOther{\discreteMoment}$ by introducing the non-diagonal matrix of operators $\transportMoment \definitionEquality \momentMatrix \diagMatrix (\boldOther{\basicShiftLetter}^{\vectorial{\discreteVelocityNormalized}_1}, \dots, \boldOther{\basicShiftLetter}^{\vectorial{\discreteVelocityNormalized}_{\velocityNumber}}) \momentMatrix^{-1}$, using the multi-index notation, where $\boldOther{\basicShiftLetter} = (\basicShiftLetter_1, \dots, \basicShiftLetter_{\spatialDimensionality})$ and thus $\boldOther{\basicShiftLetter}^{\vectorial{\discreteVelocityNormalized}} = \basicShiftLetter_1^{\discreteVelocityNormalized_1}\cdots \basicShiftLetter_{\spatialDimensionality}^{\discreteVelocityNormalized_{\spatialDimensionality}}$ for any $\vectorial{\discreteVelocityNormalized} \in \relatives^{\spatialDimensionality}$.
In this expression, we have considered the upwind space shift operators $\basicShiftLetter_{\indiceFreeOne}$ for $\indiceFreeOne \in \integerInterval{1}{\spatialDimensionality}$ such that
\begin{equation*}
    (\basicShiftLetter_{\indiceFreeOne}\genericFunction) (\vectorial{\spaceVariable}) = \genericFunction(\vectorial{\spaceVariable} - \spaceStep \canonicalBasisVector_{\indiceFreeOne}), \qquad \vectorial{\spaceVariable} \in \reals^{\spatialDimensionality},
\end{equation*}
with $\canonicalBasisVector_{\indiceFreeOne}$ the $\indiceFreeOne$-th vector of the canonical basis.
Considering also the forward time shift operator $\timeShift$ such that
\begin{equation*}
    (\timeShift \genericFunction) (\timeVariable)  = \genericFunction(\timeVariable + \timeStep), \qquad \timeVariable \in \reals,
\end{equation*}
the whole \lbm scheme on the moments reads
\begin{equation}\label{eq:LBMSchemeMonol}
    (\timeShift \boldOther{\discreteMoment})(\timeVariable, \vectorial{\spaceVariable}) = (\schemeMatrix \boldOther{\discreteMoment})(\timeVariable, \vectorial{\spaceVariable}), \qquad  (\timeVariable, \vectorial{\spaceVariable}) \in \timeLattice \times \lattice,
\end{equation}
where the evolution matrix of the scheme $\schemeMatrix \definitionEquality \transportMoment \collisionMatrix \in \matrixSpace{\velocityNumber}{\ringSpaceOperators}$ has entries in the ring of Laurent polynomials in the indeterminates $\basicShiftLetter_1, \dots, \basicShiftLetter_{\spatialDimensionality}$.
In what follows, we shall remove the parenthesis to indicate the action of operators on functions, for the sake of notation.
For any spatial operator $\genericDiscreteOperator = \genericDiscreteOperator(\boldOther{\basicShiftLetter}) \in \ringSpaceOperators$, we define its conjugate operator by $\conjugate{\genericDiscreteOperator} = \conjugate{\genericDiscreteOperator}(\boldOther{\basicShiftLetter}) \definitionEquality \genericDiscreteOperator(\boldOther{\basicShiftLetter^{-\vectorial{1}}})$, which allows to introduce symmetric and anti-symmetric parts as $\symmetricPart(\genericDiscreteOperator) \definitionEquality ({\genericDiscreteOperator + \conjugate{\genericDiscreteOperator}})/{2}$ and $\antisymmetricPart(\genericDiscreteOperator) \definitionEquality ({\genericDiscreteOperator - \conjugate{\genericDiscreteOperator}})/{2}$.
In \cite{bellotti2021fd}, we have shown that---by means of the Cayley-Hamilton theorem on the commutative ring $\ringSpaceOperators$---the discrete dynamics of the conserved moment $\discreteMoment_1$ computed by \Cref{alg:LBMScheme} or by \eqref{eq:LBMSchemeMonol}---away from the initial time---is the one of the corresponding \fd scheme under the form
\begin{equation}\label{eq:FDCorrespondingOneMomentInitialCondition}
    \timeShift^{\nonTriviallyRelaxingMomentsNumber + 1 - \velocityNumber}\determinant(\timeShift \identity - \schemeMatrix) \discreteMoment_1 (\timeVariable, \vectorial{\spaceVariable}) = \sum_{\indiceTime = 0}^{\velocityNumber} \coefficientFDScheme_{\indiceTime} \timeShift^{\indiceTime + \nonTriviallyRelaxingMomentsNumber + 1 - \velocityNumber} \discreteMoment_1  (\timeVariable, \vectorial{\spaceVariable}) = 0, \qquad (\timeVariable, \vectorial{\spaceVariable}) \in \timeLattice \times \lattice,
\end{equation}
where $(\coefficientFDScheme_{\indiceTime})_{\indiceTime = 0}^{\indiceTime = \velocityNumber} \subset \ringSpaceOperators$ are the coefficients of the characteristic polynomial $\determinant(\timeShift \identity - \schemeMatrix) = \sum\nolimits_{\indiceTime = 0}^{\indiceTime = \velocityNumber} \coefficientFDScheme_{\indiceTime} \timeShift^{\indiceTime}$ of $\schemeMatrix$, which is indeed the amplification polynomial of the scheme.
One can easily see that $\coefficientFDScheme_{\indiceTime} = 0$ for $\indiceTime \in \integerInterval{0}{\velocityNumber - \nonTriviallyRelaxingMomentsNumber - 2}$, whence the important role played by $\nonTriviallyRelaxingMomentsNumber$.
Furthermore, since the characteristic polynomial is monic, \emph{i.e.} $\coefficientFDScheme_{\velocityNumber} = 1$, the scheme is explicit, thus can be recast into
\begin{equation*}
    \timeShift \discreteMoment_1 (\timeVariable, \vectorial{\spaceVariable}) =  - \sum_{\indiceTime = \velocityNumber - \nonTriviallyRelaxingMomentsNumber - 1}^{\velocityNumber-1} \coefficientFDScheme_{\indiceTime} \timeShift^{\indiceTime + 1 - \velocityNumber} \discreteMoment_1  (\timeVariable, \vectorial{\spaceVariable}), \qquad (\timeVariable, \vectorial{\spaceVariable}) \in \timeStep \integerIntervalClosedOpen{\nonTriviallyRelaxingMomentsNumber}{+\infty}\times \lattice.
\end{equation*}
We call this scheme corresponding \bulkScheme{} acting on the bulk time steps $\integerIntervalClosedOpen{\nonTriviallyRelaxingMomentsNumber}{+\infty}$, which is a multi-step scheme with $\nonTriviallyRelaxingMomentsNumber + 2$ stages (or indeed $\nonTriviallyRelaxingMomentsNumber + 1$ steps).
We remark the need for initialisation data through $\nonTriviallyRelaxingMomentsNumber$ \iniSchemes{}, that we shall analyze in what follows.

\subsection{Initialisation schemes}

The \iniSchemes{}---the outcome of which eventually ``nourishes'' the \bulkScheme{}---are determined by the choice of initial datum $\boldOther{\discreteMoment}(0, \cdot)$.
They are:
\begin{equation*}
    \discreteMoment_1(\indiceTime \timeStep, \vectorial{\spaceVariable}) = (\schemeMatrix^{\indiceTime} \boldOther{\discreteMoment})_1 (0, \vectorial{\spaceVariable}), \qquad \indiceTime \in \integerInterval{1}{\nonTriviallyRelaxingMomentsNumber}, \quad \vectorial{\spaceVariable} \in \lattice.
\end{equation*}
The formulation that we have proposed is provided in an abstract yet general form.
In order to make the link with well-known \lbm schemes and illustrate our purpose, let us introduce the following example.
More of them are provided in \Cref{sec:Illustrations} and \Cref{sec:Ginzburg}.
\begin{example}[\scheme{1}{2}]\label{ex:D1Q2}
Consider the \scheme{1}{2} scheme as \cite{van2009smooth, dellacherie2014construction, graille2014approximation, caetano2019result}, where $\spatialDimensionality = 1$, $\velocityNumber = 2$, $\discreteVelocityNormalized_1 = 1$, $\discreteVelocityNormalized_2 = -1$, and the moment matrix is
\begin{equation*}
    \momentMatrix = 
    \begin{bmatrix}
        1 & 1 \\
        1 & -1
    \end{bmatrix}. \quad \text{Therefore} \quad
    \transportMoment = 
    \begin{bmatrix}
        \symmetricPart (\basicShiftLetter_1) & \antisymmetricPart(\basicShiftLetter_1)\\
        \antisymmetricPart(\basicShiftLetter_1 ) & \symmetricPart (\basicShiftLetter_1)
    \end{bmatrix}, \quad \text{and} \quad
    \collisionMatrix = 
    \begin{bmatrix}
        1 & 0 \\
        \relaxationParameter_2 \equilibriumCoefficient_2 & 1 - \relaxationParameter_2
    \end{bmatrix}.
\end{equation*}
The \bulkScheme{} comes from $\determinant(\timeShift \identity - \schemeMatrix) = \timeShift^2  +  ( (\relaxationParameter_2 - 2) \symmetricPart(\basicShiftLetter_1) - {\relaxationParameter_2 \equilibriumCoefficient_2}  \antisymmetricPart(\basicShiftLetter_1 )  )\timeShift + (1 - \relaxationParameter_2)$, encompassing the result from \cite{dellacherie2014construction}, and hence reads
\begin{equation}\label{eq:FDschemeD1Q2}
    \discreteMoment_1((\indiceTime + 1) \timeStep, {\spaceVariable}) = ((2-\relaxationParameter_2)\symmetricPart(\basicShiftLetter_1) + {\relaxationParameter_2 \equilibriumCoefficient_2}  \antisymmetricPart(\basicShiftLetter_1 )) \discreteMoment_1(\indiceTime  \timeStep, {\spaceVariable}) + (\relaxationParameter_2 - 1) \discreteMoment_1((\indiceTime - 1)  \timeStep, {\spaceVariable}),
\end{equation}
for $\indiceTime \in \integerIntervalClosedOpen{\nonTriviallyRelaxingMomentsNumber}{+\infty}$ and ${\spaceVariable} \in \spaceStep \relatives$.
This is a Lax-Friedrichs scheme when $\relaxationParameter_2 = 1$---which is first-order consistent with the transport equation at velocity $\latticeVelocity \equilibriumCoefficient_2$---and a leap-frog scheme when $\relaxationParameter_2 = 2$, which is second-order consistent.
Thus, to approximate the solution of \eqref{eq:CauchyEquation} by $\discreteMoment_1 \approx \solutionCauchy$, the choice of equilibrium is $\equilibriumCoefficient_2 = \transportVelocity/\latticeVelocity$.
The \bulkScheme{} \eqref{eq:FDschemeD1Q2} is multi-step with $\nonTriviallyRelaxingMomentsNumber = 1$ when $\relaxationParameter_2 \neq 1$: in this case, one needs to specify one \iniScheme{}, which is
\begin{equation*}
    \discreteMoment_1(\timeStep, \spaceVariable) =  ( \symmetricPart(\basicShiftLetter_1) + \relaxationParameter_2 \equilibriumCoefficient_2  \antisymmetricPart(\basicShiftLetter_1)  ) \discreteMoment_1 (0, \spaceVariable) + (1-\relaxationParameter_2) \antisymmetricPart(\basicShiftLetter_1) \discreteMoment_2(0, \spaceVariable), \qquad \spaceVariable \in \spaceStep \relatives.
\end{equation*}
We see that both the choice of the conserved moment $\discreteMoment_1 (0, \cdot)$ and the non-conserved moment $\discreteMoment_2 (0, \cdot)$ with respect to $\solutionCauchyInitial$ determine the initial scheme.
Unsurprisingly, this scheme coincides with the bulk scheme when $\relaxationParameter_2 = 1$.
\end{example}

\subsection{Overall scheme}

The \bulkScheme{} supplemented by the \iniSchemes{} reads as in \Cref{alg:CorrespondingFDScheme}.
\begin{algorithm}
    \begin{itemize}
        \item Given $\boldOther{\discreteMoment}(0, \vectorial{\spaceVariable})$ for every $\vectorial{\spaceVariable} \in \lattice$.
        \item{\strong{Initialisation schemes}}. For $\indiceTime \in \integerInterval{1}{\nonTriviallyRelaxingMomentsNumber}$
        \begin{equation}\label{eq:InitialisationSchemes}
            \discreteMoment_1(\indiceTime \timeStep, \vectorial{\spaceVariable}) = (\schemeMatrix^{\indiceTime} \boldOther{\discreteMoment})_1 (0, \vectorial{\spaceVariable}), \qquad \vectorial{\spaceVariable} \in\lattice.
        \end{equation}
        \item{\strong{Corresponding bulk \fd scheme}}. For $\indiceTime \in \integerIntervalClosedOpen{\nonTriviallyRelaxingMomentsNumber}{+\infty}$
        \begin{equation}\label{eq:BulkSchemes}
            \discreteMoment_1((\indiceTime + 1) \timeStep, \vectorial{\spaceVariable}) = -  \sum_{\indiceFreeOne = \velocityNumber - \nonTriviallyRelaxingMomentsNumber - 1}^{\velocityNumber-1} \coefficientFDScheme_{\indiceFreeOne} \discreteMoment_1 ((\indiceTime + \indiceFreeOne + 1 - \velocityNumber)\timeStep, \vectorial{\spaceVariable}), \qquad \vectorial{\spaceVariable} \in \lattice.
        \end{equation}
    \end{itemize}
    \caption{\label{alg:CorrespondingFDScheme}Corresponding \fd scheme.}
\end{algorithm}
\begin{figure} 
    \begin{center}
        \begin{footnotesize}
            \includegraphics[width=1.\textwidth]{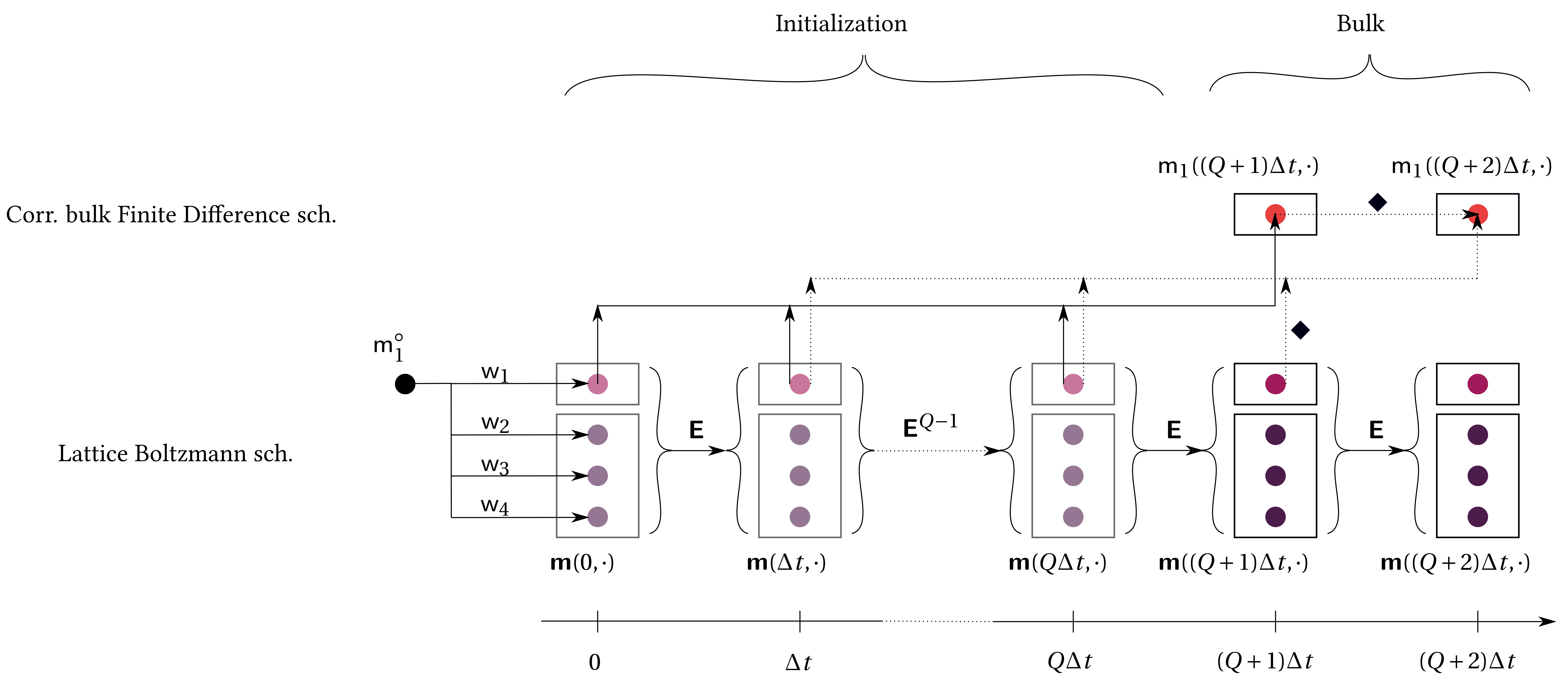}
        \end{footnotesize}
    \end{center}\caption{\label{fig:schemes} Illustration of the way of working of the \lbm scheme (bottom) and the \bulkScheme{} (top). The former acts both on the conserved (light violet) and the non-conserved (dark violet) moments. The latter implies only the conserved moment, drawn in light violet in the initialisation layer and in red in the bulk. Remark that to compute the conserved moment for the \bulkScheme{} at time $(\nonTriviallyRelaxingMomentsNumber + 2)\timeStep$, one can either rely on the information at time $(\nonTriviallyRelaxingMomentsNumber + 1)\timeStep$ in light violet (from the \lbm scheme) or on the one in red (from the \fd scheme), as highlighted by the symbol \protect\rotatebox{45}{{\tiny $\blacksquare$}}. This holds because these quantities are equal for any time step in the bulk, for they stem from a common initialisation process. Partial transparency is used to denote the initialisation steps.}
  \end{figure}
We stress that \Cref{alg:CorrespondingFDScheme} is the corresponding scheme of \Cref{alg:LBMScheme} in the sense that they issue the same discrete dynamics of the conserved moment $\discreteMoment_1$ approximating $\solutionCauchy$, see \Cref{fig:schemes}.
Of course, the non-conserved moments $\discreteMoment_2, \dots, \discreteMoment_{\velocityNumber}$ have been eliminated, at the price of handling a multi-step \fd scheme.
They still remain in the initialisation (\emph{cf.} \Cref{ex:D1Q2}), giving a first intuition of why we claimed that non-physical modes---associated with non-conserved moments---play a role in this topic.
\begin{remark}
    It is worthwhile observing that even if the \iniSchemes{} \eqref{eq:InitialisationSchemes} are considered here close to the initial time, \emph{i.e.} for $\indiceTime \in \integerInterval{1}{\nonTriviallyRelaxingMomentsNumber}$, they also represent the action of the \lbm scheme through its evolution operator $\schemeMatrix$ away from the initial time, that is, when $\indiceTime > \nonTriviallyRelaxingMomentsNumber$.
    In the sequel, we shall employ the following nomenclature:
    \begin{itemize}
        \item ``\strong{\iniSchemes{}}'', to indicate \eqref{eq:InitialisationSchemes} for $\indiceTime \in \integerInterval{1}{\nonTriviallyRelaxingMomentsNumber}$;
        \item ``\strong{\stSchemes{}}'', to indicate \eqref{eq:InitialisationSchemes} for any $\indiceTime \in \nonZeroNaturals$. 
    \end{itemize}
    Hence, the \iniSchemes{} are a proper subset of the \stSchemes{}.
    Indeed, in \Cref{sec:ModifiedEquations} and \Cref{sec:Illustrations}, we shall also consider the behaviour of \eqref{eq:InitialisationSchemes} for $\indiceTime > \nonTriviallyRelaxingMomentsNumber$, aiming at analysing the agreement between the behaviour of the numerical schemes inside the initial layer and the one purely in the bulk.
    This idea of matching is reminiscent of the singularly perturbed dynamical systems, see \cite{o1991singular, bender1999advanced}.
\end{remark}

\section{Modified equation analysis of the initial conditions under acoustic scaling}\label{sec:ModifiedEquations}

The study of the consistency of the initialisation schemes is crucial---especially when one wants to reach high-order accuracy.
For the overall method, \cite[Theorem 10.6.2]{strikwerda2004finite} states that, under acoustic scaling, if the initialisation of a stable multi-step scheme is obtained using schemes of accuracy $\maxOrder - 1$ in $\spaceStep$, where $\maxOrder$ is the accuracy of the multi-step scheme without accounting for the initialisation, then for smooth initial data, the order of accuracy of the multi-step scheme accounting for the initialisation remains $\maxOrder$.

In what follows, we shall make use of the notion of asymptotic equivalence \cite{bellotti2021equivalentequations} between discrete operators in the ring $\ringTimeSpaceOperators$ and formal series of continuous differential operators.
\begin{definition}[Asymptotic equivalence]\label{def:AsymptoticEquivalence}
    Considering a scaling between $\timeStep$ and $\spaceStep$, so that we take $\spaceStep$ as driving discretisation parameter, for any discrete time-space operator $\genericDiscreteOperator \in \ringTimeSpaceOperators$, we indicate $\genericDiscreteOperator \asymptoticEquivalence \delta$ where 
    $\delta \in \formalSeries{(\ringContinuousTimeSpaceOperators)}{\spaceStep}$ is a formal series in $\spaceStep$ with coefficients in the ring of time-space differential operators $\ringContinuousTimeSpaceOperators$, if for every smooth $\genericFunction: \reals \times \reals^{\spatialDimensionality} \to \reals$
    \begin{equation*}
        \genericDiscreteOperator \genericFunction (\timeVariable, \vectorial{\spaceVariable}) = \sum_{\indiceOrder = 0}^{+\infty} \spaceStep^{\indiceOrder} \termAtOrder{\delta}{\indiceOrder} \genericFunction  (\timeVariable, \vectorial{\spaceVariable}), \qquad (\timeVariable, \vectorial{\spaceVariable}) \in \reals \times \reals^{\spatialDimensionality},
    \end{equation*}
    in the limit $\spaceStep \to 0$.
\end{definition}
To perform the consistency analysis of the schemes \emph{via} the modified equation \cite{warming1974modified,strikwerda2004finite,gustafsson1995time}, one practical way of proceeding is to deploy the scheme on smooth functions over $\reals \times \reals^{\spatialDimensionality}$ instead of on grid functions defined over $\timeLattice \times \lattice$, and use truncated asymptotic equivalents according to \Cref{def:AsymptoticEquivalence}.
The scaling assumptions \cite{bellotti2021equivalentequations} the whole work will rely on are---unless further notice---that $\momentMatrix$, $\relaxationMatrix$, and $\boldOther{\equilibriumCoefficient}$ are independent of $\spaceStep$ as $\spaceStep \to 0$.
Following \cite{dubois2019nonlinear}, we have introduced \cite{bellotti2021equivalentequations} the matrix of first-order space differential operators
\begin{equation*}
    \duboisOperator = \momentMatrix \sum_{|\boldOther{\indiceMultiIndexDifferential}| = 1} \diagMatrix (\vectorial{\discreteVelocityNormalized}_1^{\boldOther{\indiceMultiIndexDifferential}}, \dots, \vectorial{\discreteVelocityNormalized}_{\velocityNumber}^{\boldOther{\indiceMultiIndexDifferential}}) \partial_{\vectorial{\spaceVariable}}^{\boldOther{\indiceMultiIndexDifferential}}  \momentMatrix^{-1} \in \matrixSpace{\velocityNumber}{\ringContinuousTimeSpaceOperators},
\end{equation*}
influenced both by the choice of discrete velocities and the moment matrix at hand.
The entries of this matrix shall be used to write the modified equations for general \lbm schemes.
\begin{example}
    Coming back to the context of \Cref{ex:D1Q2}, we have that
    \begin{equation*}
        \duboisOperator = 
        \begin{bmatrix}
            0 & \partial_{\spaceVariable_1} \\
            \partial_{\spaceVariable_1} & 0
        \end{bmatrix}.
    \end{equation*}
\end{example}

\subsection{Review on the modified equation in the bulk}

The consistency of the \bulkScheme{} \eqref{eq:BulkSchemes} is described in the following result.
\begin{theorem}[\cite{bellotti2021equivalentequations} Modified equation of the bulk scheme]\label{thm:EquivEqBulk}
    Under acoustic scaling, that is, when $\latticeVelocity > 0$ is fixed as $\spaceStep \to 0$, the modified equation for the \bulkScheme{} \eqref{eq:BulkSchemes} is given by
    % \begin{align}
    %     \partial_{\timeVariable}\testFunction (\timeVariable, \vectorial{\spaceVariable}) &+ \latticeVelocity \Bigl ( \matrixEntries{\duboisOperatorEntry}{1}{1} + \sum_{\indiceFreeThree = 2}^{\velocityNumber}  \matrixEntries{\duboisOperatorEntry}{1}{\indiceFreeThree} \equilibriumCoefficient_{\indiceFreeThree} \Bigr ) \testFunction (\timeVariable, \vectorial{\spaceVariable}) \label{eq:BulkEquivalentEquation} \\
    %     &- \latticeVelocity\spaceStep \sum_{\indiceMoments = 2}^{\velocityNumber} \Bigl ( \frac{1}{\relaxationParameter_{\indiceMoments}} - \frac{1}{2}\Bigr ) \duboisOperatorEntry_{1\indiceMoments} \Bigl ( \duboisOperatorEntry_{\indiceMoments 1}  +  \sum_{\indiceFreeThree = 2}^{\velocityNumber} \duboisOperatorEntry_{\indiceMoments \indiceFreeThree} \equilibriumCoefficient_{\indiceFreeThree} - \Bigl ( \matrixEntries{\duboisOperatorEntry}{1}{1} + \sum_{\indiceFreeThree = 2}^{\velocityNumber}  \matrixEntries{\duboisOperatorEntry}{1}{\indiceFreeThree} \equilibriumCoefficient_{\indiceFreeThree} \Bigr )  \equilibriumCoefficient_{\indiceMoments} \Bigr )\testFunction (\timeVariable, \vectorial{\spaceVariable})  = \bigO{\spaceStep^2}, \nonumber
    % \end{align}
    \begin{equation}\label{eq:BulkEquivalentEquation}
        \partial_{\timeVariable}\testFunction (\timeVariable, \vectorial{\spaceVariable}) + \latticeVelocity \Bigl ( \matrixEntries{\duboisOperatorEntry}{1}{1} + \sum_{\indiceFreeThree = 2}^{\velocityNumber}  \matrixEntries{\duboisOperatorEntry}{1}{\indiceFreeThree} \equilibriumCoefficient_{\indiceFreeThree} \Bigr ) \testFunction (\timeVariable, \vectorial{\spaceVariable}) 
        - \latticeVelocity\spaceStep \sum_{\indiceMoments = 2}^{\velocityNumber} \Bigl ( \frac{1}{\relaxationParameter_{\indiceMoments}} - \frac{1}{2}\Bigr ) \duboisOperatorEntry_{1\indiceMoments} \Bigl ( \duboisOperatorEntry_{\indiceMoments 1}  +  \sum_{\indiceFreeThree = 2}^{\velocityNumber} \duboisOperatorEntry_{\indiceMoments \indiceFreeThree} \equilibriumCoefficient_{\indiceFreeThree} - \Bigl ( \matrixEntries{\duboisOperatorEntry}{1}{1} + \sum_{\indiceFreeThree = 2}^{\velocityNumber}  \matrixEntries{\duboisOperatorEntry}{1}{\indiceFreeThree} \equilibriumCoefficient_{\indiceFreeThree} \Bigr )  \equilibriumCoefficient_{\indiceMoments} \Bigr )\testFunction (\timeVariable, \vectorial{\spaceVariable})  = \bigO{\spaceStep^2},
    \end{equation}
    for $(\timeVariable, \vectorial{\spaceVariable}) \in \nonNegativeReals \times \reals^{\spatialDimensionality}$.
\end{theorem}

Comparing \eqref{eq:BulkEquivalentEquation} and \eqref{eq:CauchyEquation}, the consistency with the equation of the Cauchy problem shall be enforced selecting the components of the \lbm scheme such that $\latticeVelocity (\matrixEntries{\duboisOperatorEntry}{1}{1} + \sum\nolimits_{\indiceFreeThree = 2}^{\indiceFreeThree = \velocityNumber}  \matrixEntries{\duboisOperatorEntry}{1}{\indiceFreeThree} \equilibriumCoefficient_{\indiceFreeThree}) = \vectorial{V} \cdot \nabla_{\vectorial{\spaceVariable}}$.
Since we shall employ the expression ``at order $\bigO{\spaceStep^{\indiceOrder}}$'' in the following discussion, let us specify what we mean, by taking advantage of the claim from \Cref{thm:EquivEqBulk}.
The terms $\partial_{\timeVariable}$ and $\latticeVelocity (\matrixEntries{\duboisOperatorEntry}{1}{1} + \sum\nolimits_{\indiceFreeThree = 2}^{\indiceFreeThree = \velocityNumber}  \matrixEntries{\duboisOperatorEntry}{1}{\indiceFreeThree} \equilibriumCoefficient_{\indiceFreeThree})$ appear at order $\bigO{\spaceStep}$ when the actual proof of \Cref{thm:EquivEqBulk} is done, thus we call them ``$\bigO{\spaceStep}$ terms''. Then, these terms appear at leading order in \eqref{eq:BulkEquivalentEquation} because all the $\bigO{1}$ terms simplify on both sides of the equation.
The remaining term $\bigO{\spaceStep}$ in \eqref{eq:BulkEquivalentEquation} originally shows at order $\bigO{\spaceStep^2}$ and is made up of numerical diffusion.

\subsection{Linking the discrete initial datum with the one of the Cauchy problem}

We now adapt the same techniques to concentrate on the role of the initial data.
From the initial datum of the Cauchy problem $\solutionCauchyInitial$, we consider its point-wise discretisation with a lattice function $\pointWiseDiscretisationInitialDatum$ such that $ \pointWiseDiscretisationInitialDatum(\vectorial{\spaceVariable}) = \solutionCauchyInitial (\vectorial{\spaceVariable})$ for $\vectorial{\spaceVariable} \in \lattice$.
Coherently with the fact of considering a linear problem and because the equilibria of the non-conserved moments are linear functions of the conserved one through $\vectorial{\equilibriumCoefficient}$, a linear initialisation reads
\begin{equation}\label{eq:LinearInitialisation}
    \boldOther{\discreteMoment}(0, \vectorial{\spaceVariable}) = \boldOther{\initialisationOperator} \pointWiseDiscretisationInitialDatum  (\vectorial{\spaceVariable}), \qquad \vectorial{\spaceVariable} \in \lattice,
\end{equation}
where $\boldOther{\initialisationOperator}$ can be chosen in two different fashions. 
\begin{itemize}
    \item If $\boldOther{\initialisationOperator} \in \reals^{\velocityNumber}$ is considered, we obtain what we call ``\strong{local initialisation}''. However, in order to gain more freedom on the initialisation and achieve desired numerical properties, another choice is possible.
    \item If $\boldOther{\initialisationOperator} \in (\ringSpaceOperators)^{\velocityNumber}$ is considered, we obtain the ``\strong{prepared initialisation}'', where we allow for an initial rearrangement of the information issued from the initial datum of the Cauchy problem between neighboring sites of the lattice.
\end{itemize}
It can be observed that the local initialisation is only a particular case of prepared initialisation using constant polynomials, since $\reals$ is a sub-ring of $\ringSpaceOperators$.
By allowing $\initialisationOperator_1 \in \ringSpaceOperators$, we also permit to perform a preliminary modification of the point-wise discretisation of the initial datum \eqref{eq:CauchyInitialDatum} of the Cauchy problem, which can also be interpreted as an initial filtering of the datum, before assigning it to $\discreteMoment_1$.
For example, when $\spatialDimensionality = 1$, considering $\initialisationOperator_1 = \symmetricPart(\basicShiftLetter_1)$ yields $\discreteMoment_1(0, \spaceVariable) = (\solutionCauchyInitial(\spaceVariable - \spaceStep) + \solutionCauchyInitial(\spaceVariable + \spaceStep) ) / 2$ for every $\spaceVariable \in \spaceStep \relatives$.
Observe that the following developments can be easily adapted to deal with implicit initialisations \cite{van2009smooth} of the form ${\initialisationOperator}_{\indiceMoments}  {\discreteMoment}_{\indiceMoments}(0, \vectorial{\spaceVariable}) = \discrete{b}_{\indiceMoments} \pointWiseDiscretisationInitialDatum  (\vectorial{\spaceVariable})$ with $\discrete{b}_{\indiceMoments} \in \ringSpaceOperators$ for $\indiceMoments \in \integerInterval{1}{\velocityNumber}$.

\subsection{Modified equations for the initialisation schemes: local initialisation}

 Let us now compute the modified equations for the \stSchemes{} when a local initialisation is considered.
  In the general framework, we shall stop at order $\bigO{\spaceStep}$ for two reasons.
  The first one is that we are not aware of any stable \lbm scheme which---under acoustic scaling---would be third-order consistent in the bulk with the target equation \eqref{eq:CauchyEquation} and therefore would call for second-order accurate \iniSchemes{}.
  Second, the expressions for higher order terms are excessively involved to be written down in a convenient form as functions of $\indiceTime \in \integerInterval{1}{\nonTriviallyRelaxingMomentsNumber}$ for general schemes.
  Again, this is due to the role played by the non-physical eigenvalues of $\schemeMatrix$.
  Still, one more order in the expansion shall be needed to analyse the smooth initialisation proposed by \cite{van2009smooth,junk2015l2}, as we shall do in \Cref{sec:Illustrations} for some particularly simple yet instructive examples and for a more general class of schemes in \Cref{sec:Ginzburg}.
  
  \begin{proposition}[Modified equation of the \stSchemes{} with local initialisation]\label{prop:EquivalentEquationInitialisationLocal}
    Under acoustic scaling, that is, when $\latticeVelocity > 0$ is fixed as $\spaceStep \to 0$, considering a local initialisation, \emph{i.e.} $\boldOther{\initialisationOperator} \in \reals^{\velocityNumber}$, the modified equations for the \stSchemes{} are, for any $ \indiceTime \in \nonZeroNaturals$
      \begin{align}
        \testFunction (0, \vectorial{\spaceVariable}) &+ \indiceTime  \frac{\spaceStep}{\latticeVelocity} \partial_{\timeVariable}\testFunction (0, \vectorial{\spaceVariable}) + \bigO{\spaceStep^2} \label{eq:ExpandedInitialisation} \\
        = \initialisationOperator_1 \testFunction (0,\vectorial{\spaceVariable})  &-  \indiceTime \spaceStep \Bigl ( \duboisOperatorEntry_{11}\initialisationOperator_1 + \sum_{\indiceFreeThree = 2}^{\velocityNumber}  \duboisOperatorEntry_{1\indiceFreeThree} \initialisationOperator_{\indiceFreeThree} + \frac{1}{\indiceTime} \sum_{\indiceFreeThree = 2}^{\velocityNumber}  \duboisOperatorEntry_{1\indiceFreeThree} (\equilibriumCoefficient_{\indiceFreeThree} \initialisationOperator_1 - \initialisationOperator_{\indiceFreeThree}) \sum_{\indiceFreeOne = 0}^{\indiceTime - 1} \polynomialEquilibrium_{\indiceTime - \indiceFreeOne}(\relaxationParameter_{\indiceFreeThree}) \Bigr ) \testFunction (0, \vectorial{\spaceVariable}) + \bigO{\spaceStep^2}, \qquad \vectorial{\spaceVariable} \in \reals^{\spatialDimensionality}, \nonumber
    \end{align}
    where $\polynomialEquilibrium_{\indiceFreeOne}(X) = 1 - (1-X)^{\indiceFreeOne}$ for $\indiceFreeOne \in \naturals$.
  \end{proposition}
  \begin{proof}
      We start by describing the particular structure of the powers of collision matrix $\collisionMatrix$.
      It is straightforward to see that we obtain an upper-triangular matrix with
      \begin{equation}\label{eq:StructurePowersCollisionMatrix}
          \collisionMatrix^{\indiceFreeOne} = 
          \begin{bmatrix}
              1 & 0 & 0 & \cdots & \cdots & 0 \\
              \polynomialEquilibrium_{\indiceFreeOne}(\relaxationParameter_2) \equilibriumCoefficient_2 & (1 - \relaxationParameter_2)^{\indiceFreeOne} & 0 & & & \vdots \\
              \polynomialEquilibrium_{\indiceFreeOne}(\relaxationParameter_3) \equilibriumCoefficient_3 & 0 & (1 - \relaxationParameter_3)^{\indiceFreeOne} & \ddots & & \vdots \\
              \polynomialEquilibrium_{\indiceFreeOne}(\relaxationParameter_4) \equilibriumCoefficient_4 & 0  & 0 & \ddots & \ddots & \vdots \\
              \vdots & \vdots & \vdots & \ddots & \ddots & 0 \\
              \polynomialEquilibrium_{\indiceFreeOne}(\relaxationParameter_{\velocityNumber}) \equilibriumCoefficient_{\velocityNumber} & 0 & 0 & \cdots & 0 & (1 - \relaxationParameter_{\velocityNumber})^{\indiceFreeOne}
          \end{bmatrix}, \qquad \indiceFreeOne \in \nonZeroNaturals, 
      \end{equation}
      where the polynomials $\polynomialEquilibrium_{\indiceFreeOne}$ are defined recursively as $\polynomialEquilibrium_0 (X) \definitionEquality 0$ and $\polynomialEquilibrium_{\indiceFreeOne + 1} (X) \definitionEquality X + (1-X) \polynomialEquilibrium_{\indiceFreeOne} (X)$ for $\indiceFreeOne \in \naturals$. Therefore $\polynomialEquilibrium_{\indiceFreeOne}(X) = 1 - (1-X)^{\indiceFreeOne}$ for $\indiceFreeOne \in \naturals$.
      The \stSchemes{} read
  \begin{equation}\label{eq:InitialSchemes}
      \timeShift^{\indiceTime} \discreteMoment_1 (0, \vectorial{\spaceVariable}) = (\schemeMatrix^{\indiceTime} \boldOther{\initialisationOperator} )_1  \pointWiseDiscretisationInitialDatum(\vectorial{\spaceVariable}), \qquad \indiceTime \in \nonZeroNaturals, \quad \vectorial{\spaceVariable} \in \lattice.
  \end{equation}
  Concerning the time shifts on the left hand side of \eqref{eq:InitialSchemes}, we have $\timeShift^{\indiceTime} \asymptoticEquivalence \text{exp}({\indiceTime \tfrac{\spaceStep}{\latticeVelocity} \partial_{\timeVariable}}) = 1 + \indiceTime \tfrac{\spaceStep}{\latticeVelocity} \partial_{\timeVariable} + \bigO{\spaceStep^2}$ for $ \indiceTime \in \naturals$.
  For the right hand side of \eqref{eq:InitialSchemes}, we have that $\schemeMatrix \asymptoticEquivalence \schemeMatrixAsymptotic = \transportMomentAsymptotic \collisionMatrix$ where $\transportMoment \asymptoticEquivalence \transportMomentAsymptotic = \text{exp}({-{\spaceStep} \duboisOperator} )= \identity - \spaceStep \duboisOperator + \bigO{\spaceStep^2}$, see \cite{bellotti2021equivalentequations}, and for $\indiceTime \in \nonZeroNaturals$
  \begin{align}
      \schemeMatrixAsymptotic^{\indiceTime} &=  (\termAtOrder{\schemeMatrixAsymptotic}{0} + \spaceStep \termAtOrder{\schemeMatrixAsymptotic}{1}  + \bigO{\spaceStep^2})^{\indiceTime}\nonumber \\
      &= (\termAtOrder{\schemeMatrixAsymptotic}{0} )^{\indiceTime} + \spaceStep \sum \{ \text{permutations of }\termAtOrder{\schemeMatrixAsymptotic}{0} \text{ (}\indiceTime-1\text{ times) and }\termAtOrder{\schemeMatrixAsymptotic}{1} \text{ (once)}\} + \bigO{\spaceStep^2} \nonumber \\
      &= (\termAtOrder{\schemeMatrixAsymptotic}{0} )^{\indiceTime}  + \spaceStep \sum\nolimits_{\indiceFreeOne = 0}^{\indiceFreeOne = \indiceTime - 1} (\termAtOrder{\schemeMatrixAsymptotic}{0})^{\indiceFreeOne} \termAtOrder{\schemeMatrixAsymptotic}{1} (\termAtOrder{\schemeMatrixAsymptotic}{0})^{\indiceTime - 1 - \indiceFreeOne} + \bigO{\spaceStep^2} = \collisionMatrix^{\indiceTime}  - \spaceStep \sum\nolimits_{\indiceFreeOne = 0}^{\indiceFreeOne = \indiceTime - 1} \collisionMatrix^{\indiceFreeOne} \duboisOperator \collisionMatrix^{\indiceTime - \indiceFreeOne} + \bigO{\spaceStep^2}, \label{eq:schemeMatrixPowerDevelopment}
  \end{align}
  where we use the fact that $\termAtOrder{\schemeMatrixAsymptotic}{\indiceOrder} = \termAtOrder{\transportMomentAsymptotic}{\indiceOrder} \collisionMatrix$ for $\indiceOrder \in \naturals$.
  Plugging into \eqref{eq:InitialSchemes}, employing a smooth function $\testFunction$ instead of $\discreteMoment_{1}$ and $\pointWiseDiscretisationInitialDatum$ and using the fact that the initialisation is local, we have for $\indiceTime \in \nonZeroNaturals$
  \begin{equation*}
      \testFunction (0, \vectorial{\spaceVariable}) + \indiceTime  \frac{\spaceStep}{\latticeVelocity} \partial_{\timeVariable}\testFunction (0, \vectorial{\spaceVariable}) + \bigO{\spaceStep^2} = (\collisionMatrix^{\indiceTime} \boldOther{\initialisationOperator} )_1 \testFunction (0,\vectorial{\spaceVariable})  - \spaceStep \Bigl ( \sum_{\indiceFreeOne = 0}^{\indiceTime - 1} \collisionMatrix^{\indiceFreeOne} \duboisOperator \collisionMatrix^{\indiceTime - \indiceFreeOne}  \boldOther{\initialisationOperator} \Bigr )_1   \testFunction (0, \vectorial{\spaceVariable}) + \bigO{\spaceStep^2}, \qquad  \vectorial{\spaceVariable} \in \reals^{\spatialDimensionality}.
  \end{equation*}
  We have that $(\collisionMatrix^{\indiceTime} \boldOther{\initialisationOperator})_1 \testFunction (0,\vectorial{\spaceVariable}) = \initialisationOperator_1 \testFunction (0,\vectorial{\spaceVariable})$ thanks to \eqref{eq:StructurePowersCollisionMatrix} and for $\indiceColumn \in \integerInterval{1}{\velocityNumber}$
  \begin{align*}
    \matrixEntries{(\collisionMatrix^{\indiceFreeOne} \duboisOperator \collisionMatrix^{\indiceTime - \indiceFreeOne})}{1}{\indiceColumn} = \sum_{\indiceFreeTwo = 1}^{\velocityNumber} \sum_{\indiceFreeThree = 1}^{\velocityNumber} \matrixEntries{(\collisionMatrix^{\indiceFreeOne})}{1}{\indiceFreeTwo} \matrixEntries{\duboisOperatorEntry}{\indiceFreeTwo}{\indiceFreeThree} \matrixEntries{(\collisionMatrix^{\indiceTime - \indiceFreeOne})}{\indiceFreeThree}{\indiceColumn} &=  \sum_{\indiceFreeThree = 1}^{\velocityNumber}  \matrixEntries{\duboisOperatorEntry}{1}{\indiceFreeThree} \matrixEntries{(\collisionMatrix^{\indiceTime - \indiceFreeOne})}{\indiceFreeThree}{\indiceColumn} 
    \\
    &= \matrixEntries{\duboisOperatorEntry}{1}{1} \delta_{1 \indiceColumn} +  \sum_{\indiceFreeThree = 2}^{\velocityNumber}  \matrixEntries{\duboisOperatorEntry}{1}{\indiceFreeThree} (\polynomialEquilibrium_{\indiceTime - \indiceFreeOne}(\relaxationParameter_{\indiceFreeThree}) \equilibriumCoefficient_{\indiceFreeThree} \delta_{1\indiceColumn} + (1-\relaxationParameter_{\indiceFreeThree})^{\indiceTime - \indiceFreeOne} \delta_{\indiceFreeThree \indiceColumn}).
  \end{align*}
  Therefore for $\indiceTime \in \nonZeroNaturals$
  \begin{align*}
    \sum_{\indiceFreeOne = 0}^{\indiceTime - 1} (\collisionMatrix^{\indiceFreeOne} \duboisOperator \collisionMatrix^{\indiceTime - \indiceFreeOne} \boldOther{\initialisationOperator})_1 &= \indiceTime \matrixEntries{\duboisOperatorEntry}{1}{1} \initialisationOperator_1 + \sum_{\indiceFreeThree = 2}^{\velocityNumber}  \matrixEntries{\duboisOperatorEntry}{1}{\indiceFreeThree}  \sum_{\indiceFreeOne = 0}^{\indiceTime - 1} (\polynomialEquilibrium_{\indiceTime - \indiceFreeOne}(\relaxationParameter_{\indiceFreeThree}) \equilibriumCoefficient_{\indiceFreeThree} + (1-\relaxationParameter_{\indiceFreeThree})^{\indiceTime - \indiceFreeOne} \initialisationOperator_{\indiceFreeThree}) \\
    &= \indiceTime \Bigl ( \duboisOperatorEntry_{11}\initialisationOperator_1 + \sum_{\indiceFreeThree = 2}^{\velocityNumber}  \duboisOperatorEntry_{1\indiceFreeThree} \initialisationOperator_{\indiceFreeThree} + \frac{1}{\indiceTime} \sum_{\indiceFreeThree = 2}^{\velocityNumber}  \duboisOperatorEntry_{1\indiceFreeThree} (\equilibriumCoefficient_{\indiceFreeThree} \initialisationOperator_1 - \initialisationOperator_{\indiceFreeThree}) \sum_{\indiceFreeOne = 0}^{\indiceTime - 1} \polynomialEquilibrium_{\indiceTime - \indiceFreeOne}(\relaxationParameter_{\indiceFreeThree}) \Bigr ),
\end{align*}
where we have used that by the definition of $\polynomialEquilibrium_{\indiceFreeOne}$, $(\indiceTime - \sum_{\indiceFreeOne = 0}^{\indiceFreeOne = \indiceTime - 1} \polynomialEquilibrium_{\indiceTime - \indiceFreeOne}(\relaxationParameter_{\indiceFreeThree}))/\sum_{\indiceFreeOne = 0}^{\indiceFreeOne = \indiceTime - 1} (1-\relaxationParameter_{\indiceFreeThree})^{\indiceTime - \indiceFreeOne} = 1$ for every $\indiceTime \in \nonZeroNaturals$, yielding the claim.
  \end{proof}

  \begin{figure} 
    \begin{center}
        \includegraphics[width=0.8\textwidth]{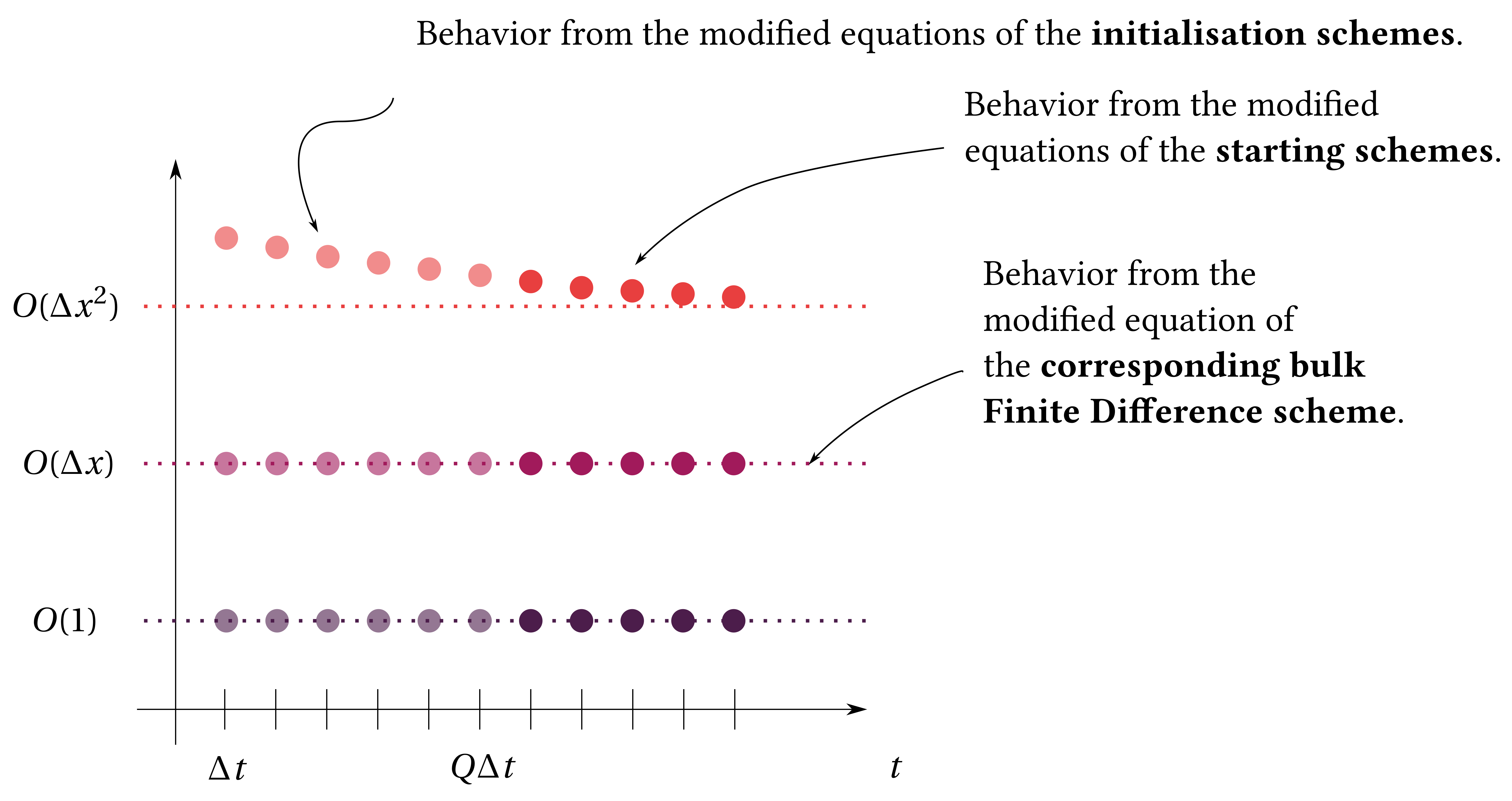}
    \end{center}\caption{\label{fig:matching}Example of behaviour of the inner expansion (dots, concerning the \stSchemes{}) and the outer expansion (dashed lines, relative to the \bulkScheme{}) at different orders in $\spaceStep$ for $\spaceStep \to 0$.}
  \end{figure}

  With \Cref{prop:EquivalentEquationInitialisationLocal}, we can now compare the modified equation for the \bulkScheme{} and the modified equations of the \stSchemes{}, so respectively the dashed lines and the dots in \Cref{fig:matching} at order $\bigO{1}$ and $\bigO{\spaceStep}$.

\subsection{Consistency of the initialisation schemes: local initialisation}

The agreement between the terms at these two orders takes place under the following conditions.
\begin{corollary}[Consistency of the \stSchemes{} with local initialisation]\label{prop:LocalInitialisation}
    Under acoustic scaling, that is, when $\latticeVelocity > 0$ is fixed as $\spaceStep \to 0$, considering a local initialisation, \emph{i.e.} $\vectorial{\initialisationOperator} \in \reals^{\velocityNumber}$, under the conditions
    \begin{align}
        &\initialisationOperator_1 = 1, \label{eq:LocalInitial} \\
        \text{for} \quad \indiceFreeThree \in \integerInterval{2}{\velocityNumber}, \quad \text{if} \quad \matrixEntries{\duboisOperatorEntry}{1}{\indiceFreeThree} \neq 0, \quad \text{then} \quad &\initialisationOperator_{\indiceFreeThree} = \equilibriumCoefficient_{\indiceFreeThree}, \label{eq:LocalBulk}
    \end{align}
    where $\boldOther{\equilibriumCoefficient}$ are the equilibrium coefficients, the \stSchemes{} are consistent with the modified equation \eqref{eq:BulkEquivalentEquation} of the \bulkScheme{} at order $\bigO{\spaceStep}$. Moreover, the initial datum feeding the \bulkScheme{} and the \stSchemes{} is consistent with the initial datum \eqref{eq:CauchyInitialDatum} of the Cauchy problem.
\end{corollary}
The first condition \eqref{eq:LocalInitial} implies that the initial datum for $\discreteMoment_1$, used both by the \stSchemes{} and the \bulkScheme{}, is left untouched compared to the one of the Cauchy problem.
The second condition \eqref{eq:LocalBulk} is expected: for the non-conserved moments involved in the modified equation \eqref{eq:BulkEquivalentEquation} at leading order, we need to consider the initial datum at equilibrium.  
It is also to observe that this requirement does not \emph{a priori} fix all the initialisation parameters, contrarily to \Cref{ex:D1Q2}, because some of them can affect only higher orders in the developments, \emph{i.e.} $\matrixEntries{\duboisOperatorEntry}{1}{\indiceFreeThree} = 0$ for some $\indiceFreeThree \in \integerInterval{1}{\velocityNumber}$.

\begin{proof}[Proof of \Cref{prop:LocalInitialisation}]
The proof proceeds order-by-order in $\spaceStep$.
\begin{itemize}
\item $\bigO{1}$.
This order indicates that the initial datum for the conserved moment has to be consistent with the one of the Cauchy problem \eqref{eq:CauchyInitialDatum}.
From \Cref{prop:EquivalentEquationInitialisationLocal}, it reads
\begin{equation}\label{eq:ExpansionLocalOrderZero}
    \testFunction (0, \vectorial{\spaceVariable}) = \initialisationOperator_1  \testFunction (0, \vectorial{\spaceVariable})  + \bigO{\spaceStep}, \qquad \indiceTime \in \nonZeroNaturals, \quad \vectorial{\spaceVariable} \in \reals^{\spatialDimensionality},
\end{equation}
hence we enforce $\initialisationOperator_1 = 1$.
Remark that \eqref{eq:ExpansionLocalOrderZero} is satisfied both for $\indiceTime \in \integerInterval{1}{\nonTriviallyRelaxingMomentsNumber}$ and for $\indiceTime > \nonTriviallyRelaxingMomentsNumber$, that is, both for \iniSchemes{} and \stSchemes{}.
This condition being fulfilled, the next order to check is
\begin{equation}\label{eq:EquivalentEquationFirstOrderInitialTime}
     \partial_{\timeVariable}\testFunction (0, \vectorial{\spaceVariable}) + \latticeVelocity \Bigl ( \duboisOperatorEntry_{11} + \sum_{\indiceFreeThree = 2}^{\velocityNumber}  \duboisOperatorEntry_{1\indiceFreeThree} \initialisationOperator_{\indiceFreeThree} + \frac{1}{\indiceTime} \sum_{\indiceFreeThree = 2}^{\velocityNumber}  \duboisOperatorEntry_{1\indiceFreeThree} (\equilibriumCoefficient_{\indiceFreeThree}  - \initialisationOperator_{\indiceFreeThree}) \sum_{\indiceFreeOne = 0}^{\indiceTime - 1} \polynomialEquilibrium_{\indiceTime - \indiceFreeOne}(\relaxationParameter_{\indiceFreeThree}) \Bigr ) = \bigO{\spaceStep}, \qquad \indiceTime \in \nonZeroNaturals, \quad \vectorial{\spaceVariable} \in \reals^{\spatialDimensionality}. 
\end{equation}

\item $\bigO{\spaceStep}$.
Evaluating the bulk modified equation \eqref{eq:BulkEquivalentEquation} at time $\timeVariable = 0$ gives
\begin{equation}\label{eq:BulkEquivalentEquationDominantOrderInitial}
    \partial_{\timeVariable}\testFunction (0, \vectorial{\spaceVariable}) + \latticeVelocity \Bigl ( \matrixEntries{\duboisOperatorEntry}{1}{1} + \sum_{\indiceFreeThree = 2}^{\velocityNumber}  \matrixEntries{\duboisOperatorEntry}{1}{\indiceFreeThree} \equilibriumCoefficient_{\indiceFreeThree} \Bigr ) \testFunction (0, \vectorial{\spaceVariable}) = \bigO{\spaceStep}, \qquad  \vectorial{\spaceVariable} \in \reals^{\spatialDimensionality},
\end{equation}
and trying to match each term with \eqref{eq:EquivalentEquationFirstOrderInitialTime} yields the condition
\begin{equation*}
    \text{for} \quad \indiceFreeThree \in \integerInterval{2}{\velocityNumber}, \quad \text{if} \quad \matrixEntries{\duboisOperatorEntry}{1}{\indiceFreeThree} \neq 0, \quad \text{then} \quad \initialisationOperator_{\indiceFreeThree} = \equilibriumCoefficient_{\indiceFreeThree}.
\end{equation*}
\end{itemize}
\end{proof}

\subsection{Modified equations for the initialisation schemes: prepared initialisation}

Now that the principles concerning the computation of modified equations for the \stSchemes{} and the way of matching terms with the modified equation of the \bulkScheme{} are clarified, we can tackle the case of prepared initialisations.
\begin{proposition}[Modified equation of the \stSchemes{} with prepared initialisation]\label{prop:ModifiedEquationPrepared}
    Under acoustic scaling, that is, when $\latticeVelocity > 0$ is fixed as $\spaceStep \to 0$, considering a prepared initialisation, \emph{i.e.} $\boldOther{\initialisationOperator} \in (\ringSpaceOperators)^{\velocityNumber}$, which can be put under the form
    \begin{equation}\label{eq:FormInitilisationNonLocal}
        \initialisationOperator_{\indiceMoments} = \sum_{\boldOther{\indiceMultiIndexDiscrete}} \initialisationOperatorCoefficients_{\indiceMoments, \boldOther{\indiceMultiIndexDiscrete}} \boldOther{\basicShiftLetter}^{\boldOther{\indiceMultiIndexDiscrete}}, \qquad \indiceMoments \in \integerInterval{1}{\velocityNumber},
    \end{equation}
    where the sequences of coefficients $(\initialisationOperatorCoefficients_{\indiceMoments, \boldOther{\indiceMultiIndexDiscrete}})_{\boldOther{\indiceMultiIndexDiscrete}} \subset \reals$ are compactly supported, the modified equations for the \stSchemes{} are, for any $\indiceTime \in \nonZeroNaturals$ and $ \vectorial{\spaceVariable} \in \reals^{\spatialDimensionality}$ 
    \begin{align*}
        \testFunction (0, \vectorial{\spaceVariable}) &+ \indiceTime  \frac{\spaceStep}{\latticeVelocity} \partial_{\timeVariable}\testFunction (0, \vectorial{\spaceVariable}) + \bigO{\spaceStep^2} = \termAtOrder{\initialisationOperatorAsymptotic}{0}_1 \testFunction (0,\vectorial{\spaceVariable}) \\
        &-  \indiceTime \spaceStep \Bigl ( \duboisOperatorEntry_{11}  \termAtOrder{\initialisationOperatorAsymptotic}{0}_1 + \sum_{\indiceFreeThree = 2}^{\velocityNumber}  \duboisOperatorEntry_{1\indiceFreeThree}  \termAtOrder{\initialisationOperatorAsymptotic}{0}_{\indiceFreeThree} + \frac{1}{\indiceTime} \sum_{\indiceFreeThree = 2}^{\velocityNumber}  \duboisOperatorEntry_{1\indiceFreeThree} (\equilibriumCoefficient_{\indiceFreeThree}  \termAtOrder{\initialisationOperatorAsymptotic}{0}_1 -  \termAtOrder{\initialisationOperatorAsymptotic}{0}_{\indiceFreeThree}) \sum_{\indiceFreeOne = 0}^{\indiceTime - 1} \polynomialEquilibrium_{\indiceTime - \indiceFreeOne}(\relaxationParameter_{\indiceFreeThree}) - \frac{1}{\indiceTime}  \termAtOrder{\initialisationOperatorAsymptotic}{1}_1 \Bigr ) \testFunction (0, \vectorial{\spaceVariable}) + \bigO{\spaceStep^2},
    \end{align*}
    where 
    \begin{equation*}
        \termAtOrder{\initialisationOperatorAsymptotic}{0}_{\indiceMoments} = \sum_{\boldOther{\indiceMultiIndexDiscrete}} \initialisationOperatorCoefficients_{\indiceMoments, \boldOther{\indiceMultiIndexDiscrete}}, \qquad 
        \termAtOrder{\initialisationOperatorAsymptotic}{1}_{\indiceMoments} = 
        -\sum_{|\boldOther{\indiceMultiIndexDifferential}| = 1} \Bigl ( \sum_{\boldOther{\indiceMultiIndexDiscrete}} \initialisationOperatorCoefficients_{\indiceMoments, \boldOther{\indiceMultiIndexDiscrete}}  \boldOther{\indiceMultiIndexDiscrete}^{\boldOther{\indiceMultiIndexDifferential}} \Bigl ) \partial_{\vectorial{\spaceVariable}}^{\boldOther{\indiceMultiIndexDifferential}}, \qquad \indiceMoments \in \integerInterval{1}{\velocityNumber},
    \end{equation*}
    and such that $\initialisationOperator_{\indiceMoments} \asymptoticEquivalence \termAtOrder{\initialisationOperatorAsymptotic}{0}_{\indiceMoments}  + \spaceStep \termAtOrder{\initialisationOperatorAsymptotic}{1}_{\indiceMoments} + \bigO{\spaceStep}$ and $\polynomialEquilibrium_{\indiceFreeOne}(X) = 1 - (1-X)^{\indiceFreeOne}$ for $\indiceFreeOne \in \naturals$. 
\end{proposition}
\begin{proof}
    The asymptotic equivalent of the initialisation $\boldOther{\initialisationOperator}$ reads
\begin{equation*}
    \initialisationOperator_{\indiceMoments} \asymptoticEquivalence \initialisationOperatorAsymptotic_{\indiceMoments} =  \sum_{\boldOther{\indiceMultiIndexDiscrete}} \initialisationOperatorCoefficients_{\indiceMoments, \boldOther{\indiceMultiIndexDiscrete}} - \spaceStep \sum_{|\boldOther{\indiceMultiIndexDifferential}| = 1} \Bigl ( \sum_{\boldOther{\indiceMultiIndexDiscrete}} \initialisationOperatorCoefficients_{\indiceMoments, \boldOther{\indiceMultiIndexDiscrete}}  \boldOther{\indiceMultiIndexDiscrete}^{\boldOther{\indiceMultiIndexDifferential}} \Bigl ) \partial_{\vectorial{\spaceVariable}}^{\boldOther{\indiceMultiIndexDifferential}} + \bigO{\spaceStep^2}, \qquad \indiceMoments \in \integerInterval{1}{\velocityNumber}.
\end{equation*}
Using the Cauchy product between formal series, we have $\schemeMatrix^{\indiceTime} \boldOther{\initialisationOperator} \asymptoticEquivalence \schemeMatrixAsymptotic^{\indiceTime} \boldOther{\initialisationOperatorAsymptotic} = \termAtOrder{(\schemeMatrixAsymptotic^{\indiceTime})}{0} \termAtOrder{\boldOther{\initialisationOperatorAsymptotic}}{0} + \spaceStep (\termAtOrder{(\schemeMatrixAsymptotic^{\indiceTime})}{1} \termAtOrder{\boldOther{\initialisationOperatorAsymptotic}}{0} + \termAtOrder{(\schemeMatrixAsymptotic^{\indiceTime})}{0} \termAtOrder{\boldOther{\initialisationOperatorAsymptotic}}{1}) + \bigO{\spaceStep^2}$ for $\indiceTime \in \nonZeroNaturals$.
The $\bigO{\spaceStep}$ term in the previous expansion is made up of two contributions.
The first one is $\termAtOrder{(\schemeMatrixAsymptotic^{\indiceTime})}{1} \termAtOrder{\boldOther{\initialisationOperatorAsymptotic}}{0}$ and is not influenced by the ``prepared'' character of the initialisation, because it was also present for the local initialisation.
The second one is inherent to the prepared initialisation.
The result comes from the very same computations as \Cref{prop:EquivalentEquationInitialisationLocal}.
\end{proof}

\subsection{Consistency of the initialisation schemes: prepared initialisation}

\begin{corollary}[Consistency of the \stSchemes{} with prepared initialisation]\label{prop:PreparedInitialisation}
    Under acoustic scaling, that is, when $\latticeVelocity > 0$ is fixed as $\spaceStep \to 0$, considering a prepared initialisation, \emph{i.e.} $\boldOther{\initialisationOperator} \in (\ringSpaceOperators)^{\velocityNumber}$, with \eqref{eq:FormInitilisationNonLocal}, under the conditions
    \begin{align}
        &\sum_{\boldOther{\indiceMultiIndexDiscrete}} \initialisationOperatorCoefficients_{1, \boldOther{\indiceMultiIndexDiscrete}} = 1, \label{eq:PreparedInitial} \\
        \text{for every}\quad |\boldOther{\indiceMultiIndexDifferential}| = 1, \quad &\sum_{\boldOther{\indiceMultiIndexDiscrete}} \initialisationOperatorCoefficients_{1, \boldOther{\indiceMultiIndexDiscrete}}  \boldOther{\indiceMultiIndexDiscrete}^{\boldOther{\indiceMultiIndexDifferential}} = 0, \label{eq:PreparedInitial2} \\
        \text{for} \quad \indiceFreeThree \in \integerInterval{2}{\velocityNumber}, \quad \text{if} \quad \matrixEntries{\duboisOperatorEntry}{1}{\indiceFreeThree} \neq 0, \quad \text{then} \quad &\sum_{\boldOther{\indiceMultiIndexDiscrete}} \initialisationOperatorCoefficients_{\indiceFreeThree, \boldOther{\indiceMultiIndexDiscrete}} = \equilibriumCoefficient_{\indiceFreeThree}, \label{eq:PreparedBulk}
    \end{align}
    the \stSchemes{} are consistent with the modified equation \eqref{eq:BulkEquivalentEquation} of the \bulkScheme{} at order $\bigO{\spaceStep}$. Moreover, the initial datum feeding the \bulkScheme{} and the \stSchemes{} is consistent with the initial datum \eqref{eq:CauchyInitialDatum} of the Cauchy problem up to order $\bigO{\spaceStep^2}$.
\end{corollary}

Condition \eqref{eq:PreparedInitial} is the analogue of \eqref{eq:LocalInitial}. However, since the initialisation of the conserved moment can also be prepared, an additional condition \eqref{eq:PreparedInitial2} has to be taken into account.
This guarantees, in particular, that the initial datum of the Cauchy problem used for $\discreteMoment_1$ is not perturbed by some drift term at order $\bigO{\spaceStep}$.
This is useful because of the multi-step nature of the \bulkScheme{} \eqref{eq:BulkSchemes}, which shall also be fed with \eqref{eq:LinearInitialisation}.
Finally, \eqref{eq:PreparedBulk} has to be compared with \eqref{eq:LocalBulk}.
This condition maintains that the non-conserved moments participating to the consistency at leading order have to be chosen---at leading order---at equilibrium.

\begin{proof}[Proof of \Cref{prop:PreparedInitialisation}]
Proceeding order-by-order in $\spaceStep$, we obtain:
\begin{itemize}
\item $\bigO{1}$.
The dominant order in the analogous of \eqref{eq:ExpandedInitialisation}.
Hence the consistency with the datum of the Cauchy problem reads $\termAtOrder{\initialisationOperatorAsymptotic_1}{0} = \sum\nolimits_{\boldOther{\indiceMultiIndexDiscrete}} \initialisationOperatorCoefficients_{1, \boldOther{\indiceMultiIndexDiscrete}} = 1$.
\item $\bigO{\spaceStep}$.
We see that now, there is the additional term associated with $\termAtOrder{\initialisationOperatorAsymptotic_1}{1}$ corresponding to a drift term in the initialisation of the conserved moment.
In general, we now have wider possibilities in terms of how to initialise, still remaining consistent with the modified equation of the \bulkScheme{} at the desired order, at least for the \iniSchemes{} (\emph{i.e.} $\indiceTime \in \integerInterval{1}{\nonTriviallyRelaxingMomentsNumber}$).
Indeed, it is sufficient to enforce that 
\begin{equation*}
    \duboisOperatorEntry_{11}   + \sum_{\indiceFreeThree = 2}^{\velocityNumber}  \duboisOperatorEntry_{1\indiceFreeThree}  \termAtOrder{\initialisationOperatorAsymptotic}{0}_{\indiceFreeThree} + \frac{1}{\indiceTime} \sum_{\indiceFreeThree = 2}^{\velocityNumber}  \duboisOperatorEntry_{1\indiceFreeThree} (\equilibriumCoefficient_{\indiceFreeThree}  -  \termAtOrder{\initialisationOperatorAsymptotic}{0}_{\indiceFreeThree}) \sum_{\indiceFreeOne = 0}^{\indiceTime - 1} \polynomialEquilibrium_{\indiceTime - \indiceFreeOne}(\relaxationParameter_{\indiceFreeThree}) - \frac{1}{\indiceTime}  \termAtOrder{\initialisationOperatorAsymptotic}{1}_1   = \matrixEntries{\duboisOperatorEntry}{1}{1} + \sum_{\indiceFreeThree = 2}^{\velocityNumber}  \matrixEntries{\duboisOperatorEntry}{1}{\indiceFreeThree} \equilibriumCoefficient_{\indiceFreeThree}.
\end{equation*}
Occasionally, for some $\indiceTime \in \integerInterval{1}{\nonTriviallyRelaxingMomentsNumber}$, the previous equality can be satisfied even if $\termAtOrder{\initialisationOperatorAsymptotic_1}{1} \neq 0$, see examples in \Cref{sec:Illustrations}.
However, we are interested in enforcing it for every $\indiceTime \in \nonZeroNaturals$---that is---for all \stSchemes{}. This comes, as previously claimed, from the multi-step nature of the \bulkScheme{}: we have to ensure that the order of consistency with the initial datum \eqref{eq:CauchyInitialDatum} of the Cauchy problem is high enough not to lower the overall order of the method.
Hence, suppressing the drift term for the conserved moment, thus enforcing $\termAtOrder{\initialisationOperatorAsymptotic_1}{1} = 0$, we have  
\begin{align*}
    &\text{for every}\quad |\boldOther{\indiceMultiIndexDifferential}| = 1, \quad \sum_{\boldOther{\indiceMultiIndexDiscrete}} \initialisationOperatorCoefficients_{1, \boldOther{\indiceMultiIndexDiscrete}}  \boldOther{\indiceMultiIndexDiscrete}^{\boldOther{\indiceMultiIndexDifferential}} = 0, \quad \text{and} \quad \text{for} \quad \indiceFreeThree \in \integerInterval{2}{\velocityNumber}, \quad \text{if} \quad \matrixEntries{\duboisOperatorEntry}{1}{\indiceFreeThree} \neq 0, \quad \text{then} \quad \sum_{\boldOther{\indiceMultiIndexDiscrete}} \initialisationOperatorCoefficients_{\indiceFreeThree, \boldOther{\indiceMultiIndexDiscrete}} = \equilibriumCoefficient_{\indiceFreeThree}.
\end{align*}
\end{itemize}
\end{proof}

\subsection{Initialisation schemes \emph{versus} starting schemes}

Before proceeding to some numerical illustrations, we point out important facts concerning the match of terms between the \bulkScheme{} and the \iniSchemes{}/\stSchemes{}.

\begin{proposition}[Control on the \iniSchemes{} leads control on the \stSchemes{}]\label{prop:MatchEventually}
    Let $\maxOrder \in \nonZeroNaturals$. Assume that
    \begin{itemize}
        \item $\termAtOrder{\initialisationOperatorAsymptotic_1}{0} = 1$ and $\termAtOrder{\initialisationOperatorAsymptotic_1}{\indiceOrder} = 0$ for $\indiceOrder \in \integerInterval{1}{\maxOrder}$.
        \item The modified equations of the \iniSchemes{} (\eqref{eq:InitialisationSchemes} for $\indiceTime \in \integerInterval{1}{\nonTriviallyRelaxingMomentsNumber}$) match the one of the \bulkScheme{} \eqref{eq:BulkSchemes} at any order $\indiceOrder \in \integerInterval{1}{\maxOrder}$.
    \end{itemize}
    Then, the modified equations of the \stSchemes{} (\eqref{eq:InitialisationSchemes} for $\indiceTime > \nonTriviallyRelaxingMomentsNumber$) match the one of the \bulkScheme{} \eqref{eq:BulkSchemes} at any order $\indiceOrder \in \integerInterval{1}{\maxOrder}$.
\end{proposition}

\Cref{prop:MatchEventually} does not provide indications on how to equate the order $\indiceOrder \in \integerInterval{1}{\maxOrder}$---\emph{i.e.} how to fulfill its assumptions---contrarily to what \Cref{prop:LocalInitialisation} and \Cref{prop:PreparedInitialisation} do for $\maxOrder = 1$.
Again, this is due to the fact that the general expression of the asymptotic expansion of $(\schemeMatrix^{\indiceTime} \boldOther{\initialisationOperator})_1$ can quickly become messy as the considered order increases.
Still, \Cref{prop:MatchEventually} claims that if one is able to match the modified equation of the \iniSchemes{} with the one of the \bulkScheme{} until a given order $\maxOrder$ (as we shall do for specific schemes in \Cref{sec:Illustrations} with $\maxOrder = 2$), then this guarantees the same property on the \stSchemes{}.
Otherwise said---referring to \Cref{fig:matching}---if one is able to ensure that the terms represented by the dots lie on the corresponding dashed line for $\indiceTime \in \integerInterval{1}{\nonTriviallyRelaxingMomentsNumber}$, then one will be sure that these dots will lie on the very same line for any $\indiceTime \in \nonZeroNaturals$.
This result seems intuitively reasonable by virtue of the Cayley-Hamilton theorem, which allows to recast any power $\schemeMatrix^{\indiceTime}$ for $\indiceTime \geq \nonTriviallyRelaxingMomentsNumber+1$ as a combination of $\identity, \schemeMatrix, \dots, \schemeMatrix^{\nonTriviallyRelaxingMomentsNumber}$.

\begin{proof}[Proof of \Cref{prop:MatchEventually}]

Let us consider $\spatialDimensionality = 1$ for the sake of notation: for $\spatialDimensionality > 1$, the multi-index notation would suffice.
Consider a one-step linear \fd scheme on the lattice function $\discrete{\solutionCauchy}$, under the form $\timeShift \discrete{\solutionCauchy}(\timeVariable, {\spaceVariable}) = \discrete{\eigenvalueLetter}_1 \discrete{\solutionCauchy}(\timeVariable, {\spaceVariable})$ for $(\timeVariable, {\spaceVariable}) \in \timeLattice \times \spaceStep \relatives$, where $\discrete{\eigenvalueLetter}_1 \in \reals [\basicShiftLetter_1, \basicShiftLetter_1^{-1}]$.
This can be rewritten using the Fourier transform in space, that is 
\begin{equation}\label{eq:tmp1}
    \timeShift \discrete{\fourierTransformed{\solutionCauchy}}(\timeVariable, {\frequency}\spaceStep) = \fourierTransformed{\discrete{\eigenvalueLetter}}_1({\frequency} \spaceStep) \fourierTransformed{\discrete{\solutionCauchy}}(\timeVariable, {\frequency}\spaceStep), \qquad (\timeVariable, {\frequency}) \in \timeLattice \times [-\pi/\spaceStep, \pi/\spaceStep]. 
\end{equation}
The frequency-dependent eigenvalue $\fourierTransformed{\discrete{\eigenvalueLetter}}_1({\frequency} \spaceStep) \in \complex$ shall be a Laurent polynomial in the indeterminate $e^{i \frequency \spaceStep}$ and encodes both the stability features of the method, for every ${\frequency} \in [-\pi/\spaceStep, \pi/\spaceStep]$ and the consistency features, in the low-frequency limit $|{\frequency} \spaceStep | \ll 1$.
In particular, to be consistent with an equation of the form \eqref{eq:CauchyEquation} with a first-order derivative in time, one can easily see that
\begin{equation}\label{eq:physicalEigenvalueCharact}
    \fourierTransformed{\discrete{\eigenvalueLetter}}_1({\frequency} \spaceStep) = 1 + \bigO{|\frequency \spaceStep|} \qquad \text{in the limit} \quad |{\frequency} \spaceStep | \ll 1.
\end{equation}
Applying the scheme \eqref{eq:tmp1} $\indiceTime \in \nonZeroNaturals$ times provides a sort of multi-step scheme which we shall compare to the \stSchemes{} \eqref{eq:InitialisationSchemes}
\begin{equation}\label{eq:tmp2}
    \timeShift^{\indiceTime} \discrete{\fourierTransformed{\solutionCauchy}}(\timeVariable, {\frequency} \spaceStep) = \fourierTransformed{\discrete{\eigenvalueLetter}}_1({\frequency} \spaceStep)^{\indiceTime} \fourierTransformed{\discrete{\solutionCauchy}}(\timeVariable, {\frequency} \spaceStep), \qquad (\timeVariable, {\frequency} ) \in \timeLattice \times [-\pi/\spaceStep, \pi/\spaceStep],
\end{equation}
with associated amplification polynomial
\begin{equation}\label{eq:tmp3}
    \fourierTransformed{\tilde{\amplificationPolynomial}}(\timeShift, {\frequency}\spaceStep) = \timeShift^{\indiceTime} - \fourierTransformed{\discrete{\eigenvalueLetter}}_1({\frequency} \spaceStep)^{\indiceTime} = (\timeShift - \fourierTransformed{\discrete{\eigenvalueLetter}}_1({\frequency} \spaceStep)) \sum_{\indiceFreeOne = 0}^{\indiceTime - 1} \fourierTransformed{\discrete{\eigenvalueLetter}}_1({\frequency} \spaceStep)^{\indiceFreeOne} \timeShift^{\indiceTime - 1 - \indiceFreeOne} = (\timeShift - \fourierTransformed{\discrete{\eigenvalueLetter}}_1({\frequency} \spaceStep)) \prod_{\indiceFreeOne = 2}^{\indiceTime} (\timeShift - \fourierTransformed{\tilde{\discrete{\eigenvalueLetter}}}_{\indiceFreeOne}(\frequency \spaceStep)), 
\end{equation}
having roots $\fourierTransformed{\tilde{\discrete{\eigenvalueLetter}}}_1 = \fourierTransformed{\discrete{\eigenvalueLetter}}_1$ and $\fourierTransformed{\tilde{\discrete{\eigenvalueLetter}}}_2, \dots, \fourierTransformed{\tilde{\discrete{\eigenvalueLetter}}}_{\indiceTime}$ (recall that $\complex$ is an algebraically closed field).
By differentiating the amplification polynomial \eqref{eq:tmp3} using the rule for the derivative of a product, we get 
\begin{equation*}
    \frac{\text{d}\fourierTransformed{\tilde{\amplificationPolynomial}}(\timeShift, {\frequency}\spaceStep)}{\text{d}\timeShift} = \indiceTime \timeShift^{\indiceTime - 1} = \prod_{\indiceFreeOne = 2}^{\indiceTime} (\timeShift - \fourierTransformed{\tilde{\discrete{\eigenvalueLetter}}}_{\indiceFreeOne}(\frequency \spaceStep)) +  (\timeShift - \fourierTransformed{\discrete{\eigenvalueLetter}}_1({\frequency} \spaceStep)) + \sum_{\indiceFreeTwo = 2}^{\indiceTime}\prod_{\substack{\indiceFreeOne = 2\\ \indiceFreeOne \neq \indiceFreeTwo}}^{\indiceTime} (\timeShift - \fourierTransformed{\tilde{\discrete{\eigenvalueLetter}}}_{\indiceFreeOne}(\frequency \spaceStep)).
\end{equation*}
Taking $\timeShift = 1$ in the limit $|{\frequency} \spaceStep | \ll 1$ gives $0 \neq \indiceTime = \prod\nolimits_{\indiceFreeOne = 2}^{\indiceFreeOne = \indiceTime} (1 - \termAtOrder{\fourierTransformed{\tilde{{\eigenvalueLetter}}}_{\indiceFreeOne}}{0})$, where $\fourierTransformed{\tilde{\discrete{\eigenvalueLetter}}}_{\indiceFreeOne}(\frequency\spaceStep) = \termAtOrder{\fourierTransformed{\tilde{{\eigenvalueLetter}}}_{\indiceFreeOne}}{0} + \bigO{|\frequency \spaceStep|}$, thanks to \eqref{eq:physicalEigenvalueCharact}, thus all the other eigenvalues $\fourierTransformed{\tilde{\discrete{\eigenvalueLetter}}}_2, \dots, \fourierTransformed{\tilde{\discrete{\eigenvalueLetter}}}_{\indiceTime}$ are not equal to one for small frequencies and thus are not linked with consistency. The only which matters is $\fourierTransformed{\tilde{\discrete{\eigenvalueLetter}}}_1 = \fourierTransformed{\discrete{\eigenvalueLetter}}_1$, thus the scheme \eqref{eq:tmp2} with amplification polynomial \eqref{eq:tmp3} has the same modified equations as \eqref{eq:tmp1}.
An alternative way of seeing this is to use the approach from the proof of \cite[Proposition 1]{carpentier1997derivation}, which aims at automatically handling the ``reinjection'' of the previous orders in the expansions to eliminate time derivatives above first order.
Inserting the asymptotic equivalent $\text{exp}({\indiceTime \tfrac{\spaceStep}{\latticeVelocity} \partial_{\timeVariable}}) \asymptoticEquivalence \timeShift^{\indiceTime}$ into \eqref{eq:tmp2} using a smooth ``test'' function $\fourierTransformed{\testFunction}$ gives $\text{exp}({\indiceTime \tfrac{\spaceStep}{\latticeVelocity} \partial_{\timeVariable}}) \fourierTransformed{\testFunction}(\timeVariable, {\frequency}) = \fourierTransformed{\discrete{\eigenvalueLetter}}_1({\frequency} \spaceStep)^{\indiceTime} \fourierTransformed{\testFunction}(\timeVariable, {\frequency})$ for $(\timeVariable, {\frequency}) \in \nonNegativeReals \times \reals$, which means that if we do not want $\fourierTransformed{\testFunction}$ to trivially vanish, we must enforce the formal identity $\text{exp}({\indiceTime \tfrac{\spaceStep}{\latticeVelocity} \partial_{\timeVariable}}) = \fourierTransformed{\discrete{\eigenvalueLetter}}_1({\frequency} \spaceStep)^{\indiceTime}$.
Since the exponential is bijective close to zero (here we are considering the limit $|\frequency \spaceStep| \ll 1$), we can take the logarithm to yield: 
\begin{equation*}
    \partial_{\timeVariable} = \frac{\indiceTime}{\indiceTime} \frac{\latticeVelocity}{\spaceStep}\log(\fourierTransformed{\discrete{\eigenvalueLetter}}_1({\frequency} \spaceStep)),
\end{equation*}
which is thus independent of $\indiceTime$.

Differently, a $(\nonTriviallyRelaxingMomentsNumber + 2)$-stage \fd scheme, like the \bulkScheme{} \eqref{eq:BulkSchemes}, has associated amplification polynomial
\begin{equation}\label{eq:tmp4}
    \fourierTransformed{\amplificationPolynomial} (\timeShift, {\frequency}\spaceStep) \definitionEquality \frac{1}{\timeShift^{\velocityNumber - \nonTriviallyRelaxingMomentsNumber - 1}} \determinant (\timeShift \identity - \fourierTransformed{\schemeMatrix}(\frequency \spaceStep)) = \timeShift^{\nonTriviallyRelaxingMomentsNumber + 1} + \sum_{\indiceTime = 0}^{\nonTriviallyRelaxingMomentsNumber} \fourierTransformed{\coefficientFDScheme}_{\indiceTime + \velocityNumber - \nonTriviallyRelaxingMomentsNumber - 1} (\frequency \spaceStep) \timeShift^{\indiceTime} = \prod_{\indiceFreeOne = 1}^{\nonTriviallyRelaxingMomentsNumber + 1} (\timeShift - \fourierTransformed{\discrete{\eigenvalueLetter}}_{\indiceFreeOne}(\frequency \spaceStep)).
\end{equation}
Out of the roots in \eqref{eq:tmp4}, we shall number the (unique) eigenvalue providing the modified equation \eqref{eq:BulkEquivalentEquation}, \emph{i.e.} such that \eqref{eq:physicalEigenvalueCharact} holds, by $\fourierTransformed{\discrete{\eigenvalueLetter}}_1$. 
This is the amplification factor of the so-called ``pseudo-scheme'' \cite{strikwerda2004finite}.
Furthermore, the higher-order terms in the modified equation of the \bulkScheme{} stem from $\fourierTransformed{\discrete{\eigenvalueLetter}}_1(\frequency \spaceStep) = 1 + \sum_{\indiceOrder = 1}^{\indiceOrder = \maxOrder}  (\frequency \spaceStep)^{\indiceOrder}  \termAtOrder{\fourierTransformed{{\eigenvalueLetter}}_1}{\indiceOrder}+ \bigO{|\frequency \spaceStep|^{\maxOrder + 1}}$ in the limit $|\frequency\spaceStep| \ll 1$.
The \iniSchemes{} read
\begin{equation}\label{eq:tmp5}
    \timeShift^{\indiceTime} \fourierTransformed{\discreteMoment}_1(0, \frequency\spaceStep) = \underbrace{(\fourierTransformed{\schemeMatrix}(\frequency\spaceStep)^{\indiceTime} \fourierTransformed{\boldOther{\initialisationOperator}}(\frequency\spaceStep))_1}_{=: \amplificationFactorStartingScheme{\indiceTime}  (\frequency\spaceStep)} \fourierTransformed{\pointWiseDiscretisationInitialDatum}(\frequency\spaceStep), \qquad \indiceTime \in \integerInterval{1}{\nonTriviallyRelaxingMomentsNumber},
\end{equation}
with $\frequency \in [-\pi/\spaceStep, \pi/\spaceStep]$.
Using the assumption that $\termAtOrder{\initialisationOperatorAsymptotic_1}{0} = 1$, the proof of \Cref{prop:ModifiedEquationPrepared} naturally entails that $\amplificationFactorStartingScheme{\indiceTime}  (\frequency\spaceStep) = 1 + \bigO{|\frequency\spaceStep|}$ for $|\frequency\spaceStep| \ll 1$.
Comparing \eqref{eq:tmp2} and \eqref{eq:tmp5}, we cannot employ the same trick without a deeper discussion.
We have
\begin{equation*}
    \partial_{\timeVariable} = \frac{\latticeVelocity}{\spaceStep} \log(\fourierTransformed{\discrete{\eigenvalueLetter}}_1({\frequency} \spaceStep))\qquad \text{and} \qquad  \partial_{\timeVariable} = \frac{1}{\indiceTime} \frac{\latticeVelocity}{\spaceStep} \log(\amplificationFactorStartingScheme{\indiceTime}  (\frequency\spaceStep)), \qquad \indiceTime \in \integerInterval{1}{\nonTriviallyRelaxingMomentsNumber}, 
\end{equation*}
where the first equation comes from the modified equation of the \bulkScheme{} and the second one from \eqref{eq:tmp5}.
Since the \iniSchemes{} and the \bulkScheme{} have the same modified equations up to order $\maxOrder$, then we have, in the limit $|\frequency\spaceStep| \ll 1$
\begin{equation*}
    \indiceTime \log(\fourierTransformed{\discrete{\eigenvalueLetter}}_1({\frequency} \spaceStep)) = \log((\fourierTransformed{\discrete{\eigenvalueLetter}}_1({\frequency} \spaceStep))^{\indiceTime}) = \log(\amplificationFactorStartingScheme{\indiceTime}  (\frequency\spaceStep)) + \bigO{|\frequency\spaceStep|^{\maxOrder + 1}}, \qquad \indiceTime \in \integerInterval{1}{\nonTriviallyRelaxingMomentsNumber},
\end{equation*}
hence we deduce that $\amplificationFactorStartingScheme{\indiceTime}  (\frequency\spaceStep) = \fourierTransformed{\discrete{\eigenvalueLetter}}_1({\frequency} \spaceStep)^{\indiceTime} + \bigO{|\frequency\spaceStep|^{\maxOrder + 1}}$ for $\indiceTime \in \integerInterval{1}{\nonTriviallyRelaxingMomentsNumber}$.
To finish the proof, we now consider \eqref{eq:InitialisationSchemes} for $\indiceTime = \nonTriviallyRelaxingMomentsNumber + 1$
\begin{align}
    \timeShift^{\nonTriviallyRelaxingMomentsNumber + 1} \fourierTransformed{\discreteMoment}_1(0, \frequency\spaceStep) &= - \sum_{\indiceTime = 0}^{\nonTriviallyRelaxingMomentsNumber} \fourierTransformed{\coefficientFDScheme}_{\indiceTime + \velocityNumber - \nonTriviallyRelaxingMomentsNumber - 1} (\frequency \spaceStep) \timeShift^{\indiceTime} \fourierTransformed{\discreteMoment}_1(0, \frequency\spaceStep) \label{eq:tmp6}\\
    &= \underbrace{(\fourierTransformed{\schemeMatrix}(\frequency\spaceStep)^{\nonTriviallyRelaxingMomentsNumber + 1} \fourierTransformed{\boldOther{\initialisationOperator}}(\frequency\spaceStep))_1}_{=: \amplificationFactorStartingScheme{\nonTriviallyRelaxingMomentsNumber + 1} (\frequency\spaceStep)} \fourierTransformed{\pointWiseDiscretisationInitialDatum}(\frequency\spaceStep). \label{eq:tmp7}
\end{align}
We compute the modified equation of \eqref{eq:tmp7}, yielding the thesis, by using \eqref{eq:tmp6}.
We have
\begin{align*}
    \timeShift^{\nonTriviallyRelaxingMomentsNumber + 1} \fourierTransformed{\discreteMoment}_1(0, \frequency\spaceStep) &= - \sum_{\indiceTime = 1}^{\nonTriviallyRelaxingMomentsNumber} \fourierTransformed{\coefficientFDScheme}_{\indiceTime + \velocityNumber - \nonTriviallyRelaxingMomentsNumber - 1} (\frequency \spaceStep) \timeShift^{\indiceTime} \fourierTransformed{\discreteMoment}_1(0, \frequency\spaceStep) - \fourierTransformed{\coefficientFDScheme}_{\velocityNumber - \nonTriviallyRelaxingMomentsNumber - 1} (\frequency \spaceStep) \fourierTransformed{\initialisationOperator}_1 (\frequency \spaceStep) \fourierTransformed{\pointWiseDiscretisationInitialDatum}(\frequency\spaceStep) \\
    &= - \sum_{\indiceTime = 1}^{\nonTriviallyRelaxingMomentsNumber} \fourierTransformed{\coefficientFDScheme}_{\indiceTime + \velocityNumber - \nonTriviallyRelaxingMomentsNumber - 1} (\frequency \spaceStep) \amplificationFactorStartingScheme{\indiceTime}  (\frequency\spaceStep) \fourierTransformed{\pointWiseDiscretisationInitialDatum}(\frequency\spaceStep) - \fourierTransformed{\coefficientFDScheme}_{\velocityNumber - \nonTriviallyRelaxingMomentsNumber - 1} (\frequency \spaceStep) \fourierTransformed{\initialisationOperator}_1 (\frequency \spaceStep)  \fourierTransformed{\pointWiseDiscretisationInitialDatum}(\frequency\spaceStep).
\end{align*}
In the limit $|\frequency \spaceStep| \ll 1$, we have $\fourierTransformed{\initialisationOperator}_1 (\frequency \spaceStep)  = 1 + \bigO{|\frequency \spaceStep|^{\maxOrder + 1}}$ and $\amplificationFactorStartingScheme{\indiceTime}  (\frequency\spaceStep) = \fourierTransformed{\discrete{\eigenvalueLetter}}_1({\frequency} \spaceStep)^{\indiceTime} + \bigO{|\frequency\spaceStep|^{\maxOrder + 1}}$ for $\indiceTime \in \integerInterval{1}{\nonTriviallyRelaxingMomentsNumber}$, thanks to the assumption on $\initialisationOperator_1$ and to the previous computations.
In the limit $|\frequency \spaceStep| \ll 1$, we have to consider the amplification polynomial
\begin{equation*}
    \fourierTransformed{\tilde{\amplificationPolynomial}}(\timeShift, \frequency \spaceStep) = \timeShift^{\nonTriviallyRelaxingMomentsNumber + 1} + \sum_{\indiceTime = 0}^{\nonTriviallyRelaxingMomentsNumber} \fourierTransformed{\coefficientFDScheme}_{\indiceTime + \velocityNumber - \nonTriviallyRelaxingMomentsNumber - 1} (\frequency \spaceStep) \fourierTransformed{\discrete{\eigenvalueLetter}}_1({\frequency} \spaceStep)^{\indiceTime} + \bigO{|\frequency \spaceStep|^{\maxOrder + 1}} = \timeShift^{\nonTriviallyRelaxingMomentsNumber + 1} -  \fourierTransformed{\discrete{\eigenvalueLetter}}_1({\frequency} \spaceStep)^{\nonTriviallyRelaxingMomentsNumber + 1}  + \bigO{|\frequency \spaceStep|^{\maxOrder + 1}},
\end{equation*}
using the fact that $\fourierTransformed{\discrete{\eigenvalueLetter}}_1$ is a root of \eqref{eq:tmp4}.
We are therefore, up to terms $\bigO{|\frequency \spaceStep|^{\maxOrder + 1}}$, in the same setting as \eqref{eq:tmp2} and \eqref{eq:tmp3}, hence with the usual trick, we gain
\begin{equation*}
    \partial_{\timeVariable} = \frac{\nonTriviallyRelaxingMomentsNumber + 1}{\nonTriviallyRelaxingMomentsNumber + 1} \frac{\latticeVelocity}{\spaceStep} \log(\fourierTransformed{\discrete{\eigenvalueLetter}}_1({\frequency} \spaceStep)) + \bigO{|\frequency \spaceStep|^{\maxOrder + 1}}, 
\end{equation*}
hence also that $\amplificationFactorStartingScheme{\nonTriviallyRelaxingMomentsNumber + 1}  (\frequency\spaceStep) = \fourierTransformed{\discrete{\eigenvalueLetter}}_1({\frequency} \spaceStep) ^{\nonTriviallyRelaxingMomentsNumber + 1} + \bigO{|\frequency \spaceStep|^{\maxOrder + 1}}$.
This concludes the proof. The case $\indiceTime > \nonTriviallyRelaxingMomentsNumber + 1$ is done analogously.
\end{proof}

\newcommand{\greenFunctionTime}[2]{\fourierTransformed{\discrete{G}}_{#1}^{[#2]}}
\newcommand{\greenFunctionMoments}[2]{\fourierTransformed{\discrete{M}}_{#1}^{[#2]}}

In order to understand why it is difficult to study the \stSchemes{} above first-order in full generality, we consider the Fourier's setting introduced in the previous proof. To isolate the role of each initialisation scheme, we introduce the $\indiceFreeOne$-th time Green functions $\greenFunctionTime{\indiceFreeOne}{\indiceTime} = \greenFunctionTime{\indiceFreeOne}{\indiceTime}(\vectorial{\frequency}\spaceStep)$, for $\indiceFreeOne \in \integerInterval{0}{\nonTriviallyRelaxingMomentsNumber}$, see \cite{cheng1999general}, defined by
\begin{equation*}
    \begin{cases}
        \greenFunctionTime{\indiceFreeOne}{\indiceTime + 1} = -\sum_{\indiceFreeTwo = \velocityNumber - \nonTriviallyRelaxingMomentsNumber - 1}^{\indiceFreeTwo = \velocityNumber - 1} \fourierTransformed{\coefficientFDScheme}_{\indiceFreeTwo} \greenFunctionTime{\indiceFreeOne}{\indiceTime + \indiceFreeTwo - \velocityNumber + 1} = \transpose{\canonicalBasisVector_1} \fourierTransformed{\companionMatrixScheme} \transpose{ \left [ \greenFunctionTime{\indiceFreeOne}{\indiceTime} , \dots, \greenFunctionTime{\indiceFreeOne}{\indiceTime - \nonTriviallyRelaxingMomentsNumber} \right ]}, &\qquad \text{for} \quad \indiceTime \geq \nonTriviallyRelaxingMomentsNumber, \\
        \greenFunctionTime{\indiceFreeOne}{\indiceFreeTwo} = \delta_{\indiceFreeOne, \indiceFreeTwo}, &\qquad \text{for} \quad \indiceFreeTwo \in \integerInterval{0}{\nonTriviallyRelaxingMomentsNumber},
    \end{cases}
\end{equation*}
where $\companionMatrixScheme$ is the companion matrix of size $\nonTriviallyRelaxingMomentsNumber$ associated with the monic polynomial $\timeShift^{\nonTriviallyRelaxingMomentsNumber + 1 - \velocityNumber} \determinant(\timeShift \identity - \schemeMatrix)$ of degree $\nonTriviallyRelaxingMomentsNumber + 1$.
We can also associate a moment Green function $\greenFunctionMoments{\indiceMoments}{\indiceTime} = \greenFunctionMoments{\indiceMoments}{\indiceTime}(\vectorial{\frequency}\spaceStep)$ for the $\indiceMoments\in\integerInterval{1}{\velocityNumber}$ moment, defined by $\greenFunctionMoments{\indiceMoments}{\indiceTime} = \transpose{\canonicalBasisVector_1} \fourierTransformed{\schemeMatrix}^{\indiceTime} \canonicalBasisVector_{\indiceMoments}$.
The numerical solution can then be written---for $\indiceTime \geq \nonTriviallyRelaxingMomentsNumber + 1$---using the following decompositions:
\begin{align*}
    \fourierTransformed{\discreteMoment}_{1}(\indiceTime \timeStep, \vectorial{\frequency}\spaceStep) = \overbrace{\amplificationFactorStartingScheme{\indiceTime} (\vectorial{\frequency}\spaceStep) \fourierTransformed{\pointWiseDiscretisationInitialDatum} (\vectorial{\frequency}\spaceStep)}^{\text{(I) starting schemes}} = \overbrace{ \sum_{\indiceFreeOne = 1}^{\nonTriviallyRelaxingMomentsNumber + 1} \fourierTransformed{\discrete{W}}_{\indiceFreeOne}(\vectorial{\frequency}\spaceStep) \fourierTransformed{\discrete{\eigenvalueLetter}}_{\indiceFreeOne}(\vectorial{\frequency}\spaceStep)^{\indiceTime}}^{\text{(II) spectrum of }\schemeMatrix} &= \overbrace{ \sum_{\indiceFreeOne = 0}^{\nonTriviallyRelaxingMomentsNumber} \greenFunctionTime{\indiceFreeOne}{\indiceTime} (\vectorial{\frequency}\spaceStep)  \amplificationFactorStartingScheme{\indiceFreeOne} (\vectorial{\frequency}\spaceStep)}^{\text{(III) time Green functions}} \\
    &= \underbrace{\Bigl ( \greenFunctionMoments{1}{\indiceTime}(\vectorial{\frequency}\spaceStep) \fourierTransformed{\initialisationOperator}_1(\vectorial{\frequency}\spaceStep) + \sum_{\substack{\indiceMoments = 2 \\ \relaxationParameter_{\indiceMoments \neq 1}}}^{\velocityNumber} \greenFunctionMoments{\indiceMoments}{\indiceTime}(\vectorial{\frequency}\spaceStep) \fourierTransformed{\initialisationOperator}_{\indiceMoments}(\vectorial{\frequency}\spaceStep) \Bigr ) \fourierTransformed{\pointWiseDiscretisationInitialDatum} (\vectorial{\frequency}\spaceStep)}_{\text{(IV) moment Green functions}},
\end{align*}
where the decomposition (II) holds under the assumption that the roots of $\timeShift^{\nonTriviallyRelaxingMomentsNumber + 1 - \velocityNumber} \determinant(\timeShift \identity - \fourierTransformed{\schemeMatrix})$ are simple for any wavenumber. The decomposition (II) makes it clear that the consistency for the initial conditions also depends on the non-physical eigenvalues $\fourierTransformed{\discrete{\eigenvalueLetter}}_2, \dots, \fourierTransformed{\discrete{\eigenvalueLetter}}_{\nonTriviallyRelaxingMomentsNumber + 1}$ of $\fourierTransformed{\schemeMatrix}$, and the role of the choice of initial datum influences the weights $\fourierTransformed{\discrete{W}}_{\indiceFreeOne}$ for $\indiceFreeOne\in\integerInterval{1}{\nonTriviallyRelaxingMomentsNumber + 1}$ in a non-trivial way. For the same reasons, it is difficult to analyse (III) and (IV), for the involved Green functions can be difficult to compute, even in the small wavenumber limit, due to the influence of all the modes of the system.
Finally, let us insist on the fact that the Cayley-Hamilton theorem entails that the amplification factors of the starting schemes which are not initialisation schemes can be computed in two ways, which read
\begin{equation*}
    \amplificationFactorStartingScheme{\indiceTime} = \transpose{\canonicalBasisVector_1} \fourierTransformed{\schemeMatrix}^{\indiceTime} \fourierTransformed{\vectorial{\initialisationOperator}} = [\greenFunctionMoments{1}{\indiceTime}, \dots, \greenFunctionMoments{\velocityNumber}{\indiceTime}] \fourierTransformed{\vectorial{\initialisationOperator}} = \transpose{\canonicalBasisVector_1} \fourierTransformed{\companionMatrixScheme}^{\indiceTime - \nonTriviallyRelaxingMomentsNumber} \transpose{[\amplificationFactorStartingScheme{\nonTriviallyRelaxingMomentsNumber}, \dots, \amplificationFactorStartingScheme{1}, \fourierTransformed{\initialisationOperator}_1]},
\end{equation*}
for $\indiceTime \geq \nonTriviallyRelaxingMomentsNumber + 1$.

A second result states that there is little interest in considering the formal limit $\indiceTime \to +\infty$ in the modified equations of the \stSchemes{}.
\begin{proposition}[Long-time behavior: limits for $\indiceTime \to +\infty$]\label{prop:LongTimebehaviour}
    Assume that the scheme is $L^2$ stable, meaning that the roots of the amplification polynomial $\timeShift^{\nonTriviallyRelaxingMomentsNumber + 1 - \velocityNumber} \determinant(\timeShift \identity - \fourierTransformed{\schemeMatrix})$ are inside or on the unit circle and those on the unit circle are simple. Also assume that $\termAtOrder{\initialisationOperatorAsymptotic_1}{0} = 1$. Then:
    \begin{itemize}
        \item If $|1-\relaxationParameter_{\indiceMoments}|< 1$ for $\indiceMoments \in \integerInterval{2}{\velocityNumber}$, or equivalently $\relaxationParameter_{\indiceMoments} \in ]0, 2[$ for $\indiceMoments \in \integerInterval{2}{\velocityNumber}$, then the modified equations of the \stSchemes{} in the formal long-time limit $\indiceTime \to +\infty$ coincide at any order with the one of the \bulkScheme{}. 
        \item If it exists $\tilde{\indiceMoments} \in \integerInterval{2}{\velocityNumber}$ such that $|1-\relaxationParameter_{\tilde{\indiceMoments}}| = 1$, thus equivalently $\relaxationParameter_{\tilde{\indiceMoments}} = 2$. Let $\maxOrder \in \naturals$. Provided that $\termAtOrder{\initialisationOperatorAsymptotic_1}{\indiceOrder} = 0$ for $\indiceOrder \in \integerInterval{1}{\maxOrder}$ and the $\nonTriviallyRelaxingMomentsNumber$ modified equations of the  \iniSchemes{} coincide with the one of the \bulkScheme{} at any order $\indiceOrder \in \integerInterval{1}{\maxOrder}$, then the modified equations of the \stSchemes{} in the formal long-time limit $\indiceTime \to +\infty$ coincide at any order $\indiceOrder \in \integerInterval{1}{\maxOrder + 1}$ with the one of the \bulkScheme{}. 
    \end{itemize}
\end{proposition}
The first case in \Cref{prop:LongTimebehaviour} ensures that all the parasitic modes are damped. The second one deals with the case of undamped parasitic modes (\emph{i.e.} leap--frog-like schemes).
\Cref{prop:LongTimebehaviour} means that the effect of the initialisation decays, provided that the initial filtering on the datum of the Cauchy problem \eqref{eq:CauchyInitialDatum} preserves it at leading order and that either the parasitic modes damp in time, or if the parasitic modes are oscillatory, the \iniSchemes{} are accurate enough. 
The situation is the one depicted in \Cref{fig:matching}, where the dots asymptotically reach the dashed lines.
Let us point out that the assumption concerning stability may not be optimal, in the sense that we can find unstable schemes (for example, violating the Courant-Friedrichs-Lewy condition, but not having relaxation parameters exceeding $2$) for which the modified equations of the \stSchemes{} asymptotically reach those of the \bulkScheme{}. However, these schemes are practically useless.
    \begin{example}
        In order to illustrate \Cref{prop:LongTimebehaviour}, consider \Cref{ex:D1Q2}. 
        For this scheme, we consider $\initialisationOperator_1 = 1$ and $\initialisationOperator_2 = 0$.
        This choice does not satisfy \Cref{prop:LocalInitialisation} and the resulting initialization scheme is not consistent with the target equation, thus $\maxOrder = 0$.
        According to the choice of $\relaxationParameter_2$, we obtain the following modified equations up to order two:
        \begin{center}
            \addtolength{\tabcolsep}{-0.4em}
            \begin{tabular}{c|l|l}
                $\indiceTime$ & Mod. eq. \stSchemes{} for $\relaxationParameter_2 = 3/2$ & Mod. eq. \stSchemes{} for $\relaxationParameter_2 = 2$ \\
                \hline 
                \hline
                $2$  & { \footnotesize $\partial_{\timeVariable}\testFunction = -\latticeVelocity \tfrac{9}{8}    \equilibriumCoefficient_2              \partial_{\spaceVariable} \testFunction - \latticeVelocity \spaceStep \bigl ( \tfrac{9}{64}                           \equilibriumCoefficient_2^2 - \tfrac{1}{4}              \bigr ) \partial_{\spaceVariable\spaceVariable} \testFunction + \bigO{\spaceStep^2} $ } & { \footnotesize $\partial_{\timeVariable}\testFunction =                -   \latticeVelocity \equilibriumCoefficient_2 \partial_{\spaceVariable} \testFunction                          +\latticeVelocity\spaceStep\equilibriumCoefficient_2^2                            \partial_{\spaceVariable\spaceVariable} \testFunction + \bigO{\spaceStep^2}   $ } \\
                $3$  & { \footnotesize $\partial_{\timeVariable}\testFunction = -\latticeVelocity \tfrac{9}{8}    \equilibriumCoefficient_2              \partial_{\spaceVariable} \testFunction - \latticeVelocity \spaceStep \bigl ( \tfrac{51}{128}                         \equilibriumCoefficient_2^2 - \tfrac{1}{4}              \bigr ) \partial_{\spaceVariable\spaceVariable} \testFunction + \bigO{\spaceStep^2} $ } & { \footnotesize $\partial_{\timeVariable}\testFunction = -\tfrac{4}{3}      \latticeVelocity \equilibriumCoefficient_2 \partial_{\spaceVariable} \testFunction -  \latticeVelocity\spaceStep\bigl( \tfrac{4}{3}    \equilibriumCoefficient_2^2 - \tfrac{1}{6}     \bigr ) \partial_{\spaceVariable\spaceVariable} \testFunction + \bigO{\spaceStep^2}   $ } \\
                $4$  & { \footnotesize $\partial_{\timeVariable}\testFunction = -\latticeVelocity \tfrac{69}{64}    \equilibriumCoefficient_2            \partial_{\spaceVariable} \testFunction - \latticeVelocity \spaceStep \bigl ( \tfrac{441}{2048}                       \equilibriumCoefficient_2^2 - \tfrac{7}{32}             \bigr ) \partial_{\spaceVariable\spaceVariable} \testFunction + \bigO{\spaceStep^2} $ } & { \footnotesize $\partial_{\timeVariable}\testFunction =                -   \latticeVelocity \equilibriumCoefficient_2 \partial_{\spaceVariable} \testFunction                          +\latticeVelocity\spaceStep\equilibriumCoefficient_2^2                            \partial_{\spaceVariable\spaceVariable} \testFunction + \bigO{\spaceStep^2}   $ } \\
                $5$  & { \footnotesize $\partial_{\timeVariable}\testFunction = -\latticeVelocity \tfrac{171}{160}    \equilibriumCoefficient_2          \partial_{\spaceVariable} \testFunction - \latticeVelocity \spaceStep \bigl ( \tfrac{549}{2048}                       \equilibriumCoefficient_2^2 - \tfrac{17}{80}            \bigr ) \partial_{\spaceVariable\spaceVariable} \testFunction + \bigO{\spaceStep^2} $ } & { \footnotesize $\partial_{\timeVariable}\testFunction = -\tfrac{6}{5}      \latticeVelocity \equilibriumCoefficient_2 \partial_{\spaceVariable} \testFunction -  \latticeVelocity\spaceStep\bigl( \tfrac{6}{5}    \equilibriumCoefficient_2^2 - \tfrac{1}{10}    \bigr ) \partial_{\spaceVariable\spaceVariable} \testFunction + \bigO{\spaceStep^2}   $ } \\
                $6$  & { \footnotesize $\partial_{\timeVariable}\testFunction = -\latticeVelocity \tfrac{135}{128}    \equilibriumCoefficient_2          \partial_{\spaceVariable} \testFunction - \latticeVelocity \spaceStep \bigl ( \tfrac{3603}{16384}                     \equilibriumCoefficient_2^2 - \tfrac{13}{64}            \bigr ) \partial_{\spaceVariable\spaceVariable} \testFunction + \bigO{\spaceStep^2} $ } & { \footnotesize $\partial_{\timeVariable}\testFunction =                -   \latticeVelocity \equilibriumCoefficient_2 \partial_{\spaceVariable} \testFunction                          +\latticeVelocity\spaceStep\equilibriumCoefficient_2^2                            \partial_{\spaceVariable\spaceVariable} \testFunction + \bigO{\spaceStep^2}   $ } \\
                $7$  & { \footnotesize $\partial_{\timeVariable}\testFunction = -\latticeVelocity \tfrac{939}{896}    \equilibriumCoefficient_2          \partial_{\spaceVariable} \testFunction - \latticeVelocity \spaceStep \bigl ( \tfrac{52281}{229376}                   \equilibriumCoefficient_2^2 - \tfrac{89}{448}           \bigr ) \partial_{\spaceVariable\spaceVariable} \testFunction + \bigO{\spaceStep^2} $ } & { \footnotesize $\partial_{\timeVariable}\testFunction = -\tfrac{8}{7}      \latticeVelocity \equilibriumCoefficient_2 \partial_{\spaceVariable} \testFunction -  \latticeVelocity\spaceStep\bigl( \tfrac{8}{7}    \equilibriumCoefficient_2^2 - \tfrac{1}{14}    \bigr ) \partial_{\spaceVariable\spaceVariable} \testFunction + \bigO{\spaceStep^2}   $ } \\
                $8$  & { \footnotesize $\partial_{\timeVariable}\testFunction = -\latticeVelocity \tfrac{2133}{2048}    \equilibriumCoefficient_2        \partial_{\spaceVariable} \testFunction - \latticeVelocity \spaceStep \bigl ( \tfrac{222777}{1048576}                 \equilibriumCoefficient_2^2 - \tfrac{199}{1024}         \bigr ) \partial_{\spaceVariable\spaceVariable} \testFunction + \bigO{\spaceStep^2} $ } & { \footnotesize $\partial_{\timeVariable}\testFunction =                -   \latticeVelocity \equilibriumCoefficient_2 \partial_{\spaceVariable} \testFunction                          +\latticeVelocity\spaceStep\equilibriumCoefficient_2^2                            \partial_{\spaceVariable\spaceVariable} \testFunction + \bigO{\spaceStep^2}   $ } \\
                $9$  & { \footnotesize $\partial_{\timeVariable}\testFunction = -\latticeVelocity \tfrac{531}{512}    \equilibriumCoefficient_2          \partial_{\spaceVariable} \testFunction - \latticeVelocity \spaceStep \bigl ( \tfrac{110769}{524288}                  \equilibriumCoefficient_2^2 - \tfrac{49}{256}           \bigr ) \partial_{\spaceVariable\spaceVariable} \testFunction + \bigO{\spaceStep^2} $ } & { \footnotesize $\partial_{\timeVariable}\testFunction = -\tfrac{10}{9}     \latticeVelocity \equilibriumCoefficient_2 \partial_{\spaceVariable} \testFunction -  \latticeVelocity\spaceStep\bigl( \tfrac{10}{9}   \equilibriumCoefficient_2^2 - \tfrac{1}{18}   \bigr )  \partial_{\spaceVariable\spaceVariable} \testFunction + \bigO{\spaceStep^2}   $ } \\
                $10$ & { \footnotesize $\partial_{\timeVariable}\testFunction = -\latticeVelocity \tfrac{10581}{10240}    \equilibriumCoefficient_2      \partial_{\spaceVariable} \testFunction - \latticeVelocity \spaceStep \bigl ( \tfrac{4296249}{20971520}               \equilibriumCoefficient_2^2 - \tfrac{967}{5120}         \bigr ) \partial_{\spaceVariable\spaceVariable} \testFunction + \bigO{\spaceStep^2} $ } & { \footnotesize $\partial_{\timeVariable}\testFunction =                -   \latticeVelocity \equilibriumCoefficient_2 \partial_{\spaceVariable} \testFunction                          +\latticeVelocity\spaceStep\equilibriumCoefficient_2^2                            \partial_{\spaceVariable\spaceVariable} \testFunction + \bigO{\spaceStep^2}   $ } \\
                $11$ & { \footnotesize $\partial_{\timeVariable}\testFunction = -\latticeVelocity \tfrac{23211}{22528}    \equilibriumCoefficient_2      \partial_{\spaceVariable} \testFunction - \latticeVelocity \spaceStep \bigl ( \tfrac{18673209}{92274688}              \equilibriumCoefficient_2^2 - \tfrac{2105}{11264}       \bigr ) \partial_{\spaceVariable\spaceVariable} \testFunction + \bigO{\spaceStep^2} $ } & { \footnotesize $\partial_{\timeVariable}\testFunction = -\tfrac{12}{11}    \latticeVelocity \equilibriumCoefficient_2 \partial_{\spaceVariable} \testFunction -  \latticeVelocity\spaceStep\bigl( \tfrac{12}{11}  \equilibriumCoefficient_2^2 - \tfrac{1}{22}  \bigr )   \partial_{\spaceVariable\spaceVariable} \testFunction + \bigO{\spaceStep^2}   $ } \\
                $12$ & { \footnotesize $\partial_{\timeVariable}\testFunction = -\latticeVelocity \tfrac{16839}{16384}    \equilibriumCoefficient_2      \partial_{\spaceVariable} \testFunction - \latticeVelocity \spaceStep \bigl ( \tfrac{26696211}{134217728}             \equilibriumCoefficient_2^2 - \tfrac{1517}{8192}        \bigr ) \partial_{\spaceVariable\spaceVariable} \testFunction + \bigO{\spaceStep^2} $ } & { \footnotesize $\partial_{\timeVariable}\testFunction =                -   \latticeVelocity \equilibriumCoefficient_2 \partial_{\spaceVariable} \testFunction                          +\latticeVelocity\spaceStep\equilibriumCoefficient_2^2                            \partial_{\spaceVariable\spaceVariable} \testFunction + \bigO{\spaceStep^2}   $ } \\
                $13$ & { \footnotesize $\partial_{\timeVariable}\testFunction = -\latticeVelocity \tfrac{109227}{106496}    \equilibriumCoefficient_2    \partial_{\spaceVariable} \testFunction - \latticeVelocity \spaceStep \bigl ( \tfrac{343150137}{1744830464}           \equilibriumCoefficient_2^2 - \tfrac{9785}{53248}       \bigr ) \partial_{\spaceVariable\spaceVariable} \testFunction + \bigO{\spaceStep^2} $ } & { \footnotesize $\partial_{\timeVariable}\testFunction = -\tfrac{14}{13}    \latticeVelocity \equilibriumCoefficient_2 \partial_{\spaceVariable} \testFunction -  \latticeVelocity\spaceStep\bigl( \tfrac{14}{13}  \equilibriumCoefficient_2^2 - \tfrac{1}{26}  \bigr )   \partial_{\spaceVariable\spaceVariable} \testFunction + \bigO{\spaceStep^2}   $ } \\
                $14$ & { \footnotesize $\partial_{\timeVariable}\testFunction = -\latticeVelocity \tfrac{234837}{229376}    \equilibriumCoefficient_2    \partial_{\spaceVariable} \testFunction - \latticeVelocity \spaceStep \bigl ( \tfrac{1461161529}{7516192768}          \equilibriumCoefficient_2^2 - \tfrac{20935}{114688}     \bigr ) \partial_{\spaceVariable\spaceVariable} \testFunction + \bigO{\spaceStep^2} $ } & { \footnotesize $\partial_{\timeVariable}\testFunction =                -   \latticeVelocity \equilibriumCoefficient_2 \partial_{\spaceVariable} \testFunction                          +\latticeVelocity\spaceStep\equilibriumCoefficient_2^2                            \partial_{\spaceVariable\spaceVariable} \testFunction + \bigO{\spaceStep^2}   $ } \\
                $15$ & { \footnotesize $\partial_{\timeVariable}\testFunction = -\latticeVelocity \tfrac{167481}{163840}    \equilibriumCoefficient_2    \partial_{\spaceVariable} \testFunction - \latticeVelocity \spaceStep \bigl ( \tfrac{2068175379}{10737418240}         \equilibriumCoefficient_2^2 - \tfrac{14867}{81920}      \bigr ) \partial_{\spaceVariable\spaceVariable} \testFunction + \bigO{\spaceStep^2} $ } & { \footnotesize $\partial_{\timeVariable}\testFunction = -\tfrac{16}{15}    \latticeVelocity \equilibriumCoefficient_2 \partial_{\spaceVariable} \testFunction -  \latticeVelocity\spaceStep\bigl( \tfrac{16}{15}  \equilibriumCoefficient_2^2 - \tfrac{1}{30}  \bigr )   \partial_{\spaceVariable\spaceVariable} \testFunction + \bigO{\spaceStep^2}   $ } \\
                $16$ & { \footnotesize $\partial_{\timeVariable}\testFunction = -\latticeVelocity \tfrac{1070421}{1048576}    \equilibriumCoefficient_2  \partial_{\spaceVariable} \testFunction - \latticeVelocity \spaceStep \bigl ( \tfrac{26245566009}{137438953472}       \equilibriumCoefficient_2^2 - \tfrac{94663}{524288}     \bigr ) \partial_{\spaceVariable\spaceVariable} \testFunction + \bigO{\spaceStep^2} $ } & { \footnotesize $\partial_{\timeVariable}\testFunction =                -   \latticeVelocity \equilibriumCoefficient_2 \partial_{\spaceVariable} \testFunction                          +\latticeVelocity\spaceStep\equilibriumCoefficient_2^2                            \partial_{\spaceVariable\spaceVariable} \testFunction + \bigO{\spaceStep^2}   $ } \\
                $17$ & { \footnotesize $\partial_{\timeVariable}\testFunction = -\latticeVelocity \tfrac{2271915}{2228224}    \equilibriumCoefficient_2  \partial_{\spaceVariable} \testFunction - \latticeVelocity \spaceStep \bigl ( \tfrac{110717799993}{584115552256}      \equilibriumCoefficient_2^2 - \tfrac{200249}{1114112}   \bigr ) \partial_{\spaceVariable\spaceVariable} \testFunction + \bigO{\spaceStep^2} $ } & { \footnotesize $\partial_{\timeVariable}\testFunction = -\tfrac{18}{17}    \latticeVelocity \equilibriumCoefficient_2 \partial_{\spaceVariable} \testFunction -  \latticeVelocity\spaceStep\bigl( \tfrac{18}{17}  \equilibriumCoefficient_2^2 - \tfrac{1}{34}  \bigr )   \partial_{\spaceVariable\spaceVariable} \testFunction + \bigO{\spaceStep^2}   $ } \\
                $18$ & { \footnotesize $\partial_{\timeVariable}\testFunction = -\latticeVelocity \tfrac{533997}{524288}    \equilibriumCoefficient_2    \partial_{\spaceVariable} \testFunction - \latticeVelocity \spaceStep \bigl ( \tfrac{51750979761}{274877906944}       \equilibriumCoefficient_2^2 - \tfrac{46927}{262144}     \bigr ) \partial_{\spaceVariable\spaceVariable} \testFunction + \bigO{\spaceStep^2} $ } & { \footnotesize $\partial_{\timeVariable}\testFunction =                -   \latticeVelocity \equilibriumCoefficient_2 \partial_{\spaceVariable} \testFunction                          +\latticeVelocity\spaceStep\equilibriumCoefficient_2^2                            \partial_{\spaceVariable\spaceVariable} \testFunction + \bigO{\spaceStep^2}   $ } \\
                $19$ & { \footnotesize $\partial_{\timeVariable}\testFunction = -\latticeVelocity \tfrac{10136235}{9961472}    \equilibriumCoefficient_2 \partial_{\spaceVariable} \testFunction - \latticeVelocity \spaceStep \bigl ( \tfrac{1954701086265}{10445360463872}   \equilibriumCoefficient_2^2 - \tfrac{888377}{4980736}   \bigr ) \partial_{\spaceVariable\spaceVariable} \testFunction + \bigO{\spaceStep^2} $ } & { \footnotesize $\partial_{\timeVariable}\testFunction = -\tfrac{20}{19}    \latticeVelocity \equilibriumCoefficient_2 \partial_{\spaceVariable} \testFunction -  \latticeVelocity\spaceStep\bigl( \tfrac{20}{19}  \equilibriumCoefficient_2^2 - \tfrac{1}{38}  \bigr )   \partial_{\spaceVariable\spaceVariable} \testFunction + \bigO{\spaceStep^2}   $ } \\
                \hline 
                \hline
                Mod. eq. bulk & $\partial_{\timeVariable} \testFunction = -\latticeVelocity \equilibriumCoefficient_2 \partial_{\spaceVariable}\testFunction -\latticeVelocity\spaceStep \bigl ( \tfrac{1}{6} \equilibriumCoefficient_2^2 - \tfrac{1}{6} \bigr ) \partial_{\spaceVariable\spaceVariable}\testFunction + \bigO{\spaceStep^2}$ & $\partial_{\timeVariable} \testFunction = -\latticeVelocity \equilibriumCoefficient_2 \partial_{\spaceVariable}\testFunction  + \bigO{\spaceStep^2}$
            \end{tabular}            
        \end{center}
        For $\relaxationParameter_2 = 3/2$, we observe convergence at any order, whereas for $\relaxationParameter_2 = 2$, we see that convergence takes place until order $\maxOrder + 1 = 1$ (transport) and $\maxOrder + 2 = 2$ (diffusion) is not converging, as claimed by \Cref{prop:LongTimebehaviour}.
    \end{example}

\begin{proof}[Proof of \Cref{prop:LongTimebehaviour}]
        Again, we consider $\spatialDimensionality = 1$ only to simplify notations.
        Let us formulate a preliminary remark: we consider formal series in the limit $|\frequency\spaceStep|\ll 1$. Therefore, the possibility of diagonalise $\fourierTransformed{\companionMatrixScheme} (\frequency\spaceStep)$ in this limit---or alternatively being obliged to deal with a true Jordan canonical form---is determined by $\fourierTransformed{\companionMatrixScheme} (0)$, \emph{i.e.} the leading-order term in the formal series.
        We assume---without loss of generality---that all $\fourierTransformed{\discrete{\eigenvalueLetter}}_{\indiceFreeOne}(0)$ for $\indiceFreeOne \in \integerInterval{1}{\nonTriviallyRelaxingMomentsNumber + 1}$ are simple, even those strictly inside the unit circle, so that we can diagonalise the companion matrix in the desired limit. If this does not hold, for example for a SRT (or BGK) scheme, one can easily go through the same proof using the well-known expression for the powers of Jordan blocks.

        Let us start the proof. Let $\fourierTransformed{\matricial{\discrete{V}}} = \fourierTransformed{\matricial{\discrete{V}}}(\frequency\spaceStep)$ be the Vandermonde matrix associated with $\fourierTransformed{\discrete{\eigenvalueLetter}}_1 = \fourierTransformed{\discrete{\eigenvalueLetter}}_1(\frequency\spaceStep), \dots, \fourierTransformed{\discrete{\eigenvalueLetter}}_{\nonTriviallyRelaxingMomentsNumber + 1} = \fourierTransformed{\discrete{\eigenvalueLetter}}_{\nonTriviallyRelaxingMomentsNumber + 1}(\frequency\spaceStep)$. It is well-known that this Vandermonde matrix diagonalises the companion matrix $\fourierTransformed{\companionMatrixScheme} (\frequency\spaceStep)$, thus for $\indiceTime \geq \nonTriviallyRelaxingMomentsNumber + 1$
        \begin{align}
            \amplificationFactorStartingScheme{\indiceTime}(\frequency\spaceStep) &= \transpose{\canonicalBasisVector_1} \fourierTransformed{\companionMatrixScheme}(\frequency\spaceStep)^{\indiceTime - \nonTriviallyRelaxingMomentsNumber} \transpose{[\amplificationFactorStartingScheme{\nonTriviallyRelaxingMomentsNumber}(\frequency\spaceStep), \dots, \amplificationFactorStartingScheme{1}(\frequency\spaceStep), \fourierTransformed{\initialisationOperator}_1(\frequency\spaceStep)]} \label{eq:powerIteration} \\
            &= \transpose{\canonicalBasisVector_1}  \fourierTransformed{\matricial{\discrete{V}}}(\frequency\spaceStep) \diagMatrix (\fourierTransformed{\discrete{\eigenvalueLetter}}_1(\frequency\spaceStep)^{\indiceTime - \nonTriviallyRelaxingMomentsNumber}, \dots, \fourierTransformed{\discrete{\eigenvalueLetter}}_{\nonTriviallyRelaxingMomentsNumber + 1}(\frequency\spaceStep)^{\indiceTime - \nonTriviallyRelaxingMomentsNumber}) \fourierTransformed{\matricial{\discrete{V}}}(\frequency\spaceStep)^{-1} \transpose{[\amplificationFactorStartingScheme{\nonTriviallyRelaxingMomentsNumber}(\frequency\spaceStep), \dots, \amplificationFactorStartingScheme{1}(\frequency\spaceStep), \fourierTransformed{\initialisationOperator}_1(\frequency\spaceStep)]}. \nonumber
        \end{align}
        The idea of the proof is that the amplification factors associated with the \iniSchemes{} form an approximation of the eigenvector of $\fourierTransformed{\companionMatrixScheme} (\frequency\spaceStep)$ relative to the consistency eigenvalue $\fourierTransformed{\discrete{\eigenvalueLetter}}_{1}$, so that the power iteration \eqref{eq:powerIteration} converges for $\indiceTime \to +\infty$ up to some order.
        Up to a re-ordering of the non-conserved moments---in order to start with those which do not relax on the equilibrium, for notational ease---the lower-triangular structure of the collision matrix $\collisionMatrix$ entails that $\fourierTransformed{\discrete{\eigenvalueLetter}}_{\indiceFreeOne}(0) = 1-\relaxationParameter_{\indiceFreeOne}$ for $\indiceFreeOne \in \integerInterval{2}{\nonTriviallyRelaxingMomentsNumber + 1}$. Moreover, we have that $\fourierTransformed{\discrete{\eigenvalueLetter}}_{1}(0) = 1$.
        \begin{itemize}
            \item  Using the assumption on the relaxation parameters, we have $|\fourierTransformed{\discrete{\eigenvalueLetter}}_{\indiceFreeOne}(0)| < 1$ for $\indiceFreeOne \in \integerInterval{2}{\nonTriviallyRelaxingMomentsNumber + 1}$.
            Using the assumption $\termAtOrder{\initialisationOperatorAsymptotic_1}{0} = 1$ (\emph{i.e.} $\fourierTransformed{\initialisationOperator}_1(\frequency \spaceStep) = 1+\bigO{|\frequency\spaceStep|}$), \Cref{prop:ModifiedEquationPrepared} provides $\amplificationFactorStartingScheme{\indiceFreeOne}  (\frequency\spaceStep) = 1 + \bigO{|\frequency\spaceStep|}$ for $\indiceFreeOne \in \integerInterval{1}{\nonTriviallyRelaxingMomentsNumber}$. Therefore
            \begin{equation*}
                [\amplificationFactorStartingScheme{\nonTriviallyRelaxingMomentsNumber}(\frequency\spaceStep), \dots, \amplificationFactorStartingScheme{1}(\frequency\spaceStep), \fourierTransformed{\initialisationOperator}_1(\frequency\spaceStep)] = [\fourierTransformed{\discrete{\eigenvalueLetter}}_1(\frequency\spaceStep)^{\nonTriviallyRelaxingMomentsNumber}, \dots, \fourierTransformed{\discrete{\eigenvalueLetter}}_1(\frequency\spaceStep), 1] + \bigO{|\frequency\spaceStep|},
            \end{equation*}
            which means that the amplification factors of the \iniSchemes{} are the eigenvector of $\fourierTransformed{\companionMatrixScheme} (\frequency\spaceStep)$ associated with $\fourierTransformed{\discrete{\eigenvalueLetter}}_1(\frequency\spaceStep)$ at leading order.
            Back in \eqref{eq:powerIteration}, this gives that 
            \begin{multline*}
                \amplificationFactorStartingScheme{\indiceTime}(\frequency\spaceStep) = \transpose{\canonicalBasisVector_1}  \fourierTransformed{\matricial{\discrete{V}}}(\frequency\spaceStep) \diagMatrix (\fourierTransformed{\discrete{\eigenvalueLetter}}_1(\frequency\spaceStep)^{\indiceTime - \nonTriviallyRelaxingMomentsNumber}, \dots, \fourierTransformed{\discrete{\eigenvalueLetter}}_{\nonTriviallyRelaxingMomentsNumber + 1}(\frequency\spaceStep)^{\indiceTime - \nonTriviallyRelaxingMomentsNumber}) (\canonicalBasisVector_1 + \bigO{|\frequency\spaceStep|}) \\
                = \transpose{\canonicalBasisVector_1} 
                \begin{bmatrix}
                    \fourierTransformed{\discrete{\eigenvalueLetter}}_1(\frequency\spaceStep)^{\nonTriviallyRelaxingMomentsNumber} & \cdots & \fourierTransformed{\discrete{\eigenvalueLetter}}_{\nonTriviallyRelaxingMomentsNumber + 1}(\frequency\spaceStep)^{\nonTriviallyRelaxingMomentsNumber} \\
                    \vdots & & \vdots \\
                    \fourierTransformed{\discrete{\eigenvalueLetter}}_1(\frequency\spaceStep) & \cdots & \fourierTransformed{\discrete{\eigenvalueLetter}}_{\nonTriviallyRelaxingMomentsNumber + 1}(\frequency\spaceStep) \\
                    1 & \cdots & 1
                \end{bmatrix}
                \diagMatrix (\fourierTransformed{\discrete{\eigenvalueLetter}}_1(\frequency\spaceStep)^{\indiceTime - \nonTriviallyRelaxingMomentsNumber} (1+ \bigO{|\frequency\spaceStep|}), \fourierTransformed{\discrete{\eigenvalueLetter}}_2(\frequency\spaceStep)^{\indiceTime - \nonTriviallyRelaxingMomentsNumber} \bigO{|\frequency\spaceStep|}, \dots, \fourierTransformed{\discrete{\eigenvalueLetter}}_{\nonTriviallyRelaxingMomentsNumber + 1}(\frequency\spaceStep)^{\indiceTime - \nonTriviallyRelaxingMomentsNumber} \bigO{|\frequency\spaceStep|}) \\
                = \fourierTransformed{\discrete{\eigenvalueLetter}}_1(\frequency\spaceStep)^{\indiceTime} (1+ \bigO{|\frequency\spaceStep|}) + \fourierTransformed{\discrete{\eigenvalueLetter}}_2(\frequency\spaceStep)^{\indiceTime} \bigO{|\frequency\spaceStep|} + \dots + \fourierTransformed{\discrete{\eigenvalueLetter}}_{\nonTriviallyRelaxingMomentsNumber + 1}(\frequency\spaceStep)^{\indiceTime} \bigO{|\frequency\spaceStep|}),
            \end{multline*}
            where it is important to observe that the $\bigO{|\frequency\spaceStep|}$-terms are independent of $\indiceTime$.
            Considering that $|\fourierTransformed{\discrete{\eigenvalueLetter}}_{\indiceFreeOne}(0)| < 1$ for $\indiceFreeOne \in \integerInterval{2}{\nonTriviallyRelaxingMomentsNumber + 1}$, we deduce that $\lim_{\indiceTime \to +\infty} \fourierTransformed{\discrete{\eigenvalueLetter}}_{\indiceFreeOne}(\frequency\spaceStep)^{\indiceTime} = 0$ for $\indiceFreeOne \in \integerInterval{2}{\nonTriviallyRelaxingMomentsNumber + 1}$. Of course, convergence can be slow for high orders in the formal series. This entails that we have 
            \begin{equation*}
                \amplificationFactorStartingScheme{\indiceTime}(\frequency\spaceStep) = \fourierTransformed{\discrete{\eigenvalueLetter}}_1(\frequency\spaceStep)^{\indiceTime} (1 + \fourierTransformed{{\discrete{r}}}^{[\indiceTime]} (\frequency\spaceStep)),
            \end{equation*}
            where the residual $\fourierTransformed{{\discrete{r}}}^{[\indiceTime]} (\frequency\spaceStep) = \bigO{|\frequency\spaceStep|}$ is such that it converges to a fixed formal series for $\indiceTime \to +\infty$.
            The usual trick provides
            \begin{equation*}
                \lim_{\indiceTime \to +\infty} \partial_{\timeVariable} = \frac{\latticeVelocity}{\spaceStep} \lim_{\indiceTime \to +\infty} \frac{1}{\indiceTime} \log (\amplificationFactorStartingScheme{\indiceTime}(\frequency\spaceStep)) = \frac{\latticeVelocity}{\spaceStep} \Bigl ( \log (\fourierTransformed{\discrete{\eigenvalueLetter}}_1(\frequency\spaceStep)) + \lim_{\indiceTime \to +\infty} \frac{1}{\indiceTime} \log (1 + \fourierTransformed{{\discrete{r}}}^{[\indiceTime]} (\frequency\spaceStep)\Bigr ) = \frac{\latticeVelocity}{\spaceStep} \log (\fourierTransformed{\discrete{\eigenvalueLetter}}_1(\frequency\spaceStep)).
            \end{equation*}
            \item Observe that thanks to the stability assumption, there can be only one relaxation parameter $\relaxationParameter_{\tilde{\indiceMoments}} = 2$. Otherwise, there would be a multiple eigenvalue on the unit circle for $\frequency\spaceStep = 0$, contradicting the stability assumption whilst generating polynomial (in contrast with exponential) instabilities. Up to a rearrangement of the moments, we have $\tilde{\indiceMoments} = 2$.
            By the assumptions on $\initialisationOperator_1$ and the \iniSchemes{}, we deduce that 
            \begin{equation*}
                [\amplificationFactorStartingScheme{\nonTriviallyRelaxingMomentsNumber}(\frequency\spaceStep), \dots, \amplificationFactorStartingScheme{1}(\frequency\spaceStep), \fourierTransformed{\initialisationOperator}_1(\frequency\spaceStep)] = [\fourierTransformed{\discrete{\eigenvalueLetter}}_1(\frequency\spaceStep)^{\nonTriviallyRelaxingMomentsNumber}, \dots, \fourierTransformed{\discrete{\eigenvalueLetter}}_1(\frequency\spaceStep), 1] + \bigO{|\frequency\spaceStep|^{\maxOrder + 1}},
            \end{equation*}
            which means that the amplification factors of the \iniSchemes{} are the eigenvector of $\fourierTransformed{\companionMatrixScheme} (\frequency\spaceStep)$ associated with $\fourierTransformed{\discrete{\eigenvalueLetter}}_1(\frequency\spaceStep)$ up to order $\bigO{|\frequency\spaceStep|^{\maxOrder + 1}}$.
            Into \eqref{eq:powerIteration}, this yields
            \begin{align*}
                \amplificationFactorStartingScheme{\indiceTime}(\frequency\spaceStep) &= \transpose{\canonicalBasisVector_1}  \fourierTransformed{\matricial{\discrete{V}}}(\frequency\spaceStep) \diagMatrix (\fourierTransformed{\discrete{\eigenvalueLetter}}_1(\frequency\spaceStep)^{\indiceTime - \nonTriviallyRelaxingMomentsNumber}, \fourierTransformed{\discrete{\eigenvalueLetter}}_2(\frequency\spaceStep)^{\indiceTime - \nonTriviallyRelaxingMomentsNumber}, \dots, \fourierTransformed{\discrete{\eigenvalueLetter}}_{\nonTriviallyRelaxingMomentsNumber + 1}(\frequency\spaceStep)^{\indiceTime - \nonTriviallyRelaxingMomentsNumber}) (\canonicalBasisVector_1 + \bigO{|\frequency\spaceStep|^{\maxOrder + 1}}) \\
                &= \fourierTransformed{\discrete{\eigenvalueLetter}}_1(\frequency\spaceStep)^{\indiceTime} (1+ \bigO{|\frequency\spaceStep|^{\maxOrder+1}}) + \fourierTransformed{\discrete{\eigenvalueLetter}}_2(\frequency\spaceStep)^{\indiceTime} \bigO{|\frequency\spaceStep|^{\maxOrder+1}} + \dots + \fourierTransformed{\discrete{\eigenvalueLetter}}_{\nonTriviallyRelaxingMomentsNumber + 1}(\frequency\spaceStep)^{\indiceTime} \bigO{|\frequency\spaceStep|^{\maxOrder+1}}),
            \end{align*}
            where all the $\bigO{|\frequency\spaceStep|^{\maxOrder+1}}$-terms are independent of $\indiceTime$.
            Due to the fact that $\fourierTransformed{\discrete{\eigenvalueLetter}}_2(0) = 1-\relaxationParameter_2 = -1$, the formal series $\fourierTransformed{\discrete{\eigenvalueLetter}}_2(\frequency\spaceStep)^{\indiceTime}$ contains terms that can oscillate by featuring expressions involving $(-1)^{\indiceTime}$, and the term at order $\indiceOrder \in \integerIntervalClosedOpen{0}{+\infty}$ grows with $\indiceTime$ at most as a polynomial of degree $\indiceOrder$ in $\indiceTime$. We indicate this fact using the notation $\fourierTransformed{\discrete{\eigenvalueLetter}}_2(\frequency\spaceStep)^{\indiceTime} = \sum_{\indiceOrder = 0}^{\indiceOrder = +\infty} \bigO{\indiceTime^{\indiceOrder}} (\frequency\spaceStep)^{\indiceOrder}$.
            Therefore 
            \begin{equation*}
                \fourierTransformed{\discrete{\eigenvalueLetter}}_2(\frequency\spaceStep)^{\indiceTime} \bigO{|\frequency\spaceStep|^{\maxOrder+1}} =  \sum_{\indiceOrder = \maxOrder+1}^{+\infty} \bigO{\indiceTime^{\indiceOrder - \maxOrder - 1}} (\frequency\spaceStep)^{\indiceOrder}.
            \end{equation*} 
            As previously acknowledged, since $|\fourierTransformed{\discrete{\eigenvalueLetter}}_{\indiceFreeOne}(0)| < 1$ for $\indiceFreeOne \in \integerInterval{3}{\nonTriviallyRelaxingMomentsNumber + 1}$, we deduce that $\lim_{\indiceTime \to +\infty} \fourierTransformed{\discrete{\eigenvalueLetter}}_{\indiceFreeOne}(\frequency\spaceStep)^{\indiceTime} = 0$ for $\indiceFreeOne \in \integerInterval{3}{\nonTriviallyRelaxingMomentsNumber + 1}$. This ensures that 
            \begin{equation*}
                \amplificationFactorStartingScheme{\indiceTime}(\frequency\spaceStep) = \fourierTransformed{\discrete{\eigenvalueLetter}}_1 (\frequency\spaceStep)^{\indiceTime} \Bigl ( 1 + \sum_{\indiceOrder = \maxOrder+1}^{+\infty} \bigO{\indiceTime^{\indiceOrder - \maxOrder - 1}} (\frequency\spaceStep)^{\indiceOrder} \Bigr ).
            \end{equation*}
            Utilising the usual trick, we have
            \begin{multline*}
                \lim_{\indiceTime \to +\infty} \partial_{\timeVariable} = \frac{\latticeVelocity}{\spaceStep} \lim_{\indiceTime \to +\infty} \frac{1}{\indiceTime} \log (\amplificationFactorStartingScheme{\indiceTime}(\frequency\spaceStep)) = \frac{\latticeVelocity}{\spaceStep} \Bigl ( \log (\fourierTransformed{\discrete{\eigenvalueLetter}}_1(\frequency\spaceStep)) + \lim_{\indiceTime \to +\infty} \frac{1}{\indiceTime} \log \Bigl (1 + \sum_{\indiceOrder = \maxOrder+1}^{+\infty} \bigO{\indiceTime^{\indiceOrder - \maxOrder - 1}} (\frequency\spaceStep)^{\indiceOrder}  \Bigr )\Bigr ) \\
                = \frac{\latticeVelocity}{\spaceStep} \Bigl ( \log (\fourierTransformed{\discrete{\eigenvalueLetter}}_1(\frequency\spaceStep)) + \lim_{\indiceTime \to +\infty} \frac{1}{\indiceTime} \sum_{\indiceOrder = \maxOrder+1}^{+\infty} \bigO{\indiceTime^{\indiceOrder - \maxOrder - 1}} (\frequency\spaceStep)^{\indiceOrder}  \Bigr ) 
                = \frac{\latticeVelocity}{\spaceStep} \Bigl ( \log (\fourierTransformed{\discrete{\eigenvalueLetter}}_1(\frequency\spaceStep)) + \lim_{\indiceTime \to +\infty} \sum_{\indiceOrder = \maxOrder+1}^{+\infty} \bigO{\indiceTime^{\indiceOrder - \maxOrder - 2}} (\frequency\spaceStep)^{\indiceOrder}  \Bigr ) \\ 
                = \frac{\latticeVelocity}{\spaceStep} \Bigl ( \log (\fourierTransformed{\discrete{\eigenvalueLetter}}_1(\frequency\spaceStep)) + \lim_{\indiceTime \to +\infty} \Bigl (\bigO{\indiceTime^{-1}} (\frequency\spaceStep)^{\maxOrder + 1} + \sum_{\indiceOrder = \maxOrder+2}^{+\infty} \bigO{\indiceTime^{\indiceOrder - \maxOrder - 2}} (\frequency\spaceStep)^{\indiceOrder}  \Bigr )  \Bigr ) \\
                = \frac{\latticeVelocity}{\spaceStep} \Bigl ( \log (\fourierTransformed{\discrete{\eigenvalueLetter}}_1(\frequency\spaceStep)) + \lim_{\indiceTime \to +\infty} \sum_{\indiceOrder = \maxOrder+2}^{+\infty} \bigO{\indiceTime^{\indiceOrder - \maxOrder - 2}} (\frequency\spaceStep)^{\indiceOrder}  \Bigr ),
            \end{multline*}
            achieving the demonstration.
        \end{itemize}
    
\end{proof}

\subsection{Conclusions}

In this \Cref{sec:ModifiedEquations}, we have proposed a way of linking the initial datum of the \lbm scheme $\boldOther{\discreteMoment}(0, \cdot)$ to the initial datum $\solutionCauchyInitial$ of the Cauchy problem \eqref{eq:CauchyInitialDatum}.
This allowed us to propose a modified equation analysis of the initialisation phase---see \Cref{prop:EquivalentEquationInitialisationLocal} and \Cref{prop:ModifiedEquationPrepared}--- making the study of the real behaviour of the numerical schemes possible and find the constraints---see \Cref{prop:LocalInitialisation} and \Cref{prop:PreparedInitialisation}--- under which the initialisation schemes are consistent with the same equation \eqref{eq:CauchyEquation} as the bulk scheme, preventing from having order reductions. 
We have also stressed that controlling the behaviour of the scheme inside the initialisation layer implies a control on the numerical scheme eventually in time (\Cref{prop:MatchEventually}).
The general computations have been done until order $\bigO{\spaceStep}$ but can be carried further to $\bigO{\spaceStep^2}$ and above for specific schemes, as in \Cref{sec:Illustrations} and \Cref{sec:Ginzburg}.
This provides additional information on other features of the schemes close to the beginning of the simulation, such as dissipation and dispersion.

% \FloatBarrier

\section{Examples and numerical simulations}\label{sec:Illustrations}

\Cref{sec:Illustrations} first aims at checking the previously introduced theory concerning consistency on actual numerical simulations on the \scheme{1}{2} scheme (\emph{cf.} \Cref{ex:D1Q2}).
Moreover, the computations of the modified equation shall be pushed one order further providing the dissipation of the \stSchemes{}, the impact of which is precisely quantified on the numerical experiments for a \scheme{1}{2} and \scheme{1}{3} scheme.
Finally,  the example of \scheme{1}{3} scheme paves the way for the general discussion of \Cref{sec:Ginzburg} concerning a more precise counting of the number of initialisation schemes, with important consequences on the dissipation of the numerical schemes. 
We utilize 1d schemes in order to illustrate how to employ the previously developed tools in the most simple fashion. However, they can be used to deal with any spatial dimension $\spatialDimensionality = 1, 2, 3$, with computations becoming more and more involved as the complexity of the scheme grows, due to the role of parasitic eigenvalues.

\subsection{Two-velocities \scheme{1}{2} scheme}\label{sec:D1Q2}

Consider the scheme from \Cref{ex:D1Q2}.
The modified equation of the \bulkScheme{} reads as in \cite{graille2014approximation} and \Cref{thm:EquivEqBulk}
\begin{equation}\label{eq:ModifiedD1Q2Bulk}
    \partial_{\timeVariable} \testFunction(\timeVariable, \spaceVariable) + \latticeVelocity \equilibriumCoefficient_2 \partial_{\spaceVariable}\testFunction(\timeVariable, \spaceVariable) - \latticeVelocity \spaceStep \Bigl (\frac{1}{\relaxationParameter_2} - \frac{1}{2} \Bigr )   (1 - \equilibriumCoefficient_2^2  )  \partial_{\spaceVariable\spaceVariable}\testFunction(\timeVariable, \spaceVariable) = \bigO{\spaceStep^2}, \qquad (\timeVariable, \spaceVariable) \in \nonNegativeReals \times \reals,
\end{equation}
thus to be consistent with the Cauchy problem \eqref{eq:CauchyEquation}, one takes $\equilibriumCoefficient_2 = \transportVelocity/\latticeVelocity$.
For $\relaxationParameter_2 < 2$, the \bulkScheme{} is first-order accurate, thus \iniSchemes{} which are non-consistent with the target conservation law---\emph{i.e.} indeed violating \eqref{eq:LocalBulk} or \eqref{eq:PreparedBulk}---do not degrade the order of convergence. 
For $\relaxationParameter_2 = 2$, the \bulkScheme{} is second-order accurate, thus consistent \iniSchemes{} are needed, \emph{i.e.} verifying  \eqref{eq:LocalBulk} or \eqref{eq:PreparedBulk}.
Observe that the scheme is $L^2$ stable according to \emph{von Neumann} (\emph{i.e.} the roots of the amplification polynomial of the \bulkScheme{} are inside the unit disk and those on the unit disk are simple) under the conditions (\cite{graille2014approximation} and \Cref{app:StabilityD1Q2})
\begin{equation}\label{eq:StabilityConditionsD1Q2}
    \relaxationParameter_2 \in ]0, 2], \qquad \text{and} \qquad
    \begin{cases}
        |\equilibriumCoefficient_2| \leq 1, \qquad &\text{if} \quad \relaxationParameter_2 \in ]0, 2[, \\
        |\equilibriumCoefficient_2| < 1, \qquad &\text{if} \quad \relaxationParameter_2 = 2.
    \end{cases}
\end{equation}
The second conditions are the Courant-Friedrichs-Lewy (CFL) condition, which is known to be strict for the leap-frog scheme.

We consider five different choices of \iniSchemes{}.
They are designed to showcase different facets of the previous theoretical discussion.
More precisely, the first initialisation is the one where all data are taken at equilibrium, which is likely the most common way of initializing \lbm schemes \cite{graille2014approximation,caetano2019result}.
The second and the third initialisations both render a forward centered scheme as \iniScheme{}, which would be unstable if used as bulk scheme. Still, these two initialisations yield different outcomes for the associated numerical simulations, and our theory accounts for this phenomenon.
The fourth initialisation aims at obtaining a Lax-Wendroff \iniScheme{}, which allows to study the effect of a second-order \iniScheme{}.
Finally, the fifth initialisation is inspired by works from the literature \cite{van2009smooth}.
\begin{itemize}
    \item{\strong{Lax-Friedrichs scheme} (LF)}, a first-order consistent scheme which we shall obtain using the local initialisation
    \begin{equation}\label{eq:LaxFriedrichs}
        \initialisationOperator_1 = 1, \qquad \initialisationOperator_2 = \equilibriumCoefficient_2.
    \end{equation}
    Except when $\relaxationParameter_2 = 1$ (where $\nonTriviallyRelaxingMomentsNumber = 0$), the dissipation of the \bulkScheme{} is not matched by the one of the Lax-Friedrichs scheme.

    \item{\strong{Forward centered scheme} (FC)}. This is a first-order consistent scheme which is unstable even under CFL condition \eqref{eq:StabilityConditionsD1Q2} if used as bulk scheme, due to its negative dissipation.
    Still, it is perfectly suitable for the initialisation of the method (see \cite[Chapter 10]{strikwerda2004finite}). Its diffusivity shall not match the one of the \bulkScheme{}, see \eqref{eq:ModifiedD1Q2Bulk}.
    This \iniScheme{} cannot stem from a local initialisation, \emph{i.e.} $\initialisationOperator_1, \initialisationOperator_2 \in \reals$, since the only first-order consistent \iniScheme{} that can be obtained in this way is the Lax-Friedrichs scheme \eqref{eq:LaxFriedrichs}.
    We could unsuccessfully try to generate it by a local initialisation of the conserved moment, that is $\initialisationOperator_1 = 1$ and prepared initialisation of the non-conserved one, thus $\initialisationOperator_2 \in \ringSpaceOperatorsOneD$.
    Considering---see \Cref{app:derivationForwardCenteredD1Q2} for the details---a prepared initialisation for both moments, thus $\initialisationOperator_1, \initialisationOperator_2 \in \ringSpaceOperatorsOneD$, several choices are possible to recover this scheme.
    One is
    \begin{equation}\label{eq:CenteredGood}
        \initialisationOperatorCoefficients_{1, \pm 1} = \frac{1}{2}, \qquad \initialisationOperatorCoefficients_{2, \pm 1} = \mp\frac{ 1  \pm \relaxationParameter_2 \equilibriumCoefficient_2}{2(1 - \relaxationParameter_2)},  \qquad \initialisationOperatorCoefficients_{2, 0} = \frac{\equilibriumCoefficient_2}{1 - \relaxationParameter_2},
    \end{equation}
    with the notation by \eqref{eq:FormInitilisationNonLocal} and agrees with \eqref{eq:PreparedInitial}, \eqref{eq:PreparedInitial2}, and \eqref{eq:PreparedBulk}.
    Another possible choice to obtain the desired scheme would be
    \begin{equation}\label{eq:CenteredBad}
        \initialisationOperatorCoefficients_{1, \pm 2} = \pm \frac{\equilibriumCoefficient_2}{2}, \qquad \initialisationOperatorCoefficients_{1, \pm 1} = \frac{1}{2}, \qquad \initialisationOperatorCoefficients_{2, \pm 2} = -\frac{\equilibriumCoefficient_2(1 \pm \relaxationParameter_2 \equilibriumCoefficient_2)}{2(1-\relaxationParameter_2)}, \qquad \initialisationOperatorCoefficients_{2, \pm 1} = \mp \frac{1 \pm \relaxationParameter_2 \equilibriumCoefficient_2}{2(1 - \relaxationParameter_2)}.
    \end{equation}
    However, this initialisation yields only \eqref{eq:PreparedInitial} but does not fulfill either \eqref{eq:PreparedInitial2} or \eqref{eq:PreparedBulk}.
    This means that in this case $\discreteMoment_1$ is initialized as a first-order perturbation of the datum of the Cauchy problem \eqref{eq:CauchyInitialDatum} and that $\discreteMoment_2$ is not initialized at equilibrium at leading order.

    \item{\strong{Lax-Wendroff scheme} (LW)}. This is a second-order consistent scheme with no dissipation, thus matches the diffusivity of the \bulkScheme{} only when $\relaxationParameter_2 = 2$.
    Remark that since the bulk scheme is at most second-order accurate, it is somehow excessive to initialize with a scheme of the same order.
    Following an analogous procedure to the centered forward scheme, one possible initialisation is $\initialisationOperator_1, \initialisationOperator_2 \in \ringSpaceOperatorsOneD$ with
    \begin{equation}\label{eq:LawWendroff}
        \initialisationOperatorCoefficients_{1, \pm 1} = \frac{1- \equilibriumCoefficient_2^2}{2} , \quad \initialisationOperatorCoefficients_{1, 0} = \equilibriumCoefficient_2^2, \quad \initialisationOperatorCoefficients_{2, \pm 1} = \mp\frac{( 1  \pm \relaxationParameter_2 \equilibriumCoefficient_2)(1- \equilibriumCoefficient_2^2)}{2(1 - \relaxationParameter_2)},  \quad \initialisationOperatorCoefficients_{2, 0} = \frac{\equilibriumCoefficient_2 (1- \relaxationParameter_2 \equilibriumCoefficient_2^2)}{1 - \relaxationParameter_2} ,
    \end{equation}
    according to \eqref{eq:FormInitilisationNonLocal}, which respects \eqref{eq:PreparedInitial}, \eqref{eq:PreparedInitial2}, and \eqref{eq:PreparedBulk}.
    Again, it is also possible to generate initialisations yielding this scheme which do not fulfill \eqref{eq:PreparedInitial2} and \eqref{eq:PreparedBulk}.

    \item{\strong{Smooth initialisation inspired by }\cite{van2009smooth} (RE1)}.
    The idea of this initialisation is to make the most of the terms in the modified equation of the \iniSchemes{} and that of the \bulkScheme{} to match, if possible, without modification of the conserved moment, that is $\initialisationOperator_1 \in \reals$.
    In particular, in our case, this allows to match the numerical diffusion coefficient between the two schemes for every $\relaxationParameter_2 \in ]0, 2]$, as we shall see.
    We adapt Equation (13) from \cite{van2009smooth} by discretising the continuous derivative by a second-order centered formula, having
    \begin{equation}\label{eq:CoefficientInitialisationSmoothInTime}
        \initialisationOperator_1 = 1 \quad \text{and} \quad \initialisationOperator_2 \in \ringSpaceOperatorsOneD, \quad \text{where} \quad \initialisationOperatorCoefficients_{2, \pm 1} = \pm \frac{1 - \equilibriumCoefficient_2^2}{2 \relaxationParameter_2}, \qquad \initialisationOperatorCoefficients_{2, 0} = \equilibriumCoefficient_2,
    \end{equation}
    according to \eqref{eq:FormInitilisationNonLocal}.
    This initialisation fulfills \eqref{eq:PreparedInitial}, \eqref{eq:PreparedInitial2}, and \eqref{eq:PreparedBulk}.
\end{itemize}

\subsubsection{Study of the order of convergence}\label{sec:D1Q2ConvergenceOrder}

To empirically analyze the preservation of the order of the \bulkScheme, we consider the following initial data with different smoothness
\begin{align}
    \text{(a)} \qquad \solutionCauchyInitial (\spaceVariable) &= \chi_{\lvert \spaceVariable \rvert \leq 1/2}(\spaceVariable) \in H^{\sigma}, \quad \text{for any} \quad \sigma < \sigma_0 = 1/2. \label{eq:InitialDatumA}\\
    \text{(b)} \qquad \solutionCauchyInitial (\spaceVariable) &= (1-2\lvert \spaceVariable \rvert)\chi_{\lvert \spaceVariable \rvert \leq 1/2}(\spaceVariable)  \in H^{\sigma}, \quad \text{for any} \quad \sigma < \sigma_0 = 3/2. \label{eq:InitialDatumB}\\
    \text{(c)} \qquad \solutionCauchyInitial (\spaceVariable) &= \cos^2{(\pi \spaceVariable)}\chi_{\lvert \spaceVariable \rvert \leq 1/2}(\spaceVariable) \in H^{\sigma}, \quad \text{for any} \quad \sigma < \sigma_0 = 5/2. \label{eq:InitialDatumC}\\
    \text{(d)} \qquad \solutionCauchyInitial (\spaceVariable) &= \text{exp}\left (-1/{(1- \lvert 2\spaceVariable \rvert^2)} \right )\chi_{\lvert \spaceVariable \rvert \leq 1/2}(\spaceVariable) \in C_{c}^{\infty}, \label{eq:InitialDatumD}
\end{align}
issued from \cite{bellotti2021fd}.
As common in the linear framework, we monitor the $L^2$ errors.
We simulate for $\latticeVelocity = 1$, $\equilibriumCoefficient_2  = \transportVelocity/\latticeVelocity = 1/2$ with final time $1/2$ and on a bounded domain $[-1, 1]$ with periodic boundary conditions.

\begin{table} 
    \caption{\label{tab:orderConvergenceD1Q2}Expected orders of convergence in $\spaceStep$ for the \scheme{1}{2} scheme.}   
    \begin{center}
        \begin{tabular}{c|c|c}
            Test & Bulk scheme 1st order ($0 < \relaxationParameter_2 < 2$) & Bulk scheme 2nd order ($\relaxationParameter_2 = 2$) \\
            \hline
            (a) - \eqref{eq:InitialDatumA} & order $1/4$ & order $1/3$ \\
            (b) - \eqref{eq:InitialDatumB} & order $3/4$ & order $1$ \\
            (c) - \eqref{eq:InitialDatumC} & order $1$   & order $5/3$ \\
            (d) - \eqref{eq:InitialDatumD} & order $1$   & order $2$
        \end{tabular}
    \end{center}     
\end{table}

We expect the scheme  to be convergent following the orders given in \Cref{tab:orderConvergenceD1Q2} \cite{strikwerda2004finite, bellotti2021fd} and observe orders exceeding one provided that both following conditions are met:
\begin{enumerate}
    \item the \iniScheme{} is at least first-order consistent with the Cauchy problem \eqref{eq:CauchyEquation};
    \item the initial filter on the initial datum $\initialisationOperator_1$ is such that $\termAtOrder{\initialisationOperatorAsymptotic_1}{1} = 0$, meaning that it perturbs from $\bigO{\spaceStep^2}$ or for higher orders.
\end{enumerate}

% Colors for the legend
\definecolor{ColorS11}{RGB}{53, 24, 62}
\definecolor{ColorS12}{RGB}{112, 31, 87}
\definecolor{ColorS14}{RGB}{173, 23, 89}
\definecolor{ColorS16}{RGB}{226, 50, 66}
\definecolor{ColorS18}{RGB}{244, 118, 81}
\definecolor{ColorS20}{RGB}{247, 181, 143}

  \begin{figure} 
    \begin{center}
        Initialisation \eqref{eq:CenteredGood}  \\
        \includegraphics{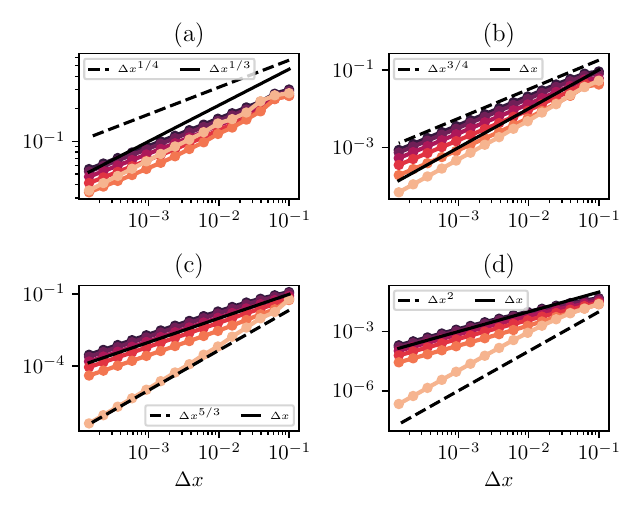} \\
        Initialisation \eqref{eq:CenteredBad} \\
        \includegraphics{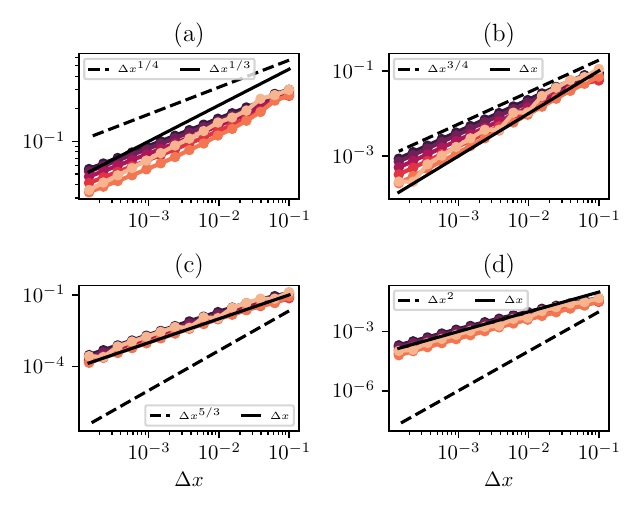}            
    \end{center}\caption{\label{fig:d1q2_centered}$L^2$ errors at the final time for two forward centered initialisations \eqref{eq:CenteredGood} (top) and \eqref{eq:CenteredBad} (bottom). Since the letter irremediably perturbs the conserved moment feeding the \bulkScheme{}, the orders of convergence above one are lowered. Color legend for the relaxation parameter $\relaxationParameter_2$: \textcolor{ColorS11}{{\LARGE $\bullet$} for $\relaxationParameter_2 = 1.1$}, \textcolor{ColorS12}{{\LARGE $\bullet$} for $\relaxationParameter_2 = 1.2$}, \textcolor{ColorS14}{{\LARGE $\bullet$} for $\relaxationParameter_2 = 1.4$}, \textcolor{ColorS16}{{\LARGE $\bullet$} for $\relaxationParameter_2 = 1.6$}, \textcolor{ColorS18}{{\LARGE $\bullet$} for $\relaxationParameter_2 = 1.8$}, and \textcolor{ColorS20}{{\LARGE $\bullet$} for $\relaxationParameter_2 = 2$}.}
  \end{figure}

The results are in agreement with the theory.
We just present few of them for the sake of avoiding redundancy, in particular, those concerning the forward centered initialisation schemes \eqref{eq:CenteredGood} and \eqref{eq:CenteredBad} given in \Cref{fig:d1q2_centered}.
As expected, despite the fact that the obtained \iniScheme{} is the same, \eqref{eq:CenteredBad} pollutes the initial datum with respect to the one from the Cauchy problem \eqref{eq:CauchyInitialDatum} due to a first-order term $\termAtOrder{\initialisationOperatorAsymptotic_1}{1} \neq 0$.
Hence, even for $\relaxationParameter_2 = 2$, the order of convergence is lowered.
We shall reinterpret why  \eqref{eq:CenteredBad} yields a poor behaviour.

% \FloatBarrier

\subsubsection{Study of the time smoothness of the numerical solution}\label{sec:TimeSmoothnessD1Q2}

\begin{figure}[h]
    \begin{center}
        \includegraphics[scale=0.99]{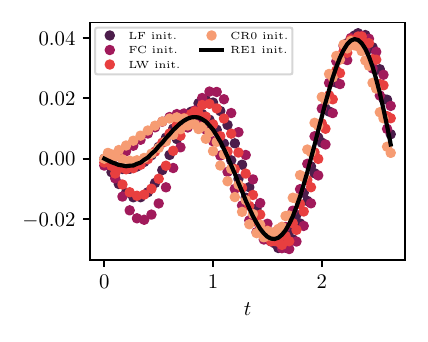}
        \includegraphics[scale=0.99]{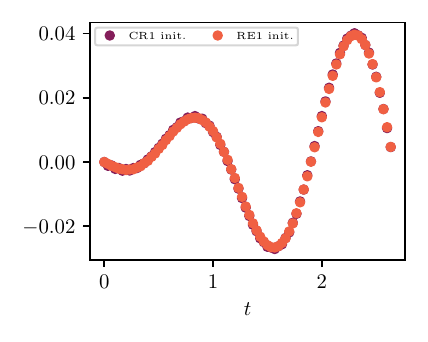}
    \end{center}\caption{\label{fig:InitialError}Test for the smoothness in time close to $\timeVariable = 0$ for $\relaxationParameter_2 = 1.99$: difference between exact and numerical solution at the eighth lattice point.}
  \end{figure}

  \begin{figure}[h]
    \begin{center}
        \includegraphics{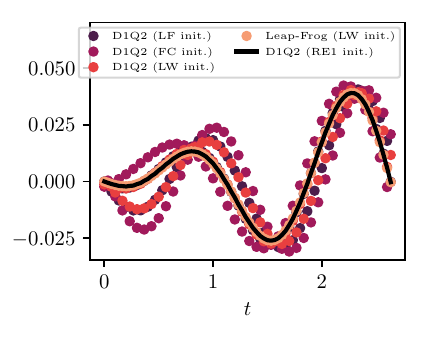}
    \end{center}\caption{\label{fig:InitialErrors2}Test for the smoothness in time close to $\timeVariable = 0$ for $\relaxationParameter_2 = 2$: difference between exact and numerical solution at the eighth lattice point. Compared to \Cref{fig:InitialError}, CR0 and CR1 from \cite{van2009smooth} cannot be used.}
  \end{figure}

We have observed that the only proposed initialisation matching the dissipation of the bulk scheme for every $\relaxationParameter_2 \in ]0, 2]$ is the one given by \eqref{eq:CoefficientInitialisationSmoothInTime}.
To confirm that this is the origin of its good performances in term of time smoothness of the discrete solution close to the initial time, we repeat the numerical experiment found in \cite{van2009smooth}.
The simulation is carried on the periodic domain $[0, 1]$ discretized with $\spaceStep = 1/30$, $\relaxationParameter_2 = 1.99$, $\latticeVelocity = 1$ and $\equilibriumCoefficient_2  = \transportVelocity / \latticeVelocity = 0.66$.
The initial datum of the Cauchy problem is $\solutionCauchyInitial(\spaceVariable) = \cos(2\pi \spaceVariable)$.

We initialize using the Lax-Friedrichs initialisation \eqref{eq:LaxFriedrichs} (coinciding with what \cite{van2009smooth} calls RE0 scheme), the forward centered initialisation \eqref{eq:CenteredGood}, the Lax-Wendroff initialisation \eqref{eq:LawWendroff}, the RE1 initialisation \eqref{eq:CoefficientInitialisationSmoothInTime} and the implicit initialisations CR0 and CR1 proposed in \cite{van2009smooth}, which are not detailed here.
We measure the difference between the exact solution and the approximate solution at the eighth cell of the lattice.
The results are given in \Cref{fig:InitialError} and are in accordance with the previous analysis as well as the computations in \cite{van2009smooth}.
Indeed, since the dissipation of the \bulkScheme{} is almost zero for $\relaxationParameter_2 = 1.99$, the Lax-Wendroff scheme is supposed to almost match this dissipation.
However here, the same phenomenon that took place in \Cref{sec:D1Q2ConvergenceOrder} at leading order for the forward centered initialisation between \eqref{eq:CenteredGood} and \eqref{eq:CenteredBad}, due to the introduction of a first-order perturbation on the conserved moment, now takes place for \eqref{eq:LawWendroff}, because it introduces a second-order perturbation on the conserved moment $\discreteMoment_1$ feeding the multi-step \bulkScheme{} \eqref{eq:BulkSchemes}, namely $\termAtOrder{\initialisationOperatorAsymptotic_1}{2} \neq 0$.
Taking $\relaxationParameter_2 = 2$, hence no dissipation from the bulk scheme, we obtain the result in \Cref{fig:InitialErrors2}, which is not different from the previous one (notice that here the implicit initialisation RE1 cannot be utilized).
In this Figure, we have also repeated the simulation using a leap-frog scheme (coinciding with the \bulkScheme{}) initialized with a Lax-Wendroff scheme, which conversely leads the expected smoothness, since we do not have to filter the initial datum of the Cauchy problem.

% 
% \FloatBarrier

\subsubsection{Theoretical analysis using the modified equations}\label{sec:ModEqD1Q2order2}

Let us proceed to a more quantitative study of what can be observed in \Cref{fig:InitialError} and \Cref{fig:InitialErrors2}.
To this end, we push the computation of the modified equation of the \stSchemes{} for $\indiceTime \in \nonZeroNaturals$ to order $\bigO{\spaceStep^2}$.
At the end of this section, we will fully justify why this procedure allows us to study the behavior of the numerical solution even when we cannot suppose that it is smooth in the time variable, which seems to be the case for the previous numerical experiments.
To complete the previous computations, we are left to consider
\begin{align*}
    \termAtOrder{(\schemeMatrixAsymptotic^{\indiceTime})}{2} 
    &= \sum \{ \text{per. of }\termAtOrder{\schemeMatrixAsymptotic}{0} \text{ (}\indiceTime-1\text{ tm.) and }\termAtOrder{\schemeMatrixAsymptotic}{2} \text{ (once)}\} + \sum \{ \text{per. of }\termAtOrder{\schemeMatrixAsymptotic}{0} \text{ (}\indiceTime-2\text{ tm.) and }\termAtOrder{\schemeMatrixAsymptotic}{1} \text{ (twice)}\} \\
    &= \sum\nolimits_{\indiceFreeOne = 0}^{\indiceFreeOne = \indiceTime - 1} (\termAtOrder{\schemeMatrixAsymptotic}{0})^{\indiceFreeOne} \termAtOrder{\schemeMatrixAsymptotic}{2} (\termAtOrder{\schemeMatrixAsymptotic}{0})^{\indiceTime - 1 - \indiceFreeOne}  + \sum\nolimits_{\indiceFreeOne = 0}^{\indiceFreeOne = \indiceTime - 2} \sum\nolimits_{\indiceFreeTwo = 0}^{ \indiceFreeTwo = \indiceTime - 2 - \indiceFreeOne}(\termAtOrder{\schemeMatrixAsymptotic}{0})^{\indiceFreeOne} \termAtOrder{\schemeMatrixAsymptotic}{1}  (\termAtOrder{\schemeMatrixAsymptotic}{0})^{\indiceFreeTwo} \termAtOrder{\schemeMatrixAsymptotic}{1}  (\termAtOrder{\schemeMatrixAsymptotic}{0})^{\indiceTime - 2 - \indiceFreeOne - \indiceFreeTwo}.\nonumber \\
\end{align*}
Using the matrix from the particular \scheme{1}{2} scheme, we obtain for every $\indiceTime \in \nonZeroNaturals$
\begin{align*}
    \termAtOrder{(\schemeMatrixAsymptotic^{\indiceTime})}{2}_{11} &= \Bigl ( \frac{\indiceTime}{2}  + \sum\nolimits_{\indiceFreeOne = 0}^{\indiceFreeOne = \indiceTime - 2} \sum\nolimits_{\indiceFreeTwo = 1}^{ \indiceFreeTwo = \indiceTime - 1 - \indiceFreeOne} (1-\relaxationParameter_2)^{\indiceFreeTwo} \\
    &+ \equilibriumCoefficient_2^2 \sum\nolimits_{\indiceFreeOne = 0}^{\indiceFreeOne = \indiceTime - 2} \sum\nolimits_{\indiceFreeTwo = 0}^{ \indiceFreeTwo = \indiceTime - 2 - \indiceFreeOne} \Bigl ( \relaxationParameter_2^2  + \relaxationParameter_2 (1-\relaxationParameter_2) \polynomialEquilibrium_{\indiceTime-2-\indiceFreeOne - \indiceFreeTwo}(\relaxationParameter_2) + (1-\relaxationParameter_2) \polynomialEquilibrium_{\indiceFreeTwo}(\relaxationParameter_2) \polynomialEquilibrium_{\indiceTime-1-\indiceFreeOne - \indiceFreeTwo}(\relaxationParameter_2) \Bigr ) \Bigr )\partial_{\spaceVariable \spaceVariable},
\end{align*}
and 
\begin{equation*}
    \termAtOrder{(\schemeMatrixAsymptotic^{\indiceTime})}{2}_{12} = \equilibriumCoefficient_2\sum_{\indiceFreeOne = 0}^{\indiceTime - 2} \sum_{\indiceFreeTwo = 0}^{\indiceTime - 2 - \indiceFreeOne} (1-\relaxationParameter_2)^{\indiceTime - 1 - \indiceFreeOne - \indiceFreeTwo} \polynomialEquilibrium_{\indiceFreeTwo+1}(\relaxationParameter_2) \partial_{\spaceVariable\spaceVariable}.
\end{equation*}

For all the initialisations we have considered, we have $\termAtOrder{\initialisationOperatorAsymptotic_1}{0} = 1$, corresponding to \eqref{eq:PreparedInitial}.
No other assumption is needed in the following derivation.
With the usual procedure, we obtain for $\indiceTime \in \nonZeroNaturals$ and $\spaceVariable \in \reals$
\begin{align*}
    & \partial_{\timeVariable}\testFunction(0, \spaceVariable)- \frac{\latticeVelocity}{\indiceTime} \Bigl ( \termAtOrder{(\schemeMatrixAsymptotic^{\indiceTime})}{1}_{11} +  \termAtOrder{(\schemeMatrixAsymptotic^{\indiceTime})}{1}_{12} \termAtOrder{\initialisationOperatorAsymptotic_2}{0} + \termAtOrder{\initialisationOperatorAsymptotic_1}{1}\Bigr ) \testFunction(0, \spaceVariable)\\ 
    &+ \indiceTime  \frac{\spaceStep}{2\latticeVelocity} \partial_{\timeVariable\timeVariable}\testFunction(0, \spaceVariable)  -  \frac{\latticeVelocity \spaceStep}{\indiceTime} \Bigl (  \termAtOrder{(\schemeMatrixAsymptotic^{\indiceTime})}{2}_{11} +  \termAtOrder{(\schemeMatrixAsymptotic^{\indiceTime})}{2}_{12} \termAtOrder{\initialisationOperatorAsymptotic_2}{0} + \termAtOrder{(\schemeMatrixAsymptotic^{\indiceTime})}{1}_{11} \termAtOrder{\initialisationOperatorAsymptotic_1}{1} + \termAtOrder{(\schemeMatrixAsymptotic^{\indiceTime})}{1}_{12} \termAtOrder{\initialisationOperatorAsymptotic_2}{1} + \termAtOrder{\initialisationOperatorAsymptotic_1}{2}\Bigr ) \testFunction(0, \spaceVariable) = \bigO{\spaceStep^2}.
\end{align*}

\begin{itemize}
    \item{\strong{Lax-Friedrichs} \eqref{eq:LaxFriedrichs}}.
    \begin{proposition}
        Under acoustic scaling, the modified equations for the \stSchemes{} for the Lax-Friedrichs initialisation given by \eqref{eq:LaxFriedrichs}  are, for $\indiceTime \in \nonZeroNaturals$
        \begin{equation}\label{eq:ModifiedD1Q2InitialLaxFridrichs}
            \partial_{\timeVariable} \testFunction(0, \spaceVariable) + \latticeVelocity \equilibriumCoefficient_2 \partial_{\spaceVariable} \testFunction(0, \spaceVariable) - \latticeVelocity \spaceStep \Bigl ( \frac{1}{2} + \sum_{\indiceFreeOne = 1}^{\indiceTime - 1}\Bigl (1 - \frac{\indiceFreeOne}{\indiceTime}\Bigr )(1-\relaxationParameter_2)^{\indiceFreeOne} \Bigr ) ( 1 - \equilibriumCoefficient_2^2 )\partial_{\spaceVariable\spaceVariable}  \testFunction(0, \spaceVariable) = \bigO{\spaceStep^2}, \qquad \spaceVariable \in \reals. 
        \end{equation}
    \end{proposition}
    \begin{proof}
    This initialisation fulfils the requirements by \Cref{prop:PreparedInitialisation}, which leads to 
    \begin{equation}\label{eq:ModifiedD1Q2InitialLaxFridrichsIncomplete}
        \partial_{\timeVariable}\testFunction(0, \spaceVariable) + \latticeVelocity \equilibriumCoefficient_2 \partial_{\spaceVariable} \testFunction(0, \spaceVariable) 
        + \indiceTime \frac{\spaceStep}{2\latticeVelocity} \partial_{\timeVariable\timeVariable}\testFunction(0, \spaceVariable)  - \frac{\latticeVelocity \spaceStep}{\indiceTime} \Bigl (  \termAtOrder{(\schemeMatrixAsymptotic^{\indiceTime})}{2}_{11} +  \termAtOrder{(\schemeMatrixAsymptotic^{\indiceTime})}{2}_{12} \equilibriumCoefficient_2 \Bigr ) \testFunction(0, \spaceVariable) = \bigO{\spaceStep^2},
    \end{equation}
    for $\indiceTime \in \nonZeroNaturals$ and $\spaceVariable \in \reals$.
    Using the previous order to get rid of the second-order time derivative $\partial_{\timeVariable \timeVariable}$ \cite{warming1974modified,carpentier1997derivation, dubois2008equivalent,dubois2019nonlinear} boils down to
    \begin{equation*}
        \partial_{\timeVariable}\testFunction(0, \spaceVariable) + \latticeVelocity \equilibriumCoefficient_2 \partial_{\spaceVariable} \testFunction(0, \spaceVariable) 
        - \latticeVelocity \spaceStep \Bigl ( - \frac{\indiceTime}{2}\equilibriumCoefficient_2^2  \partial_{\spaceVariable \spaceVariable}+ \frac{1}{\indiceTime} \Bigl (  \termAtOrder{(\schemeMatrixAsymptotic^{\indiceTime})}{2}_{11} +  \termAtOrder{(\schemeMatrixAsymptotic^{\indiceTime})}{2}_{12} \equilibriumCoefficient_2 \Bigr ) \Bigr ) \testFunction(0, \spaceVariable)= \bigO{\spaceStep^2},
    \end{equation*}
    for $\indiceTime \in \nonZeroNaturals$ and $\spaceVariable \in \reals $.
    We are left to deal with the diffusion term, for $\indiceTime \in \nonZeroNaturals$
    \begin{multline*}
        \termAtOrder{(\schemeMatrixAsymptotic^{\indiceTime})}{2}_{11} +  \termAtOrder{(\schemeMatrixAsymptotic^{\indiceTime})}{2}_{12} \equilibriumCoefficient_2 
        = \Bigl ( \frac{\indiceTime}{2}  + \sum_{\indiceFreeOne = 0}^{\indiceTime - 2} \sum_{\indiceFreeTwo = 1}^{ \indiceTime - 1 - \indiceFreeOne} (1-\relaxationParameter_2)^{\indiceFreeTwo} + \equilibriumCoefficient_2^2 \sum_{\indiceFreeOne = 0}^{\indiceTime - 2} \sum_{\indiceFreeTwo = 0}^{ \indiceTime - 2 - \indiceFreeOne} \Bigl ( \relaxationParameter_2^2  + \relaxationParameter_2 (1-\relaxationParameter_2) \polynomialEquilibrium_{\indiceTime-2-\indiceFreeOne - \indiceFreeTwo}(\relaxationParameter_2) \\
        + (1-\relaxationParameter_2) \polynomialEquilibrium_{\indiceFreeTwo}(\relaxationParameter_2) \polynomialEquilibrium_{\indiceTime-1-\indiceFreeOne - \indiceFreeTwo}(\relaxationParameter_2) + (1-\relaxationParameter_2)^{\indiceTime - 1 - \indiceFreeOne - \indiceFreeTwo} \polynomialEquilibrium_{\indiceFreeTwo+1}(\relaxationParameter_2)\Bigr ) \Bigr )\partial_{\spaceVariable \spaceVariable}.
    \end{multline*}
    Using the expression for $\polynomialEquilibrium_{\indiceFreeOne}$ to handle the last term shows that
    \begin{equation*}
        \termAtOrder{(\schemeMatrixAsymptotic^{\indiceTime})}{2}_{11} +  \termAtOrder{(\schemeMatrixAsymptotic^{\indiceTime})}{2}_{12} \equilibriumCoefficient_2 
        = \Bigl ( \frac{\indiceTime}{2}  + \sum_{\indiceFreeOne = 1}^{\indiceTime - 1}(\indiceTime - \indiceFreeOne)(1-\relaxationParameter_2)^{\indiceFreeOne} + \equilibriumCoefficient_2^2  \Bigl (\frac{\indiceTime (\indiceTime - 1)}{2}-\sum_{\indiceFreeOne = 1}^{\indiceTime - 1}(\indiceTime - \indiceFreeOne)(1-\relaxationParameter_2)^{\indiceFreeOne}  \Bigr ) \Bigr )\partial_{\spaceVariable \spaceVariable},  
    \end{equation*}
    for $\indiceTime \in \nonZeroNaturals$.
    Plugging into the expansion \eqref{eq:ModifiedD1Q2InitialLaxFridrichsIncomplete} provides
    \begin{equation*}
        \partial_{\timeVariable} \testFunction(0, \spaceVariable) + \latticeVelocity \equilibriumCoefficient_2 \partial_{\spaceVariable} \testFunction(0, \spaceVariable) - \latticeVelocity \spaceStep \Bigl ( \frac{1}{2} + \sum_{\indiceFreeOne = 1}^{\indiceTime - 1}\Bigl (1 - \frac{\indiceFreeOne}{\indiceTime}\Bigr )(1-\relaxationParameter_2)^{\indiceFreeOne} \Bigr ) ( 1 - \equilibriumCoefficient_2^2 )\partial_{\spaceVariable\spaceVariable}  \testFunction(0, \spaceVariable) = \bigO{\spaceStep^2}, \qquad \indiceTime \in \nonZeroNaturals, 
    \end{equation*}
    for $\spaceVariable \in \reals$.
\end{proof}
    \begin{figure}
        \begin{center}
            \includegraphics{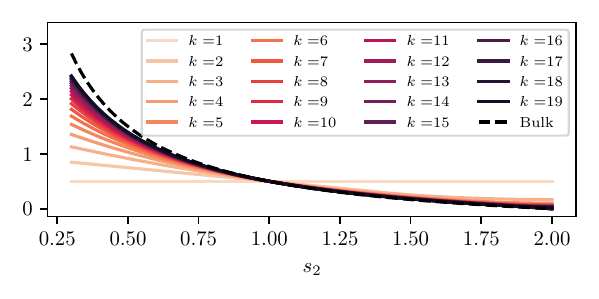}
        \end{center}\caption{\label{fig:d1q2_equilibrium_diffusion}Plot of the polynomial ${1}/{2} + \sum\nolimits_{\indiceFreeOne = 1}^{\indiceFreeOne = \indiceTime - 1} (1 - {\indiceFreeOne}/{\indiceTime} ) (1-\relaxationParameter_2)^{\indiceFreeOne} $ appearing in \eqref{eq:ModifiedD1Q2InitialLaxFridrichs} for different $\indiceTime$ compared to $1/\relaxationParameter_2 - 1/2$ (bulk).}
      \end{figure}
    This proves once more that the origin of the initial boundary layer is the mismatch in the dissipation coefficient of the scheme, see \Cref{fig:d1q2_equilibrium_diffusion}.
    Of course, it must be kept in mind that these expansions are meaningful as long as $\indiceTime \timeStep \ll 1$, this is, for small times.
    However, from the simulations and \Cref{fig:d1q2_equilibrium_diffusion}, we see that the boundary layer damps in time, since the dissipation coefficient  in \eqref{eq:ModifiedD1Q2InitialLaxFridrichs} converges to the bulk one in \eqref{eq:ModifiedD1Q2Bulk} by taking the formal limit $\indiceTime \to +\infty$:
    \begin{equation*}
        \lim_{\indiceTime \to +\infty} \Bigl ( \frac{1}{2} + \sum_{\indiceFreeOne = 1}^{\indiceTime - 1}\Bigl (1 - \frac{\indiceFreeOne}{\indiceTime}  \Bigr ) (1-\relaxationParameter_2)^{\indiceFreeOne} \Bigr )  = \lim_{\indiceTime \to +\infty} \Bigl ( \frac{1}{2} + \frac{(1-\relaxationParameter_2)^{\indiceTime + 1}}{\indiceTime \relaxationParameter_2^2} - \frac{(1-\relaxationParameter_2)}{\indiceTime \relaxationParameter_2^2} + \frac{(1-\relaxationParameter_2)}{\relaxationParameter_2} \Bigr ) = \frac{1}{\relaxationParameter_2} - \frac{1}{2},
    \end{equation*}
    unsurprisingly by virtue of \Cref{prop:LongTimebehaviour}.
    We see that---as previously claimed---this formal limit holds regardless of the fulfilment of the CFL condition. However, it strongly depends on the fact that $\relaxationParameter_2 \leq 2$, otherwise, it would not hold and indeed exponentially diverge.
    We can also study the behaviour for $\relaxationParameter_2 \simeq 2$:
    \begin{equation*}
        \lim_{\relaxationParameter_2 \to 2^{-}} \Bigl ( \frac{1}{2} + \sum_{\indiceFreeOne = 1}^{\indiceTime - 1}\Bigl (1 - \frac{\indiceFreeOne}{\indiceTime}  \Bigr ) (1-\relaxationParameter_2)^{\indiceFreeOne} \Bigr ) =  \frac{1}{2} + \sum_{\indiceFreeOne = 1}^{ \indiceTime - 1}\Bigl (1 - \frac{\indiceFreeOne}{\indiceTime}  \Bigr ) (-1)^{\indiceFreeOne} = \frac{1 - (-1)^{\indiceTime}}{4\indiceTime} = 
        \begin{cases}
            0, \qquad &\text{for }\indiceTime\text{ even}, \\
            \tfrac{1}{2\indiceTime}, \qquad &\text{for }\indiceTime\text{ odd},
        \end{cases}
    \end{equation*}
    for $\indiceTime \in \nonZeroNaturals$.
    This explains why the errors in  \Cref{fig:InitialError} and \Cref{fig:InitialErrors2} are close to the ones of RE1 \eqref{eq:CoefficientInitialisationSmoothInTime} (up to high order contributions) for even time steps.
    On the one hand, for $\indiceTime$ even, the dissipation of the \bulkScheme{} is matched by the \stSchemes{}, producing good agreement. On the other hand, for $\indiceTime$ odd, the dissipation is strictly positive, though decreasing linearly with $\indiceTime$, creating the jumping behaviour of the errors. 
    This suggests that the damping of the initial boundary layer should be proportional to $\timeVariable^{-1}$ and explains the discrepancies with respect to RE1 \eqref{eq:CoefficientInitialisationSmoothInTime} for the odd time steps.
    Finally, observe that this decoupling---even as far as the dissipation is concerned---between even and odd time steps for $\relaxationParameter_2 = 2$ is expected since the \bulkScheme{} is a leap-frog.

    \item{\strong{Forward centered scheme} \eqref{eq:CenteredGood}}. 
    We have, see \Cref{app:DerivationModifiedEquationsD1Q2} for the proof:
    \begin{proposition}
        Under acoustic scaling, the modified equations for the \stSchemes{} for the forward centered initialisation given by \eqref{eq:CenteredGood}  are, for $\indiceTime \in \nonZeroNaturals$
        \begin{align}
            \partial_{\timeVariable} \testFunction(0, \spaceVariable) &+ \latticeVelocity \equilibriumCoefficient_2 \partial_{\spaceVariable} \testFunction(0, \spaceVariable) \label{eq:ModifiedD1Q2InitialCenteredGood} \\
            &- \latticeVelocity \spaceStep \Bigl ( \Bigl ( \frac{1}{2} + \sum_{\indiceFreeOne = 1}^{\indiceTime - 1}\Bigl (1 - \frac{\indiceFreeOne}{\indiceTime}\Bigr )(1-\relaxationParameter_2)^{\indiceFreeOne} \Bigr )  ( 1 - \equilibriumCoefficient_2^2  )+ \frac{1}{2\indiceTime} \Bigl ( 1 - 2\sum_{\indiceFreeOne = 0}^{\indiceTime - 1} (1- \relaxationParameter_2)^{\indiceFreeOne} \Bigr ) \Bigl ) \partial_{\spaceVariable \spaceVariable} \testFunction(0, \spaceVariable) = \bigO{\spaceStep^2},\nonumber
        \end{align}
        with $\spaceVariable \in \reals$.
    \end{proposition}
    Again, according to \Cref{prop:LongTimebehaviour}, the bulk viscosity coefficient is asymptotically reached, since    
    \begin{equation*}
        \lim_{\indiceTime \to +\infty} \Bigl ( \Bigl ( \frac{1}{2} + \sum_{\indiceFreeOne = 1}^{\indiceTime - 1}\Bigl (1 - \frac{\indiceFreeOne}{\indiceTime}\Bigr )(1-\relaxationParameter_2)^{\indiceFreeOne} \Bigr )  ( 1 - \equilibriumCoefficient_2^2 ) + \frac{1}{2\indiceTime} \Bigl ( 1 - 2\sum_{\indiceFreeOne = 0}^{\indiceTime - 1} (1- \relaxationParameter_2)^{\indiceFreeOne} \Bigr ) \Bigl ) = \Bigl ( \frac{1}{\relaxationParameter_2} - \frac{1}{2} \Bigr )  ( 1 - \equilibriumCoefficient_2^2  ).
    \end{equation*}
    Concerning the behaviour close to $\relaxationParameter_2 \simeq 2$, we have
    \begin{multline*}
        \lim_{\relaxationParameter_2 \to 2^{-}} \Bigl ( \Bigl ( \frac{1}{2} + \sum_{\indiceFreeOne = 1}^{\indiceTime - 1}\Bigl (1 - \frac{\indiceFreeOne}{\indiceTime}\Bigr )(1-\relaxationParameter_2)^{\indiceFreeOne} \Bigr )  ( 1 - \equilibriumCoefficient_2^2  ) + \frac{1}{2\indiceTime} \Bigl ( 1 - 2\sum_{\indiceFreeOne = 0}^{\indiceTime - 1} (1- \relaxationParameter_2)^{\indiceFreeOne} \Bigr ) \Bigl ) \\
        = \frac{(1 - (-1)^{\indiceTime})}{4\indiceTime}  ( 1 - \equilibriumCoefficient_2^2 ) + \frac{(-1)^{\indiceTime}}{2\indiceTime}= 
        \begin{cases}
            \tfrac{1}{2\indiceTime}, \qquad &\text{for }\indiceTime\text{ even}, \\
            -\tfrac{\equilibriumCoefficient^2_2}{2 \indiceTime}, \qquad &\text{for }\indiceTime\text{ odd},
        \end{cases}
    \end{multline*}
    for $\indiceTime \in \nonZeroNaturals$.
    We observe that the even steps of \stSchemes{} have the same diffusivity as the odd steps for the Lax-Friedrichs initialisation \eqref{eq:LaxFriedrichs}, whereas the odd ones have negative diffusivity, which remains from having an \iniScheme{} with negative dissipation, coupled with the fact that the \bulkScheme{} is a leap-frog scheme.
    The question which might be risen is on how the overall scheme can remain stable.
    In terms of Finite Differences, the choice of initial datum only changes the spectrum of the data feeding the \bulkScheme{}, which is stable under \eqref{eq:StabilityConditionsD1Q2}, for every initial datum.
    Concerning the previous computation, we have that under the CFL condition $-\equilibriumCoefficient_2^2/(2 \indiceTime) \geq -1/(2 \indiceTime)$, hence steps with negative dissipation are compensated by steps with sufficiently positive dissipation, yielding an overall stable scheme.

    \item{\strong{Forward centered scheme} \eqref{eq:CenteredBad}}. For this scheme, it is useless to analyze until second order because we know that issues start at $\bigO{\spaceStep}$, see \Cref{sec:D1Q2ConvergenceOrder}.
    We have, see \Cref{app:DerivationModifiedEquationsD1Q2}:
    \begin{proposition}
        Under acoustic scaling, the modified equations for the \stSchemes{} for the forward centered initialisation given by \eqref{eq:CenteredBad}  are, for $\indiceTime \in \nonZeroNaturals$
        \begin{equation*}
            \partial_{\timeVariable}\testFunction(0, \spaceVariable)+ \latticeVelocity \equilibriumCoefficient_2 \Bigl ( 1 + \frac{2}{\indiceTime} \Bigl ( 1 - \sum_{\indiceFreeOne = 0}^{\indiceTime - 1}(1 - \relaxationParameter_2)^{\indiceFreeOne} \Bigr ) \Bigr ) \partial_{\spaceVariable} \testFunction(0, \spaceVariable)
            = \bigO{\spaceStep}, \qquad \spaceVariable \in \reals.
        \end{equation*}
    \end{proposition}

    Unsurprisingly, the \iniScheme{} is consistent ($\indiceTime = 1$), but the general \stSchemes{} ($\indiceTime > 1$) are not.
    This does not prevent the overall scheme to converge, since $\termAtOrder{\initialisationOperatorAsymptotic_1}{0} = 1$ but only at first-order even when $\relaxationParameter_2 = 2$, see \Cref{fig:d1q2_centered}.
    Following \Cref{prop:LongTimebehaviour}
    \begin{equation*}
        \lim_{\indiceTime \to +\infty}\Bigl ( 1 + \frac{2}{\indiceTime} \Bigl ( 1 - \sum_{\indiceFreeOne = 0}^{\indiceTime - 1}(1 - \relaxationParameter_2)^{\indiceFreeOne} \Bigr ) \Bigr ) = 1.
    \end{equation*}

    \item{\strong{Lax-Wendroff} \eqref{eq:LawWendroff}}. We have, \emph{cf.} \Cref{app:DerivationModifiedEquationsD1Q2}:
    \begin{proposition}
        Under acoustic scaling, the modified equations for the \stSchemes{} for the Lax-Wendroff initialisation given by \eqref{eq:LawWendroff}  are, for $\indiceTime \in \nonZeroNaturals$
        \begin{align}
            \partial_{\timeVariable} \testFunction(0, \spaceVariable) &+ \latticeVelocity \equilibriumCoefficient_2 \partial_{\spaceVariable} \testFunction(0, \spaceVariable) \label{eq:ModifiedD1Q2InitialLaxWendroff} \\
            &- \latticeVelocity \spaceStep \Bigl (\frac{1}{2} + \sum_{\indiceFreeOne = 1}^{\indiceTime - 1}\Bigl (1 - \frac{\indiceFreeOne}{\indiceTime}\Bigr )(1-\relaxationParameter_2)^{\indiceFreeOne} + \frac{1}{2\indiceTime} \Bigl ( 1 - 2\sum_{\indiceFreeOne = 0}^{\indiceTime - 1} (1- \relaxationParameter_2)^{\indiceFreeOne} \Bigr ) \Bigl )  ( 1 -  \equilibriumCoefficient_2^2 ) \partial_{\spaceVariable \spaceVariable} \testFunction(0, \spaceVariable) = \bigO{\spaceStep^2}, \nonumber
        \end{align}
        for $\spaceVariable \in \reals$.
    \end{proposition}
    As expected, the dissipation coefficients tend to the one of the bulk scheme for $\indiceTime \to +\infty$ and for $\relaxationParameter_2 \simeq 2$, we find
    \begin{equation*}
        \lim_{\relaxationParameter_2 \to 2^{-}} \Bigl (\frac{1}{2} + \sum_{\indiceFreeOne = 1}^{\indiceTime - 1}\Bigl (1 - \frac{\indiceFreeOne}{\indiceTime}\Bigr )(1-\relaxationParameter_2)^{\indiceFreeOne} + \frac{1}{2\indiceTime} \Bigl ( 1 - 2\sum_{\indiceFreeOne = 0}^{\indiceTime - 1} (1- \relaxationParameter_2)^{\indiceFreeOne} \Bigr ) \Bigl ) = \frac{1 + (-1)^{\indiceTime}}{4\indiceTime} = 
        \begin{cases}
            \tfrac{1}{2\indiceTime}, \qquad &\text{for }\indiceTime\text{ even}, \\
            0, \qquad &\text{for }\indiceTime\text{ odd},
        \end{cases}
    \end{equation*}
    for $\indiceTime \in \nonZeroNaturals$.
    This is the opposite situation compared to the Lax-Friedrichs initialisation \eqref{eq:LaxFriedrichs} and again justifies the jumping behaviour compared to RE1 \eqref{eq:CoefficientInitialisationSmoothInTime}, see \Cref{fig:InitialError} and \Cref{fig:InitialErrors2}.
    Moreover, we further understand why we still observe the boundary layer: even if the \iniScheme{} matches the zero diffusivity of the bulk scheme, the second-order modification $\termAtOrder{\initialisationOperatorAsymptotic_1}{2} \neq 0$ we have imposed on the initial datum to obtain such \iniScheme{} reverberates over the following (even) time steps.

    \item{\strong{Smooth initialisation RE1} \eqref{eq:CoefficientInitialisationSmoothInTime}}. 
    \begin{proposition}
        Under acoustic scaling, the modified equations for the \stSchemes{} for the RE1 initialisation given by \eqref{eq:CoefficientInitialisationSmoothInTime}  are, for $\indiceTime \in \nonZeroNaturals$
        \begin{equation*}
            \partial_{\timeVariable} \testFunction(0, \spaceVariable) + \latticeVelocity \equilibriumCoefficient_2 \partial_{\spaceVariable} \testFunction(0, \spaceVariable) - \latticeVelocity \spaceStep \Bigl ( \frac{1}{\relaxationParameter_2} - \frac{1}{2} \Bigr )  ( 1 - \equilibriumCoefficient_2^2 )\partial_{\spaceVariable\spaceVariable}  \testFunction(0, \spaceVariable) = \bigO{\spaceStep^2}, \qquad \spaceVariable \in \reals.
        \end{equation*}
    \end{proposition}
    \begin{proof}
        In \Cref{app:DerivationModifiedEquationsD1Q2}, we obtain that 

    \begin{equation}\label{eq:ModifiedD1Q2InitialSmoothTime}
        \partial_{\timeVariable} \testFunction(0, \spaceVariable) + \equilibriumCoefficient_2 \partial_{\spaceVariable} \testFunction(0, \spaceVariable) - \latticeVelocity \spaceStep \Bigl ( \frac{1}{2} - \sum_{\indiceFreeOne = 1}^{\indiceTime - 1}\Bigl (1 - \frac{\indiceFreeOne}{\indiceTime}\Bigr )(1-\relaxationParameter_2)^{\indiceFreeOne} + \frac{1}{\indiceTime \relaxationParameter_2} \sum_{\indiceFreeOne = 1}^{\indiceTime}(1 - \relaxationParameter_2)^{\indiceFreeOne} \Bigr )  ( 1 - \equilibriumCoefficient_2^2 )\partial_{\spaceVariable\spaceVariable}  \testFunction(0, \spaceVariable) = \bigO{\spaceStep^2}.
    \end{equation}
    One can easily show by induction that
    \begin{equation*}
        \frac{1}{2} - \sum_{\indiceFreeOne = 1}^{\indiceTime - 1}\Bigl (1 - \frac{\indiceFreeOne}{\indiceTime}\Bigr )(1-\relaxationParameter_2)^{\indiceFreeOne} + \frac{1}{\indiceTime \relaxationParameter_2} \sum_{\indiceFreeOne = 1}^{\indiceTime}(1 - \relaxationParameter_2)^{\indiceFreeOne}  = \frac{1}{\relaxationParameter_2} - \frac{1}{2}, \qquad \indiceTime \in \nonZeroNaturals,
    \end{equation*}
    yielding the same modified equation as the \bulkScheme{}.
    \end{proof}
    This explains, once more, the smooth behaviour observed in \Cref{fig:InitialError} and \Cref{fig:InitialErrors2} and also shows an actual application of \Cref{prop:MatchEventually} for $\maxOrder = 2$.
    The smooth behavior comes from the fact that the schemes dissipate in the same way---which is the same as the bulk scheme---regardless of the time indices $\indiceTime$.
\end{itemize}

    \begin{remark}[Justification on the use of the modified equations]
        Observe that we can employ the modified equations---as we did---to assess the behavior of the schemes, in particular as far as time oscillatory (thus non-smooth) boundary layers are concerned. Even if these modified equations boil down to a low-frequency analysis in space, they are obtained without any approximation for the time variable, since we keep the discrete time indices $\indiceTime$, and no smoothness assumption or Taylor expansions with respect to the time variable are done.
    \end{remark}

\newcommand{\basicX}{\basicShiftLetter_1}
\newcommand{\FirstRelParDOneQThree}{\relaxationParameter_2}
\newcommand{\SecondRelParDOneQThree}{\relaxationParameter_3}
\newcommand{\velocityDOneQThree}{\equilibriumCoefficient_2}
\newcommand{\kDOneQThree}{\equilibriumCoefficient_3}
\newcommand{\henonSecond}{\frac{1}{\FirstRelParDOneQThree} - \frac{1}{2}}

\subsection{Three-velocities \scheme{1}{3} scheme}\label{sec:D1Q3}

The previous case of \scheme{1}{2} scheme suggests that particular care must be adopted when prepared initialisations for the conserved moment $\discreteMoment_1$ are used (\emph{i.e.} $\initialisationOperator_1 \in \ringSpaceOperatorsOneD$).
Therefore, in what follows, we treat only local initialisations for any moment.
We are now interested in equating the  dissipation of the initialisation schemes with the one of the bulk scheme for a richer scheme: the \scheme{1}{3}.
In particular, we look for a full characterisation of the conditions under which $\initialisationOperator_1, \initialisationOperator_2, \initialisationOperator_3 \in \reals$ yield initialisation schemes with the same dissipation as the \bulkScheme. 

\subsubsection{Description of the scheme}

We consider the \scheme{1}{3} scheme \cite{dubois2020notion, bellotti2021fd}, having $\spatialDimensionality = 1$, $\velocityNumber = 3$, $\discreteVelocityNormalized_1 = 0$, $\discreteVelocityNormalized_2 = 1$, $\discreteVelocityNormalized_3 = -1$, and
\begin{equation*}
    \momentMatrix = 
    \begin{bmatrix}
        1 & 1 & 1 \\
        0 & 1 & -1 \\
        -2 & 1 & 1
    \end{bmatrix}, ~~ 
    \transportMoment = 
    \begin{bmatrix}
        \frac{1}{3}(2\symmetricPart(\basicX) + 1) &  \antisymmetricPart(\basicX) & \frac{1}{3}(\symmetricPart(\basicX) - 1)\\
        \frac{2}{3} \antisymmetricPart(\basicX) & \symmetricPart (\basicX ) & \frac{1}{3} \antisymmetricPart(\basicX ) \\
        \frac{2}{3}(\symmetricPart(\basicX) - 1) & \antisymmetricPart(\basicX) & \frac{1}{3}(\symmetricPart(\basicX) + 2) 
    \end{bmatrix}, ~~
    \collisionMatrix = 
    \begin{bmatrix}
        1 & 0 & 0 \\
        \FirstRelParDOneQThree \velocityDOneQThree & 1 - \FirstRelParDOneQThree & 0 \\
        \SecondRelParDOneQThree \kDOneQThree & 0 & 1 - \SecondRelParDOneQThree
    \end{bmatrix}.
\end{equation*}
The modified equation of the \bulkScheme{} from \Cref{thm:EquivEqBulk} is
\begin{equation}\label{eq:EquivalenD1Q3Bulk}
    \partial_{\timeVariable} \testFunction (\timeVariable, \spaceVariable) + \latticeVelocity \velocityDOneQThree \partial_{\spaceVariable} \testFunction (\timeVariable, \spaceVariable) - \latticeVelocity \spaceStep \Bigl ( \henonSecond \Bigr )
    \Bigl ( \frac{2}{3} - \velocityDOneQThree^2 + \frac{\kDOneQThree}{3}  \Bigr )  \partial_{\spaceVariable \spaceVariable} \testFunction (\timeVariable, \spaceVariable) = \bigO{\spaceStep^2}, \qquad (\timeVariable, \spaceVariable) \in \nonNegativeReals \times \reals.
\end{equation}
To have a stable bulk method in the $L^2$ metric, the dissipation coefficient must not be negative, hence $\kDOneQThree < -2 + 3 \velocityDOneQThree^2$ is forbidden, because the modulus of the consistency (or ``physical'') eigenvalue would initially increase above one for small wavenumbers, causing the \bulkScheme{} to be unstable.
Sufficient conditions are more involved to determine but can be checked numerically.
Observe that the \eqref{eq:EquivalenD1Q3Bulk} does not depend on the choice of $\SecondRelParDOneQThree$. 
To obtain consistency with \eqref{eq:CauchyEquation}, we have to enforce $\equilibriumCoefficient_2 = \transportVelocity / \latticeVelocity$.
Furthermore, two leverages are available to make the \bulkScheme{} second-order consistent with the \eqref{eq:CauchyEquation}, namely taking $\FirstRelParDOneQThree = 2$ or $\FirstRelParDOneQThree \in ]0, 2[$ and $\kDOneQThree = -2 + 3\velocityDOneQThree^2$.

\subsubsection{Conditions to achieve time smoothness of the numerical solution}

Assuming that $\FirstRelParDOneQThree, \SecondRelParDOneQThree \neq 1$, we have that $\nonTriviallyRelaxingMomentsNumber = 2$, thus two \iniSchemes{} are to consider.
Their modified equations, computed with the previous techniques and considering local initialisations following the conditions by \Cref{prop:LocalInitialisation}---\emph{i.e.} $\initialisationOperator_1 = 1$ and $\initialisationOperator_2 = \velocityDOneQThree$---are as follows.
\begin{itemize}
    \item \strong{First \iniScheme{}}: \eqref{eq:InitialisationSchemes} for $\indiceTime = 1$
    \begin{equation}\label{eq:EquivalenD1Q3First}
        \partial_{\timeVariable} \testFunction (0, \spaceVariable) + \latticeVelocity \velocityDOneQThree \partial_{\spaceVariable} \testFunction (0, \spaceVariable) - \latticeVelocity \spaceStep \Bigl ( \frac{1}{3} -  \frac{\velocityDOneQThree^2}{2} + \frac{\SecondRelParDOneQThree \kDOneQThree}{6} + \frac{(1-\SecondRelParDOneQThree) \initialisationOperator_3}{6} \Bigr )  \partial_{\spaceVariable \spaceVariable} \testFunction (0, \spaceVariable)  = \bigO{\spaceStep^2}, \qquad \spaceVariable \in \reals.
    \end{equation}
    This scheme makes sense as \iniScheme{} unless both $\FirstRelParDOneQThree = \SecondRelParDOneQThree = 1$ (\emph{i.e.} $\nonTriviallyRelaxingMomentsNumber = 0$), where we observe that the diffusion coefficient in \eqref{eq:EquivalenD1Q3First} becomes equal to the one from \eqref{eq:EquivalenD1Q3Bulk}. In this case, the choice of $\initialisationOperator_3$ is unimportant, as expected.

    \item \strong{Second \iniScheme{}}: \eqref{eq:InitialisationSchemes} for $\indiceTime = 2$
    \begin{align}
        &\partial_{\timeVariable} \testFunction (0, \spaceVariable) + \latticeVelocity \velocityDOneQThree \partial_{\spaceVariable} \testFunction (0, \spaceVariable) \nonumber \\
        &- \latticeVelocity \spaceStep \Bigl ( \frac{(2 - \FirstRelParDOneQThree)}{3} + \frac{(\FirstRelParDOneQThree - 2) \velocityDOneQThree^2}{2} + \frac{\SecondRelParDOneQThree (5 - 2\FirstRelParDOneQThree-\SecondRelParDOneQThree) \kDOneQThree}{12}  + \frac{(1-\SecondRelParDOneQThree)(4  - 2\FirstRelParDOneQThree - \SecondRelParDOneQThree)\initialisationOperator_3}{12}\Bigr )  \partial_{\spaceVariable \spaceVariable} \testFunction (0, \spaceVariable)  = \bigO{\spaceStep^2},  \label{eq:EquivalenD1Q3Second}
    \end{align}
    for $\spaceVariable \in \reals$.
    In the case where both $\FirstRelParDOneQThree = \SecondRelParDOneQThree = 1$ ($\nonTriviallyRelaxingMomentsNumber = 0$), we have the previously described situation.
    Taking $\FirstRelParDOneQThree \neq 1$ and $\SecondRelParDOneQThree = 1$ ($\nonTriviallyRelaxingMomentsNumber = 1$), we obtain the modified equation of the first \stScheme{} which is not an \iniScheme{}
    \begin{equation*}
        \partial_{\timeVariable} \testFunction (0, \spaceVariable) + \latticeVelocity \velocityDOneQThree \partial_{\spaceVariable} \testFunction (0, \spaceVariable) - \latticeVelocity \spaceStep \FirstRelParDOneQThree \Bigl ( \henonSecond \Bigr )
        \Bigl ( \frac{2}{3} - \velocityDOneQThree^2 + \frac{\kDOneQThree}{3}  \Bigr )  \partial_{\spaceVariable \spaceVariable} \testFunction (0, \spaceVariable) = \bigO{\spaceStep^2}, \qquad \spaceVariable \in \reals,
    \end{equation*}
    which equals \eqref{eq:EquivalenD1Q3Bulk} up to the multiplication of the diffusion coefficient by $\FirstRelParDOneQThree$.
    This discrepancy is the remaining contribution of the initialisation on the evolution of the solution, as we have already observed for the \scheme{1}{2} in \Cref{sec:D1Q2} scheme for all initialisations except \eqref{eq:CoefficientInitialisationSmoothInTime}.
    Taking $\FirstRelParDOneQThree = 1$ and $\SecondRelParDOneQThree \neq 1$ ($\nonTriviallyRelaxingMomentsNumber = 1$), we have
    \begin{equation*}
        \partial_{\timeVariable} \testFunction (0, \spaceVariable) + \latticeVelocity \velocityDOneQThree \partial_{\spaceVariable} \testFunction (0, \spaceVariable) - \latticeVelocity \spaceStep \Bigl ( \frac{1}{3} - \frac{\velocityDOneQThree^2}{2} + \frac{\SecondRelParDOneQThree (3-\SecondRelParDOneQThree )\kDOneQThree}{12} + \frac{(1-\SecondRelParDOneQThree) (2 - \SecondRelParDOneQThree) \initialisationOperator_3}{12} \Bigr ) \partial_{\spaceVariable \spaceVariable} \testFunction (0, \spaceVariable) = \bigO{\spaceStep^2},
    \end{equation*}
    for $\spaceVariable \in \reals$, which is utterly different from \eqref{eq:EquivalenD1Q3Bulk}: the choice of initialisation $\initialisationOperator_3$ and the relaxation parameter $\SecondRelParDOneQThree$ influence the diffusivity, contrarily to  \eqref{eq:EquivalenD1Q3Bulk}.
\end{itemize}
\begin{remark}
    The previous discussion again confirms that, for \stSchemes{} which are not \iniSchemes{}, the choice of initialisations and relaxation parameters can change the modified equations compared to the \bulkScheme{} and thus the dynamics of the method close to the beginning of the simulation.
    Moreover, even some parameters that do not influence the modified equation of the \bulkScheme{} at a given order (see $\relaxationParameter_3$ in this example) impact the modified equations of the \stSchemes{}.    
    This is due to the role of the parasitic eigenvalues in the initialisation process.
\end{remark}

According to \Cref{prop:MatchEventually}, it is enough to study the order $\bigO{\spaceStep^2}$ for the \iniSchemes{} to deduce the modified equations for any \stScheme{}.
In order to match the diffusivity in both \iniSchemes{}, we set the following system 
\begin{equation}\label{eq:SystemMatchedDiffusionD1Q3}
    \begin{cases}
        \frac{1}{3} -  \frac{\velocityDOneQThree^2}{2} + \frac{\SecondRelParDOneQThree \kDOneQThree}{6} + \frac{(1-\SecondRelParDOneQThree) \initialisationOperator_3}{6} &= \Bigl ( \henonSecond \Bigr ) \Bigl ( \frac{2}{3} - \velocityDOneQThree^2 + \frac{\kDOneQThree}{3}  \Bigr ), \\
        \frac{(2 - \FirstRelParDOneQThree)}{3} + \frac{(\FirstRelParDOneQThree - 2) \velocityDOneQThree^2}{2} + \frac{\SecondRelParDOneQThree (5 - 2\FirstRelParDOneQThree-\SecondRelParDOneQThree) \kDOneQThree}{12}  + \frac{(1-\SecondRelParDOneQThree)(4  - 2\FirstRelParDOneQThree - \SecondRelParDOneQThree)\initialisationOperator_3}{12} &= \Bigl ( \henonSecond \Bigr ) \Bigl ( \frac{2}{3} - {\velocityDOneQThree^2} + \frac{\kDOneQThree}{3}  \Bigr ).
    \end{cases}
\end{equation}
We have to interpret $\velocityDOneQThree$ as fixed by the target problem and $\kDOneQThree$ as well as $\FirstRelParDOneQThree$ by the choice of numerical dissipation of the \bulkScheme{}, \emph{i.e.} the right hand sides in \eqref{eq:SystemMatchedDiffusionD1Q3}.
Therefore, the unknowns (or the leverages) are $\SecondRelParDOneQThree$ and $\initialisationOperator_3$, forming a non-linear system.
Eliminating $\initialisationOperator_3$ from the second equation in \eqref{eq:SystemMatchedDiffusionD1Q3} using the first one yields---following some algebra---the equation for $\SecondRelParDOneQThree$:
\begin{equation*}
    (1-\FirstRelParDOneQThree) \Bigl ( \frac{2}{3} - \velocityDOneQThree^2 + \frac{\kDOneQThree}{3}  \Bigr )\SecondRelParDOneQThree = (2-\FirstRelParDOneQThree)(1-\FirstRelParDOneQThree)  \Bigl ( \frac{2}{3} - \velocityDOneQThree^2 + \frac{\kDOneQThree}{3}  \Bigr ).
\end{equation*}

\begin{table} 
    \caption{Different choices of parameters for the \scheme{1}{3} scheme ensuring match at order $\bigO{\spaceStep^2}$ between \iniSchemes{} and \bulkScheme{}.}
    \begin{center}
    \begin{tabular}{cc|c|c}
        \multicolumn{2}{c|}{Factors controlling dissipation} & Leverages to obtain compatible dissipation & \\
        \hline
        \multirow{2}{*}{$\FirstRelParDOneQThree = 1$} & \multirow{2}{*}{$\kDOneQThree \geq - 2 + 3\velocityDOneQThree^2 $} & $\SecondRelParDOneQThree = 1$, any $\initialisationOperator_3$ & (a)\\ 
        & & $\SecondRelParDOneQThree \neq 1$, $\initialisationOperator_3 = \kDOneQThree$ & (b)\\
        \hline 
        \multirow{3}{*}{$\FirstRelParDOneQThree \neq 1$} & $\kDOneQThree > - 2 + 3\velocityDOneQThree^2 $ & $\SecondRelParDOneQThree = 2 - \FirstRelParDOneQThree$,  $\initialisationOperator_3 = ( 2 (-2 + 3\velocityDOneQThree^2) + (\FirstRelParDOneQThree - 2) \kDOneQThree  )/\FirstRelParDOneQThree$ & (c)\\
        & \multirow{2}{*}{$\kDOneQThree = - 2 + 3\velocityDOneQThree^2 $} & $\SecondRelParDOneQThree = 1$, any $\initialisationOperator_3$ & (d)\\
        & & $\SecondRelParDOneQThree \neq 1$, $\initialisationOperator_3 = \kDOneQThree $ & (e)\\
        \hline
    \end{tabular}
    \end{center}
    \label{tab:choicesParD1Q3}
    \end{table}

We have different cases to discuss which are summarized in \Cref{tab:choicesParD1Q3} and which are obtained as detailed in \Cref{app:ConditionsMatchDissipatinD1Q3}.
We solely comment on (c): in \cite{bellotti2021fd}, we have found that the choice 
\begin{equation}\label{eq:MagicRelation}
    \SecondRelParDOneQThree = 2 - \FirstRelParDOneQThree,
\end{equation}
could yield a \bulkScheme{} with three stages instead of four, in a way reminiscent of \cite{d2009viscosity, kuzmin2011role}.
This phenomenon shall be the focus of \Cref{sec:Ginzburg}.
As far as the stability under this condition is concerned, the analytical conditions in this case are 
\begin{equation*}
    \relaxationParameter_2 \in ]0, 2], \qquad \text{and} \qquad 
    \begin{cases}
        |\velocityDOneQThree| \leq 1, \qquad -2 + 3 \velocityDOneQThree^2 \leq \kDOneQThree \leq 1, \quad &\text{if} \quad \relaxationParameter_2 \in ]0, 2[, \\
        |\velocityDOneQThree| < 1, \qquad &\text{if} \quad \relaxationParameter_2 = 2,
    \end{cases}
\end{equation*}
see \Cref{app:StabilityD1Q3}, where in $-2 + 3 \velocityDOneQThree^2 \leq \kDOneQThree \leq 1$, the left constraint enforces non-negative dissipation (stability for small wavenumbers) whereas the right one concerns large wavenumbers.
Notice that the case $\relaxationParameter_2 = 2$ corresponds to $\relaxationParameter_3 = 0$, meaning that $\discreteMoment_3$ is conserved and thus goes against the assumptions of the paper. This particular occurrence will be discussed in what follows and will prove to be harmless.

\subsubsection{Study of the time smoothness of the numerical solution}

We repeat the numerical experiment by \cite{van2009smooth} introduced in \Cref{sec:TimeSmoothnessD1Q2}.
Only $L^2$ stable configurations are considered.
As long as the dissipation of the \bulkScheme{} is large, time oscillations are damped and thus cannot be observed even if the diffusivities of the \bulkScheme{} and the \iniSchemes{} are not the same.
We therefore look for situations where the numerical diffusion is small or zero.
\begin{figure} 
    \begin{center}
        \includegraphics[scale=0.99]{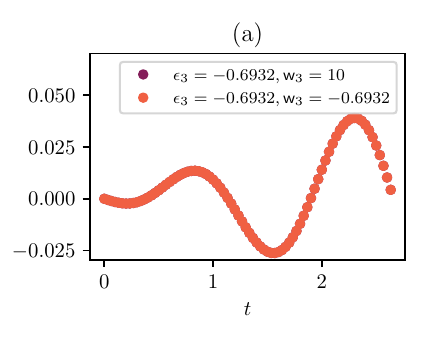}
        \includegraphics[scale=0.99]{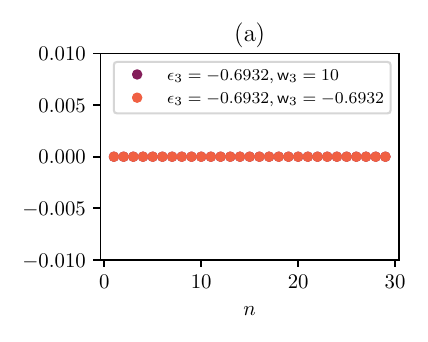}
    \end{center}\caption{\label{fig:a}Left: test for smoothness in time close to $\timeVariable = 0$ for the case (a) in \Cref{tab:choicesParD1Q3}: difference between exact and numerical solution at the eighth lattice point. As expected, regardless of the choice on $\initialisationOperator_3$, the profile is smooth. Right: diffusion coefficient (factor in front of $-\latticeVelocity\spaceStep \partial_{\spaceVariable\spaceVariable}$) in the modified equations for different $\indiceTime$.}
  \end{figure}
  \begin{figure} 
    \begin{center}
        \includegraphics[scale=0.99]{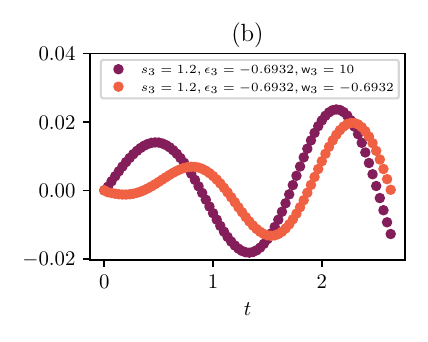}
        \includegraphics[scale=0.99]{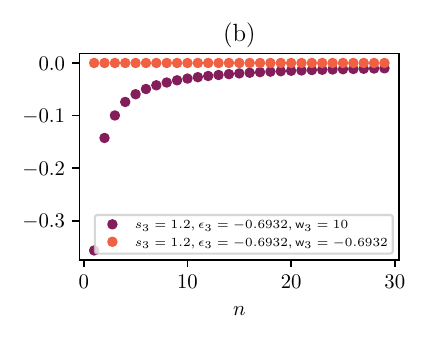}
    \end{center}\caption{\label{fig:b}Left: test for smoothness in time close to $\timeVariable = 0$ for the case (b) in \Cref{tab:choicesParD1Q3} ($\initialisationOperator_3 = -0.6932$) or violating this condition  ($\initialisationOperator_3 = 10$): difference between exact and numerical solution at the eighth lattice point. We observe radical differences in the profiles but the smoothness is not affected. Right: diffusion coefficient (factor in front of $-\latticeVelocity\spaceStep \partial_{\spaceVariable\spaceVariable}$) in the modified equations for different $\indiceTime$.}
  \end{figure}
  \begin{figure} 
    \begin{center}
        \includegraphics[scale=0.99]{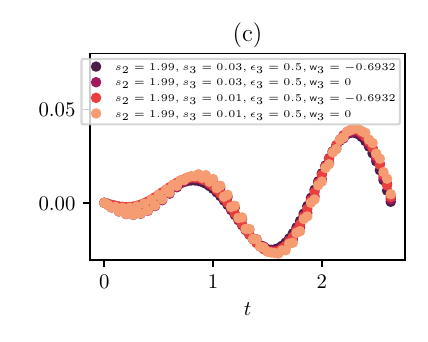}
        \includegraphics[scale=0.99]{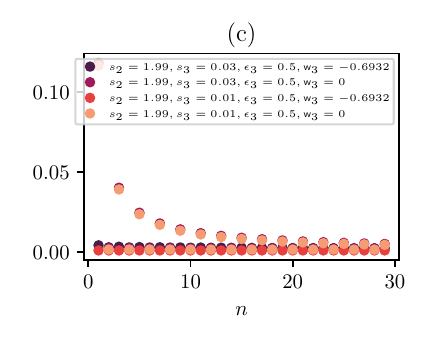}
    \end{center}\caption{\label{fig:c}Left: test for smoothness in time close to $\timeVariable = 0$ for the case (c) in \Cref{tab:choicesParD1Q3}: difference between exact and numerical solution at the eighth lattice point. The cases where $\relaxationParameter_3 = 0.01$ violate the magic relation \eqref{eq:MagicRelation} $\relaxationParameter_2 + \relaxationParameter_3 = 2$ with minor influences on the spurious oscillation, whereas $\initialisationOperator_3 = 0$ violates \eqref{eq:ChoiceInitizialisationMagicD1Q3}, with more tendency towards an initial boundary layer. Right: diffusion coefficient (factor in front of $-\latticeVelocity\spaceStep \partial_{\spaceVariable\spaceVariable}$) in the modified equations for different $\indiceTime$.}
  \end{figure}
\begin{itemize}
    \item $\FirstRelParDOneQThree = 1$, $\kDOneQThree = -2 + 3\velocityDOneQThree^2$, no dissipation, and $\SecondRelParDOneQThree = 1$. This is the framework of (a) (\emph{cf.} \Cref{tab:choicesParD1Q3}), where we can consider arbitrary $\initialisationOperator_3$.
    This case is trivial because $\nonTriviallyRelaxingMomentsNumber = 0$.
    We see in \Cref{fig:a} that the profile remains smooth no matter the choice of $\initialisationOperator_3$, as predicted by the theory.
    \item $\FirstRelParDOneQThree = 1$, $\kDOneQThree = -2 + 3\velocityDOneQThree^2$, no dissipation, and $\SecondRelParDOneQThree = 1.2$, close to one for stability reasons.
    Thus we are in the setting of (b).
    In \Cref{fig:b}, we see that the choice of $\initialisationOperator_3$ changes the outcome, even if the time smoothness seems to be preserved in both cases.
    To explain this, on the one hand, we have to take into account that since we are compelled to take $\SecondRelParDOneQThree$ close to one, we are not far from the previous case.
    On the other hand, even when the dissipation is not matched, it does not oscillate between time steps, unlike many initialisations for the \scheme{1}{2} scheme in \Cref{sec:D1Q2}.
    This is confirmed by the right image in \Cref{fig:b}: the diffusivity behaves smoothly in $\indiceTime$ and tends monotonically and quite rapidly to the bulk vanishing one.
    \item $\FirstRelParDOneQThree = 1.99$, almost zero dissipation. We test (c), since (d) and (e) cannot be considered for stability reasons.
    In \Cref{fig:c}, we observe that violating the magic relation \eqref{eq:MagicRelation} still enforcing \eqref{eq:ChoiceInitizialisationMagicD1Q3} does not produce large spurious oscillations, likely because this has limited effects on the diffusion coefficient. 
    Quite the opposite, violating \eqref{eq:ChoiceInitizialisationMagicD1Q3} both with and without \eqref{eq:MagicRelation} produces an initial oscillating boundary layer.
    This is corroborated by the right image in \Cref{fig:c}, where the reason for the observed oscillations is the highly non-smooth behaviour of the diffusion coefficient in $\indiceTime$, as a result of having taken $\relaxationParameter_2 \simeq 2$.

\end{itemize}

\subsection{Conclusions}

In this \Cref{sec:Illustrations}, we have observed in practice that the conditions to obtain consistent \stSchemes{} found in \Cref{sec:ModifiedEquations} preserve second-order convergence when the bulk scheme is second-order consistent.
Using an additional order for the modified equations introduced in \Cref{sec:ModifiedEquations}, we obtain an extremely precise description of the behaviour of the \scheme{1}{2} scheme close to the initial time, according to the initialisation at hand.
The same has been done for a \scheme{1}{3} scheme.
Finally, discussing the conditions to have the same dissipation between initialisation and bulk schemes for the \scheme{1}{3} scheme has made the magic relations \eqref{eq:MagicRelation} known in the literature \cite{d2009viscosity, kuzmin2011role} turn up once more \cite{bellotti2021fd}. The investigation of these relations is central in the following \Cref{sec:Ginzburg}.

% \FloatBarrier

\newcommand{\observabilityMatrix}{\boldOther{\discrete{\Omega}}}
\newcommand{\outputLetter}{y}
\newcommand{\outputMatrix}{\matricial{C}}
\newcommand{\rank}[1]{\text{rank}(#1)}
\newcommand{\observabilityIndex}{o}
\newcommand{\unObservableSubSpace}{\mathcal{N}}
\newcommand{\kernel}[1]{\text{ker}(#1)}
\newcommand{\cokernel}[1]{\text{coker}(#1)}
\newcommand{\annhilitaingPolyGinzburgCoefficientLetter}{p}
\newcommand{\annhilitaingPolyGinzburgCoefficient}{\discrete{\annhilitaingPolyGinzburgCoefficientLetter}}

\section{A more precise evaluation of the number of initialisation schemes}\label{sec:Ginzburg}

In \Cref{sec:ModifiedEquations}, we have observed that describing the behaviour of general \lbm schemes close to the initial time above $\bigO{\spaceStep}$ order---using the modified equations---seems out of reach, due to the presence of many parasitic modes in the system.
The question which we try to answer here---inspired by the findings on the \scheme{1}{3} scheme in \Cref{sec:D1Q3}---concerns the existence of vast classes of \lbm schemes for which a detailed description of the behaviour of the initialisation schemes is indeed possible.
The idea is to investigate the possibility of having, from a purely algebraic standpoint, a very small number of \iniSchemes{} to be considered, or equivalently, a large number of trivial eigenvalues.
For example, this would allow to avoid dealing---when trying to have the same dissipation coefficient between initialisation and bulk---with large non-linear systems such as \eqref{eq:SystemMatchedDiffusionD1Q3}, where the number and the complexity of equations would grow with $\nonTriviallyRelaxingMomentsNumber$.
The conditions to control the initialisation until a certain order in $\spaceStep$ could be simpler thanks to the fact that we have a small number of initialisation steps. In this way, if something similar to \Cref{prop:MatchEventually} were valid, we could conclude that this control is enough to master the dynamics of the scheme at the considered orders eventually in time.

\subsection{Lattice Boltzmann schemes as dynamical systems and observability}

A preliminary step in this direction is to consider any \lbm scheme \Cref{alg:LBMScheme} as a linear time-invariant discrete-time system 
\begin{align*}
    \timeShift \boldOther{\discreteMoment}(\timeVariable, \vectorial{\spaceVariable}) = \schemeMatrix \boldOther{\discreteMoment}(\timeVariable, \vectorial{\spaceVariable}), \qquad  &(\timeVariable, \vectorial{\spaceVariable}) \in \timeLattice \times \lattice, \\
    \boldOther{\discreteMoment}(0,  \vectorial{\spaceVariable}) \qquad &\text{given for }\vectorial{\spaceVariable} \in \lattice,
\end{align*}
where the output is $\boldOther{\discrete{\outputLetter}} = \outputMatrix \boldOther{\discreteMoment}$ with a smatrix $\outputMatrix$ of appropriate dimensions.
Since, from the very beginning of the paper, we are solely interested in the conserved moment $\discreteMoment_1$, we select $\outputMatrix = \transpose{\canonicalBasisVector_1} \in \reals^{\velocityNumber}$.
As we have already pointed out, see \eqref{eq:InitialisationSchemes}
\begin{equation*}
    \discrete{\outputLetter} (\indiceTime \timeStep, \vectorial{\spaceVariable}) = \discreteMoment_1(\indiceTime \timeStep, \vectorial{\spaceVariable}) = (\schemeMatrix^{\indiceTime} \boldOther{\discreteMoment})_1 (0, \vectorial{\spaceVariable}) = \outputMatrix  \schemeMatrix^{\indiceTime} \boldOther{\discreteMoment} (0, \vectorial{\spaceVariable}), \qquad\indiceTime \in \naturals, \quad \vectorial{\spaceVariable} \in \lattice,
\end{equation*}
thus we introduce the observability matrix of the system
\begin{equation*}
    \observabilityMatrix \definitionEquality 
    \begin{bmatrix}
        \outputMatrix \\
        \outputMatrix\schemeMatrix \\
        \vdots \\
        \outputMatrix \schemeMatrix^{\velocityNumber - 1}
    \end{bmatrix} \in \matrixSpace{\velocityNumber}{\ringSpaceOperators}.
\end{equation*}
If the system were set on a field (\emph{e.g.} $\observabilityMatrix \in \matrixSpace{\velocityNumber}{\reals}$ or $\observabilityMatrix \in \matrixSpace{\velocityNumber}{\complex}$), it would be customary to call the system ``observable'' if and only if $\rank{\observabilityMatrix} = \velocityNumber$.
This would mean that we could reconstruct the initial data $\boldOther{\discreteMoment}(0)$ from the observation of $\discrete{\outputLetter} = \discreteMoment_1$ at times $\indiceTime \in \integerIntervalClosedOpen{0}{\velocityNumber}$.
Quite the opposite, in our case, since the non-zero entries of $\observabilityMatrix$ are in general not invertible (for $\spatialDimensionality = 1$, $\symmetricPart (\basicX)$ and $\antisymmetricPart (\basicX)$ are examples of this), we cannot proceed in the same way, because the observability matrix $\observabilityMatrix$ can never be a unit.

For systems over commutative rings, different definitions of observability are available in the literature: we list a few of them in the following Definition.
\begin{definition}[Observability]\label{def:Observability}
    The system is said to be 
    \begin{itemize}
        \item ``observable'' according to \cite[Theorem 2.6]{brewer1986linear}, if the application represented by the left action of $\observabilityMatrix$ is injective.
        \item ``observable'' according to \cite{fliess1998controllability}, if $\observabilityMatrix$ has left inverse.
        \item ``hyper-observable'' according to \cite{fliess1998controllability}, if the unobservable sub-space $\unObservableSubSpace \definitionEquality \kernel{\observabilityMatrix}$---where operators act on lattice functions\footnote{Observe that the kernel is the left null space: indeed the left action of elements in $\ringSpaceOperators$ can operate both on lattice functions and operators in $\ringSpaceOperators$, whereas the right action is reserved for operators in $\ringSpaceOperators$.}---is trivial: $\unObservableSubSpace = \{ \vectorial{0} \}$.
    \end{itemize}
\end{definition}

Furthermore, \cite[Theorem 2.6]{brewer1986linear} gives the following criterion to check observability.
\begin{theorem}[\cite{brewer1986linear} Observability criterion]\label{thm:Observability}
    The system is ``observable'' according to \cite{brewer1986linear} if and only if the ideal of $\ringSpaceOperators$ generated by $\determinant (\observabilityMatrix )$ is such that its annihilator is zero.
\end{theorem}

We also define the ``observability index'' $\observabilityIndex \leq \nonTriviallyRelaxingMomentsNumber + 1$ mimicking the definition for systems over fields as 
\begin{equation*}
    \observabilityIndex \definitionEquality \max_{\indiceFreeOne \in \naturals} \rank{\observabilityMatrix_{\indiceFreeOne}}, \quad \text{where} \quad \observabilityMatrix_{\indiceFreeOne} \definitionEquality
    \begin{bmatrix}
        \outputMatrix \\
        \outputMatrix\schemeMatrix \\
        \vdots \\
        \outputMatrix \schemeMatrix^{\indiceFreeOne - 1}
    \end{bmatrix} \in \matrixSpace{\indiceFreeOne \times \velocityNumber}{\ringSpaceOperators},
\end{equation*}
and $\rank{\cdot}$ stands for the row rank of a matrix over a ring according to Definition 10.6 in \cite{blyth2018module}.

\begin{example}\label{ex:D1Q2Again}
    Considering \Cref{ex:D1Q2} treated in \Cref{sec:D1Q2}, we have that
    \begin{equation*}
        \observabilityMatrix = 
        \begin{bmatrix}
            1 & 0 \\
            \symmetricPart (\basicShiftLetter_1) + \relaxationParameter_2 \equilibriumCoefficient_2 \antisymmetricPart (\basicShiftLetter_1) & (1-\relaxationParameter_2) \antisymmetricPart (\basicShiftLetter_1)
        \end{bmatrix}, 
    \end{equation*}
    hence $\observabilityIndex = \nonTriviallyRelaxingMomentsNumber + 1 = 2$ if $\relaxationParameter_2 \neq 1$ and $\observabilityIndex = \nonTriviallyRelaxingMomentsNumber + 1 = 1$ if $\relaxationParameter_2 = 1$.
When $\relaxationParameter_2 = 1$, we have $\unObservableSubSpace =  \{ \transpose{(0, \discreteMoment_2)} ~ : ~ \text{for arbitrary }\discreteMoment_2  = \discreteMoment_2({\spaceVariable}) \text{ lattice function}\}$, which adheres to the intuition that we cannot know the non-conserved moment $\discreteMoment_2$ by looking at the conserved moment $\discreteMoment_1$ if the relaxation is made on the equilibrium, regardless of the structure of $\discreteMoment_2$.
When $\relaxationParameter_2 \neq 1$, we have $\unObservableSubSpace =  \{ \transpose{(0, \discreteMoment_2)} ~ : ~ \text{for any }\discreteMoment_2  = \discreteMoment_2({\spaceVariable}) \text{ lattice function such that }\antisymmetricPart (\basicShiftLetter_1)\discreteMoment_2 = 0 \}$.
We see that the unobservable sub-space is non-trivial even when $\observabilityIndex = \velocityNumber = 2$, contrarily to the case of systems with matrix $\schemeMatrix$ and $\observabilityMatrix$ with entries in a field.
The unobservable states are those in which the first component is zero and the discrete derivative $\antisymmetricPart (\basicShiftLetter_1)$ of the second component is zero everywhere, for example because the second component is constant or takes one given value on all even points and another one on all odd points.
The interesting reader can consult \Cref{app:numExpUnobsD1Q2} to find a numerical experiment showcasing the structure of $\unObservableSubSpace$ for this scheme.
We finally comment on the notions from \Cref{def:Observability}.
\begin{itemize}
    \item $\determinant (\observabilityMatrix) = {(1-\relaxationParameter_2)} \antisymmetricPart (\basicShiftLetter_1)$, thus the ideal to consider (\emph{cf.} \Cref{thm:Observability}) is $\{ \genericDiscreteOperator {(1-\relaxationParameter_2)} \antisymmetricPart (\basicShiftLetter_1) ~ : ~ \genericDiscreteOperator \in \ringSpaceOperatorsOneD \}$.
    On the one hand, if $\relaxationParameter_2 = 1$, then any operator in $\ringSpaceOperatorsOneD$ multiplied at the left of any element of the ideal is an annihilator, thus the system is not observable according to \cite{brewer1986linear}. On the other hand, if $\relaxationParameter_2 \neq 1$, then the only element annihilating any element of the ideal is zero, thus the system is observable according to \cite{brewer1986linear}.
    \item For any $\relaxationParameter_2$, we see that $\observabilityMatrix$ does not admit left inverse, therefore it is not observable according to \cite{fliess1998controllability}.
    \item For any $\relaxationParameter_2$, the system is not hyper-observable  according to \cite{fliess1998controllability} due to the non-trivial $\unObservableSubSpace$.
\end{itemize}
For these reasons, we infer that the observability according to \cite{brewer1986linear} is the one more closely adhering---between those issued from \Cref{def:Observability}---to our definition of observability index $\observabilityIndex$.
\end{example}

\begin{remark}[Regularized lattice Boltzmann schemes and similar models]
    We observe that the so-called ``regularized'' lattice Boltzmann schemes \cite{coreixas2017recursive, coreixas2019comprehensive}, where the relaxation rates of the non-hydrodynamic moments are set to one, and approaches where all relaxation parameter equals one, like ``Finite Boltzmann'' schemes \cite{van2006finite}, ``macroscopic lattice Boltzmann'' methods \cite{zhou2020macroscopic} and ``simplified'' lattice Boltzmann methods \cite{chen2017simplified}, are somehow trivial examples of unobservable schemes.
    Non-trivial examples of unobservable schemes will be introduced in \Cref{sec:ReducedNumberInitialisationSchemes} and \Cref{sec:LinkScheme}.
    Indeed, we can see from \Cref{sec:LBMSchemes} and \Cref{sec:CorrespondingFD} that since $\nonTriviallyRelaxingMomentsNumber < \velocityNumber - 1$, because the non-hydrodynamic (respectively, all the non-conserved) moments relax to their equilibrium, we have that $\observabilityIndex < \velocityNumber$ and the  initialization of the non-hydrodynamic (respectively, all the non-conserved) moments does not impact the numerical scheme.
\end{remark}

\subsection{Reduced number of initialisation schemes for non-observable systems}\label{sec:ReducedNumberInitialisationSchemes}

Following the discussion in \cite{bellotti2021fd}, we can introduce $\boldOther{\annhilitaingPolyGinzburgCoefficient}_{\observabilityIndex} \in (\ringSpaceOperators)^{\observabilityIndex}$ such that 
\begin{equation}\label{eq:SystemGinzburgCoefficients}
    \boldOther{\annhilitaingPolyGinzburgCoefficient}_{\observabilityIndex} \observabilityMatrix_{\observabilityIndex} = -\outputMatrix \schemeMatrix^{\observabilityIndex}.    
\end{equation}
The solution of this problem exists thanks to the definition of the observability index $\observabilityIndex$.
We then introduce the monic polynomial (keep in mind that the indices in vectors like $\boldOther{\annhilitaingPolyGinzburgCoefficient}_{\observabilityIndex}$ start from one)
\begin{equation}\label{eq:DefinitionGinzburgPolynomial}
    \annhilitaingPolyGinzburgWithOrder{\observabilityIndex} (\timeShift) \definitionEquality \timeShift^{\observabilityIndex} + \sum_{\indiceTime = 1}^{\observabilityIndex} {\annhilitaingPolyGinzburgCoefficient}_{\observabilityIndex, \indiceTime} \timeShift^{\indiceTime - 1},
\end{equation}
which by construction \eqref{eq:SystemGinzburgCoefficients} annihilates the first row of $\schemeMatrix$, since $\outputMatrix = \transpose{\canonicalBasisVector_1}$.
Moreover, we have shown in \cite{bellotti2021fd} that $\annhilitaingPolyGinzburgWithOrder{\observabilityIndex}  (\timeShift)$ divides $\determinant(\timeShift \identity - \schemeMatrix)$, whence if $\observabilityIndex = \nonTriviallyRelaxingMomentsNumber + 1$, we naturally have $\annhilitaingPolyGinzburgWithOrder{\observabilityIndex}  (\timeShift) = \timeShift^{\nonTriviallyRelaxingMomentsNumber + 1 - \velocityNumber} \determinant(\timeShift \identity - \schemeMatrix)$.
We therefore obtain the following corresponding \bulkScheme{} based on $\annhilitaingPolyGinzburgWithOrder{\observabilityIndex}$ given by \Cref{alg:CorrespondingFDSchemeGinzburh}, coinciding with \Cref{alg:CorrespondingFDScheme} when $\observabilityIndex = \nonTriviallyRelaxingMomentsNumber + 1$.
\begin{algorithm}
    \begin{itemize}
        \item Given $\boldOther{\discreteMoment}(0, \vectorial{\spaceVariable})$ for every $\vectorial{\spaceVariable} \in \lattice$.
        \item{\strong{Initialisation schemes}}. For $\indiceTime \in \integerInterval{1}{\observabilityIndex-1}$
        \begin{equation}\label{eq:InitialisationSchemesGinzburg}
            \discreteMoment_1(\indiceTime \timeStep, \vectorial{\spaceVariable}) = \outputMatrix \schemeMatrix^{\indiceTime} \boldOther{\discreteMoment} (0, \vectorial{\spaceVariable}), \qquad \vectorial{\spaceVariable} \in\lattice.
        \end{equation}
        \item{\strong{Corresponding \bulkScheme{}}}. For $\indiceTime \in \integerIntervalClosedOpen{\observabilityIndex-1}{+\infty}$
        \begin{equation}\label{eq:BulkSchemeGinzburg}
            \discreteMoment_1((\indiceTime + 1) \timeStep, \vectorial{\spaceVariable}) = - \sum_{\indiceFreeOne = \velocityNumber - \observabilityIndex}^{\velocityNumber-1}  \annhilitaingPolyGinzburgCoefficient_{\observabilityIndex, \observabilityIndex + \indiceFreeOne + 1 - \velocityNumber} \discreteMoment_1 ((\indiceTime + \indiceFreeOne + 1 - \velocityNumber)\timeStep, \vectorial{\spaceVariable}) , \qquad \vectorial{\spaceVariable} \in \lattice.
        \end{equation}
    \end{itemize}
    \caption{\label{alg:CorrespondingFDSchemeGinzburh}Corresponding \fd scheme based on $\annhilitaingPolyGinzburgWithOrder{\observabilityIndex}$.}
\end{algorithm}

\newcommand{\transferFunction}{\mathsf{H}}
\newcommand{\schemeMatrixMoments}{\boldOther{\discrete{A}}}
\newcommand{\schemeMatrixEquilibria}{\boldOther{\discrete{B}}}

The lack of observability is indeed the reason why, as previously announced in \Cref{sec:D1Q3}, one can find a \bulkScheme{} with less time steps than what is prescribed by the characteristic polynomial of $\schemeMatrix$.
From a different perspective, this is the so-called ``pole-zero cancellation'' in the transfer function---see for example \cite[Chapter 8.3]{aastrom2008feedback} or \cite[Chapter 3.9]{hendricks2008linear}---associated with the system and taken from control theory.
In our framework, the transfer function is 
\begin{equation*}
    \transferFunction(\timeShift) = \outputMatrix \frac{\overbrace{\adjugateMatrix(\timeShift \identity - \schemeMatrixMoments) \schemeMatrixEquilibria \boldOther{\equilibriumCoefficient}}^{\text{control by equilibria}}}{\underbrace{\determinant(\timeShift \identity - \schemeMatrixMoments)}_{\text{state}}} = \outputMatrix (\timeShift \identity - \schemeMatrixMoments)^{-1} \schemeMatrixEquilibria \boldOther{\equilibriumCoefficient},
\end{equation*}
with $\schemeMatrixMoments \definitionEquality \transportMoment (\identity - \relaxationMatrix)$ and $\schemeMatrixEquilibria \definitionEquality \transportMoment \relaxationMatrix$ defined as in \cite{bellotti2021fd,bellotti2021equivalentequations}, with $\schemeMatrix = \schemeMatrixMoments + \schemeMatrixEquilibria \vectorial{\equilibriumCoefficient} \otimes \canonicalBasisVector_1$.

\begin{example}\label{ex:D1Q3Ginzburg}
    We come back to the scheme of \Cref{sec:D1Q3} where we select $\relaxationParameter_2 + \relaxationParameter_3 = 2$.
    We also assume that $\relaxationParameter_2 \neq 1$ to keep things non-trivial.
    In this case, it can be seen that $\observabilityIndex = 2 < 3$, whereas $\nonTriviallyRelaxingMomentsNumber + 1 = 3$.
    Moreover, we obtain
    \begin{align*}
        \determinant (\timeShift \identity - \schemeMatrix) &= (\timeShift + (1-\relaxationParameter_2)) \annhilitaingPolyGinzburg(\timeShift), \\
        &\text{with} \quad  \annhilitaingPolyGinzburg(\timeShift) = \timeShift^2 +  (- \relaxationParameter_2 \velocityDOneQThree\antisymmetricPart(\basicX) + \tfrac{1}{3}(\relaxationParameter_2 - 2)  (2\symmetricPart(\basicX) + 1) + \tfrac{1}{3} \kDOneQThree (\relaxationParameter_2 - 2) (\symmetricPart(\basicX) - 1) ) \timeShift + (1-\relaxationParameter_2), \\
        &\text{or equivalently} \quad \transferFunction(\timeShift) =  \frac{(\timeShift + (1-\relaxationParameter_2)) (\relaxationParameter_2 \velocityDOneQThree\antisymmetricPart(\basicX) + \tfrac{1}{3} \kDOneQThree (2 - \relaxationParameter_2) (\symmetricPart(\basicX) - 1)) \timeShift }{(\timeShift + (1-\relaxationParameter_2)) (\timeShift^2 + \tfrac{1}{3}(\relaxationParameter_2 - 2)  (2\symmetricPart(\basicX) + 1) \timeShift + (1-\relaxationParameter_2)) }.
    \end{align*}
    The \fd scheme coming from $\annhilitaingPolyGinzburg(\timeShift)$ becomes a leap-frog scheme for $\relaxationParameter_2 = 2$.
    Otherwise, it is a centered discretisation with a certain amount of numerical dissipation.
    A first question which might arise concerns the modified equation for the \bulkScheme{} obtained using $\determinant (\timeShift \identity - \schemeMatrix)$, see \Cref{alg:CorrespondingFDScheme}, \emph{versus} that obtained by $\annhilitaingPolyGinzburg(\timeShift)$, see \Cref{alg:CorrespondingFDSchemeGinzburh}. The answer is that they are same at any order because the eigenvalue $(\relaxationParameter_2 - 1)$ does not contribute to the consistency (being constant through wavenumbers and thus being a mere numerical eigenvalue) and it can be easily checked that $\annhilitaingPolyGinzburg(\timeShift)$ yields the same modified equation, since it contains the consistency eigenvalue \cite{strikwerda2004finite}.
    As far as stability is concerned, the stability constraints for the two \bulkScheme s are the same because $|\relaxationParameter_2 - 1| < 1$ for $\relaxationParameter_2 \in ]0, 2[$.
    The case $\relaxationParameter_2 = 2$ might produce instabilities because of the presence of multiple roots in $\determinant(\timeShift \identity - \fourierTransformed{\schemeMatrix})$ on the unit circle. However, in this case, there is an additional conserved moment $\discreteMoment_3$ and we know that the \emph{von Neumann} condition for systems is that no root is outside the unit circle (no precise indication is provided for those on the unit circle), but this is only necessary for stability \cite[Theorem 5.2.2]{gustafsson1995time}. Therefore, the presence of
    multiple eigenvalues on the unit circle (and in particular those concerning consistency which are now more than one)
    cannot allow to deduce that the scheme is unstable.
    The stability conditions are analytically computed in \Cref{app:StabilityD1Q3}.

    Concerning the notion of observability by \cite{brewer1986linear}, we have that $\determinant (\observabilityMatrix) = 0$, hence the system is not observable, according to \Cref{thm:Observability}.
    If we want to characterize the unobservable sub-space, we have that, since $\relaxationParameter_2 \neq 1$, it is given by $\unObservableSubSpace =  \{ \transpose{(0, \discreteMoment_2, \discreteMoment_3)} ~ : ~ \text{for any }\discreteMoment_2  = \discreteMoment_2({\spaceVariable}), ~\discreteMoment_3  = \discreteMoment_3({\spaceVariable}) \text{ lattice functions such that }\antisymmetricPart (\basicShiftLetter_1)\discreteMoment_2 = {1}/{3} (\symmetricPart(\basicShiftLetter_1) - 1)\discreteMoment_3  \}$.
    Recall that $\antisymmetricPart (\basicShiftLetter_1) = {(\basicShiftLetter_1 - {\basicShiftLetter_1}^{-1})}/{2}$ and $\symmetricPart(\basicShiftLetter_1) - 1 = {(\basicShiftLetter_1 - 2 + {\basicShiftLetter_1}^{-1})}/{2}$, which means that the initial states belonging to $\unObservableSubSpace$ are those with zero first moment everywhere and such that the centered approximation of the first derivative of the second moment is proportional---with ratio $1/6$---to the centered approximation of the second derivative of the third moment, at any point of the lattice. 
    The numerical verification of the expression found for $\unObservableSubSpace$ is provided in \Cref{app:numExpUnobsD1Q3} for the interested reader.
      
\end{example}

As already remarked in \cite{saad1989overview, bellotti2021fd}, the cases where $\nonTriviallyRelaxingMomentsNumber + 1 \neq \observabilityIndex$ are extremely peculiar.
Indeed, the situations described in \Cref{ex:D1Q3Ginzburg} and in the forthcoming \Cref{sec:LinkScheme} are the only examples we were able to find.
Loosely speaking, both $\observabilityIndex$ and $\nonTriviallyRelaxingMomentsNumber$ measure the speed of saturation of the image of the scheme $\schemeMatrix$ concerning the conserved moment. 
Once the generated sub-spaces saturate, the evolution of the conserved moment at the new time-step can be recast as function of itself at the previous steps.
The fact that the Cayley-Hamilton theorem holds (concerning $\nonTriviallyRelaxingMomentsNumber$) and that the polynomial $\annhilitaingPolyGinzburgWithOrder{\observabilityIndex}$ (concerning $\observabilityIndex$) annihilates the first row of $\schemeMatrix$ introduce---as previously shown---a set of linear constraints on $\discreteMoment_1$, solution of the \lbm scheme.

    \begin{remark}[Unobservable schemes are ``weakly-kinetic'' schemes]
        We can interpret unobservable schemes, where $\observabilityIndex < \nonTriviallyRelaxingMomentsNumber + 1$, as being ``very little kinetic'', ``almost non-kinetic'', or finally ``weakly-kinetic''.
        This fact can be exploited to avoid the storage of $\nonTriviallyRelaxingMomentsNumber + 1$ unknowns on every gridpoint, by storing just $\observabilityIndex$ unknowns, and  implement the algorithm using the corresponding Finite Difference scheme.
    \end{remark}

It should be emphasized that \Cref{prop:MatchEventually} is still valid turning $\nonTriviallyRelaxingMomentsNumber$ into $\observabilityIndex - 1$, \eqref{eq:InitialisationSchemes} into \eqref{eq:InitialisationSchemesGinzburg}, and \eqref{eq:BulkSchemes} into \eqref{eq:BulkSchemeGinzburg}.
This is fundamental, because \Cref{prop:MatchEventually} ensures to control the whole dynamics of the scheme by mastering it in the initialisation layer. The aim of studying the observability is to characterise in what case the initialisation layer \eqref{eq:InitialisationSchemesGinzburg}, thus what we need to control, is simple but still determines the dynamics eventually in time.
This property comes from the fact that the root of $\determinant(\timeShift \identity - \fourierTransformed{\schemeMatrix}(\vectorial{\frequency}\spaceStep))$ setting the consistency of the bulk scheme---\emph{i.e.} being one in the low-frequency limit---is also a root of $\annhilitaingPolyGinzburgWithOrderFourier{\observabilityIndex} (\timeShift)$. This is a consequence of the fact that $\annhilitaingPolyGinzburgWithOrderFourier{\observabilityIndex} (\timeShift)$ annihilates the first row of $\fourierTransformed{\schemeMatrix}(\vectorial{\frequency} \spaceStep)$. 
\begin{proposition}
    Let $\fourierTransformed{\discrete{\eigenvalueLetter}}_1 = \fourierTransformed{\discrete{\eigenvalueLetter}}_1 (\vectorial{\frequency}\spaceStep)$ be the unique root of $\determinant(\timeShift \identity - \fourierTransformed{\schemeMatrix}(\vectorial{\frequency}\spaceStep))$ such that 
    \begin{equation}\label{eq:desiredProperty}
        \fourierTransformed{\discrete{\eigenvalueLetter}}_1(\vectorial{\frequency}\spaceStep)= 1 + \bigO{|\vectorial{\frequency}\spaceStep|}
    \end{equation}
    in the limit $|\vectorial{\frequency}\spaceStep| \ll 1$, which is the one determining the consistency and the modified equation of the bulk scheme. 
    Then, $\fourierTransformed{\discrete{\eigenvalueLetter}}_1$  is also a root of $\annhilitaingPolyGinzburgWithOrderFourier{\observabilityIndex} (\timeShift)$.
\end{proposition}
\begin{proof}
Let us show this, using the Fourier representation and considering $\spatialDimensionality = 1$ for the sake of keeping notations simple.
Recall that 
\begin{equation}\label{eq:tmpSameEigenvalue}
    \renewcommand{\arraystretch}{0.7}
    \annhilitaingPolyGinzburgWithOrderFourier{\observabilityIndex} (\fourierTransformed{\schemeMatrix}(\frequency \spaceStep)) = \fourierTransformed{\schemeMatrix}(\frequency \spaceStep)^{\observabilityIndex} + \sum_{\indiceTime = 1}^{\observabilityIndex} \fourierTransformed{\annhilitaingPolyGinzburgCoefficient}_{\observabilityIndex, \indiceTime} (\frequency \spaceStep) \fourierTransformed{\schemeMatrix}(\frequency \spaceStep)^{\indiceTime - 1} = 
    \begin{bmatrix}
        0 & \cdots & 0 \\
        \star & \cdots & \star \\
        \vdots & & \vdots \\
        \star & \cdots & \star 
    \end{bmatrix},
\end{equation}
where the starred $\star$ entries are not necessarily zero.
Notice that, whatever the scaling between time and space, we have that for every $\indiceFreeOne \in \naturals$
\begin{equation}\label{eq:tmpSharingEigenvalue}
    \renewcommand{\arraystretch}{0.7}
    \fourierTransformed{\schemeMatrix}(\frequency \spaceStep)^{\indiceFreeOne} = 
    \begin{bmatrix}
        1 & \cdots & 0 \\
        \star & \cdots & \star \\
        \vdots & & \vdots \\
        \star & \cdots & \star 
    \end{bmatrix} + \bigO{|\frequency \spaceStep|}
\end{equation}
in the limit $|\frequency \spaceStep| \ll 1$.
In particular, for the acoustic scaling, we have $\fourierTransformed{\schemeMatrix}(\frequency \spaceStep)^{\indiceFreeOne} = \collisionMatrix^{\indiceFreeOne}  + \bigO{|\frequency \spaceStep|}$, where $\collisionMatrix^{\indiceFreeOne}$ has the property stated by \eqref{eq:tmpSharingEigenvalue} and $\collisionMatrix$ is the collision matrix.
Again taking  $|\frequency \spaceStep| \ll 1$ and considering that $\fourierTransformed{\annhilitaingPolyGinzburgCoefficient}_{\observabilityIndex, \indiceTime} (\frequency \spaceStep) = \termAtOrder{\fourierTransformed{\annhilitaingPolyGinzburgCoefficientLetter}_{\observabilityIndex, \indiceTime}}{0} + \bigO{|\frequency \spaceStep|}$, selecting the very first entry in \eqref{eq:tmpSameEigenvalue} yields, using \eqref{eq:tmpSharingEigenvalue}
\begin{equation}\label{eq:tmpSameCleDeVoute}
    1 + \sum_{\indiceTime = 1}^{\observabilityIndex} \termAtOrder{\fourierTransformed{\annhilitaingPolyGinzburgCoefficientLetter}_{\observabilityIndex, \indiceTime}}{0} = \bigO{|\frequency \spaceStep|}.
\end{equation}
Since $\complex$ is an algebraically closed field, we can write $\annhilitaingPolyGinzburgWithOrderFourier{\observabilityIndex} (\timeShift) = \prod_{\indiceFreeOne = 1}^{\indiceFreeOne = \observabilityIndex} (\timeShift - \fourierTransformed{\discrete{r}}_{\indiceFreeOne}(\frequency \spaceStep))$, where $\fourierTransformed{\discrete{r}}_{\indiceFreeOne}$ for $\indiceFreeOne \in \integerInterval{1}{\observabilityIndex}$ are the roots of $\annhilitaingPolyGinzburgWithOrderFourier{\observabilityIndex} (\timeShift)$.
These are also part of the roots of $\determinant(\timeShift \identity - \fourierTransformed{\schemeMatrix}(\frequency \spaceStep))$ since $\annhilitaingPolyGinzburgWithOrderFourier{\observabilityIndex} (\timeShift)$ divides $\determinant(\timeShift \identity - \fourierTransformed{\schemeMatrix}(\frequency \spaceStep))$.
The question is whether the roots of $\annhilitaingPolyGinzburgWithOrderFourier{\observabilityIndex} (\timeShift)$ include the one of $\determinant(\timeShift \identity - \fourierTransformed{\schemeMatrix}(\frequency \spaceStep))$, indicated by $\fourierTransformed{\discrete{\eigenvalueLetter}}_1(\frequency\spaceStep)$, being the only one such that $\fourierTransformed{\discrete{\eigenvalueLetter}}_1(\frequency\spaceStep) = 1 + \bigO{|\frequency\spaceStep|}$ in the limit $|\frequency\spaceStep| \ll 1$, see \eqref{eq:physicalEigenvalueCharact}, and which totally dictates consistency (and the modified equations).
Considering $\timeShift = 1$ in \eqref{eq:DefinitionGinzburgPolynomial} gives
\begin{equation*}
    \annhilitaingPolyGinzburgWithOrderFourier{\observabilityIndex} (1) = 1 + \sum_{\indiceTime = 1}^{\observabilityIndex} \fourierTransformed{\annhilitaingPolyGinzburgCoefficient}_{\observabilityIndex, \indiceTime}(\frequency\spaceStep).
\end{equation*}
Taking the limit  $|\frequency\spaceStep| \ll 1$, we are left with 
\begin{equation*}
    \prod_{\indiceFreeOne = 1}^{\observabilityIndex} (1 - \termAtOrder{\fourierTransformed{{r}}_{\indiceFreeOne}}{0}) + \bigO{|\frequency\spaceStep| } = 1 + \sum_{\indiceTime = 1}^{\observabilityIndex} \termAtOrder{\fourierTransformed{\annhilitaingPolyGinzburgCoefficientLetter}_{\observabilityIndex, \indiceTime}}{0} + \bigO{|\frequency\spaceStep| } = \bigO{|\frequency\spaceStep| },
\end{equation*}
thanks to \eqref{eq:tmpSameCleDeVoute}, where $\fourierTransformed{\discrete{r}}_{\indiceFreeOne} = \termAtOrder{\fourierTransformed{{r}}_{\indiceFreeOne}}{0}  +  \bigO{|\frequency\spaceStep| }$.
This gives $\prod_{\indiceFreeOne = 1}^{\indiceFreeOne = \observabilityIndex} (1 - \termAtOrder{\fourierTransformed{{r}}_{\indiceFreeOne}}{0}) = 0$, hence at least one $\termAtOrder{\fourierTransformed{{r}}_{\indiceFreeOne}}{0} = 1$. Since the roots $\fourierTransformed{\discrete{r}}_{\indiceFreeOne}$ for $\indiceFreeOne \in \integerInterval{1}{\observabilityIndex}$  are a subset of those of $\determinant(\timeShift \identity - \fourierTransformed{\schemeMatrix}(\frequency \spaceStep))$, where only one has the desired property \eqref{eq:desiredProperty}, then the latter is also a root of  $\annhilitaingPolyGinzburgWithOrderFourier{\observabilityIndex} (\timeShift)$, let us say $\fourierTransformed{\discrete{r}}_{1} \equiv \fourierTransformed{\discrete{\eigenvalueLetter}}_1$.
\end{proof}

\subsection{An important case: link \scheme{\spatialDimensionality}{1+2\blockNumberGinzburg} two-relaxation-times schemes with magic parameters equal to $1/4$}\label{sec:LinkScheme}

We are now ready to consider a quite wide class of schemes \cite{d2009viscosity} for which very little \iniSchemes{} are to consider, namely $\observabilityIndex$ is particularly small.
The ``observable'' features of these schemes are to some extent independent from $\spatialDimensionality$ and the choice of the $\velocityNumber = 1 + 2\blockNumberGinzburg$ discrete velocities.
This boils down to a quite general application of the ideas of \Cref{sec:ReducedNumberInitialisationSchemes}.

\subsubsection{Description of the schemes}

Consider any spatial dimension $\spatialDimensionality$ and $\velocityNumber = 1 + 2\blockNumberGinzburg$ velocities with $\blockNumberGinzburg \in \nonZeroNaturals$, which is the number of so-called ``links''.
The velocities should be opposite along each link, so such that
\begin{equation}\label{eq:LinkSchemeVelocities}
    \vectorial{\discreteVelocityNormalized}_1 = \vectorial{0}, \qquad \vectorial{\discreteVelocityNormalized}_{2\indiceDistributions} = -\vectorial{\discreteVelocityNormalized}_{2\indiceDistributions + 1} \in \relatives^{\spatialDimensionality}, \qquad \indiceDistributions \in \integerInterval{1}{\blockNumberGinzburg},
\end{equation}
and the moment matrix
\begin{equation}\label{eq:LinkSchemeMoments}
    \momentMatrix = \left [ ~~
    \begin{matrix}
        1 & \rvline & 1 & 1 & \rvline & \cdots  & \rvline & 1 & 1 \\
        \hline 
        0 & \rvline & 1 & -1 & \rvline &  & \rvline & & \\
        0 & \rvline & 1 &  1 & \rvline &  & \rvline & & \\
        \hline
        \vdots & \rvline & & & \rvline & \ddots & \rvline & & \\
        \hline
        0 & \rvline &  &  & \rvline &  & \rvline & 1 & -1 \\
        0 & \rvline &  &  & \rvline &  & \rvline & 1 & 1 \\
    \end{matrix} ~~ \right ] \in \matrixSpace{1+2\blockNumberGinzburg}{\reals}.
\end{equation}
Here, empty blocks shall indicate null blocks of suitable size.
The relaxation parameters should be such that 
\begin{equation}\label{eq:LinkSchemeRelaxation}
    \relaxationParameter_{2\indiceBlock} = \relaxationParameter \in ]0, 2], \qquad \relaxationParameter_{2\indiceBlock + 1} = 2 - \relaxationParameter, \qquad \indiceBlock \in \integerInterval{1}{\blockNumberGinzburg}.
\end{equation}
This is equivalent to having the so-called ``magic parameter'' of every block $\indiceBlock \in \integerInterval{1}{\blockNumberGinzburg}$, given by $(1/\relaxationParameter_{2\indiceBlock} - 1/2)(1/\relaxationParameter_{2\indiceBlock+1} - 1/2)$, equal to $1/4$.
These magic parameters stem from the product of the so-called ``H\'enon parameters'' for even and odd terms.
Remark that other values for the magic parameters---that do not impact observability and are beyond the scope of this paper---have been investigated in the past \cite{kuzmin2011role}.
We mention the choice $1/12$ (respectively $1/6$) that eliminates third (respectively fourth) order spatial errors and $3/16$, resulting in desired effects when introducing boundary conditions.

\subsubsection{Observability and number of initialisation steps}\label{sec:GinzburgObservability}

The study of the observability of the previously described schemes is carried in the following result.
\begin{proposition}\label{prop:Ginzburg}
    The characteristic polynomial of the scheme matrix $\schemeMatrix$ for the schemes given by \eqref{eq:LinkSchemeVelocities}, \eqref{eq:LinkSchemeMoments}, and \eqref{eq:LinkSchemeRelaxation} is given by
    \begin{equation*}
        \determinant(\timeShift \identity - \schemeMatrix) = (\timeShift + (1-\relaxationParameter)) ( \timeShift^2 - (1- \relaxationParameter)^2 )^{\blockNumberGinzburg - 1} \annhilitaingPolyGinzburg(\timeShift),
    \end{equation*}
    where 
    \begin{equation}\label{eq:PolyGinzburgTRT}
        \annhilitaingPolyGinzburg(\timeShift) = \timeShift^2 + (\relaxationParameter - 2) \timeShift + (1-\relaxationParameter) - \timeShift {\relaxationParameter} \sum_{\indiceBlock = 1}^{\blockNumberGinzburg} \antisymmetricPart(\vectorial{\basicShiftLetter}^{\vectorial{\discreteVelocityNormalized}_{2\indiceBlock}}) \equilibriumCoefficient_{2\indiceBlock} + \timeShift {(\relaxationParameter - 2)} \sum_{\indiceBlock = 1}^{\blockNumberGinzburg} (\symmetricPart(\vectorial{\basicShiftLetter}^{\vectorial{\discreteVelocityNormalized}_{2\indiceBlock}}) - 1)\equilibriumCoefficient_{2\indiceBlock+1}
    \end{equation}
    annihilates the first row of the matrix $\schemeMatrix$. Therefore $\observabilityIndex = 2$ if $\relaxationParameter \neq 1$ and $\observabilityIndex = 1$ if $\relaxationParameter = 1$.
    Equivalently, the transfer function of the system is given by
    \begin{equation*}
        \transferFunction(\timeShift) = \frac{(\timeShift + (1-\relaxationParameter)) ( \timeShift^2 - (1- \relaxationParameter)^2 )^{\blockNumberGinzburg - 1} \Bigl ({\relaxationParameter} \sum_{\indiceBlock = 1}^{\indiceBlock = \blockNumberGinzburg} \antisymmetricPart(\vectorial{\basicShiftLetter}^{\vectorial{\discreteVelocityNormalized}_{2\indiceBlock}}) \equilibriumCoefficient_{2\indiceBlock} +  {(2 - \relaxationParameter)} \sum_{\indiceBlock = 1}^{\indiceBlock = \blockNumberGinzburg} (\symmetricPart(\vectorial{\basicShiftLetter}^{\vectorial{\discreteVelocityNormalized}_{2\indiceBlock}}) - 1)\equilibriumCoefficient_{2\indiceBlock+1}  \Bigr )\timeShift}{(\timeShift + (1-\relaxationParameter)) ( \timeShift^2 - (1- \relaxationParameter)^2 )^{\blockNumberGinzburg - 1} (\timeShift^2 + (\relaxationParameter - 2) \timeShift + (1-\relaxationParameter))}.
    \end{equation*}
\end{proposition}
By \Cref{prop:Ginzburg}, $\annhilitaingPolyGinzburg (\timeShift)$ yields a \bulkScheme{} according to \Cref{alg:CorrespondingFDSchemeGinzburh}.
As for \Cref{ex:D1Q3Ginzburg}, the modified equations of the method obtained by $\annhilitaingPolyGinzburg(\timeShift)$ and by $ \determinant(\timeShift \identity - \schemeMatrix)$ are the same because the remaining roots do not concern consistency with respect to \eqref{eq:CauchyEquation}. The only consistency eigenvalue is one of the two roots of $\annhilitaingPolyGinzburg(\timeShift)$, thus present in both schemes.
The case $\relaxationParameter = 2$ apparently questions the previous claim since by looking at the proof of \Cref{prop:MatchEventually}, a scheme consistent with \eqref{eq:CauchyEquation} has only one eigenvalue equal to one for small wavenumbers.
This is not a contradiction because in this case $\relaxationParameter_{2\indiceBlock + 1} = 0$ for $\indiceBlock \in \integerInterval{1}{\blockNumberGinzburg}$, thus the corresponding moments are conserved \cite{ginzburg2008two}, whereas \Cref{thm:EquivEqBulk} has been demonstrated under the assumption that $\relaxationParameter_{\indiceMoments} \neq 0$ for $\indiceMoments \in \integerInterval{2}{\velocityNumber}$ and the whole paper relies on the assumption that we deal only with one conserved moment. The moments $\discrete{\momentLetter}_{2\indiceBlock + 1}$ for $\indiceBlock \in \integerInterval{1}{\blockNumberGinzburg}$ are conserved not because their equilibrium value equals the respective moment itself, but rather since their corresponding relaxation parameter is zero. A valid proof of \Cref{thm:EquivEqBulk} for several conserved moments has to follow the indications of \cite{bellotti2021fd,bellotti2021equivalentequations} and would still lead to \eqref{eq:BulkEquivalentEquation}.
Using the results of \cite{bellotti2021equivalentequations} with $\relaxationParameter = 2$, we would get, for the first moment
\begin{equation}\label{eq:ModEqOfInterest}
    \partial_{\timeVariable} \testFunction(\timeVariable, \vectorial{\spaceVariable}) + \latticeVelocity \sum_{\indiceBlock = 1}^{\blockNumberGinzburg} \equilibriumCoefficient_{2\indiceBlock} \sum_{|\boldOther{\indiceMultiIndexDifferential}| = 1} \vectorial{\discreteVelocityNormalized}_{2\indiceBlock}^{\boldOther{\indiceMultiIndexDifferential}} \partial_{\vectorial{\spaceVariable}}^{\boldOther{\indiceMultiIndexDifferential}} \testFunction(\timeVariable, \vectorial{\spaceVariable}) = \bigO{\spaceStep^2}.
\end{equation}
For the conserved moments $\discrete{\momentLetter}_{2\indiceBlock + 1}$ for $\indiceBlock \in \integerInterval{1}{\blockNumberGinzburg}$, we obtain 
\begin{equation}\label{eq:ModEqFict}
    \partial_{\timeVariable} \momentLetter_{2\indiceBlock + 1} (\timeVariable, \vectorial{\spaceVariable})  + \latticeVelocity  \equilibriumCoefficient_{2\indiceBlock} \sum_{|\boldOther{\indiceMultiIndexDifferential}| = 1} \vectorial{\discreteVelocityNormalized}_{2\indiceBlock}^{\boldOther{\indiceMultiIndexDifferential}} \partial_{\vectorial{\spaceVariable}}^{\boldOther{\indiceMultiIndexDifferential}} \testFunction(\timeVariable, \vectorial{\spaceVariable}) = \bigO{\spaceStep^2}.
\end{equation}
Observe that the equation \eqref{eq:ModEqOfInterest} for the moment of interest is indeed independent of the other conserved moments, as desired (no possible coupling \emph{via} the equilibria, for they depend only on the first conserved moment).
Quite the opposite, the equations \eqref{eq:ModEqFict} for the ``inadvertently'' conserved moments couple them with the first one.
The first conserved moment is going to evolve alone, as usual, and the dynamics of the other conserved moments is going to be coupled with the one of $\discrete{\momentLetter}_1$ according to  \eqref{eq:ModEqFict}. Still, we are not interested in the latter moments.
Coming back to our case, where we operate as only one moment was conserved even when $\relaxationParameter = 2$, the multiplicative factor $ (\timeShift + (1-\relaxationParameter))  ( \timeShift^2 - (1- \relaxationParameter)^2  )^{\blockNumberGinzburg - 1} \asymptoticEquivalence (\text{exp}({\tfrac{\spaceStep}{\latticeVelocity} \partial_{\timeVariable}}) - 1)  ( \text{exp}({2\tfrac{\spaceStep}{\latticeVelocity} \partial_{\timeVariable}}) - 1  )^{\blockNumberGinzburg - 1} = \bigO{\spaceStep^{\blockNumberGinzburg}}$ in front of the amplification polynomial $\annhilitaingPolyGinzburgWithOrder{2}$ of a leap-frog scheme is a series of time differential operators starting with a term of kind $\spaceStep^{\blockNumberGinzburg} \partial_{\timeVariable}^{\blockNumberGinzburg}$.
Thus, if we compute the modified equation of the corresponding \fd{} scheme obtained as we were dealing only with one conserved moment (\emph{i.e.} \eqref{eq:FDCorrespondingOneMomentInitialCondition}) whereas several conserved moments are present, we would obtain a sort of wave equation with time derivative of order $1 + \blockNumberGinzburg$. This is unsurprising since this kind of equation features $1 + \blockNumberGinzburg$ ``consistency'' eigenvalues (one of which, the actual one, is inside $\annhilitaingPolyGinzburg (\timeShift)$) which values equal one for small wavenumbers.

Coming back to generic $\relaxationParameter$, the stability conditions of the corresponding \fd obtained by using $\annhilitaingPolyGinzburg(\timeShift)$ instead of $ \determinant(\timeShift \identity - \schemeMatrix)$ are the same because the remaining roots are constant in wavenumber and do not exceed modulus one when $\relaxationParameter \in ]0, 2[$.
The case $\relaxationParameter = 2$ might produce instabilities because of the presence of multiple roots of $\determinant(\timeShift\identity-\fourierTransformed{\schemeMatrix})$ on the unit circle. 
However, in this case, there are additional conserved moments and the \emph{von Neumann} condition for systems is that no root is outside the unit circle, with no precision concerning multiple ones
on the unit circle. Still this condition is only necessary for stability. Therefore, the presence of multiple eigenvalues on the unit circle (and in particular those concerning consistency which are now $1+\blockNumberGinzburg$)
cannot allow to deduce that the scheme is unstable. This should be precisely tested in the case where $\blockNumberGinzburg \geq 2$, for example, taking a \scheme{1}{5} scheme.

Since $\determinant (\observabilityMatrix) = 0$, the system is not observable according to \cite{brewer1986linear}.
However, it is not easy to generally characterize the unobservable sub-space $\unObservableSubSpace$, because this sub-space inflates with $\spatialDimensionality$ and $\blockNumberGinzburg$ due to the rank-nullity theorem.
To explain this difficulty, consider that $\annhilitaingPolyGinzburg$ from \eqref{eq:PolyGinzburgTRT} is essentially scheme-independent and concerns the ``observable'' part of the system relative to $\text{span}(\observabilityMatrix)$, whereas $\unObservableSubSpace = \kernel{\observabilityMatrix}$ must be highly scheme-dependent because it pertains to the remaining ``unobservable'' part of the system, which is encoded in the quotient $\determinant(\timeShift \identity - \schemeMatrix) / \annhilitaingPolyGinzburg(\timeShift)$ between polynomials.
Let us now proceed to the proof of \Cref{prop:Ginzburg}.

\begin{proof}[Proof of \Cref{prop:Ginzburg}]
The transport matrix is given by
\begin{equation*}
    \transportMoment = \left [ ~~
    \begin{matrix}
        1 & \rvline & \antisymmetricPart(\vectorial{\basicShiftLetter}^{\vectorial{\discreteVelocityNormalized}_2}) & \symmetricPart(\vectorial{\basicShiftLetter}^{\vectorial{\discreteVelocityNormalized}_2})-1 & \rvline & \cdots  & \rvline & \antisymmetricPart(\vectorial{\basicShiftLetter}^{\vectorial{\discreteVelocityNormalized}_{2\blockNumberGinzburg}}) & \symmetricPart(\vectorial{\basicShiftLetter}^{\vectorial{\discreteVelocityNormalized}_{2\blockNumberGinzburg}})-1 \\
        \hline 
        0 & \rvline & \symmetricPart(\vectorial{\basicShiftLetter}^{\vectorial{\discreteVelocityNormalized}_{2}}) & \antisymmetricPart(\vectorial{\basicShiftLetter}^{\vectorial{\discreteVelocityNormalized}_{2}}) & \rvline &  & \rvline & & \\
        0 & \rvline & \antisymmetricPart(\vectorial{\basicShiftLetter}^{\vectorial{\discreteVelocityNormalized}_{2}}) & \symmetricPart(\vectorial{\basicShiftLetter}^{\vectorial{\discreteVelocityNormalized}_{2}}) & \rvline &  & \rvline & & \\
        \hline
        \vdots & \rvline & & & \rvline & \ddots & \rvline & & \\
        \hline
        0 & \rvline &  &  & \rvline &  & \rvline &  \symmetricPart(\vectorial{\basicShiftLetter}^{\vectorial{\discreteVelocityNormalized}_{2\blockNumberGinzburg}}) & \antisymmetricPart(\vectorial{\basicShiftLetter}^{\vectorial{\discreteVelocityNormalized}_{2\blockNumberGinzburg}}) \\
        0 & \rvline &  &  & \rvline &  & \rvline & \antisymmetricPart(\vectorial{\basicShiftLetter}^{\vectorial{\discreteVelocityNormalized}_{2\blockNumberGinzburg}}) & \symmetricPart(\vectorial{\basicShiftLetter}^{\vectorial{\discreteVelocityNormalized}_{2\blockNumberGinzburg}}) \\
    \end{matrix} ~~ \right ] \in \matrixSpace{1+2\blockNumberGinzburg}{\ringSpaceOperators}.
\end{equation*}
We have that $\determinant(\timeShift \identity - \schemeMatrix) = \determinant(\timeShift \identity - \schemeMatrixMoments- \schemeMatrixEquilibria \vectorial{\equilibriumCoefficient} \otimes \canonicalBasisVector_1) = \determinant(\timeShift \identity - \schemeMatrixMoments) - \transpose{\canonicalBasisVector_1} \adjugateMatrix(\timeShift \identity - \schemeMatrixMoments ) \schemeMatrixEquilibria \boldOther{\equilibriumCoefficient}$, using the matrix determinant lemma \cite{horn2012matrix}.
Also
\begin{equation*}
    \adjugateMatrix (\timeShift \identity - \schemeMatrixMoments )_{11} = \prod_{\indiceBlock = 1}^{\blockNumberGinzburg} \determinant(\timeShift \identity - \tilde{\transportMoment}_{\indiceBlock} \diagMatrix(1-\relaxationParameter, \relaxationParameter-1) ) = \prod_{\indiceBlock = 1}^{\blockNumberGinzburg}  ( \timeShift^2 - (1- \relaxationParameter)^2  ) =  ( \timeShift^2 - (1- \relaxationParameter)^2  )^{\blockNumberGinzburg},
\end{equation*}
with 
\begin{equation*}
    \tilde{\transportMoment}_{\indiceBlock} = 
    \begin{bmatrix}
        \symmetricPart(\vectorial{\basicShiftLetter}^{\vectorial{\discreteVelocityNormalized}_{2\indiceBlock}}) & \antisymmetricPart(\vectorial{\basicShiftLetter}^{\vectorial{\discreteVelocityNormalized}_{2\indiceBlock}}) \\
        \antisymmetricPart(\vectorial{\basicShiftLetter}^{\vectorial{\discreteVelocityNormalized}_{2\indiceBlock}}) & \symmetricPart(\vectorial{\basicShiftLetter}^{\vectorial{\discreteVelocityNormalized}_{2\indiceBlock}})
    \end{bmatrix}.
\end{equation*}
We only treat $\adjugateMatrix (\timeShift \identity - \schemeMatrixMoments )_{12}$ and $\adjugateMatrix (\timeShift \identity - \schemeMatrixMoments )_{13}$, since the following entries read the same except for the indices of the involved shift operators.
\begin{align*}
    &\adjugateMatrix (\timeShift \identity - \schemeMatrixMoments )_{12} \\
    &= - \determinant \left [ ~~
    \begin{matrix}
        \begin{matrix}
            (\relaxationParameter-1) \antisymmetricPart (\vectorial{\basicShiftLetter}^{\vectorial{\discreteVelocityNormalized}_{2}} ) & (1-\relaxationParameter)(\symmetricPart(\vectorial{\basicShiftLetter}^{\vectorial{\discreteVelocityNormalized}_{2}}) -1 ) \\
            {(\relaxationParameter - 1)}\antisymmetricPart(\vectorial{\basicShiftLetter}^{\vectorial{\discreteVelocityNormalized}_{2}} ) & \timeShift + (1-\relaxationParameter) \symmetricPart (\vectorial{\basicShiftLetter}^{\vectorial{\discreteVelocityNormalized}_{2}})
        \end{matrix} & \rvline & \star & \rvline & \cdots & \rvline & \star \\
        \hline
        & \rvline & \timeShift \identity - \tilde{\transportMoment}_{2}\diagMatrix(1-\relaxationParameter, \relaxationParameter-1) & \rvline & \star & \rvline & \star\\
        \hline 
        & \rvline &  & \rvline & \ddots & \rvline &\star \\
        \hline 
        & \rvline & & \rvline& & \rvline & \timeShift \identity - \tilde{\transportMoment}_{\blockNumberGinzburg}\diagMatrix(1-\relaxationParameter, \relaxationParameter-1)\\
    \end{matrix} ~~ \right ],
\end{align*}
where the $\star$ blocks are not necessarily zero but do not need to be further characterized, since this is a determinant of a block upper triangular matrix, whence
\begin{align*}
    \adjugateMatrix (\timeShift \identity - \schemeMatrixMoments )_{1, 2\indiceBlock} &= -  ( \timeShift^2 - (1- \relaxationParameter)^2  )^{\blockNumberGinzburg-1} \determinant
    \begin{bmatrix}
        {(\relaxationParameter - 1)} \antisymmetricPart (\vectorial{\basicShiftLetter}^{\vectorial{\discreteVelocityNormalized}_{2\indiceBlock}} ) & {(1-\relaxationParameter)} (\symmetricPart(\vectorial{\basicShiftLetter}^{\vectorial{\discreteVelocityNormalized}_{2\indiceBlock}}) -1 ) \\
        {(\relaxationParameter - 1)}\antisymmetricPart(\vectorial{\basicShiftLetter}^{\vectorial{\discreteVelocityNormalized}_{2\indiceBlock}} ) & \timeShift + (1-\relaxationParameter) \symmetricPart (\vectorial{\basicShiftLetter}^{\vectorial{\discreteVelocityNormalized}_{2\indiceBlock}})
    \end{bmatrix} , \\
    &={(1-\relaxationParameter)} (\timeShift + (1-\relaxationParameter))  ( \timeShift^2 - (1- \relaxationParameter)^2  )^{\blockNumberGinzburg-1} \antisymmetricPart(\vectorial{\basicShiftLetter}^{\vectorial{\discreteVelocityNormalized}_{2\indiceBlock}}), \qquad \indiceBlock\in \integerInterval{1}{\blockNumberGinzburg}.
\end{align*}
The analogous computation for the odd moments yields
\begin{align*}
    \adjugateMatrix (\timeShift \identity - \schemeMatrixMoments )_{1, 2\indiceBlock+1} &=  ( \timeShift^2 - (1- \relaxationParameter)^2  )^{\blockNumberGinzburg-1} \determinant
    \begin{bmatrix}
        {(\relaxationParameter - 1)}\antisymmetricPart (\vectorial{\basicShiftLetter}^{\vectorial{\discreteVelocityNormalized}_{2\indiceBlock}}) & (1-\relaxationParameter) (\symmetricPart(\vectorial{\basicShiftLetter}^{\vectorial{\discreteVelocityNormalized}_{2\indiceBlock}}) -1 ) \\
        \timeShift {(\relaxationParameter - 1)} \symmetricPart(\vectorial{\basicShiftLetter}^{\vectorial{\discreteVelocityNormalized}_{2\indiceBlock}}) & (1-\relaxationParameter) \symmetricPart(\vectorial{\basicShiftLetter}^{\vectorial{\discreteVelocityNormalized}_{2\indiceBlock}})
    \end{bmatrix} \\
    &={(\relaxationParameter - 1)} (\timeShift + (1-\relaxationParameter))  ( \timeShift^2 - (1- \relaxationParameter)^2  )^{\blockNumberGinzburg-1} (\symmetricPart(\vectorial{\basicShiftLetter}^{\vectorial{\discreteVelocityNormalized}_{2\indiceBlock}}) -1 ), \qquad \indiceBlock\in \integerInterval{1}{\blockNumberGinzburg}.
\end{align*}
Some algebra provides, for $\indiceBlock \in \integerInterval{1}{\blockNumberGinzburg}$
\begin{equation*}
    (\schemeMatrixEquilibria \boldOther{\equilibriumCoefficient})_{\indiceFreeTwo} = 
    \begin{cases}
        {\relaxationParameter} \sum_{\indiceFreeThree = 1}^{\blockNumberGinzburg} \antisymmetricPart(\vectorial{\basicShiftLetter}^{\vectorial{\discreteVelocityNormalized}_{2\indiceFreeThree}})\equilibriumCoefficient_{2\indiceFreeThree} + (2-\relaxationParameter)\sum_{\indiceFreeThree = 1}^{\indiceFreeThree = \blockNumberGinzburg} (\symmetricPart(\vectorial{\basicShiftLetter}^{\vectorial{\discreteVelocityNormalized}_{2\indiceFreeThree}}) -1 )\equilibriumCoefficient_{2\indiceFreeThree+1}, \qquad &\indiceFreeTwo = 1, \\
        {\relaxationParameter} \symmetricPart(\vectorial{\basicShiftLetter}^{\vectorial{\discreteVelocityNormalized}_{2\indiceBlock}})\equilibriumCoefficient_{2\indiceBlock} + (2-\relaxationParameter) \antisymmetricPart(\vectorial{\basicShiftLetter}^{\vectorial{\discreteVelocityNormalized}_{2\indiceBlock}} )\equilibriumCoefficient_{2\indiceBlock + 1}, \qquad &\indiceFreeTwo = 2\indiceBlock, \\
        { \relaxationParameter} \antisymmetricPart(\vectorial{\basicShiftLetter}^{\vectorial{\discreteVelocityNormalized}_{2\indiceBlock}} )\equilibriumCoefficient_{2\indiceBlock} +(2-\relaxationParameter) \symmetricPart(\vectorial{\basicShiftLetter}^{\vectorial{\discreteVelocityNormalized}_{2\indiceBlock}} )\equilibriumCoefficient_{2\indiceBlock + 1}, \qquad &\indiceFreeTwo = 2\indiceBlock + 1,
    \end{cases}
\end{equation*}
and thus, after tedious computations
\begin{equation*}
    \transpose{\canonicalBasisVector_1} \adjugateMatrix(\timeShift \identity - \schemeMatrixMoments )\schemeMatrixEquilibria \boldOther{\equilibriumCoefficient} = \timeShift (\timeShift + (1-\relaxationParameter))   ( \timeShift^2 - (1- \relaxationParameter)^2  )^{\blockNumberGinzburg-1} \Bigl ( {\relaxationParameter} \sum_{\indiceBlock = 1}^{\blockNumberGinzburg} \antisymmetricPart(\vectorial{\basicShiftLetter}^{\vectorial{\discreteVelocityNormalized}_{2\indiceBlock}})\equilibriumCoefficient_{2\indiceBlock} + {(2-\relaxationParameter)} \sum_{\indiceBlock = 1}^{\blockNumberGinzburg} (\symmetricPart(\vectorial{\basicShiftLetter}^{\vectorial{\discreteVelocityNormalized}_{2\indiceBlock}}) - 1 )\equilibriumCoefficient_{2\indiceBlock+1}  \Bigr ).
\end{equation*}
To finish up, since the matrix $\timeShift \identity - \schemeMatrixMoments$ is upper block triangular, we have
\begin{equation*}
    \determinant (\timeShift \identity - \schemeMatrixMoments) = (\timeShift - 1) \prod_{\indiceBlock = 1}^{\blockNumberGinzburg} \determinant(\timeShift \identity - \tilde{\transportMoment}_{\indiceBlock} \diagMatrix(1-\relaxationParameter, \relaxationParameter-1) ) = (\timeShift - 1) ( \timeShift^2 - (1- \relaxationParameter)^2  )^{\blockNumberGinzburg},
\end{equation*}
giving the characteristic polynomial of the scheme.
The property of $\annhilitaingPolyGinzburg(\timeShift)$ annihilating for the first row of $\schemeMatrix$ can be checked analogously to \cite{bellotti2021fd}. Observe that $\annhilitaingPolyGinzburg(\timeShift)$ could also be found solving \eqref{eq:SystemGinzburgCoefficients} by hand.
\end{proof}

\subsubsection{Modified equations under acoustic scaling}

The discussion of \Cref{sec:GinzburgObservability} is fully discrete. 
Now we come back to the asymptotic analysis using modified equations of \Cref{sec:ModifiedEquations} and considering local initialisations, \emph{i.e.} $\boldOther{\initialisationOperator} \in \reals^{\velocityNumber}$.

\begin{proposition}[Modified equations]\label{prop:ModifiedEquationsGinzburg}
    Under acoustic scaling, that is, when $\latticeVelocity > 0$ is fixed as $\spaceStep \to 0$, the modified equation for the \bulkScheme{} \eqref{eq:BulkSchemeGinzburg} where the \lbm scheme is determined by  \eqref{eq:LinkSchemeVelocities}, \eqref{eq:LinkSchemeMoments}, and \eqref{eq:LinkSchemeRelaxation} is
\begin{align*}
    \partial_{\timeVariable} \testFunction(\timeVariable, \vectorial{\spaceVariable}) &+ \latticeVelocity \sum_{\indiceBlock = 1}^{\blockNumberGinzburg} \equilibriumCoefficient_{2\indiceBlock} \sum_{|\boldOther{\indiceMultiIndexDifferential}| = 1} \vectorial{\discreteVelocityNormalized}_{2\indiceBlock}^{\boldOther{\indiceMultiIndexDifferential}} \partial_{\vectorial{\spaceVariable}}^{\boldOther{\indiceMultiIndexDifferential}} \testFunction(\timeVariable, \vectorial{\spaceVariable}) \\
    &- \latticeVelocity\spaceStep \Bigl ( \frac{1}{\relaxationParameter} - \frac{1}{2} \Bigr )  \Bigl ( 2\sum_{\indiceBlock = 1}^{\blockNumberGinzburg} \equilibriumCoefficient_{2\indiceBlock+1} \sum_{|\boldOther{\indiceMultiIndexDifferential}| = 2} \frac{\vectorial{\discreteVelocityNormalized}_{2\indiceBlock}^{\boldOther{\indiceMultiIndexDifferential}}}{\boldOther{\indiceMultiIndexDifferential}!} \partial_{\vectorial{\spaceVariable}}^{\boldOther{\indiceMultiIndexDifferential}} - \Bigl ( \sum_{\indiceBlock = 1}^{\blockNumberGinzburg} \equilibriumCoefficient_{2\indiceBlock} \sum_{|\boldOther{\indiceMultiIndexDifferential}| = 1} \vectorial{\discreteVelocityNormalized}_{2\indiceBlock}^{\boldOther{\indiceMultiIndexDifferential}} \partial_{\vectorial{\spaceVariable}}^{\boldOther{\indiceMultiIndexDifferential}} \Bigr )^2 \Bigr )  \testFunction(\timeVariable, \vectorial{\spaceVariable})  = \bigO{\spaceStep^2},
\end{align*}
for $(\timeVariable, \vectorial{\spaceVariable}) \in \nonNegativeReals \times \reals^{\spatialDimensionality}$.
Under the assumption of local initialisation $\boldOther{\initialisationOperator} \in \reals^{\velocityNumber}$ fulfilling \Cref{prop:LocalInitialisation}, thus having $ \initialisationOperator_1 = 1$ and $ \initialisationOperator_{2\indiceBlock} = \equilibriumCoefficient_{2\indiceBlock}$ for $ \indiceBlock \in \integerInterval{1}{\blockNumberGinzburg}$, the modified equation for the unique \iniScheme{} (\eqref{eq:InitialisationSchemesGinzburg} with $\indiceTime = 1$) is, for $\vectorial{\spaceVariable} \in \reals^{\spatialDimensionality}$
\begin{align*}
    \partial_{\timeVariable} \testFunction(0, \vectorial{\spaceVariable}) &+\latticeVelocity \sum_{\indiceBlock = 1}^{\blockNumberGinzburg} \equilibriumCoefficient_{2\indiceBlock} \sum_{|\boldOther{\indiceMultiIndexDifferential}| = 1} \vectorial{\discreteVelocityNormalized}_{2\indiceBlock}^{\boldOther{\indiceMultiIndexDifferential}} \partial_{\vectorial{\spaceVariable}}^{\boldOther{\indiceMultiIndexDifferential}} \testFunction(0, \vectorial{\spaceVariable}) \\
    &- \frac{\latticeVelocity\spaceStep}{2}  \Bigl ( 2  \sum_{\indiceBlock = 1}^{\blockNumberGinzburg} \bigl ( (2-\relaxationParameter) \equilibriumCoefficient_{2\indiceBlock+1} + (\relaxationParameter - 1) \initialisationOperator_{2\indiceBlock + 1} \bigl ) \sum_{|\boldOther{\indiceMultiIndexDifferential}| = 2} \frac{\vectorial{\discreteVelocityNormalized}_{2\indiceBlock}^{\boldOther{\indiceMultiIndexDifferential}}}{\boldOther{\indiceMultiIndexDifferential}!} \partial_{\vectorial{\spaceVariable}}^{\boldOther{\indiceMultiIndexDifferential}} - \Bigl ( \sum_{\indiceBlock = 1}^{\blockNumberGinzburg} \equilibriumCoefficient_{2\indiceBlock}  \sum_{|\boldOther{\indiceMultiIndexDifferential}| = 1} \vectorial{\discreteVelocityNormalized}_{2\indiceBlock}^{\boldOther{\indiceMultiIndexDifferential}} \partial_{\vectorial{\spaceVariable}}^{\boldOther{\indiceMultiIndexDifferential}} \Bigr )^2 \Bigr )  \testFunction(0, \vectorial{\spaceVariable})  = \bigO{\spaceStep^2}.
\end{align*}
\end{proposition}
\begin{proof}
We have 
\begin{equation*}
    \antisymmetricPart(\vectorial{\basicShiftLetter}^{\vectorial{\discreteVelocityNormalized}_{2\indiceBlock}}) \asymptoticEquivalence -\spaceStep \sum_{|\boldOther{\indiceMultiIndexDifferential}| = 1} \vectorial{\discreteVelocityNormalized}_{2\indiceBlock}^{\boldOther{\indiceMultiIndexDifferential}} \partial_{\vectorial{\spaceVariable}}^{\boldOther{\indiceMultiIndexDifferential}} + \bigO{\spaceStep^3}, \qquad \symmetricPart(\vectorial{\basicShiftLetter}^{\vectorial{\discreteVelocityNormalized}_{2\indiceBlock}}) \asymptoticEquivalence 1 + {\spaceStep^2} \sum_{|\boldOther{\indiceMultiIndexDifferential}| = 2} \frac{\vectorial{\discreteVelocityNormalized}_{2\indiceBlock}^{\boldOther{\indiceMultiIndexDifferential}}}{\boldOther{\indiceMultiIndexDifferential}!} \partial_{\vectorial{\spaceVariable}}^{\boldOther{\indiceMultiIndexDifferential}} + \bigO{\spaceStep^4}.
\end{equation*}
It can be easily checked that the modified equation of the \bulkScheme{} reads as in the claim.
For the \iniScheme{}, only the computation of $\termAtOrder{\schemeMatrixAsymptotic}{0}$, $\termAtOrder{\schemeMatrixAsymptotic}{1}$ and $\termAtOrder{\schemeMatrixAsymptotic}{2}$ is needed:
\begin{align*}
    \termAtOrder{\schemeMatrixAsymptotic}{0}_{1\indiceDistributions} = \delta_{1\indiceDistributions}, \quad \termAtOrder{\schemeMatrixAsymptotic}{1}_{1, \cdot} &= \Bigl [  -\relaxationParameter \sum_{\indiceBlock = 1}^{\blockNumberGinzburg} \equilibriumCoefficient_{2\indiceBlock} \sum_{|\boldOther{\indiceMultiIndexDifferential}| = 1} \vectorial{\discreteVelocityNormalized}_{2\indiceBlock}^{\boldOther{\indiceMultiIndexDifferential}} \partial_{\vectorial{\spaceVariable}}^{\boldOther{\indiceMultiIndexDifferential}}, (\relaxationParameter-1)  \sum_{|\boldOther{\indiceMultiIndexDifferential}| = 1} \vectorial{\discreteVelocityNormalized}_{2}^{\boldOther{\indiceMultiIndexDifferential}} \partial_{\vectorial{\spaceVariable}}^{\boldOther{\indiceMultiIndexDifferential}}, 0, \dots, (\relaxationParameter - 1)  \sum_{|\boldOther{\indiceMultiIndexDifferential}| = 1} \vectorial{\discreteVelocityNormalized}_{2\blockNumberGinzburg}^{\boldOther{\indiceMultiIndexDifferential}} \partial_{\vectorial{\spaceVariable}}^{\boldOther{\indiceMultiIndexDifferential}}, 0 \Bigr ], \\
    \termAtOrder{\schemeMatrixAsymptotic}{2}_{1, \cdot} &=  \Bigl [  (2-\relaxationParameter) \sum_{\indiceBlock = 1}^{\blockNumberGinzburg} \equilibriumCoefficient_{2\indiceBlock+1}  \sum_{|\boldOther{\indiceMultiIndexDifferential}| = 2} \frac{\vectorial{\discreteVelocityNormalized}_{2\indiceBlock}^{\boldOther{\indiceMultiIndexDifferential}}}{\boldOther{\indiceMultiIndexDifferential}!} \partial_{\vectorial{\spaceVariable}}^{\boldOther{\indiceMultiIndexDifferential}}, 0, (\relaxationParameter - 1)  \sum_{|\boldOther{\indiceMultiIndexDifferential}| = 2} \frac{\vectorial{\discreteVelocityNormalized}_{2}^{\boldOther{\indiceMultiIndexDifferential}}}{\boldOther{\indiceMultiIndexDifferential}!} \partial_{\vectorial{\spaceVariable}}^{\boldOther{\indiceMultiIndexDifferential}}, \dots,0, (\relaxationParameter - 1)   \sum_{|\boldOther{\indiceMultiIndexDifferential}| = 2} \frac{\vectorial{\discreteVelocityNormalized}_{2\blockNumberGinzburg}^{\boldOther{\indiceMultiIndexDifferential}}}{\boldOther{\indiceMultiIndexDifferential}!} \partial_{\vectorial{\spaceVariable}}^{\boldOther{\indiceMultiIndexDifferential}} \Bigr ].
\end{align*}
Using the assumptions on the choice of initialisation, the modified equation for the \iniScheme{}, which reads for $\vectorial{\spaceVariable} \in \reals^{\spatialDimensionality}$
\begin{align*}
    \partial_{\timeVariable} \testFunction(0, \vectorial{\spaceVariable}) &+\latticeVelocity \sum_{\indiceBlock = 1}^{\blockNumberGinzburg} \equilibriumCoefficient_{2\indiceBlock} \sum_{|\boldOther{\indiceMultiIndexDifferential}| = 1} \vectorial{\discreteVelocityNormalized}_{2\indiceBlock}^{\boldOther{\indiceMultiIndexDifferential}} \partial_{\vectorial{\spaceVariable}}^{\boldOther{\indiceMultiIndexDifferential}} \testFunction(0, \vectorial{\spaceVariable}) \\
    &- \frac{\latticeVelocity\spaceStep}{2}  \Bigl ( 2  \sum_{\indiceBlock = 1}^{\blockNumberGinzburg} \bigl ( (2-\relaxationParameter) \equilibriumCoefficient_{2\indiceBlock+1} + (\relaxationParameter - 1) \initialisationOperator_{2\indiceBlock + 1} \bigl ) \sum_{|\boldOther{\indiceMultiIndexDifferential}| = 2} \frac{\vectorial{\discreteVelocityNormalized}_{2\indiceBlock}^{\boldOther{\indiceMultiIndexDifferential}}}{\boldOther{\indiceMultiIndexDifferential}!} \partial_{\vectorial{\spaceVariable}}^{\boldOther{\indiceMultiIndexDifferential}} - \Bigl ( \sum_{\indiceBlock = 1}^{\blockNumberGinzburg} \equilibriumCoefficient_{2\indiceBlock}  \sum_{|\boldOther{\indiceMultiIndexDifferential}| = 1} \vectorial{\discreteVelocityNormalized}_{2\indiceBlock}^{\boldOther{\indiceMultiIndexDifferential}} \partial_{\vectorial{\spaceVariable}}^{\boldOther{\indiceMultiIndexDifferential}} \Bigr )^2 \Bigr )  \testFunction(0, \vectorial{\spaceVariable})  = \bigO{\spaceStep^2},
\end{align*}
depends on the choice of initialisation $\initialisationOperator_{2\indiceBlock + 1}$ of the odd moments, which still need to be fixed.
\end{proof}

Enforcing the equality between the dissipation coefficients of the \iniScheme{} and the \bulkScheme{} according to \Cref{prop:ModifiedEquationsGinzburg} provides the differential constraint
\begin{equation*}
    \sum_{\indiceBlock = 1}^{\blockNumberGinzburg}  \initialisationOperator_{2\indiceBlock + 1} \sum_{|\boldOther{\indiceMultiIndexDifferential}| = 2} \frac{\vectorial{\discreteVelocityNormalized}_{2\indiceBlock}^{\boldOther{\indiceMultiIndexDifferential}}}{\boldOther{\indiceMultiIndexDifferential}!} \partial_{\vectorial{\spaceVariable}}^{\boldOther{\indiceMultiIndexDifferential}} = \frac{1}{\relaxationParameter} \Bigl (\Bigl ( \sum_{\indiceBlock = 1}^{\blockNumberGinzburg} \equilibriumCoefficient_{2\indiceBlock} \sum_{|\boldOther{\indiceMultiIndexDifferential}| = 1} \vectorial{\discreteVelocityNormalized}_{2\indiceBlock}^{\boldOther{\indiceMultiIndexDifferential}} \partial_{\vectorial{\spaceVariable}}^{\boldOther{\indiceMultiIndexDifferential}} \Bigr )^2 + (\relaxationParameter - 2)\sum_{\indiceBlock = 1}^{\blockNumberGinzburg}  \equilibriumCoefficient_{2\indiceBlock + 1} \sum_{|\boldOther{\indiceMultiIndexDifferential}| = 2} \frac{\vectorial{\discreteVelocityNormalized}_{2\indiceBlock}^{\boldOther{\indiceMultiIndexDifferential}}}{\boldOther{\indiceMultiIndexDifferential}!} \partial_{\vectorial{\spaceVariable}}^{\boldOther{\indiceMultiIndexDifferential}}\Bigr ).
\end{equation*}
We now provide some examples where this differential constraint can or cannot be fulfilled.

\begin{example}\label{ex:MatchAcousticGinzburg}
    \begin{itemize}
        \item{\scheme{1}{3}} scheme, having $\spatialDimensionality = 1$, $\blockNumberGinzburg = 1$, and $\discreteVelocityNormalized_2 = 1$. After simplifying the second-order derivative operator, the condition reads $\initialisationOperator_3 =  (2\equilibriumCoefficient_2^2 + (\relaxationParameter - 2)\equilibriumCoefficient_3)/\relaxationParameter$, which has to be compared with \eqref{eq:ChoiceInitizialisationMagicD1Q3}.
        \item{\scheme{2}{5}} scheme, having $\spatialDimensionality = 2$, $\blockNumberGinzburg = 2$, $\vectorial{\discreteVelocityNormalized}_2 = \transpose{[1, 0]}$, and $\vectorial{\discreteVelocityNormalized}_4 = \transpose{[0, 1]}$, we obtain
        \begin{equation*}
            \initialisationOperator_3 \partial_{\spaceVariable_1\spaceVariable_1} + \initialisationOperator_5 \partial_{\spaceVariable_2\spaceVariable_2}  = \frac{1}{\relaxationParameter} \bigl ( (2\equilibriumCoefficient_2^2 + (\relaxationParameter - 2)\equilibriumCoefficient_3) \partial_{\spaceVariable_1\spaceVariable_1} + 4\equilibriumCoefficient_2 \equilibriumCoefficient_4 \partial_{\spaceVariable_1 \spaceVariable_2} + (2\equilibriumCoefficient_4^2 + (\relaxationParameter - 2)\equilibriumCoefficient_5) \partial_{\spaceVariable_2\spaceVariable_2} \bigl ),
        \end{equation*}
        which cannot be fulfilled---except when either $\equilibriumCoefficient_2$ or $\equilibriumCoefficient_4$ are zero rendering an essentially 1d problem---due to the presence of the mixed term in $\partial_{\spaceVariable_1 \spaceVariable_2}$ on the right hand side, arising from the hyperbolic part.
        In order to deal with this term, one is compelled to consider a richer scheme with diagonal discrete velocities, such as the \scheme{2}{9} scheme.
        \item{\scheme{2}{9}} scheme, having $\spatialDimensionality = 2$, $\blockNumberGinzburg = 4$, $\vectorial{\discreteVelocityNormalized}_2 = \transpose{[1, 0]}$, $\vectorial{\discreteVelocityNormalized}_4 = \transpose{[0, 1]}$, $\vectorial{\discreteVelocityNormalized}_6 = \transpose{[1, 1]}$, and $\vectorial{\discreteVelocityNormalized}_8 = \transpose{[-1, 1]}$, we obtain
        \begin{align*}
            \frac{1}{2}(\initialisationOperator_3 &+ \initialisationOperator_7 + \initialisationOperator_9) \partial_{\spaceVariable_1\spaceVariable_1} + (\initialisationOperator_7 - \initialisationOperator_9) \partial_{\spaceVariable_1\spaceVariable_2} + \frac{1}{2}(\initialisationOperator_5 + \initialisationOperator_7 + \initialisationOperator_9) \partial_{\spaceVariable_2\spaceVariable_2}\\
            = \frac{1}{\relaxationParameter} \Bigl ( &\underbrace{\Bigl (\equilibriumCoefficient_2^2 + \equilibriumCoefficient_6^2 + \equilibriumCoefficient_8^2 + \equilibriumCoefficient_2\equilibriumCoefficient_6 - \equilibriumCoefficient_2\equilibriumCoefficient_8 - \equilibriumCoefficient_6\equilibriumCoefficient_8+\frac{1}{2}(\relaxationParameter-2) (\equilibriumCoefficient_3 + \equilibriumCoefficient_7 + \equilibriumCoefficient_9)\Bigr )}_{R_{\spaceVariable_1\spaceVariable_1}} \partial_{\spaceVariable_1\spaceVariable_1} \\
            +&\underbrace{\Bigl (2\equilibriumCoefficient_6^2 - 2\equilibriumCoefficient_8^2 + 2\equilibriumCoefficient_2\equilibriumCoefficient_4 + \equilibriumCoefficient_2\equilibriumCoefficient_6 + \equilibriumCoefficient_2\equilibriumCoefficient_8+\equilibriumCoefficient_4\equilibriumCoefficient_6-\equilibriumCoefficient_4\equilibriumCoefficient_8  + (\relaxationParameter-2)(\equilibriumCoefficient_7 - \equilibriumCoefficient_9)\Bigr)}_{R_{\spaceVariable_1\spaceVariable_2}}\partial_{\spaceVariable_1\spaceVariable_2} \\
            +&\underbrace{\Bigl (\equilibriumCoefficient_4^2 + \equilibriumCoefficient_6^2 + \equilibriumCoefficient_8^2 + \equilibriumCoefficient_4\equilibriumCoefficient_6 + \equilibriumCoefficient_4\equilibriumCoefficient_8 + \equilibriumCoefficient_6\equilibriumCoefficient_8+\frac{1}{2}(\relaxationParameter-2) (\equilibriumCoefficient_5 + \equilibriumCoefficient_7 + \equilibriumCoefficient_9)\Bigr )}_{R_{\spaceVariable_2\spaceVariable_2}} \partial_{\spaceVariable_2\spaceVariable_2}\Bigr ).
        \end{align*}
        This system is under-determined, thus it has several solutions.
        For example, picking $\initialisationOperator_9 = 0$, we necessarily enforce $\initialisationOperator_7 = R_{\spaceVariable_1\spaceVariable_2}/\relaxationParameter$ and then we have that $\initialisationOperator_3 =(2R_{\spaceVariable_1\spaceVariable_1} - R_{\spaceVariable_1\spaceVariable_2})/\relaxationParameter$ and $\initialisationOperator_5 = (2R_{\spaceVariable_2\spaceVariable_2} - R_{\spaceVariable_1\spaceVariable_2})/\relaxationParameter$.
    \end{itemize}    
\end{example}

\newcommand{\diffusiveScalingParameter}{\mu}

\subsubsection{Modified equations under diffusive scaling}
As mentioned in the Introduction, the literature also features \lbm schemes used under diffusive scaling \cite{zhao2017maxwell} between time and space discretisations.
We therefore finally consider this scaling where $\timeStep \propto \spaceStep^2$ as $\spaceStep \to 0$, allowing to approximate the solution of 
\begin{numcases}{}
    \partial_{\timeVariable} \solutionCauchy (\timeVariable, \vectorial{\spaceVariable}) + \vectorial{\transportVelocity} \cdot \nabla_{\vectorial{\spaceVariable}} \solutionCauchy (\timeVariable, \vectorial{\spaceVariable}) - \nabla_{\vectorial{\spaceVariable}} \cdot (\matricial{D} \nabla_{\vectorial{\spaceVariable}} \solutionCauchy) (\timeVariable, \vectorial{\spaceVariable})= 0, \qquad (\timeVariable, \vectorial{\spaceVariable}) \in \nonNegativeReals \times \reals^{\spatialDimensionality},\label{eq:CauchyEquationDiffusion} \\
    \solutionCauchy (0, \vectorial{\spaceVariable}) = \solutionCauchyInitial (\vectorial{\spaceVariable}), \qquad \vectorial{\spaceVariable} \in \reals^{\spatialDimensionality}, \label{eq:CauchyInitialDatumDiffusion}
\end{numcases}
where the diffusion matrix is $\matricial{D} \in \matrixSpace{\spatialDimensionality}{\reals}$.
This scaling is difficult to treat in full generality because it requires a consistency study up to order $\bigO{\timeStep} = \bigO{\spaceStep^2}$ included.
Still, as previously highlighted, the unobservable framework of the current \Cref{sec:Ginzburg} allows us to circumvent these difficulties.
The assumptions are slightly different than the rest of the paper.

\begin{theorem}[\cite{bellotti2021equivalentequations} Modified equation of the bulk scheme]
    Under diffusive scaling, that is, when $\latticeVelocity = \diffusiveScalingParameter / \spaceStep$ with $\diffusiveScalingParameter > 0$ fixed as $\spaceStep \to 0$, assuming that $\equilibriumCoefficient_{2\indiceBlock} = \spaceStep \tilde{\equilibriumCoefficient}_{2\indiceBlock}$ where $\tilde{\equilibriumCoefficient}_{2\indiceBlock}$ and $\equilibriumCoefficient_{2\indiceBlock + 1}$ are fixed as $\spaceStep \to 0$ for $\indiceBlock \in \integerInterval{1}{\blockNumberGinzburg}$,   the modified equation for the \bulkScheme{} \eqref{eq:BulkSchemeGinzburg} where the \lbm scheme is determined by \eqref{eq:LinkSchemeVelocities}, \eqref{eq:LinkSchemeMoments}, and \eqref{eq:LinkSchemeRelaxation} is
    \begin{equation*}
        \partial_{\timeVariable} \testFunction(\timeVariable, \vectorial{\spaceVariable}) + \diffusiveScalingParameter \sum_{\indiceBlock = 1}^{\blockNumberGinzburg} \tilde{\equilibriumCoefficient}_{2\indiceBlock} \sum_{|\boldOther{\indiceMultiIndexDifferential}| = 1} \vectorial{\discreteVelocityNormalized}_{2\indiceBlock}^{\boldOther{\indiceMultiIndexDifferential}} \partial_{\vectorial{\spaceVariable}}^{\boldOther{\indiceMultiIndexDifferential}} \testFunction(\timeVariable, \vectorial{\spaceVariable}) - 2 \diffusiveScalingParameter \Bigl ( \frac{1}{\relaxationParameter} - \frac{1}{2}\Bigr )\sum_{\indiceBlock = 1}^{\blockNumberGinzburg} \equilibriumCoefficient_{2\indiceBlock+1} \sum_{|\boldOther{\indiceMultiIndexDifferential}| = 2} \frac{\vectorial{\discreteVelocityNormalized}_{2\indiceBlock}^{\boldOther{\indiceMultiIndexDifferential}}}{\boldOther{\indiceMultiIndexDifferential}!} \partial_{\vectorial{\spaceVariable}}^{\boldOther{\indiceMultiIndexDifferential}}   \testFunction(\timeVariable, \vectorial{\spaceVariable})  = \bigO{\spaceStep^2},
    \end{equation*}
    for $(\timeVariable, \vectorial{\spaceVariable}) \in \nonNegativeReals \times \reals^{\spatialDimensionality}$.
\end{theorem}
However, we observe that since $\timeStep \propto \spaceStep^2$, the second-order consistency of the bulk scheme is preserved even when the initialisation schemes are not consistent, provided that $\initialisationOperator_1 = 1$.
This is radically different from the acoustic scaling $\timeStep \propto \spaceStep$ and comes from the fact that the errors from the initialisation routine are now of order $\bigO{\timeStep} = \bigO{\spaceStep^2}$.
Hence, under diffusive scaling, enforcing that the initialisation schemes are consistent is merely a question of obtaining time smoothness of the numerical solution.
\begin{proposition}
    Under diffusive scaling, that is, when $\latticeVelocity = \diffusiveScalingParameter / \spaceStep$ with $\diffusiveScalingParameter > 0$ fixed as $\spaceStep \to 0$, assuming that $\equilibriumCoefficient_{2\indiceBlock} = \spaceStep \tilde{\equilibriumCoefficient}_{2\indiceBlock}$ where $\tilde{\equilibriumCoefficient}_{2\indiceBlock}$ and $\equilibriumCoefficient_{2\indiceBlock + 1}$ are fixed as $\spaceStep \to 0$ for $\indiceBlock \in \integerInterval{1}{\blockNumberGinzburg}$,   the modified equation of the unique initialisation scheme for the \lbm scheme determined by \eqref{eq:LinkSchemeVelocities}, \eqref{eq:LinkSchemeMoments}, and \eqref{eq:LinkSchemeRelaxation}--- considering a local initialisation $\boldOther{\initialisationOperator} \in \reals^{\velocityNumber}$ with $\initialisationOperator_1 = 1$---is
    \begin{align*}
        \partial_{\timeVariable} \testFunction(0, \vectorial{\spaceVariable}) &+ \diffusiveScalingParameter \sum_{\indiceBlock = 1}^{\blockNumberGinzburg} (\relaxationParameter \tilde{\equilibriumCoefficient}_{2\indiceBlock}  + (1-\relaxationParameter) \tilde{\initialisationOperator}_{2\indiceBlock} )  \sum_{|\boldOther{\indiceMultiIndexDifferential}| = 1} \vectorial{\discreteVelocityNormalized}_{2\indiceBlock}^{\boldOther{\indiceMultiIndexDifferential}} \partial_{\vectorial{\spaceVariable}}^{\boldOther{\indiceMultiIndexDifferential}} \testFunction(0, \vectorial{\spaceVariable}) \\
        &- \diffusiveScalingParameter \sum_{\indiceBlock = 1}^{\blockNumberGinzburg} ((2-\relaxationParameter)\equilibriumCoefficient_{2\indiceBlock+1} + (\relaxationParameter - 1)\initialisationOperator_{2\indiceBlock + 1}) \sum_{|\boldOther{\indiceMultiIndexDifferential}| = 2} \frac{\vectorial{\discreteVelocityNormalized}_{2\indiceBlock}^{\boldOther{\indiceMultiIndexDifferential}}}{\boldOther{\indiceMultiIndexDifferential}!} \partial_{\vectorial{\spaceVariable}}^{\boldOther{\indiceMultiIndexDifferential}}   \testFunction(0, \vectorial{\spaceVariable}) = \bigO{\spaceStep},
    \end{align*}
    for $\vectorial{\spaceVariable} \in \reals^{\spatialDimensionality}$, where ${\initialisationOperator}_{2\indiceBlock} = \spaceStep \tilde{\initialisationOperator}_{2\indiceBlock}$ with fixed $\tilde{\initialisationOperator}_{2\indiceBlock}$ as $\spaceStep\to 0$ for $\indiceBlock \in \integerInterval{1}{\blockNumberGinzburg}$.
\end{proposition}
Therefore, the initialisation scheme is consistent with the bulk scheme under the conditions
\begin{equation*}
    \tilde{\initialisationOperator}_{2\indiceBlock} = \tilde{\equilibriumCoefficient}_{2\indiceBlock}, \qquad \initialisationOperator_{2\indiceBlock + 1} = \frac{\relaxationParameter-2}{\relaxationParameter} \equilibriumCoefficient_{2\indiceBlock + 1}, \qquad \indiceBlock\in\integerInterval{1}{\blockNumberGinzburg},
\end{equation*}
which are only set to ensure---as previously stated---time smoothness.

To numerically verify the previous claims, we consider the \scheme{1}{3} scheme introduced in \Cref{ex:MatchAcousticGinzburg}.
Considering the bounded domain $[0, 1]$ with periodic boundary conditions, using $\solutionCauchyInitial(\spaceVariable) = \cos(2\pi \spaceVariable)$ renders the exact solution $\solutionCauchy(\timeVariable, \spaceVariable) = e^{-4\pi^2 D \timeVariable}\cos(2\pi(\spaceVariable - \transportVelocity\timeVariable))$ to \eqref{eq:CauchyEquationDiffusion}/\eqref{eq:CauchyInitialDatumDiffusion}. We utilize $\diffusiveScalingParameter = 1$, $\transportVelocity = 2$ and $D = 1/32$.
These physical constants are set taking $\tilde{\equilibriumCoefficient}_2 = \transportVelocity$, $\equilibriumCoefficient_3 = 1$ and $\relaxationParameter = 1/(D + 1/2)$.
We consider two kinds of initialisations, which are
\begin{align*}
    \text{(a)} \qquad \initialisationOperator_1 = 1, \quad \initialisationOperator_2 = \spaceStep \tilde{\equilibriumCoefficient}_2, \quad \initialisationOperator_3 = \frac{\relaxationParameter - 2}{\relaxationParameter} \equilibriumCoefficient_3, \\
    \text{(b)} \qquad \initialisationOperator_1 = 1, \quad \initialisationOperator_2 = \frac{\spaceStep \tilde{\equilibriumCoefficient}_2}{2}, \quad \initialisationOperator_3 = 10 \frac{\relaxationParameter - 2}{\relaxationParameter} \equilibriumCoefficient_3,
\end{align*}
with the first condition $\text{(a)}$ yielding a consistent initialisation scheme and the second condition $\text{(b)}$ giving an inconsistent one. 

\paragraph{Study of the order of convergence}

\begin{figure} 
    \begin{center}
        \includegraphics{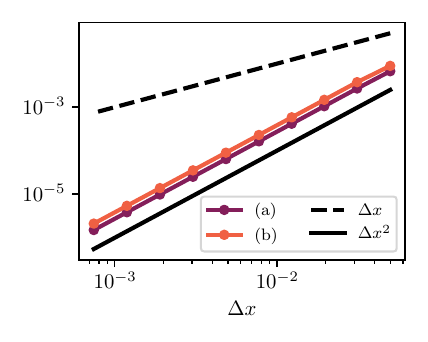}             
    \end{center}\caption{\label{fig:d1q3_magic_diffusive_convergence}$L^2$ errors at the final time for the initialisations $\text{(a)}$ and $\text{(b)}$. We observe second-order convergence irrespective of the consistency of the initialisation scheme.}
  \end{figure}

We simulate until the final time $0.05$ and measure the $L^2$ errors progressively decreasing the space step $\spaceStep$.
The results are given in \Cref{fig:d1q3_magic_diffusive_convergence}, confirming that, regardless of the consistency of the initialisation scheme, the overall method is second-order convergent, since $\timeStep \propto \spaceStep^2$.
As expected, the error constant is slightly better when the initialisation scheme is consistent.

\paragraph{Study of the time smoothness of the numerical solution}

\begin{figure} 
    \begin{center}
        \includegraphics[scale=0.99]{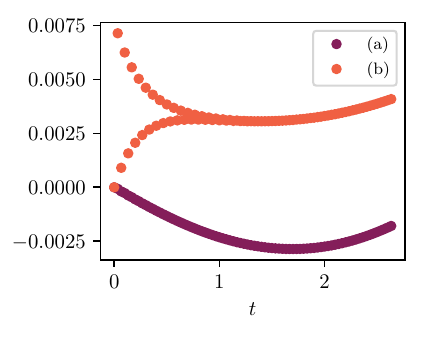}
        \includegraphics[scale=0.99]{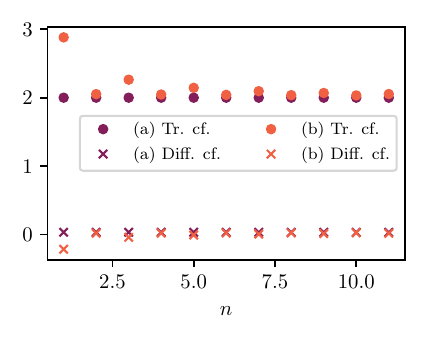}
    \end{center}\caption{\label{fig:d1q3_magic_diffusive_time_smoothness} Left: test for smoothness in time close to $\timeVariable = 0$  for the initialisations $\text{(a)}$ and $\text{(b)}$: difference between exact and numerical solution at the eighth lattice point. The first one gives a smooth profile whereas the second one oscillates with damping. Right: transport and diffusion coefficients in the modified equations for different $\indiceTime$.}
  \end{figure}

We consider a framework analogous to the one of \Cref{sec:TimeSmoothnessD1Q2} with $\spaceStep = 1/30$ and the previously introduced parameters, measuring the discrepancy between numerical and exact solution at the eighth point of the lattice.
The results in \Cref{fig:d1q3_magic_diffusive_time_smoothness} confirm the previous theoretical discussion: considering consistent initialisation schemes allows to avoid initial oscillating boundary layers in the simulation.
Furthermore, we see that for the even steps, the transport and diffusion coefficients from the modified equations of the \stSchemes{} are always closer to the one in the bulk than the ones for the odd steps, explaining why the discrepancies in terms of error with respect to the exact solution are smaller for even steps than for odd steps.

\subsection{Conclusions}

In \Cref{sec:Ginzburg}, we have defined a notion of observability for \lbm schemes, allowing to identify schemes with very little initialisation schemes as those being strongly unobservable.
In control theory, it is well known that unobservable systems can be represented by other systems which orders have been reduced by removing unobservable modes.
It is therefore easy to analyze the initialisation phase of these schemes with the technique of the modified equation.
In particular, we have found that a well-known and vast class of \lbm schemes, namely the so-called link two-relaxation-times schemes with magic parameters equal to one-fourth, fits this framework. We have exploited this fact in order to provide the constraints on the initial data for having a smooth initialisation, both under acoustic and diffusive scaling.

% \FloatBarrier

\section{General conclusions and perspectives}\label{sec:Conclusions}

Due to the fact that \lbm schemes feature more unknowns than variables of interest, their initialisation---especially for the non-conserved moments---can have an important impact on the outcomes of the simulations.
This is a side effect due to the fact that the discrete scheme supports parasitic---or spurious---modes.
The aim of the present contribution was indeed to study the role of the initialisation on the numerical behaviour of general \lbm schemes.
To this end, we have introduced a modified equation analysis which has ensured to propose initialisations yielding consistent \stSchemes{}, which is crucial to preserve the second-order accuracy of many methods.
The modified equation has also allowed to precisely describe the behaviour of the \lbm schemes close to the beginning of the numerical simulation---where the different dynamics were essentially driven by numerical dissipation---and identify initialisations yielding smooth discrete solutions without oscillatory initial layers.
Finally, we have introduced a notion of observability for \lbm schemes, which has allowed to characterize those with a small number of \iniSchemes{}, which are somehow almost ``non-kinetic''. 
This feature makes the study of the initialisation for these schemes way more accessible than for general ones.
Consistent and smooth initialisations have been a hot topic in the \lbm community for quite a long time \cite{van2009smooth, caiazzo2005analysis,junk2015l2, huang2015initial}.
However, to the best of our knowledge, no general approaches to the analysis of these features were available.
They are important in order to ensure that the order of the schemes is preserved.
From another perspective, although the novel notion of observability for \lbm schemes has been exploited solely to study the number of needed initialisation schemes, we do believe that it can be useful to investigate other features of these schemes. 
For example, one interesting topic would be the one linked to ``realisation'' \cite[Chapter 4]{brewer1986linear} and ``minimal realisations'' \cite{de2000minimal}: given a target \fd scheme (\emph{i.e.} a transfer function), how can we construct the smallest \lbm scheme of which it is the corresponding \fd scheme.
This will be the object of future investigations.

\section*{Acknowledgments}
The author thanks his PhD advisors B. Graille and M. Massot for the useful advice, and L. Fran\c cois for having hinted the importance of dealing with simple eigenvalues of modulus one in the study of initial conditions.
Finally, the author thanks the two anonymous referees for their valuable inputs towards the improvement of this work.

\section*{Funding}
The author was supported by a PhD funding (year 2019) from Ecole polytechnique.
His current post-doctoral position is funded by a ITI IRMIA++ post-doctoral funding (year 2022) from Universit\'e de Strasbourg.
This work of the Interdisciplinary Thematic Institute IRMIA++, as part of the ITI 2021-2028 program of the University of Strasbourg, CNRS and Inserm, was supported by IdEx Unistra (ANR-10-IDEX-0002), and by SFRI-STRAT'US project (ANR-20-SFRI-0012) under the framework of the French Investments for the Future Program.

\bibliographystyle{apalike} % With names and year
\bibliography{biblio}

\begin{thebibliography}{}

\bibitem[Ascher and Petzold, 1998]{ascher1998}
Ascher, U.~M. and Petzold, L.~R. (1998).
\newblock {\em {Computer Methods for Ordinary Differential Equations and
  Differential-Algebraic Equations}}.
\newblock SIAM.

\bibitem[{\AA}str{\"o}m and Murray, 2008]{aastrom2008feedback}
{\AA}str{\"o}m, K.~J. and Murray, R.~M. (2008).
\newblock {\em Feedback systems: an introduction for scientists and engineers}.
\newblock Princeton University Press.

\bibitem[Bellotti, 2023]{bellotti2021equivalentequations}
Bellotti, T. (2023).
\newblock {Truncation errors and modified equations for the lattice Boltzmann
  method via the corresponding Finite Difference schemes}.
\newblock {\em ESAIM: Mathematical Modelling and Numerical Analysis},
  57(3):1225--1255.

\bibitem[Bellotti et~al., 2022]{bellotti2021fd}
Bellotti, T., Graille, B., and Massot, M. (2022).
\newblock Finite difference formulation of any lattice {B}oltzmann scheme.
\newblock {\em {Numerische Mathematik}}, 152:1--40.

\bibitem[Bender et~al., 1999]{bender1999advanced}
Bender, C.~M., Orszag, S., and Orszag, S.~A. (1999).
\newblock {\em {Advanced mathematical methods for scientists and engineers I:
  Asymptotic methods and perturbation theory}}, volume~1.
\newblock Springer Science \& Business Media.

\bibitem[Blyth, 2018]{blyth2018module}
Blyth, T.~S. (2018).
\newblock {\em Module theory: an approach to linear algebra}.
\newblock University of St Andrews.

\bibitem[Brewer et~al., 1986]{brewer1986linear}
Brewer, J.~W., Bunce, J.~W., and Van~Vleck, F.~S. (1986).
\newblock {\em Linear systems over commutative rings}.
\newblock CRC Press.

\bibitem[Caetano et~al., 2023]{caetano2019result}
Caetano, F., Dubois, F., and Graille, B. (2023).
\newblock {A result of convergence for a mono-dimensional two-velocities
  lattice Boltzmann scheme}.
\newblock {\em Discrete and Continuous Dynamical Systems - S}.

\bibitem[Caiazzo, 2005]{caiazzo2005analysis}
Caiazzo, A. (2005).
\newblock {Analysis of lattice Boltzmann initialization routines}.
\newblock {\em Journal of Statistical Physics}, 121(1):37--48.

\bibitem[Carpentier et~al., 1997]{carpentier1997derivation}
Carpentier, R., de~La~Bourdonnaye, A., and Larrouturou, B. (1997).
\newblock On the derivation of the modified equation for the analysis of linear
  numerical methods.
\newblock {\em ESAIM: Mathematical Modelling and Numerical Analysis},
  31(4):459--470.

\bibitem[Chai and Shi, 2020]{chai2020multiple}
Chai, Z. and Shi, B. (2020).
\newblock {Multiple-relaxation-time lattice Boltzmann method for the
  Navier-Stokes and nonlinear convection-diffusion equations: Modeling,
  analysis, and elements}.
\newblock {\em Physical Review E}, 102(2):023306.

\bibitem[Chen et~al., 2017]{chen2017simplified}
Chen, Z., Shu, C., Wang, Y., Yang, L., and Tan, D. (2017).
\newblock {A simplified lattice Boltzmann method without evolution of
  distribution function}.
\newblock {\em Advances in Applied Mathematics and Mechanics}, 9(1):1--22.

\bibitem[Cheng and Lu, 1999]{cheng1999general}
Cheng, S.~S. and Lu, Y.-F. (1999).
\newblock General solutions of a three-level partial difference equation.
\newblock {\em Computers \& Mathematics with Applications}, 38(7-8):65--79.

\bibitem[Coreixas et~al., 2019]{coreixas2019comprehensive}
Coreixas, C., Chopard, B., and Latt, J. (2019).
\newblock {Comprehensive comparison of collision models in the lattice
  Boltzmann framework: Theoretical investigations}.
\newblock {\em Physical Review E}, 100(3):033305.

\bibitem[Coreixas et~al., 2017]{coreixas2017recursive}
Coreixas, C., Wissocq, G., Puigt, G., Boussuge, J.-F., and Sagaut, P. (2017).
\newblock {Recursive regularization step for high-order lattice Boltzmann
  methods}.
\newblock {\em Physical Review E}, 96(3):033306.

\bibitem[De~Schutter, 2000]{de2000minimal}
De~Schutter, B. (2000).
\newblock Minimal state-space realization in linear system theory: an overview.
\newblock {\em Journal of Computational and Applied Mathematics},
  121(1-2):331--354.

\bibitem[Dellacherie, 2014]{dellacherie2014construction}
Dellacherie, S. (2014).
\newblock {Construction and analysis of lattice Boltzmann methods applied to a
  1D convection-diffusion equation}.
\newblock {\em Acta Applicandae Mathematicae}, 131(1):69--140.

\bibitem[D'Humi\`eres, 1992]{dhumieres1992}
D'Humi\`eres, D. (1992).
\newblock {\em Generalized {L}attice-{B}oltzmann {E}quations}, pages 450--458.
\newblock American Institute of Aeronautics and Astronautics, Inc.

\bibitem[D'Humi{\`e}res and Ginzburg, 2009]{d2009viscosity}
D'Humi{\`e}res, D. and Ginzburg, I. (2009).
\newblock Viscosity independent numerical errors for lattice boltzmann models:
  From recurrence equations to ``magic'' collision numbers.
\newblock {\em Computers \& Mathematics with Applications}, 58(5):823--840.

\bibitem[Dubois, 2008]{dubois2008equivalent}
Dubois, F. (2008).
\newblock Equivalent partial differential equations of a lattice {B}oltzmann
  scheme.
\newblock {\em {Computers \& Mathematics with Applications}}, 55(7):1441--1449.

\bibitem[Dubois, 2021]{dubois2019nonlinear}
Dubois, F. (2021).
\newblock Nonlinear fourth order {T}aylor expansion of lattice {B}oltzmann
  schemes.
\newblock {\em {Asymptotic Analysis}}, (1 Jan. 2021):1--41.

\bibitem[Dubois et~al., 2020]{dubois2020notion}
Dubois, F., Graille, B., and Rao, S.~R. (2020).
\newblock A notion of non-negativity preserving relaxation for a
  mono-dimensional three velocities scheme with relative velocity.
\newblock {\em Journal of Computational Science}, 47:101181.

\bibitem[Fliess and Mounier, 1998]{fliess1998controllability}
Fliess, M. and Mounier, H. (1998).
\newblock Controllability and observability of linear delay systems: an
  algebraic approach.
\newblock {\em ESAIM: Control, Optimisation and Calculus of Variations},
  3:301--314.

\bibitem[Fu{\v{c}}{\'\i}k and Straka, 2021]{fuvcik2021equivalent}
Fu{\v{c}}{\'\i}k, R. and Straka, R. ({2021}).
\newblock {Equivalent finite difference and partial differential equations for
  the lattice Boltzmann method}.
\newblock {\em {Computers \& Mathematics with Applications}}, {90}:{96--103}.

\bibitem[Ginzburg et~al., 2008]{ginzburg2008two}
Ginzburg, I., Verhaeghe, F., and d'Humi\`eres, D. (2008).
\newblock {Two-relaxation-time lattice Boltzmann scheme: About parametrization,
  velocity, pressure and mixed boundary conditions}.
\newblock {\em Communications in Computational Physics}, 3(2):427--478.

\bibitem[Graille, 2014]{graille2014approximation}
Graille, B. (2014).
\newblock {Approximation of mono-dimensional hyperbolic systems: A lattice
  Boltzmann scheme as a relaxation method}.
\newblock {\em {Journal of Computational Physics}}, 266:74--88.

\bibitem[Gustafsson et~al., 1995]{gustafsson1995time}
Gustafsson, B., Kreiss, H.-O., and Oliger, J. (1995).
\newblock {\em Time dependent problems and difference methods}, volume~24.
\newblock John Wiley \& Sons.

\bibitem[Hairer et~al., 2008]{hairer2008}
Hairer, E., N{\o}rsett, S.~P., and Wanner, G. (2008).
\newblock {\em {Solving Ordinary Differential Equations I}}, volume~8.
\newblock Springer.
\newblock Second Revised Edition.

\bibitem[Hendricks et~al., 2008]{hendricks2008linear}
Hendricks, E., Jannerup, O., and S{\o}rensen, P.~H. (2008).
\newblock {\em Linear systems control: deterministic and stochastic methods}.
\newblock Springer.

\bibitem[Horn and Johnson, 2012]{horn2012matrix}
Horn, R.~A. and Johnson, C.~R. (2012).
\newblock {\em Matrix analysis}.
\newblock Cambridge University Press.

\bibitem[Huang et~al., 2015]{huang2015initial}
Huang, J., Wu, H., and Yong, W.-A. (2015).
\newblock {On initial conditions for the lattice Boltzmann method}.
\newblock {\em Communications in Computational Physics}, 18(2):450--468.

\bibitem[Hundsdorfer and Ruuth, 2006]{hundsdorfer2006monotonicity}
Hundsdorfer, W. and Ruuth, S. (2006).
\newblock On monotonicity and boundedness properties of linear multistep
  methods.
\newblock {\em Mathematics of Computation}, 75(254):655--672.

\bibitem[Hundsdorfer et~al., 2003]{hundsdorfer2003monotonicity}
Hundsdorfer, W., Ruuth, S.~J., and Spiteri, R.~J. (2003).
\newblock Monotonicity-preserving linear multistep methods.
\newblock {\em SIAM Journal on Numerical Analysis}, 41(2):605--623.

\bibitem[Jin and Xin, 1995]{jin1995relaxation}
Jin, S. and Xin, Z. (1995).
\newblock The relaxation schemes for systems of conservation laws in arbitrary
  space dimensions.
\newblock {\em Communications on Pure and Applied Mathematics}, 48(3):235--276.

\bibitem[Junk et~al., 2005]{junk2005asymptotic}
Junk, M., Klar, A., and Luo, L.-S. (2005).
\newblock Asymptotic analysis of the lattice {B}oltzmann equation.
\newblock {\em {Journal of Computational Physics}}, 210(2):676--704.

\bibitem[Junk and Yang, 2015]{junk2015l2}
Junk, M. and Yang, Z. (2015).
\newblock {L2 convergence of the lattice Boltzmann method for one dimensional
  convection-diffusion-reaction equations}.
\newblock {\em Communications in Computational Physics}, 17(5):1225--1245.

\bibitem[Kuzmin et~al., 2011]{kuzmin2011role}
Kuzmin, A., Ginzburg, I., and Mohamad, A. (2011).
\newblock {The role of the kinetic parameter in the stability of
  two-relaxation-time advection--diffusion lattice Boltzmann schemes}.
\newblock {\em Computers \& Mathematics with Applications}, 61(12):3417--3442.

\bibitem[Kuznik et~al., 2013]{kuznik2013mesoscopic}
Kuznik, F., Luo, L.-S., and Krafczyk, M. (2013).
\newblock {Mesoscopic Methods in Engineering and Science}.
\newblock {\em Computers \& Mathematics with Applications}, 65(6):813--814.

\bibitem[Lallemand and Luo, 2000]{lallemand2000theory}
Lallemand, P. and Luo, L.-S. (2000).
\newblock Theory of the lattice {B}oltzmann method: {D}ispersion, dissipation,
  isotropy, {G}alilean invariance, and stability.
\newblock {\em {Physical Review E}}, 61(6):6546.

\bibitem[Miller, 1971]{miller1971location}
Miller, J.~J. (1971).
\newblock On the location of zeros of certain classes of polynomials with
  applications to numerical analysis.
\newblock {\em IMA Journal of Applied Mathematics}, 8(3):397--406.

\bibitem[O'Malley, 1991]{o1991singular}
O'Malley, R.~E. (1991).
\newblock {\em Singular perturbation methods for ordinary differential
  equations}, volume~89.
\newblock Springer.

\bibitem[Saad, 1989]{saad1989overview}
Saad, Y. (1989).
\newblock {Overview of Krylov subspace methods with applications to control
  problems}.
\newblock Technical report.

\bibitem[Simonis et~al., 2020]{simonis2020relaxation}
Simonis, S., Frank, M., and Krause, M.~J. (2020).
\newblock {On relaxation systems and their relation to discrete velocity
  Boltzmann models for scalar advection--diffusion equations}.
\newblock {\em Philosophical Transactions of the Royal Society A},
  378(2175):20190400.

\bibitem[Strikwerda, 2004]{strikwerda2004finite}
Strikwerda, J.~C. (2004).
\newblock {\em Finite difference schemes and partial differential equations}.
\newblock SIAM.

\bibitem[Suga, 2010]{suga2010accurate}
Suga, S. (2010).
\newblock {An accurate multi-level finite difference scheme for 1D diffusion
  equations derived from the lattice Boltzmann method}.
\newblock {\em Journal of Statistical Physics}, 140(3):494--503.

\bibitem[Trefethen, 1996]{trefethen1996}
Trefethen, L.~N. (1996).
\newblock {\em Finite Difference and Spectral Methods for Ordinary and Partial
  Differential Equations}.
\newblock unpublished text.

\bibitem[Van~der Sman, 2006]{van2006finite}
Van~der Sman, R. (2006).
\newblock {Finite Boltzmann schemes}.
\newblock {\em Computers \& Fluids}, 35(8-9):849--854.

\bibitem[Van~Leemput et~al., 2009]{van2009smooth}
Van~Leemput, P., Rheinl{\"a}nder, M., and Junk, M. (2009).
\newblock {Smooth initialization of lattice Boltzmann schemes}.
\newblock {\em {Computers \& Mathematics with Applications}}, 58(5):867--882.

\bibitem[Warming and Hyett, 1974]{warming1974modified}
Warming, R.~F. and Hyett, B. (1974).
\newblock The modified equation approach to the stability and accuracy analysis
  of finite-difference methods.
\newblock {\em Journal of Computational Physics}, 14(2):159--179.

\bibitem[Yong et~al., 2016]{yong2016theory}
Yong, W.-A., Zhao, W., Luo, L.-S., et~al. (2016).
\newblock Theory of the lattice {B}oltzmann method: Derivation of macroscopic
  equations via the {M}axwell iteration.
\newblock {\em {Physical Review E}}, 93(3):033310.

\bibitem[Zhang et~al., 2019]{zhang2019lattice}
Zhang, M., Zhao, W., and Lin, P. (2019).
\newblock {Lattice Boltzmann method for general convection-diffusion equations:
  MRT model and boundary schemes}.
\newblock {\em {Journal of Computational Physics}}, 389:147--163.

\bibitem[Zhao and Yong, 2017]{zhao2017maxwell}
Zhao, W. and Yong, W.-A. (2017).
\newblock {Maxwell iteration for the lattice {B}oltzmann method with diffusive
  scaling}.
\newblock {\em {Physical Review E}}, 95(3):033311.

\bibitem[Zhou, 2020]{zhou2020macroscopic}
Zhou, J.~G. (2020).
\newblock {Macroscopic lattice Boltzmann method}.
\newblock {\em Water}, 13(1):61.

\end{thebibliography}

\newcommand{\wavenumber}{\theta}

\appendix

\section{Stability of the \scheme{1}{2} scheme of \Cref{sec:D1Q2}}\label{app:StabilityD1Q2}

There are several ways of checking the roots of amplification polynomial of the corresponding bulk scheme for \Cref{sec:D1Q2}.
In this case, we can proceed directly by solving the characteristic equation or by using the procedure by \cite{graille2014approximation}.
Thanks to its generality, we here present the computations using the technique by Miller \cite{miller1971location,strikwerda2004finite}.
The amplification polynomial reads $\fourierTransformed{\amplificationPolynomial}_2(\timeShift, \frequency\spaceStep) = \timeShift^2 +  ( (\relaxationParameter_2 - 2) \cos{(\frequency \spaceStep)} + i \relaxationParameter_2 \equilibriumCoefficient_2 \sin{(\frequency \spaceStep)}  ) \timeShift + (1-\relaxationParameter_2)$, where $\frequency  \in [-\pi/\spaceStep, \pi/\spaceStep]$.
We have that 
\begin{equation*}
    \fourierTransformed{\amplificationPolynomial}_2^{\star}(\timeShift, \frequency\spaceStep) \definitionEquality \timeShift^{2} {\fourierTransformed{\amplificationPolynomial}_2({\timeShift^{-1}}, -\frequency\spaceStep)} = (1-\relaxationParameter_2) \timeShift^{2} +   ( (\relaxationParameter_2 - 2) \cos{(\frequency \spaceStep)} - i \relaxationParameter_2 \equilibriumCoefficient_2  \sin{(\frequency \spaceStep)} ) \timeShift + 1.
\end{equation*}
\begin{itemize}
    \item Let $\relaxationParameter_2 \in ]0, 2[$. A first condition to bound the roots of $\fourierTransformed{\amplificationPolynomial}_2(\timeShift, \frequency\spaceStep)$ in modulus by one regardless of the frequency is that $|\fourierTransformed{\amplificationPolynomial}_2(0, \frequency\spaceStep)| < |\fourierTransformed{\amplificationPolynomial}_2^{\star}(0, \frequency\spaceStep)|$, which yields the condition $0 < \relaxationParameter_2 < 2$.
    Then we compute $\fourierTransformed{\amplificationPolynomial}_1(\timeShift, \frequency\spaceStep)$ as 
\begin{equation*}
    \fourierTransformed{\amplificationPolynomial}_1(\timeShift, \frequency\spaceStep) \definitionEquality \timeShift^{-1}  ( \fourierTransformed{\amplificationPolynomial}_2^{\star}(0, \frequency\spaceStep) \fourierTransformed{\amplificationPolynomial}_2(\timeShift, \frequency\spaceStep) - \fourierTransformed{\amplificationPolynomial}_2(0, \frequency\spaceStep)\fourierTransformed{\amplificationPolynomial}_2^{\star}(\timeShift, \frequency\spaceStep) ) = \relaxationParameter_2 (2-\relaxationParameter_2) ( \timeShift - \cos{(\frequency \spaceStep)} + i \equilibriumCoefficient_2 \sin{(\frequency \spaceStep)}  ).
\end{equation*}
The final condition to check is that the root of $\fourierTransformed{\amplificationPolynomial}_1(\timeShift, \frequency\spaceStep)$ is bounded by one in modulus for any frequency. This is $\cos^2{(\frequency \spaceStep)} + \equilibriumCoefficient_2^2  \sin^2{(\frequency \spaceStep)} = 1 + (\equilibriumCoefficient_2^2 - 1) \sin^2{(\frequency \spaceStep)} \leq 1$
taking place for any $\frequency \in [-\pi/\spaceStep, \pi/\spaceStep]$ if and only if $\equilibriumCoefficient_2^2  \leq 1$.
\item Let $\relaxationParameter_2 = 2$. In this case $\fourierTransformed{\amplificationPolynomial}_1 \equiv 0$. We then have to use the second condition from \cite[Theorem 4.3.2]{strikwerda2004finite}, hence we check 
\begin{equation*}
    \frac{\text{d}\fourierTransformed{\amplificationPolynomial}_2(\timeShift, \frequency\spaceStep) }{\text{d} \timeShift} = 2\timeShift + 2 i \equilibriumCoefficient_2 \sin{(\frequency \spaceStep)},
\end{equation*}
which unique root should be strictly in the unit circle for any frequency $\frequency \in [-\pi/\spaceStep, \pi/\spaceStep]$. This is achieved by $|\equilibriumCoefficient_2| < 1$. 
\end{itemize}
This concludes the proof of the stability conditions \eqref{eq:StabilityConditionsD1Q2}.

\section{Derivation of the forward centered initialisation schemes for the \scheme{1}{2} scheme of \Cref{sec:D1Q2}}\label{app:derivationForwardCenteredD1Q2}

We can first unsuccessfully attempt to obtain a forward centered scheme as initialisation scheme, using a local initialisation of the conserved moment, that is $\initialisationOperator_1 = 1$ and prepared initialisation of the non-conserved one, thus $\initialisationOperator_2 \in \ringSpaceOperatorsOneD$.
Using the notation \eqref{eq:FormInitilisationNonLocal}, this corresponds to find a compactly supported solution of the following infinite system
\begin{align*}
    \dots, \quad
    \initialisationOperatorCoefficients_{2, 1 } - \initialisationOperatorCoefficients_{2, 3 } = 0, \quad
    \initialisationOperatorCoefficients_{2, 0 } - \initialisationOperatorCoefficients_{2, 2 } &=  - \frac{1 - \equilibriumCoefficient_2+\relaxationParameter_2 \equilibriumCoefficient_2}{1-\relaxationParameter_2} , \quad 
    \initialisationOperatorCoefficients_{2, -1} - \initialisationOperatorCoefficients_{2, 1 } = \frac{2}{1-\relaxationParameter_2}, \\
    \initialisationOperatorCoefficients_{2, -2} - \initialisationOperatorCoefficients_{2, 0 } &= - \frac{1 + \equilibriumCoefficient_2 -\relaxationParameter_2 \equilibriumCoefficient_2}{1-\relaxationParameter_2} , \quad
    \initialisationOperatorCoefficients_{2, -3} - \initialisationOperatorCoefficients_{2, -1} = 0, \quad
    \dots
\end{align*}
This problem cannot be solved by a compactly supported sequence, in particular, because of the median term.
This would go back to perform a deconvolution in the ring of \fd operators, which is not solvable because the operator $\antisymmetricPart(\basicShiftLetter_1)$ is not invertible in such ring.
If we consider to work on a bounded domain with $\boundedDomainNumberPoints \in \nonZeroNaturals$ points and endow the shift operators with periodic boundary conditions \cite{van2009smooth}, some of these deconvolution problems become solvable at the price of dealing with non-compactly supported solutions, \emph{i.e.} stemming from a full inverse of a sparse matrix. 
The previous problem can be seen as the one of inverting a circulant matrix, which eigenvalues are $\sigma_{\indiceFreeOne} = \text{exp} (2\pi i {(\boundedDomainNumberPoints-1)} \indiceFreeOne/{\boundedDomainNumberPoints}  ) - \text{exp} (2\pi i \indiceFreeOne /{\boundedDomainNumberPoints}  )$ for $\indiceFreeOne \in \integerIntervalClosedOpen{0}{\boundedDomainNumberPoints}$.
Since $\sigma_0 = 0$, the circulant matrix is not invertible.
Therefore, even in the periodic setting, this procedure does not work.
This can be interpreted---if we see the equilibria as a control on the system---as due to the lack of ``reachability'' of the system at hand, \emph{cf.} \cite[Chapter 2]{brewer1986linear}. Since the term $\antisymmetricPart (\basicShiftLetter_1 )$ is not a unit, which causes the lack of reachability, it cannot be compensated by its inverse contained in the equilibrium to generate the desired initialisation scheme. This is why we are compelled to consider $\initialisationOperator_1 \in \ringSpaceOperatorsOneD$ to obtain the requested forward centered scheme.

Considering a prepared initialisation for both moments, thus $\initialisationOperator_1, \initialisationOperator_2 \in \ringSpaceOperatorsOneD$, several choices are possible to recover this scheme.
The infinite system to solve reads
\begin{align*}
    \dots\\
    \frac{1+\relaxationParameter_2 \equilibriumCoefficient_2 }{2} \initialisationOperatorCoefficients_{1,  1} + \frac{1-\relaxationParameter_2}{2 } \initialisationOperatorCoefficients_{2,  1}  + \frac{1-\relaxationParameter_2 \equilibriumCoefficient_2 }{2} \initialisationOperatorCoefficients_{1,  3} - \frac{1-\relaxationParameter_2}{2 } \initialisationOperatorCoefficients_{2,  3} = 0, \\
    \frac{1+\relaxationParameter_2 \equilibriumCoefficient_2 }{2} \initialisationOperatorCoefficients_{1,  0} + \frac{1-\relaxationParameter_2}{2 } \initialisationOperatorCoefficients_{2,  0}  + \frac{1-\relaxationParameter_2 \equilibriumCoefficient_2 }{2} \initialisationOperatorCoefficients_{1,  2} - \frac{1-\relaxationParameter_2}{2 } \initialisationOperatorCoefficients_{2,  2} = \frac{\equilibriumCoefficient_2}{2 }, \\
    \frac{1+\relaxationParameter_2 \equilibriumCoefficient_2 }{2} \initialisationOperatorCoefficients_{1, -1} + \frac{1-\relaxationParameter_2}{2 } \initialisationOperatorCoefficients_{2, -1}  + \frac{1-\relaxationParameter_2 \equilibriumCoefficient_2 }{2} \initialisationOperatorCoefficients_{1,  1} - \frac{1-\relaxationParameter_2}{2 } \initialisationOperatorCoefficients_{2,  1} = 1, \\
    \frac{1+\relaxationParameter_2 \equilibriumCoefficient_2 }{2} \initialisationOperatorCoefficients_{1, -2} + \frac{1-\relaxationParameter_2}{2 } \initialisationOperatorCoefficients_{2, -2}  + \frac{1-\relaxationParameter_2 \equilibriumCoefficient_2 }{2} \initialisationOperatorCoefficients_{1,  0} - \frac{1-\relaxationParameter_2}{2 } \initialisationOperatorCoefficients_{2,  0} = -\frac{\equilibriumCoefficient_2}{2 }, \\
    \frac{1+\relaxationParameter_2 \equilibriumCoefficient_2 }{2} \initialisationOperatorCoefficients_{1, -3} + \frac{1-\relaxationParameter_2}{2 } \initialisationOperatorCoefficients_{2, -3}  + \frac{1-\relaxationParameter_2 \equilibriumCoefficient_2 }{2} \initialisationOperatorCoefficients_{1, -1} - \frac{1-\relaxationParameter_2}{2 } \initialisationOperatorCoefficients_{2, -1} = 0, \\
    \dots
\end{align*}
In order to construct a (non-unique) solution, we first enforce the compactness: $\initialisationOperatorCoefficients_{1, \indiceFreeOne} = \initialisationOperatorCoefficients_{2, \indiceFreeOne} = 0$ for $|\indiceFreeOne| \geq 2$.
From this, we obtain the finite system
\begin{align*}
    {  (1+{\relaxationParameter_2 \equilibriumCoefficient_2}{ }  )} \initialisationOperatorCoefficients_{1,  1} + ({1-\relaxationParameter_2}{ }) \initialisationOperatorCoefficients_{2, 1}  &= 0, \\
    {  (1+{\relaxationParameter_2 \equilibriumCoefficient_2}{ }  )} \initialisationOperatorCoefficients_{1,  0} + ({1-\relaxationParameter_2}{ }) \initialisationOperatorCoefficients_{2, 0} &= {\equilibriumCoefficient_2}{ }, \\
    {  (1+{\relaxationParameter_2 \equilibriumCoefficient_2}{ }  )} \initialisationOperatorCoefficients_{1, -1} + ({1-\relaxationParameter_2}{ } )\initialisationOperatorCoefficients_{2, -1}  + {(1-\relaxationParameter_2 \equilibriumCoefficient_2 )}  \initialisationOperatorCoefficients_{1, 1} - {(1-\relaxationParameter_2)}{ } \initialisationOperatorCoefficients_{2, 1} &= 2, \\
    {  (1-{\relaxationParameter_2 \equilibriumCoefficient_2}{ }  )} \initialisationOperatorCoefficients_{1,  0} - ({1-\relaxationParameter_2}{ }) \initialisationOperatorCoefficients_{2, 0} &= -{\equilibriumCoefficient_2}{ }, \\
    {  (1-{\relaxationParameter_2 \equilibriumCoefficient_2}{ }  )} \initialisationOperatorCoefficients_{1, -1} - ({1-\relaxationParameter_2}{ }) \initialisationOperatorCoefficients_{2, -1} &= 0.
\end{align*}
We then split the central equation using a parameter $\theta \in \reals$, having ${(1+{\relaxationParameter_2 \equilibriumCoefficient_2}{ })} \initialisationOperatorCoefficients_{1, -1} + {(1-\relaxationParameter_2)} \initialisationOperatorCoefficients_{2, -1} = \theta$ and ${(1-{\relaxationParameter_2 \equilibriumCoefficient_2} { })} \initialisationOperatorCoefficients_{1,  1} - {(1-\relaxationParameter_2)} { } \initialisationOperatorCoefficients_{2,  1} = 2 - \theta$.
Introducing the matrix
\begin{equation*}
\matricial{A} = 
\begin{bmatrix}
    1+ \relaxationParameter_2 \equilibriumCoefficient_2 &   1-\relaxationParameter_2\\
    1-\relaxationParameter_2 \equilibriumCoefficient_2 & \relaxationParameter_2 -1
\end{bmatrix},
\end{equation*}
we solve the systems $\matricial{A} \transpose{[\initialisationOperatorCoefficients_{1, 1}, \initialisationOperatorCoefficients_{2, 1} ]} = \transpose{[0, 2 - \theta]}$, $\matricial{A}  \transpose{[\initialisationOperatorCoefficients_{1, 0}, \initialisationOperatorCoefficients_{2, 0} ]} = \transpose{[\equilibriumCoefficient_2, - \equilibriumCoefficient_2]}$ and $\matricial{A}  \transpose{[\initialisationOperatorCoefficients_{1, -1}, \initialisationOperatorCoefficients_{2, -1} ]} = \transpose{[\theta, 0]}$, yielding
\begin{align*}
&\initialisationOperatorCoefficients_{1,1} = \frac{2 - \theta}{2} , \qquad \initialisationOperatorCoefficients_{2,1} = -\frac{  (1 + \relaxationParameter_2 \equilibriumCoefficient_2 )(2 - \theta)}{2 (1-\relaxationParameter_2)}, \qquad \initialisationOperatorCoefficients_{1, 0} = 0, \qquad \initialisationOperatorCoefficients_{2, 0} = \frac{\equilibriumCoefficient_2}{1-\relaxationParameter_2}, \\
&\initialisationOperatorCoefficients_{1, -1} = \frac{\theta}{2}, \qquad \initialisationOperatorCoefficients_{2, -1} = \frac{  (1 - \relaxationParameter_2 \equilibriumCoefficient_2 ) \theta}{2 (1-\relaxationParameter_2)}. 
\end{align*}
Unsurprisingly, these coefficients are defined for $\relaxationParameter_2 \neq 1$, since otherwise there is no \iniScheme{} to devise.
The only way to fulfill \eqref{eq:PreparedInitial}, \eqref{eq:PreparedInitial2}, and \eqref{eq:PreparedBulk} is to take $\theta = 1$, giving
\begin{equation*}
    \initialisationOperatorCoefficients_{1, \pm 1} = \frac{1}{2}, \qquad \initialisationOperatorCoefficients_{2, \pm 1} = \mp\frac{ 1  \pm \relaxationParameter_2 \equilibriumCoefficient_2}{2(1 - \relaxationParameter_2)},  \qquad \initialisationOperatorCoefficients_{2, 0} = \frac{\equilibriumCoefficient_2}{1 - \relaxationParameter_2}.
\end{equation*}
Allowing more non-vanishing coefficients and through a similar procedure, another possible choice to obtain the desired scheme would be
\begin{equation*}
    \initialisationOperatorCoefficients_{1, \pm 2} = \pm \frac{\equilibriumCoefficient_2}{2}, \qquad \initialisationOperatorCoefficients_{1, \pm 1} = \frac{1}{2}, \qquad \initialisationOperatorCoefficients_{2, \pm 2} = -\frac{\equilibriumCoefficient_2(1 \pm \relaxationParameter_2 \equilibriumCoefficient_2)}{2(1-\relaxationParameter_2)}, \qquad \initialisationOperatorCoefficients_{2, \pm 1} = \mp \frac{1 \pm \relaxationParameter_2 \equilibriumCoefficient_2}{2(1 - \relaxationParameter_2)}.
\end{equation*}

\section{Derivation of the remaining modified equations of the \stSchemes{} for the \scheme{1}{2} scheme of \Cref{sec:ModEqD1Q2order2}}\label{app:DerivationModifiedEquationsD1Q2}

In \Cref{sec:ModEqD1Q2order2}, we have left the derivation of several modified equations of the \stSchemes{} for the \scheme{1}{2} scheme.
We now explain how to reach them.
\begin{itemize}
    \item{\strong{Forward centered scheme} \eqref{eq:CenteredGood}}.
    This scheme fulfills the conditions of \Cref{prop:PreparedInitialisation}, hence for $\indiceTime \in \nonZeroNaturals$
    \begin{equation*}
        \partial_{\timeVariable}\testFunction(0, \spaceVariable) + \latticeVelocity \equilibriumCoefficient_2 \partial_{\spaceVariable} \testFunction(0, \spaceVariable) - \latticeVelocity \spaceStep \Bigl ( - \frac{\indiceTime}{2} \equilibriumCoefficient_2^2 \partial_{\spaceVariable\spaceVariable} + \frac{1}{\indiceTime} \Bigl (  \termAtOrder{(\schemeMatrixAsymptotic^{\indiceTime})}{2}_{11} +  \termAtOrder{(\schemeMatrixAsymptotic^{\indiceTime})}{2}_{12} \equilibriumCoefficient_2 + \termAtOrder{(\schemeMatrixAsymptotic^{\indiceTime})}{1}_{12} \termAtOrder{\initialisationOperatorAsymptotic_2}{1} + \termAtOrder{\initialisationOperatorAsymptotic_1}{2}\Bigr ) \Bigr ) \testFunction(0, \spaceVariable) = \bigO{\spaceStep^2},
    \end{equation*}
    for $\spaceVariable \in \reals$, where only the terms $\termAtOrder{\initialisationOperatorAsymptotic_2}{1} = {1}/{(1-\relaxationParameter_2)}\partial_{\spaceVariable}$ and $\termAtOrder{\initialisationOperatorAsymptotic_1}{2} = \tfrac{1}{2} \partial_{\spaceVariable \spaceVariable}$ introduce discrepancies from the Lax-Friedrichs initialisation \eqref{eq:LaxFriedrichs}.
    Using \eqref{eq:OrderOneMatrixD1Q2} we obtain for $\indiceTime \in \nonZeroNaturals$ and $\spaceVariable \in \reals$ 
    \begin{align*}
        \partial_{\timeVariable} \testFunction(0, \spaceVariable) &+ \latticeVelocity \equilibriumCoefficient_2 \partial_{\spaceVariable} \testFunction(0, \spaceVariable) \\
        &- \latticeVelocity \spaceStep \Bigl ( \Bigl ( \frac{1}{2} + \sum_{\indiceFreeOne = 1}^{\indiceTime - 1}\Bigl (1 - \frac{\indiceFreeOne}{\indiceTime}\Bigr )(1-\relaxationParameter_2)^{\indiceFreeOne} \Bigr )  ( 1 - \equilibriumCoefficient_2^2  )+ \frac{1}{2\indiceTime} \Bigl ( 1 - 2\sum_{\indiceFreeOne = 0}^{\indiceTime - 1} (1- \relaxationParameter_2)^{\indiceFreeOne} \Bigr ) \Bigl ) \partial_{\spaceVariable \spaceVariable} \testFunction(0, \spaceVariable) = \bigO{\spaceStep^2}.
    \end{align*}
    
    \item{\strong{Forward centered scheme} \eqref{eq:CenteredBad}}. 
    We have
    \begin{equation*}
        \partial_{\timeVariable}\testFunction(0, \spaceVariable) - \frac{\latticeVelocity}{\indiceTime} \Bigl ( \termAtOrder{(\schemeMatrixAsymptotic^{\indiceTime})}{1}_{11} +  \termAtOrder{(\schemeMatrixAsymptotic^{\indiceTime})}{1}_{12} \termAtOrder{\initialisationOperatorAsymptotic_2}{0} + \termAtOrder{\initialisationOperatorAsymptotic_1}{1}\Bigr ) \testFunction(0, \spaceVariable)
        = \bigO{\spaceStep}, \qquad \indiceTime \in \nonZeroNaturals, \quad \spaceVariable \in \reals,
    \end{equation*}
    where in this case $\termAtOrder{\initialisationOperatorAsymptotic_2}{0} = -{(1+\relaxationParameter_2)}/{(1 - \relaxationParameter_2)}\equilibriumCoefficient_2$ and $\termAtOrder{\initialisationOperatorAsymptotic_1}{1} = -2 {\equilibriumCoefficient_2} \partial_{\spaceVariable}$.
    Recalling that
    \begin{equation}\label{eq:OrderOneMatrixD1Q2}
        \termAtOrder{(\schemeMatrixAsymptotic^{\indiceTime})}{1}_{11} = -\equilibriumCoefficient_2 \sum_{\indiceFreeOne = 0}^{\indiceTime - 1} \polynomialEquilibrium_{\indiceTime - \indiceFreeOne}(\relaxationParameter_2) \partial_{\spaceVariable}, \qquad \termAtOrder{(\schemeMatrixAsymptotic^{\indiceTime})}{1}_{12} = - \sum_{\indiceFreeOne = 0}^{\indiceTime - 1} (1 - \relaxationParameter_2)^{\indiceTime - \indiceFreeOne} \partial_{\spaceVariable}, \qquad \indiceTime \in \nonZeroNaturals,
    \end{equation}
    yields
    \begin{equation*}
        \termAtOrder{(\schemeMatrixAsymptotic^{\indiceTime})}{1}_{11} +  \termAtOrder{(\schemeMatrixAsymptotic^{\indiceTime})}{1}_{12} \termAtOrder{\initialisationOperatorAsymptotic_2}{0} + \termAtOrder{\initialisationOperatorAsymptotic_1}{1} = - \equilibriumCoefficient_2  \Bigl ( \indiceTime + 2 \Bigl ( 1 - \sum_{\indiceFreeOne = 0}^{\indiceTime - 1}(1 - \relaxationParameter_2)^{\indiceFreeOne} \Bigr ) \Bigr )  \partial_{\spaceVariable}, \qquad \indiceTime \in \nonZeroNaturals, 
    \end{equation*}
    thus
    \begin{equation*}
        \partial_{\timeVariable}\testFunction(0, \spaceVariable)+ \latticeVelocity \equilibriumCoefficient_2 \Bigl ( 1 + \frac{2}{\indiceTime} \Bigl ( 1 - \sum_{\indiceFreeOne = 0}^{\indiceTime - 1}(1 - \relaxationParameter_2)^{\indiceFreeOne} \Bigr ) \Bigr ) \partial_{\spaceVariable} \testFunction(0, \spaceVariable)
        = \bigO{\spaceStep}, \qquad \indiceTime \in \nonZeroNaturals, \quad \spaceVariable \in \reals.
    \end{equation*}

    \item{\strong{Lax-Wendroff} \eqref{eq:LawWendroff}}. 
    The computation is similar to the previous ones, taking into account that the only terms to change are $\termAtOrder{\initialisationOperatorAsymptotic_2}{1} = (1 - {\equilibriumCoefficient_2^2})/{(1-\relaxationParameter_2)}  \partial_{\spaceVariable}$ and $\termAtOrder{\initialisationOperatorAsymptotic_1}{2} = (1 - \equilibriumCoefficient_2^2 )/{2} \partial_{\spaceVariable \spaceVariable}$.
    This provides the modified equations, for  $\indiceTime \in \nonZeroNaturals$ and $\spaceVariable \in \reals$
    \begin{align*}
        \partial_{\timeVariable} \testFunction(0, \spaceVariable) &+ \latticeVelocity \equilibriumCoefficient_2 \partial_{\spaceVariable} \testFunction(0, \spaceVariable)\\
        &- \latticeVelocity \spaceStep \Bigl (\frac{1}{2} + \sum_{\indiceFreeOne = 1}^{\indiceTime - 1}\Bigl (1 - \frac{\indiceFreeOne}{\indiceTime}\Bigr )(1-\relaxationParameter_2)^{\indiceFreeOne} + \frac{1}{2\indiceTime} \Bigl ( 1 - 2\sum_{\indiceFreeOne = 0}^{\indiceTime - 1} (1- \relaxationParameter_2)^{\indiceFreeOne} \Bigr ) \Bigl )  ( 1 -  \equilibriumCoefficient_2^2 ) \partial_{\spaceVariable \spaceVariable} \testFunction(0, \spaceVariable) = \bigO{\spaceStep^2}. \nonumber
    \end{align*}

    \item{\strong{Smooth initialisation RE1} \eqref{eq:CoefficientInitialisationSmoothInTime}}. 
    This scheme fulfills \Cref{prop:PreparedInitialisation} and we have for $\indiceTime \in \nonZeroNaturals$ and $ \spaceVariable \in \reals$
    \begin{equation*}
        \partial_{\timeVariable}\testFunction(0, \spaceVariable) + \latticeVelocity \equilibriumCoefficient_2 \partial_{\spaceVariable} \testFunction(0, \spaceVariable) - \latticeVelocity \spaceStep \Bigl ( - \frac{\indiceTime}{2}\equilibriumCoefficient_2^2 \partial_{\spaceVariable\spaceVariable} + \frac{1}{\indiceTime} \Bigl (  \termAtOrder{(\schemeMatrixAsymptotic^{\indiceTime})}{2}_{11} +  \termAtOrder{(\schemeMatrixAsymptotic^{\indiceTime})}{2}_{12} \equilibriumCoefficient_2 + \termAtOrder{(\schemeMatrixAsymptotic^{\indiceTime})}{1}_{12} \termAtOrder{\initialisationOperatorAsymptotic_2}{1} \Bigr ) \Bigr ) \testFunction(0, \spaceVariable) = \bigO{\spaceStep^2},
    \end{equation*}
    where only $\termAtOrder{\initialisationOperatorAsymptotic_2}{1} = - ( 1 - {\equilibriumCoefficient_2^2})/{\relaxationParameter_2} \partial_{\spaceVariable}$ introduces differences compared to the Lax-Friedrichs initialisation \eqref{eq:LaxFriedrichs}.
    We therefore obtain for $\indiceTime \in \nonZeroNaturals$ and $\spaceVariable \in \reals$ 
    \begin{equation*}
        \partial_{\timeVariable} \testFunction(0, \spaceVariable) + \equilibriumCoefficient_2 \partial_{\spaceVariable} \testFunction(0, \spaceVariable) - \latticeVelocity \spaceStep \Bigl ( \frac{1}{2} - \sum_{\indiceFreeOne = 1}^{\indiceTime - 1}\Bigl (1 - \frac{\indiceFreeOne}{\indiceTime}\Bigr )(1-\relaxationParameter_2)^{\indiceFreeOne} + \frac{1}{\indiceTime \relaxationParameter_2} \sum_{\indiceFreeOne = 1}^{\indiceTime}(1 - \relaxationParameter_2)^{\indiceFreeOne} \Bigr )  ( 1 - \equilibriumCoefficient_2^2 )\partial_{\spaceVariable\spaceVariable}  \testFunction(0, \spaceVariable) = \bigO{\spaceStep^2}.
    \end{equation*}
    
\end{itemize}

\section{Conditions to obtain dissipation match for the \scheme{1}{3} scheme of \Cref{sec:D1Q3}}\label{app:ConditionsMatchDissipatinD1Q3}

Here, we present the detailed discussion of the conditions to obtain dissipation match for the \scheme{1}{3} scheme of \Cref{sec:D1Q3}
\begin{itemize}
    \item $\FirstRelParDOneQThree = 1$. Then the equation is trivially satisfied for any choice of $\SecondRelParDOneQThree$. Enforcing the choice of $\FirstRelParDOneQThree = 1$ in the first equation of \eqref{eq:SystemMatchedDiffusionD1Q3} yields $(1 - \SecondRelParDOneQThree) (\initialisationOperator_3 - \kDOneQThree) = 0$.
    This equation is trivially satisfied for $\SecondRelParDOneQThree = 1$. If $\SecondRelParDOneQThree \neq 1$, then we must initialize at equilibrium, that is, consider $\initialisationOperator_3 = \kDOneQThree$.
    \item $\FirstRelParDOneQThree \neq 1$. Then the equation for $\SecondRelParDOneQThree$ reads $( {2}/{3} - {\velocityDOneQThree^2}+ {\kDOneQThree}/{3}  )\SecondRelParDOneQThree = (2-\FirstRelParDOneQThree)  ( {2}/{3} - {\velocityDOneQThree^2} + {\kDOneQThree}/{3}  )$.
    We distinguish two cases
    \begin{itemize}
        \item $\kDOneQThree > -2 + 3\velocityDOneQThree^2$. 
        In this case, we have to enforce
        \begin{equation}
            \SecondRelParDOneQThree = 2 - \FirstRelParDOneQThree.
        \end{equation}
        Using this choice of $\SecondRelParDOneQThree$ into the first equation from \eqref{eq:SystemMatchedDiffusionD1Q3}, we obtain that $\initialisationOperator_3$ has to be taken as
        \begin{equation}\label{eq:ChoiceInitizialisationMagicD1Q3}
            \initialisationOperator_3 = \frac{1}{\FirstRelParDOneQThree}  ( 2 (-2 + 3\velocityDOneQThree^2) + ( \FirstRelParDOneQThree - 2) \kDOneQThree  ).
        \end{equation}

        Remark that in this case, the only way of making the bulk scheme to be of second-order is to take $\FirstRelParDOneQThree = 2$. This results in $\SecondRelParDOneQThree = 0$, which means that one more moment is conserved by the scheme. Still, the equilibria do not depend on it.
        Moreover, the initialisation has to be $\initialisationOperator_3 = -2 + 3\velocityDOneQThree^2 \neq \kDOneQThree$.
        
        \item $\kDOneQThree = -2 + 3\velocityDOneQThree^2$. The equation is trivially true. Considering the first equation in \eqref{eq:SystemMatchedDiffusionD1Q3} once more, we obtain $(1 - \SecondRelParDOneQThree) (\initialisationOperator_3 + 2 - 3\velocityDOneQThree^2) = 0$.
        If $\SecondRelParDOneQThree = 1$, this equation is satisfied regardless of the choice of $\initialisationOperator_3$. If $\SecondRelParDOneQThree \neq 1$, then the initialisation should be $\initialisationOperator_3 = - 2 + 3\velocityDOneQThree^2 = \kDOneQThree$.
    \end{itemize}
\end{itemize}

\section{Stability of the \scheme{1}{3} scheme of \Cref{sec:D1Q3} when $\relaxationParameter_2 + \relaxationParameter_3 = 2$}\label{app:StabilityD1Q3}

We again employ the technique by Miller \cite{miller1971location,strikwerda2004finite}.
We have to control the roots of $\annhilitaingPolyGinzburgFourier(\timeShift, \frequency\spaceStep) = \timeShift^2 + ((\relaxationParameter_2 - 2)(2\cos{(\frequency \spaceStep)} + 1)/3 + (\relaxationParameter_2 - 2) \kDOneQThree (\cos{(\frequency \spaceStep)} -1)/3 + i \relaxationParameter_2 \velocityDOneQThree \sin{(\frequency \spaceStep)}  )\timeShift + (1-\relaxationParameter_2)$, by computing
\begin{equation*}
    \annhilitaingPolyGinzburgFourier^{\star}(\timeShift, \frequency\spaceStep) = (1-\relaxationParameter_2)\timeShift^2 + ((\relaxationParameter_2 - 2)(2\cos{(\frequency \spaceStep)} + 1)/3 + (\relaxationParameter_2 - 2) \kDOneQThree (\cos{(\frequency \spaceStep)} -1)/3 - i \relaxationParameter_2 \velocityDOneQThree \sin{(\frequency \spaceStep)} )\timeShift + 1.
\end{equation*}
\begin{itemize}
    \item Let $\relaxationParameter_2 \in ]0, 2[$. Checking the first condition $|\annhilitaingPolyGinzburgFourier(0, \frequency\spaceStep)| < |\annhilitaingPolyGinzburgFourier^{\star}(0, \frequency\spaceStep)|$ trivially gives  $0 < \relaxationParameter_2 < 2$.
    Then we have
\begin{equation*}
    \annhilitaingPolyGinzburgWithOrderFourier{1}(\timeShift, \frequency\spaceStep) = \relaxationParameter_2 (2-\relaxationParameter_2)  ( \timeShift - (2\cos{(\frequency \spaceStep)} + 1)/3 - \kDOneQThree (\cos{(\frequency \spaceStep)} -1)/3 + i  \velocityDOneQThree \sin{(\frequency \spaceStep)} ).
\end{equation*}

Checking that the unique root of this polynomial has modulus bounded by one comes back at considering $((2\cos{(\frequency \spaceStep)} + 1) + \kDOneQThree (\cos{(\frequency \spaceStep)} -1)   )^2 / 9 +  \velocityDOneQThree^2 \sin^2{(\frequency \spaceStep)} \leq 1$.
Using the trigonometric identities $\cos{(\frequency \spaceStep)} = 1 - 2 \sin^2{(\frequency \spaceStep/2)}$ and $\sin^2{(\frequency \spaceStep)}  = 4 \sin^2{(\frequency \spaceStep/2)}(1-\sin^2{(\frequency \spaceStep/2)} )$ and calling $\mu = \sin^2{(\frequency \spaceStep/2)} \in [0, 1]$, we obtain
\begin{equation*}
    \mu  ( (\kDOneQThree + 2)^2/9 - \velocityDOneQThree^2  ) + ( -(\kDOneQThree + 2)/3 + \velocityDOneQThree^2  ) \leq 0, \qquad \forall \mu \in [0, 1].
\end{equation*}
This is an affine expression on $\mu$, thus the maximum is reached on the boundary of $[0, 1]$.
Assume, without loss of generality that $\velocityDOneQThree > 0$ and the standard CFL condition $\velocityDOneQThree \leq 1$.
\begin{itemize}
    \item $(\kDOneQThree + 2)^2/9 - \velocityDOneQThree^2 \geq 0$, corresponding to 
    \begin{equation*}
        \kDOneQThree \leq -2 - 3\velocityDOneQThree, \qquad \text{or} \qquad \kDOneQThree \geq -2 + 3 \velocityDOneQThree.
    \end{equation*}
    In this case the maximum is reached at $\mu = 1$, thus we want $(\kDOneQThree + 2)(\kDOneQThree - 1) \leq 0$, hence $-2 \leq \kDOneQThree \leq 1$.
    Under the CFL condition $\velocityDOneQThree \leq 1$ (otherwise all the computations can be adapted accordingly but no stability can be deduced), we easily find the first overall condition $-2 + 3 \velocityDOneQThree \leq \kDOneQThree \leq 1$.
    \item  $(\kDOneQThree + 2)^2/9 - \velocityDOneQThree^2 < 0$, corresponding to 
    \begin{equation*}
        -2 - 3 \velocityDOneQThree < \kDOneQThree < -2 + 3\velocityDOneQThree.
    \end{equation*}
    In this case the maximum is reached on $\mu = 0$, providing $-(\kDOneQThree + 2)/3 + \velocityDOneQThree^2 \leq 0$ thus comparing with the other conditions taking the CFL condition into account, we have $-2 + 3\velocityDOneQThree^2 \leq \kDOneQThree \leq -2 + 3 \velocityDOneQThree$.
\end{itemize}
In this case, the necessary and sufficient stability condition  reads  $|\velocityDOneQThree|  \leq 1$ and $-2 + 3\velocityDOneQThree^2 \leq \kDOneQThree\leq 1$.
\item Let $\relaxationParameter_2 = 2$. In this case $\annhilitaingPolyGinzburgWithOrderFourier{1} \equiv 0$, hence we compute
\begin{equation*}
    \frac{\text{d}\annhilitaingPolyGinzburgWithOrderFourier{2} (\timeShift, \frequency \spaceStep)}{\text{d} \timeShift} = 2\timeShift + 2 i \velocityDOneQThree \sin (\frequency \spaceStep),
\end{equation*}
hence to have its roots strictly in the unit circle for any frequency, we have the strict CFL condition $|\velocityDOneQThree| < 1$.
\end{itemize}

\section{Numerical experiments on the unobservable sub-space for \Cref{ex:D1Q2Again}}\label{app:numExpUnobsD1Q2}

\begin{figure} 
    \begin{center}
        \includegraphics{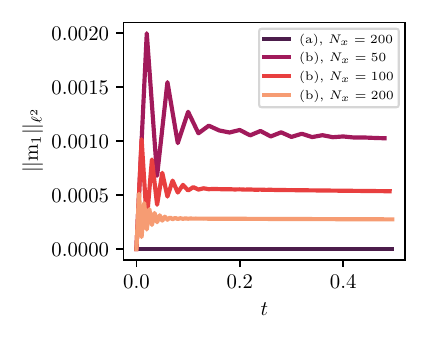}        
    \end{center}\caption{\label{fig:d1q2_unobservable}$L^2$ norm of the conserved moment as function of the time for the \scheme{1}{2} scheme choosing $\latticeVelocity = 1$, $\equilibriumCoefficient_2 = 1/2$, and $\relaxationParameter_2 = 1.8$. The test is performed for different initial data (both observable and unobservable) with different $\spaceStep$.}
  \end{figure}

To validate the finding concerning the unobservable sub-space $\unObservableSubSpace$ given in \Cref{ex:D1Q2Again}, we consider two sets of initial data
\begin{equation*}
    \text{(a)} \qquad \discreteMoment_1(0, \cdot) = 0, \qquad \discreteMoment_2(0, j\spaceStep) = \frac{1 + 3 (-1)^{j}}{8},
\end{equation*}
\begin{equation*}
    \text{(b)} \qquad \discreteMoment_1(0, \cdot) = 0, \qquad \discreteMoment_2(0, j\spaceStep) = 
    \frac{1}{10} \text{exp} \Bigl ( -\frac{1}{1-(4(j\spaceStep - 0.5))^2}\Bigr ).
\end{equation*}
The first datum (a) lies in $\unObservableSubSpace$ whereas the second one (b) does not.
Observe that both data do not adhere to the guidelines to choose initial data according to the analysis in \Cref{sec:ModifiedEquations}: they are uniquely selected for the current test.
We shall take $j \in  \integerIntervalClosedOpen{0}{N_x}$ in the simulations and $\spaceStep = 1/N_x$.
Periodic boundary conditions are enforced.
The results of the simulation given in \Cref{fig:d1q2_unobservable} confirm the theory. 
The unobservable initial datum (a) yields zero conserved (observed) moment for any time step, whereas the observable one (b) does not, even if the conserved moment is initialized as zero everywhere. 
For the observable datum (b), we see that the solution converges linearly to the exact solution of the Cauchy problem, meaning the zero solution.

\section{Numerical experiments on the unobservable sub-space for \Cref{ex:D1Q3Ginzburg}}\label{app:numExpUnobsD1Q3}

  \begin{figure}
      \begin{center}
          \begin{tabular}{cc}
            $\relaxationParameter_2 = 1.8, \qquad \relaxationParameter_3 = 2 - \relaxationParameter_2$ & $\relaxationParameter_2  = 1.8, \qquad \relaxationParameter_3 = 1.2$ \\
            \includegraphics[width=0.49\textwidth]{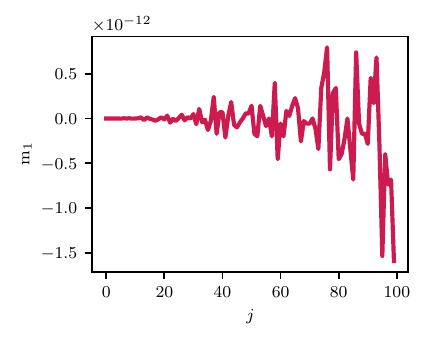} & \includegraphics[width=0.49\textwidth]{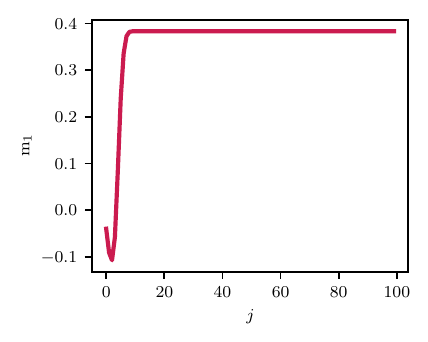}
          \end{tabular}
      \end{center}\caption{\label{fig:d1q3_unobservable_extrapolation}Conserved moment after 10 iterations for the \scheme{1}{3} scheme choosing $\latticeVelocity = 1$, $\equilibriumCoefficient_2 = 1/2$, $\equilibriumCoefficient_3 = 1/10$, $N_x = 100$, and different relaxation parameters.}
  \end{figure}

  To check the findings concerning $\unObservableSubSpace$ for the scheme in \Cref{ex:D1Q3Ginzburg}, we select 
  \begin{equation*}
      \discreteMoment_1 (0, \cdot) = 0, \qquad \discreteMoment_2 (0, j\spaceStep) = j, \qquad \discreteMoment_3 (0, j\spaceStep) = -3 j^2,
  \end{equation*}
  which thus belongs to $\unObservableSubSpace$.
  We discretize with $j \in \integerIntervalClosedOpen{0}{100}$ using an anti-bounce-back boundary condition $\discreteDistributionDensity_2((\indiceTime + 1)\timeStep, 0) = -\discreteDistributionDensity_3^{\collided}(\indiceTime \timeStep, 0)$ on the inflow and a second-order extrapolation $\discreteDistributionDensity_3((\indiceTime + 1)\timeStep, 99\spaceStep) = 3\discreteDistributionDensity_3^{\collided}(\indiceTime \timeStep, 99\spaceStep) - 3\discreteDistributionDensity_3^{\collided}(\indiceTime \timeStep, 98\spaceStep) + 3\discreteDistributionDensity_3^{\collided}(\indiceTime \timeStep, 97\spaceStep)$ on the outflow. 
  These boundary conditions are exact with the initial data, because they are polynomials of degree less or equal than two.
  The result of the simulation is proposed in \Cref{fig:d1q3_unobservable_extrapolation}.
  For the choice where $\relaxationParameter_3 = 2-\relaxationParameter_2$, thus for which the initial datum belongs to $\unObservableSubSpace$, we see that the conserved moment remains zero (up to machine precision).
  When $\relaxationParameter_3 \neq 2-\relaxationParameter_2$, thus the initial datum is observable, we remark that inside the domain, the conserved moment is non-zero (around $0.383$).

\end{document}